\documentclass[a4paper,11pt]{article}
\usepackage{amsmath,amssymb, amsbsy, amsfonts, amsthm}
\usepackage{mathrsfs}
\usepackage{bigints}
\usepackage{color,psfrag}
\usepackage[dvips]{graphicx}
\usepackage{enumerate}
\usepackage{tikz}
\usepackage{url}
\usepackage{etoc}
\usepackage{minitoc}
\usepackage{comment}

\numberwithin{equation}{section}

\newcommand{\R}{{\mathbb R}}
\newcommand{\N}{{\mathbb N}}

\newcommand{\Alk}{A_{\lambda,\kappa}}

\def\endproof{\hfill $\Box$\par\vskip3mm}

\newcommand{\eps}{\varepsilon}

\newcommand{\cotan}{\textnormal{cotan}\,}

\newcommand{\modd}{\,\textnormal{mod}\,}

\renewcommand{\ge }{\geqslant}
\renewcommand{\geq }{\geqslant}
\renewcommand{\le }{\leqslant}
\renewcommand{\leq }{\leqslant}

\def\neweq#1{\begin{equation}\label{#1}}
\def\endeq{\end{equation}}
\def\eq#1{(\ref{#1})}

\newtheorem{theorem}{Theorem}
\newtheorem{proposition}[theorem]{Proposition}
\newtheorem{lemma}[theorem]{Lemma}
\newtheorem{corollary}[theorem]{Corollary}
\newtheorem{definition}[theorem]{Definition}

\title{Linear and nonlinear equations
for beams and degenerate plates \\ with multiple intermediate piers}
\author{Maurizio GARRIONE - Filippo GAZZOLA}
\date{\small Dipartimento di Matematica, Politecnico di Milano \vspace{0.15cm}\\
Piazza Leonardo da Vinci, 32 - 20133 Milano, Italy \vspace{0.1cm}\\
e-mail addresses: maurizio.garrione@polimi.it, filippo.gazzola@polimi.it }

\begin{document}

\maketitle


\begin{abstract}
A full theory for hinged beams and degenerate plates with multiple intermediate piers is developed. The analysis starts with the variational setting and the study of the linear stationary problem in one dimension. Well-posedness results are provided and the possible loss of regularity, due to the presence of the piers, is underlined. A complete spectral theorem is then proved, explicitly determining the eigenvalues on varying of the position of the piers and exhibiting the fundamental modes of oscillation. The obtained eigenfunctions are used to tackle the study of the nonlinear evolution problem in presence of different nonlinearities, focusing on the appearance of \emph{linear} and (a suitable notion of) \emph{nonlinear instability}, with a twofold goal: finding the most reliable nonlinearity describing the oscillations of real structures and determining the position of the piers that maximizes the stability of the structure. The last part of the paper is devoted to the study of degenerate plate models, where the structure is composed by a central beam moving vertically and by a continuum of cross-sections free to rotate around the beam and allowed to deflect from the horizontal equilibrium position. The \emph{torsional instability} of the structure is investigated, taking into account the impact of different nonlinear terms aiming at modeling the action of cables and hangers in a suspension bridge. Again, the optimal position of the piers in terms of stability is discussed. The stability analysis is carried out both by means of analytical tools, such as Floquet theory, and numerical experiments. Several open problems and possible future developments are presented.
\end{abstract}

{\small
\vfill\eject
}

\dosecttoc
\renewcommand{\stifont}{\normalsize}
\renewcommand{\stcSSfont}{\small}
\renewcommand{\stctitle}{}
\tableofcontents

\vfill\eject

\section{Introduction}

Many bridges suffered unexpected oscillations both during construction and after inauguration, sometimes also leading to collapses, see
e.g.\ \cite{akesson,imhof}. Thanks to the videos available on the web \cite{tacoma}, most people have seen the spectacular collapse of the
Tacoma Narrows Bridge (TNB), occurred in 1940: the torsional oscillations were considered the main cause of this dramatic event \cite{ammvkwoo,scott}.
But torsional oscillations leading to failures also appeared in several other bridges. Let us just mention the collapses of the Brighton Chain
Pier (1836), of the Menai Straits Bridge (1839), of the Wheeling Suspension Bridge (1854), of the Matukituki Suspension Footbridge (1977).
We refer to \cite[Chapter I]{bookgaz} for more details and more historical events. Wide oscillations were also seen during the construction of the TNB and during the erection of the Storebaelt East Bridge (the second largest suspension bridge in the world) in 1998: in this occasion, it was necessary to fix some additional anchorages on the bottom of the sea \cite[p.32]{mdemiranda}.
Let us emphasize that, although being generated by a different phenomenon (synchronization of the pedestrians instead of vortex shedding),
also footbridges are prone to display (lateral) oscillations, see again \cite[Chapter I]{bookgaz}.
Here we just recall that these oscillations were seen also the very same day of the opening of the London Millennium Bridge in 2000: the bridge
had to be closed in order to prevent a possible tragedy \cite{strogatz}. This quick survey shows that new solutions are necessary to solve old problems
and that unexpected oscillations still appear nowadays in bridges. These accidents raised some fundamental questions of deep mathematical interest,
such as finding possible ways to improve the stability of suspension bridges.
Needless to say, these phenomena are fairly complicated and to tackle them theoretically requires a wide variety of different skills.\par
Since we constantly refer to the main components of a suspension bridge throughout the paper, let us explain in detail their roles, see
Figure \ref{suspension}.
\begin{figure}[ht]
\begin{center}
\includegraphics[height=50mm, width=155mm]{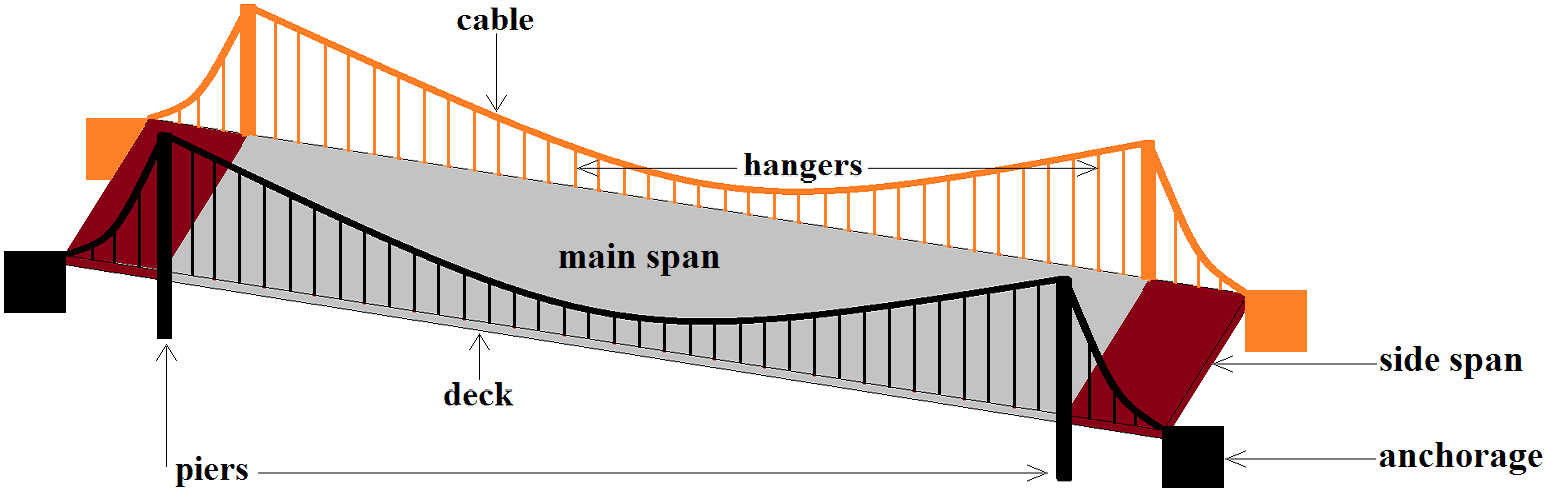}
\caption{Sketch of a suspension bridge.}\label{suspension}
\end{center}
\end{figure}
Four piers (or towers) sustain two parallel cables which, in turn, sustain the hangers. At their lower endpoints, the hangers are linked to the
deck and sustain it from above. A suspension bridge is usually erected starting from the anchorages and the piers. Then the sustaining cables
are installed between the two couples of piers. Once the cables are in position, they furnish a stable working base from which the deck can be
raised from floating barges. The hangers are hooked to the cables and the deck is hooked to the hangers; this deforms the cable and stretches
the hangers, which start their restoring action on the deck. We refer to \cite[Section 15.23]{podolny} for further details. The part of
the deck between the piers is called the main span while the parts outside the piers are called side spans. The relative lengths of the three
spans vary from bridge to bridge, see again \cite{podolny}, and one of our purposes is precisely to establish if there is an optimal position of
the piers in order to maximize stability in a suitable sense. The complex composition of a suspension bridge yields two major difficulties: first
it is aerodynamically quite vulnerable, second it appears very challenging to describe its behavior through simple and reliable mathematical models.\par
In particular, due to the large number of nonlinear interactions between its components, a fully reliable model appears out of reach.
A typical example of this conflict can be found in the paper by Abdel-Ghaffar \cite{abdel}, who makes use of variational principles to
obtain the combined equations of a suspension bridge motion in a fairly general nonlinear form. The effects of coupled vertical-torsional
oscillations, as well as cross-distortional deformations of the stiffening structure, are described by separating them into four different kinds of
displacements: the vertical displacement $v$, the torsional angle $\theta$, the cross section distortional angle $\psi$ and the warping displacement
$u$, which can be expressed in terms of $\theta$ and $\psi$. After making such a huge effort, Abdel-Ghaffar simplifies the problem by neglecting
the cross section deformation, the shear deformation and the rotatory inertia; he obtains a coupled nonlinear vertical-torsional system of two
equations in the two unknowns $v$ and $\theta$. These equations are finally linearized by neglecting terms which are considered small
and the coupling effect disappears, see \cite[(34)-(35)]{abdel}. This procedure is quite common in the
engineering literature: first an attempt to describe the full bridge model, then simplifications in order to deal with tractable equations.
In this paper we follow the opposite procedure: first we highlight the relevant phenomena in simple models, then we dress the models in order
to make them as close as possible to a suspension bridge. We choose to model the deck of the bridge as a degenerate plate consisting of a beam
with a continuum of cross sections free to rotate around the beam and we introduce a variety of possible nonlinearities. Due to the elastic behavior of several
components in a bridge (in particular, the sustaining cables), one should also expect {\em nonlocal nonlinearities}. We study in detail the behavior of
several nonlinearities in order to determine the ones giving responses in line with the phenomena visible in actual bridges.\par
A suspension bridge requires a fairly involved modeling and functional setting and we need first to analyze a hinged beam divided in three
spans separated by two internal piers. We start from the very beginning with the analysis of linear stationary beam
equations, going through the spectral analysis, studying the instabilities that arise in nonlinear evolution beam equations and, finally,
ending with a possible simplified suspension bridge model that consists of a degenerate plate with intermediate piers. All these steps require
a sound variational setting based on classical principles of functional analysis and basic theory of linear differential equations. The stationary beam equation with piers fits in the framework of the so-called {\em multi-point} problems for ODE's, that were introduced
in the pioneering work by Wilder \cite{wilder}, see also \cite{Loc73} for some developments in the subsequent decades. As far as we are aware,
no evolution problem has been studied in this setting and we insert within this framework also the degenerate plate for which the multi-points
become ``multi-cross-sections''.\par
The variational formulation of the linear stationary beam equation requires the introduction of new functional spaces. In Theorem \ref{VI}, we
show that they are subspaces of codimension two of the Sobolev space $H^2\cap H^1_0$ and we fully characterize them. The two missing
dimensions prevent a nice regularity theory for weak solutions: in Theorem \ref{regular} we show why, in general, one cannot expect to have more
than $C^2$-smoothness, the reason being that the piers yield impulses on the beam, to be added to most forcing terms. In Corollary \ref{coroll}
we provide a simple way to recognize the cases where the weak solution is also a classical solution.\par
Our linear analysis then tackles the spectral properties of the fourth order differential operator. Also for the related eigenvalue problem we
cannot expect regularity of the eigenfunctions and standard forms of Sturm-Liouville-type results fail. Therefore, we proceed ``by hand''.
A full description
of the spectrum is extremely complicated, see Theorem \ref{autovalorias}; since most suspension bridges have equal side spans, we restrict our
attention to the symmetric case with equal side spans, for which simpler and elegant results can be proved. By taking advantage of symmetries,
in Theorem \ref{symmetriceigen} we determine explicitly all the eigenvalues and the associated eigenfunctions. As expected, they strongly depend on the position of the piers, as described in Theorems \ref{Michelle} and \ref{constant}, which also characterize the placement and the number of zeros of
the eigenfunctions.
As noticed in the Federal Report \cite{ammvkwoo}, this is extremely important for the stability of bridges, see also the reproduction in Figure
\ref{zeroTNB} where a qualitative inventory of the oscillation modes seen at the TNB is drawn. The pattern of ``zeros moving on the spans'' is
governed by several rules which can be described through beautiful pictures, see Figures \ref{bello} and \ref{figurateo}.\par
The linear analysis, in particular the behavior of the eigenvalues as the position of the piers varies, enables us to tackle the stability issue
for some nonlinear evolution beam equations, for which we aim at determining the ``best position'' of the piers within the beam in order to maximize its
stability. Motivated by both physical and mathematical arguments, we perform our analysis on {\em finite-dimensional} approximations of the
phase space. We focus our attention on two different kinds of stability, the classical {\em linear stability} whose analysis is based on some
Hill equation (see Definition \ref{defstabb}) and a more ``practical'' stability that we introduced in \cite{GarGaz} and that appears suitable
for applications and nonlinear problems (see Definition \ref{unstable}). In the course, we will show an existing connection between these two
instabilities, the former being a clue for the latter.
We restrict the instability analysis to a particular class of solutions of the nonlinear beam equation, the ones possessing a \emph{prevailing mode}, i.e., for which the initial energy of the system is almost completely concentrated on a unique Fourier component, see Definition \ref{prevalente}. These solutions well describe
the behavior of actual bridges, see \cite[p.20]{ammvkwoo} and Section \ref{length}. One is then
interested in studying whether, for positive time, the energy remains mostly on the very same component or if it transfers towards other components.
In other words, our stability analysis considers beams which initially oscillate close to a prevailing mode and wonders whether this remains true
in some interval of time. If not, we say that the prevailing mode is unstable and one is then interested in finding the {\em least energy} (or
least amplitude of oscillation) which makes it unstable. The energy threshold for the stability of the beam is then the least of these energies among
all the possible prevailing modes. This means that, below this energy, the beam is stable and we expect the lower modes to have more chances to achieve this bound. Clearly, the energy threshold depends on the relative length of the side spans and we are finally interested in finding
the placement of the piers that maximizes it. As a by-product of our analysis, in Section \ref{sezioneconfronti} we identify the nonlinearities that make
the beam behave more similarly to real bridges.\par
Finally, we complete the ``model dressing'' and we focus our attention on the degenerate plate with a central beam and with cross sections free to rotate
around it. By exploiting the results obtained for the simple beam we make a precise choice of the nonlinearities to be inserted into the model.
We analyze the impact of the two nonlinear restoring forces acting on the deck of a bridge, those exerted by the cables and the hangers (Sections \ref{rigidhangers} and \ref{slack}).
A precise description of the stability is much more involved when the model is fully dressed but we are still able to retrieve most of the phenomena
visible for simple beams. The main targets are to study the impact of the hangers and of their elasticity on the stability thresholds, and to determine once
again the optimal position of the piers that now do not constrain simple points but full segments (the cross sections between the piers).
Some parts of our stability analysis are out of reach with purely theoretical tools and we use numerics. Also engineers
are nowadays aware that many expensive experiments in wind tunnels may be replaced by numerics, see e.g.\ the monograph \cite{jurado}, in particular
its preface. And from \cite[p.13]{jurado} we also pick the motivation of their work: {\em we are not trying to substitute a designer with these optimization
techniques, which would be impossible because of the complexity of real problems, but rather intending to help a designer not to fall into false steps that
can be very probable for a design with great complexity}. A further goal of this paper is precisely to introduce all the tools and nonlinearities
needed for a systematic study of the stability in multi-point problems, in particular those modeling suspension bridges.\par
Summarizing, we provide a full variational characterization of PDEs describing the stability of structures with internal piers, rigorously defining the proper functional setting. As for the considered nonlinearities, we give strong hints about their impact on the stability analysis. We believe that
all the necessary instruments to reach a complete understanding of the stability of suspension bridges are here introduced; our qualitative analysis should be seen as the starting point for a precise quantitative study of more complete models, taking into account the action of aerodynamic forces. We conclude the paper (Section \ref{finalsect}) with a summary of the main results and their consequences, with some comments and some open problems, and
with possible future perspectives and developments.

\section{The physical models}\label{physmod}

We first describe the model for a multiply hinged beam divided in three adjacent spans (segments): the main (middle) span and two side spans separated by piers.
Without loss of generality, we normalize the total length to $2\pi$ and represent the beam as in Figure \ref{figbeam}.
\newline
\begin{figure}[!ht]
\begin{center}
\begin{tikzpicture}[xscale=0.9,yscale=0.9]
\draw (0,0) -- (8,0);
\draw (0, -0.2) -- (0, 0.2);
\draw (2, -0.2) -- (2, 0.2);
\draw (7, -0.2) -- (7, 0.2);
\draw (8, -0.2) -- (8, 0.2);
\node at (-0.2, -0.8){$-\pi$};
\node at (2, -0.8){$-b\pi$};
\node at (6.9, -0.8){$a\pi$};
\node at (8.1, -0.8){$\pi$};
\node at (4.4, 1.2){$\textnormal{main span}$};
\end{tikzpicture}
\qquad\qquad\qquad\includegraphics[height=25mm, width=45mm]{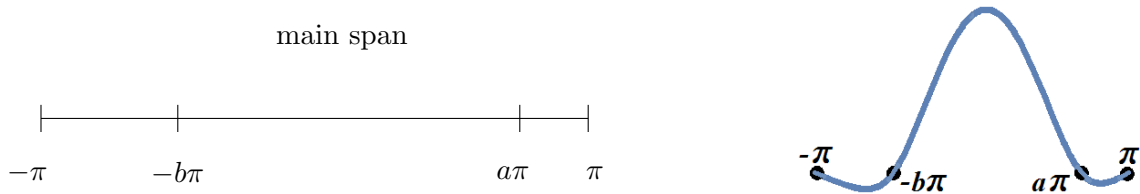}
\caption{A beam with two piers and $0 < a, b<1$.}\label{figbeam}
\end{center}
\end{figure}
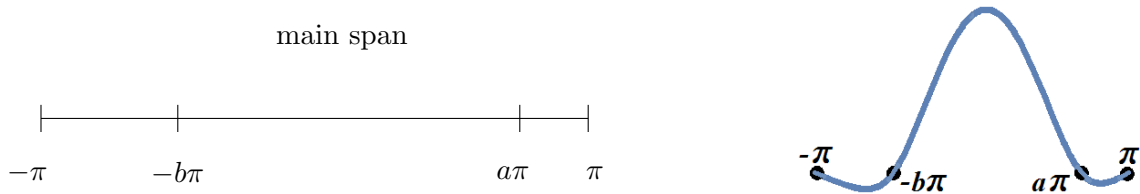
\newline
The parameters $0< a,b<1$ determine the relative measure of the side spans with respect to the main span. The beam is hinged at the extremal points $\pm\pi$ and,
in correspondence of the positions of the piers, at the points $-b\pi$ and $a\pi$. If $\gamma>0$ denotes the elastic restoring parameter, the linear beam equation reads
\begin{equation}\label{beamlineare}
u_{tt} + u_{xxxx} + \gamma u = 0 \qquad x \in I=(-\pi, \pi), \quad t > 0,
\end{equation}
where $u$ represents the vertical displacement;
we will specify in Section \ref{nonlinevol} in which sense \eqref{beamlineare} is satisfied. We complement \eq{beamlineare} with the boundary and internal conditions
\begin{equation}\label{bc}
u(-\pi,t)=u(\pi,t)=u(-b\pi,t)=u(a\pi,t)=0 \qquad t \geq 0,
\end{equation}
where the last two conditions state that no displacement occurs in correspondence of the piers $\{-b\pi,a\pi\}$ and give rise to a multi-point problem for
\eq{beamlineare}. In the dynamics of the beam modeled by
\eq{beamlineare}-\eqref{bc}, the two internal conditions in the piers create an interaction between the three spans, otherwise the three
segments would behave independently of each other. Due to the extreme complexity of the asymmetric case ($b\neq a$), at some point we restrict
our attention to the case of symmetric beams ($b=a$), also motivated by the fact that most suspension bridges have equal side spans with
\neweq{physicalrange}
\frac12 \le a\le\frac{2}{3},
\endeq
see \cite[Table 15.10]{podolny}.\par
Equation \eq{beamlineare} may be solved by separation of variables but it is too far away from the final target of the present paper, that is,
modeling a bridge with a quantitative analysis of its instability with respect to the position of the piers. Two main ingredients are missing
in \eq{beamlineare}: some nonlinearity which is intrinsic in complicated elastic structures such as bridges and the possibility of displaying
torsional oscillations. We will reach a reliable bridge model in three steps: first we analyze the behavior of beams, then we study the response of
different nonlinearities in order to detect the most reliable, finally we turn to the analysis of a nonlinear degenerate plate.\par
As for the nonlinear term to be introduced in \eqref{beamlineare}, a cubic nonlinearity naturally arises when large deflections of a beam or a plate are involved: in this case, the stretching effects suggest to
use variants of the von K\'arm\'an theory \cite{karman}, see also \cite{karmanbook,ciarletvk} for a modern point of view and \cite{gazwang} for the
adaptation of this theory to plates modeling bridges. In fact, when dealing with bridges, the nonlinearity should as well take into account the
behavior of the sustaining cables and, for this reason, also the engineering literature deals with cubic nonlinearities, see e.g.\
\cite{sepe1,bartoli,plautdavis}.\par
Cubic nonlinearities within the equation arise when quartic energy terms are inserted into the model. We add to the bending energy the three different nonlocal energies
$$
\mathcal{E}_1(u)= \frac12 \int_Iu_{xx}^2+\frac14\left(\int_Iu_{xx}^2\right)^2,\quad
\mathcal{E}_2(u)= \frac12 \int_Iu_{xx}^2+\frac14\left(\int_Iu_x^2\right)^2,\quad \mathcal{E}_3(u)=\frac12
\int_Iu_{xx}^2+\frac14\left(\int_Iu^2\right)^2,
$$
and the local energy
$$
\mathcal{E}_4(u)=\frac12 \int_Iu_{xx}^2+\frac14\int_Iu^4.
$$
By inserting the kinetics, one finds that the total energy is $\mathcal{E}(u)=\frac12\int_Iu_t^2+\mathcal{E}_i(u)$ ($i=1,2,3,4$), so
that the resulting evolution equations read, respectively:
$$
\begin{array}{ll}
u_{tt}+\big(1+\|u_{xx}\|^2_{L^2}\big) u_{xxxx}=0,\quad & u_{tt}+u_{xxxx}-\|u_x\|_{L^2}^2 u_{xx}=0, \\
u_{tt}+u_{xxxx}+\|u\|_{L^2}^2 u=0,\quad & u_{tt}+u_{xxxx}+u^3=0,
\end{array}
$$
for $x \in I$ and $t >0$. These equations are here written in strong form but, as we shall see, there is not enough regularity for this formulation: we instead stick to their weak form. This will be the object of Section \ref{evolutionbeam}. In all the energies $\mathcal{E}_i$
($i=1,2,3,4$), the first (quadratic) term represents the bending energy while the quartic term has different physical meanings that we now discuss. \par
There are several reasons for considering $\mathcal{E}_1$. First, it may be derived from a variational principle of a stiffened beam with
bending energy behaving superquadratically and nonlocally: this means that if the beam is bent somewhere, then this increases the resistance to further bending in all the other
points. Moreover, as explained by Cazenave-Weissler \cite[p.110]{cazwei}, we choose
this energy in the hope that studying the behavior of the solution of the equation could provide an indication of what to expect from
more complicated energies.\par
In 1950, Woinowsky-Krieger \cite{woinowsky} modified the classical beam models by Bernoulli and
Euler assuming a nonlinear dependence of the axial strain on the deformation gradient, by taking into account the stretching of the beam due to its
elongation. Independently, Burgreen \cite{burg} obtained the very same nonlinear beam equation. This leads to the energy $\mathcal{E}_2$ from which the corresponding
Euler-Lagrange equation can also be derived from a variational principle, now describing a stiffened beam with stretching energy behaving
superquadratically and nonlocally: if the beam is stretched somewhere, then this increases the resistance to
further stretching in all the other points (we refer to \cite{ggh} for a damped and forced version). In spite of this analogy, we will see in Section \ref{nonlinevol} that the stretching energy {\em goes through the
piers}, thereby behaving differently from the bending energy. More similar to $\mathcal{E}_1$ is the energy $\mathcal{E}_3$ which
describes a stiffened beam where the displacement behaves superquadratically and nonlocally: if the beam is displaced from its equilibrium
position in some point, then this increases the resistance to further displacements in all the other points. The corresponding equation
is the fourth order version of the wave equation considered in \cite{cazwei,cazw}.\par
The energy $\mathcal{E}_4$ is more standard and represents a stiffened beam where the displacement behaves superquadratically but locally.
The beam increases pointwise its resistance to displacements from the equilibrium position: in every point the resistance to displacement
depends only on the displacement itself in the same point. The corresponding equation will be studied in Section \ref{nonlinevol}. After a careful study of their behavior, we select the most reliable of these four nonlinearities: by ``reliable'' we mean here ``well-describing the phenomena visible
in real bridges''.\par
The video \cite{barbara} shows that there is {\em transmission of displacements} between spans. But
if one digits ``tacoma narrows bridge collapse images'' on Google, one finds pictures of the wide oscillations prior to the TNB collapse and
one sees that, while a torsional motion is visible on the main span, the side spans do not display torsional displacements, see the
qualitative reproduction in Figure \ref{tac1}.
\begin{figure}[!h]
\centering
\includegraphics[scale=0.45]{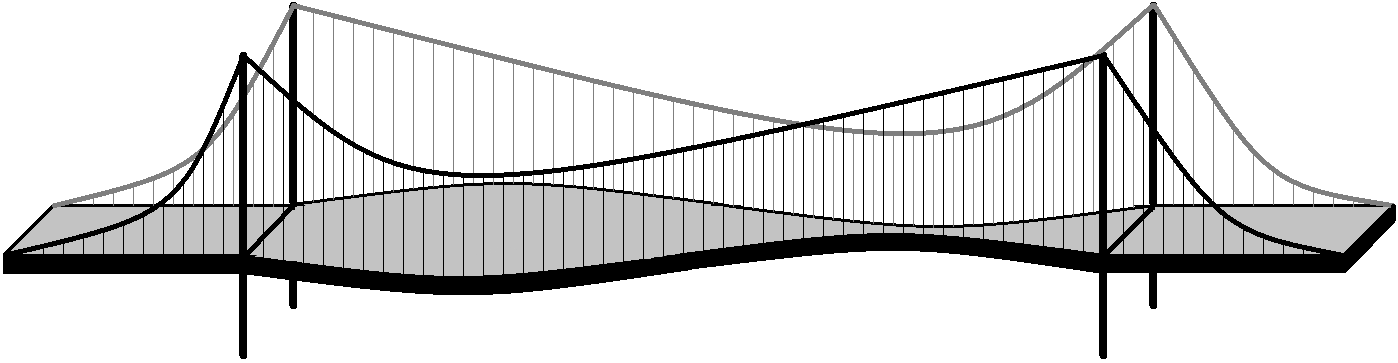}
\caption{Qualitative behavior of the oscillations at the TNB the day of the collapse.}
\label{tac1}
\end{figure}
From a mathematical point of view, this means that
$$
\mbox{\bf the matching between the displacements on the three spans is in general not smooth.}
$$
This fact is confirmed by a careful look at the video \cite{tacoma} which clearly shows that, during the oscillations, the connection between the main span and the
side spans {\em is not} $C^1$. In fact, Figure \ref{tac1}, the pictures on the web, and the video also show that
\begin{center}
{\bf within the deck, only the displacement of the midline is smooth during the oscillations}.
\end{center}

This is why among several possible ways of modeling the deck of a bridge with piers, we choose to see it as a degenerate plate, composed by a beam representing the midline of the plate and by
cross sections that are free to rotate around the beam, see Figure \ref{ffishbone}.
\begin{figure}[!ht]
\makebox[\linewidth][c]{
\begin{tikzpicture}[xscale=0.6,yscale=0.6]
\draw (0,0) -- (2, 0);

\draw[very thick](3.5, 0.3) -- (3.5, -2.3);
\draw[very thick](10.5, 0.3) -- (10.5, -2.3);
\draw[very thick](-0.1, -2.3) -- (-0.1, 0.3);
\draw  (3.5, -2) -- (0, -2);
\draw (0, -2) -- (0, 0);

\draw[fill=gray!30!white]  (0,0) -- (3.4,0) -- (3.4, -2) -- (0, -2) -- cycle;
\draw[fill=gray!30!white]  (3.5,0) -- (10.5,0) -- (10.5, -2) -- (3.5, -2) -- cycle;
\draw[fill=gray!30!white]  (10.6,0) -- (14,0) -- (14, -2) -- (10.6, -2) -- cycle;
\draw[very thick, color=white] (0, -1) -- (14, -1);

\draw[thin, color=black] (0.5, 0)--(0.5, -2);
\draw[thin, color=black] (1, 0)--(1, -2);
\draw[thin, color=black] (1.5, 0)--(1.5, -2);
\draw[thin, color=black] (2, 0)--(2, -2);
\draw[thin, color=black] (2.5, 0)--(2.5, -2);
\draw[thin, color=black] (3, 0)--(3, -2);
\draw[thin, color=black] (3.5, 0)--(3.5, -2);
\draw[thin, color=black] (4, 0)--(4, -2);
\draw[thin, color=black] (4.5, 0)--(4.5, -2);
\draw[thin, color=black] (5, 0)--(5, -2);
\draw[thin, color=black] (5.5, 0)--(5.5, -2);
\draw[thin, color=black] (6, 0)--(6, -2);
\draw[thin, color=black] (6.5, 0)--(6.5, -2);
\draw[thin, color=black] (7, 0)--(7, -2);
\draw[thin, color=black] (7.5, 0)--(7.5, -2);
\draw[thin, color=black] (8, 0)--(8, -2);
\draw[thin, color=black] (8.5, 0)--(8.5, -2);
\draw[thin, color=black] (9, 0)--(9, -2);
\draw[thin, color=black] (9.5, 0)--(9.5, -2);
\draw[thin, color=black] (10, 0)--(10, -2);
\draw[thin, color=black] (10.5, 0)--(10.5, -2);
\draw[thin, color=black] (11, 0)--(11, -2);
\draw[thin, color=black] (11.5, 0)--(11.5, -2);
\draw[thin, color=black] (12, 0)--(12, -2);
\draw[thin, color=black] (12.5, 0)--(12.5, -2);
\draw[thin, color=black] (13, 0)--(13, -2);
\draw[thin, color=black] (13.5, 0)--(13.5, -2);
\draw[thin, color=black] (14, 0)--(14, -2);

\draw[fill=black]  (3.4,0.3) -- (3.5,0.3) -- (3.5, -2.3) -- (3.4, -2.3) -- cycle;
\draw[fill=black]  (10.49,0.3) -- (10.6,0.3) -- (10.6, -2.3) -- (10.49, -2.3) -- cycle;
\draw[fill=black]  (-0.1,0.3) -- (-0.1,-2.3) -- (0, -2.3) -- (0, 0.3) -- cycle;
\draw[fill=black]  (14,0.3) -- (14.1,0.3) -- (14.1, -2.3) -- (14, -2.3) -- cycle;
\draw[fill=black]  (14.1,0.3) -- (14.3,0.3) -- (14.3, 0) -- (14.1, 0) -- cycle;
\draw[fill=black]  (14.1,-2) -- (14.3,-2) -- (14.3, -2.3) -- (14.1, -2.3) -- cycle;

\draw[fill=black]  (-0.1,0.3) -- (-0.3,0.3) -- (-0.3, 0) -- (-0.1, 0) -- cycle;
\draw[fill=black]  (-0.1,-2.3) -- (-0.3,-2.3) -- (-0.3, -2) -- (-0.1, -2) -- cycle;

\draw[very thin] (0, 0) -- (-0.8, 0);
\draw[very thin] (0, -2) -- (-0.8, -2);
\node at (-1.6, 0.1){$\ell$};
\node at (-1.75, -2.1){$-\ell$};
\node at (-0.15, -2.9){$-\pi$};
\node at (3.4, -2.9){$-a\pi$};
\node at (10.5, -2.9){$a\pi$};
\node at (14.15, -2.9){$\pi$};

\filldraw [black] (3.45, 0.2) circle (5pt);
\filldraw [black] (3.45, -2.2) circle (5pt);
\filldraw [black] (10.555, 0.2) circle (5pt);
\filldraw [black] (10.555, -2.2) circle (5pt);

\end{tikzpicture}}
\caption{A fish-bone model for a bridge with piers.}\label{ffishbone}
\end{figure}
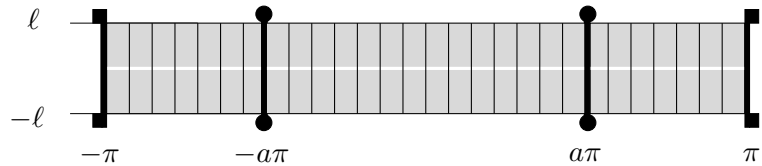
This guarantees both a smooth midline displacement as in \cite{barbara}
and nonsmooth connections between spans. Not only this model appears particularly appropriate to describe the behavior of a real bridge, but it also enables us to
exploit the analysis performed for the stability in beams. We limit ourselves to consider the case of symmetric side spans, namely $a=b$; setting
$$
I=(-\pi, \pi), \quad I_-=(-\pi, -a\pi), \quad I_0=(-a\pi, a\pi), \quad I_+=(a\pi, \pi),
$$
the plate is identified with the planar rectangle
$$
\Omega=I\times(-\ell,\ell)\subset\R^2,
$$
while the three spans are identified with
$$
\Omega_0:=I_0\times(-\ell,\ell)\quad\mbox{(main span)}, \qquad \Omega_-=I_-\times(-\ell,\ell), \quad \Omega_+=I_+\times(-\ell,\ell)\quad\mbox{(side spans).}
$$
The deck in actual bridges has a rectangular shape with two long edges of the order of $1$km and two shorter edges of the order of $14$m.
Therefore, if we assume that the deck $\Omega$ has width $2\ell$ and length $2\pi$, after scaling this means that $\ell\approx\pi/75$.
The white midline in Figure \ref{ffishbone} divides the roadway into two lanes and its vertical displacement is denoted by $u=u(x,t)$, for $x\in I$ and $t>0$. The
equilibrium position of the midline is $u=0$, with
\begin{equation}\label{spostagi첫}
u>0 \textrm{ corresponding to a } \bf{downwards} \textrm{ displacement}.
\end{equation}
Each cross-section is free to rotate
around the midline and its angle of rotation is denoted by $\alpha=\alpha(x,t)$. This model was called a {\em fish-bone} in \cite{bergaz}.\par
The vertical displacements of the two endpoints of the cross sections (in position $x$ and at time $t$) are given by
\neweq{sinalpha}
u(x,t)+\ell\sin\alpha(x,t)\qquad\mbox{and}\qquad u(x,t)-\ell\sin\alpha(x,t).
\endeq
Since we are not interested in describing accurately the behavior of the plate under large torsional angles, for small $\alpha$ the
following approximations are legitimate:
\neweq{trigo}
\cos\alpha\cong1\quad\text{and}\quad\sin\alpha\cong\alpha.
\endeq
If we set $\theta=\ell\alpha$, this cancels the dependence on the width $\ell$, see \cite{bergaz} for the details.
In view of \eq{trigo}, the displacements \eq{sinalpha} now read $u(x,t)\pm\theta(x,t)$.
In a suspension bridge, these displacements generate forces due to the restoring action of the cables+hangers system at the endpoints of the
cross sections; a crucial step is to find a reasonable form for these forces. We denote them by $f(u\pm\theta)$ and the corresponding potentials by
$F(u\pm\theta)$, so that $F'=f$.\par
We derive the Euler-Lagrange equations for this structure using variational methods, as a consequence of an energy balance. Denoting by $M>0$ the mass density of the beam, the kinetic energy of the beam is given by
the well-known law:
$$\frac{M}{2}\int_I u_t^2.$$
Since the plate is assumed to be homogeneous, the mass density of each cross section is also equal to $M>0$. On the other hand, the kinetic energy of a rotating object is
$\frac12 J\alpha_t^2$, where $J$ is the moment of inertia and $\alpha_t$ is the angular velocity. The moment of inertia of a rod of
length $2\ell$ about the perpendicular axis through its center is given by $\frac{M\ell^2}{3}$. Hence, the kinetic energy of the rod having
half-length $\ell$, rotating around its center with angular velocity $\alpha_t$, is given by
$$\frac{M}{6}\ell^2\int_I\alpha_t^2=\frac{M}{6}\int_I\theta_t^2.$$

Moreover, there exists a constant $\mu>0$, depending on the shear modulus and on the moment of inertia of the pure torsion, such that the total potential energy of the
cross sections is given by
$$\frac{\mu\ell^2}{2}\int_I\alpha_x^2=\frac{\mu}{2}\int_I\theta_x^2.$$
The bending energy of the beam depends on its curvature: if $EI>0$ is the flexural rigidity of the beam, it is given by
$$\frac{EI}{2}\int_I u_{xx}^2.$$

Finally, we derive the most delicate energy terms which create the {\em coupling} between the longitudinal displacement $u$ and the
torsional angle $\theta$. Both the contributions of the cables and the hangers have to be taken into account.\par
$\bullet$ {\bf Restoring force due to the cables.} As pointed out in the engineering literature (see, e.g., \cite{bartoli,luco}),
in suspension bridges the most relevant source of nonlinearity, and hence the main contribution to instability, comes from the
sustaining cables. The displacements of the endpoints $u\pm\theta$ of the cross sections of the deck stretch the cables which increase nonlinearly their tension
and have a nonlocal restoring effect distributed throughout themselves and, in turn, on the deck. The tension at rest of the cables is proportional to the
span$/$sag ratio. If a single beam is sustained by a cable having tension at rest $\gamma>0$ and if the beam has a displacement $w$, then the term
$\gamma\|w\|^2_{L^2}$ is a good measure for the geometric nonlinearity of the beam due to its displacement: as noticed in \cite{haraux}, the corresponding
equation is a Hamiltonian system that combines certain properties of the linear beam equation and of the central force motion. The potential $G=G(w)$ reads
\neweq{potential}
G(w)=\frac{\gamma}{4}\left(\int_Iw^2\right)^2.
\endeq
We use here this principle to study the behavior of the edges $I\times\{\pm\ell\}$ of the cross sections of the plate $\Omega=I\times(-\ell,\ell)$. Therefore,
the potential \eq{potential} is applied at both the endpoints of the cross sections. This means that $w$ should be replaced by both $u\pm\theta$, giving rise
to the potentials
$$
G(u+\theta)=\frac{\gamma}{4}\left(\int_I(u+\theta)^2\right)^2, \qquad G(u-\theta)=\frac{\gamma}{4}\left(\int_I(u-\theta)^2\right)^2.
$$

$\bullet$ {\bf Restoring force due to the hangers.} The hangers have the delicate role to connect the cables and the deck, see Figure \ref{suspension}. In some cases,
they may be considered inextensible, see \cite{luco}, but their main feature is that they can lose tension
if the deck and the cables are too close to each other, a phenomenon called {\em slackening}, that was observed during the TNB collapse,
see \cite[V-12]{ammvkwoo}.\par
Several different forms for the restoring force due to the hangers have been suggested in literature, see \cite{benforgaz} for a quick survey.
At this stage, we simply denote the restoring force by $f$, postponing its explicit form to Section \ref{slack}. The potentials $F=F(u \pm \theta)$ depend
only on the {\em local} unknowns $u,\theta$ and not on their global behavior.\par
Thanks to all the above terms, we finally obtain the total energy associated with the fish-bone model:
\begin{eqnarray*}
E(u, \theta) & = & \frac{M}{2} \int_I u_t^2+ \frac{M}{6} \int_I \theta_t^2 +\frac{\mu}{2} \int_I \theta_x^2 + \frac{EI}{2} \int_I u_{xx}^2
+\frac{\gamma}{4}\left(\int_I(u+\theta)^2\right)^2+\frac{\gamma}{4}\left(\int_I(u-\theta)^2\right)^2\, \\ & & +  \int_{I} F(u+\theta)+\int_I F(u-\theta).
\end{eqnarray*}
This energy then leads to the following {\em nonlocal} system, for $x\in I$ and $t>0$:
\begin{equation}\label{system0}
\left\{
\begin{array}
[c]{l}%
Mu_{tt}+EIu_{xxxx}+2\gamma \Big( \int_I(u^2+\theta^2)\Big)u+4\gamma \Big(\int_I u\theta\Big)\theta+f(u+\theta)+f(u-\theta)=0\\
\frac{M}{3}\theta_{tt}-\mu\theta_{xx}+4\gamma \Big(\int_I u\theta \Big) u+2\gamma\Big(\int_I(u^2+\theta^2)\Big)\theta+f(u+\theta)-f(u-\theta)=0.
\end{array}
\right.
\end{equation}
If $\gamma=0$ and $f$ is linear, then system \eq{system0} decouples into two linear equations. The definition of {\em weak solution} of \eq{system0}
will be given in Section \ref{weakformTI}.\par
In a slightly different setting, involving mixed space-time fourth order derivatives, a linear version of (\ref{system0}) with coupling terms
was suggested by Pittel-Yakubovich \cite{pittel}, see also \cite[p.458, Chapter VI]{yakubovich}. A nonlinear $f$ was considered in (\ref{system0}) by
Holubov\'{a}-Matas \cite{holubova}, who were able to prove well-posedness for a forced-damped version of (\ref{system0}). Also in
\cite{bergaz} the well-posedness of an initial-boundary value problem for (\ref{system0}) is proved for a wide class of nonlinear forces $f$.
The fish-bone model described by (\ref{system0}), with {\em nonlinear} $f$, is able to display a possible transition between vertical and
torsional oscillations within the main span: the former are described by $u$ whereas the latter are described by $\theta$.\par
To \eq{system0} we associate the boundary-internal-initial conditions
\begin{equation}\label{boundary}
u(-\pi,t)=u(\pi,t)=\theta(-\pi,t)=\theta(\pi,t)=0\qquad t \geq 0,
\end{equation}
\begin{equation}\label{jc1}
u(-a\pi,t)=u(a\pi,t)=\theta(-a\pi,t)=\theta(a\pi,t)=0\qquad t \geq 0,
\end{equation}
\begin{equation}\label{initial2}
u(x,0)=u_0(x),\quad u_t(x,0)=u_1(x),\quad\theta(x,0)=\theta_0(x),\quad\theta_t(x,0)=\theta_1(x) \qquad x \in I.
\end{equation}

\section{Functional setting: linear stationary beam equations}\label{functional}

In this section we consider a hinged beam of length $2\pi$, represented by the segment $I=(-\pi,\pi)$, which has two intermediate piers in correspondence of the points $-b\pi$ and $a\pi$, $0 < a, b < 1$. We set
$$
I_-=(-\pi, -b\pi),\qquad I_0=(-b\pi, a\pi),\qquad I_+=(a\pi,\pi),
$$
so that $\overline{I}=\overline{I}_- \cup \overline{I}_0 \cup \overline{I}_+$.
The aim is to provide a functional framework in which problem \eqref{beamlineare}-\eqref{bc} can be settled and to determine the vibrating modes of this kind of beam.

\subsection{Weak stationary solutions and regularity}\label{3.1}

In order to study both the functional setting and the regularity of the solutions,
fairly instructive appears the analysis of the stationary forced version of \eq{beamlineare}.
To this end, we introduce the space
\begin{equation}\label{V}
V(I):=\{u\in H^2\cap H^1_0(I);\, u(-b\pi)=u(a\pi)=0\}\, ;
\end{equation}
notice that the boundary and internal conditions
\neweq{bcstat}
u(-\pi)=u(\pi)=u(-b\pi)=u(a\pi)=0\,
\endeq
make sense since $V(I)$ embeds into $C^0(\overline{I})$.
We characterize $V(I)$ through the following statement; in Section \ref{pfVI} we provide its proof which, though elementary, we were not able to find in literature.

\begin{theorem}\label{VI}
The space $V(I)$ defined in \eqref{V} is a subspace of $H^2\cap H^1_0(I)$ having codimension $2$, whose orthogonal complement is given by
$$
V(I)^\perp=\{v \in C^2(\overline{I}) \mid v(\pm\pi)=v''(\pm\pi)=0, \, v''\textrm{ is piecewise affine on }\overline{I}_-,\overline{I}_0\textrm{ and }\overline{I}_+\}.
$$
Therefore, $V(I)^\perp \cap C^3(I)=\{0\}$.
A basis $\{v_1, v_2\}$ of $V(I)^\perp$ is given by
\begin{equation}\label{v1}
v_1(x)=\left\{
\begin{array}{ll}
-(b+1)\big(x^3+3\pi x^2 + \pi^2 (b^2 + 2 b) x + \pi^3 (b^2 + 2 b - 2)\big)\quad  & x \in [-\pi,-b\pi] \\
(1 - b)\big(x^3 -3\pi x^2 + \pi^2 (b^2 - 2 b)x - \pi^3 (b^2 - 2 b - 2)\big)\quad & x \in [-b\pi,\pi]
\end{array}
\right.
\end{equation}
\begin{equation}\label{v2}
v_2(x)=\left\{
\begin{array}{ll}
(a-1)\big(x^3+3\pi x^2 + \pi^2 (a^2 - 2 a)x + \pi^3 (a^2 - 2 a - 2)\big)\quad & x \in [-\pi,a\pi] \\
(1 + a)\big(x^3-3\pi x^2 + \pi^2 (a^2 + 2 a) x - \pi^3 (a^2 + 2 a - 2)\big)\quad & x \in [a\pi,\pi].
\end{array}
\right.
\end{equation}
\end{theorem}
Notice that $V(I)^\perp$ is made by functions that are more regular than $H^2(I)$; they are $C^2(\overline{I})$, but they fail to be $C^3$ (except for the zero function) since each pier
produces a discontinuity in the third derivative. This effect is well seen by looking into expressions \eqref{v1} and \eqref{v2}. Notice that the basis $\{v_1(x), v_2(x)\}$ is not orthogonal, but it has the advantage of highlighting separately the two singularities at the piers. We will complement Theorem \ref{VI} by determining an explicit orthogonal
basis of $V(I)$ in Theorem \ref{Michelle}: all the elements of this basis will be of class $C^2(\overline{I})$, some of them may even be analytic. In Figure \ref{plots}
we represent the graphs of $v_1$ and $v_2$, as well as their second derivatives, in the case $b=3/4$ and $a=1/4$.
The plots for other values of $a$ and $b$, possibly different, are qualitatively similar.

\begin{figure}[ht]
\begin{center}
\includegraphics[height=25mm, width=40mm]{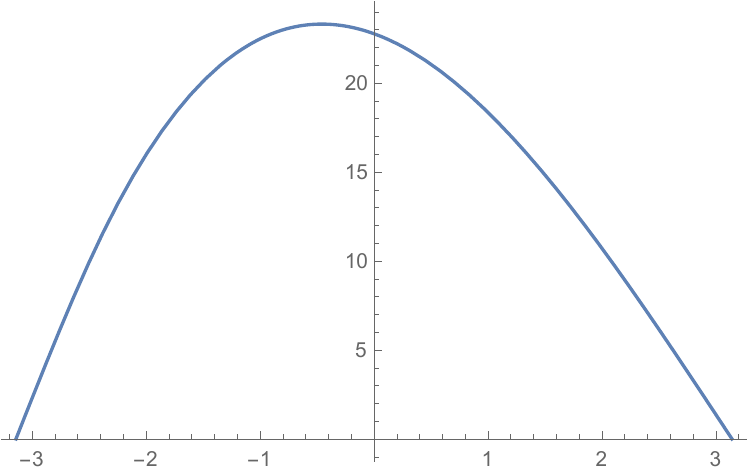}\quad\includegraphics[height=25mm, width=40mm]{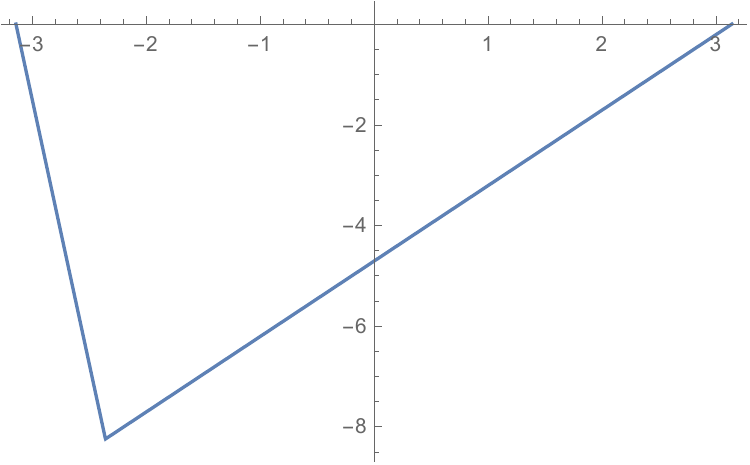}\quad
\includegraphics[height=25mm, width=40mm]{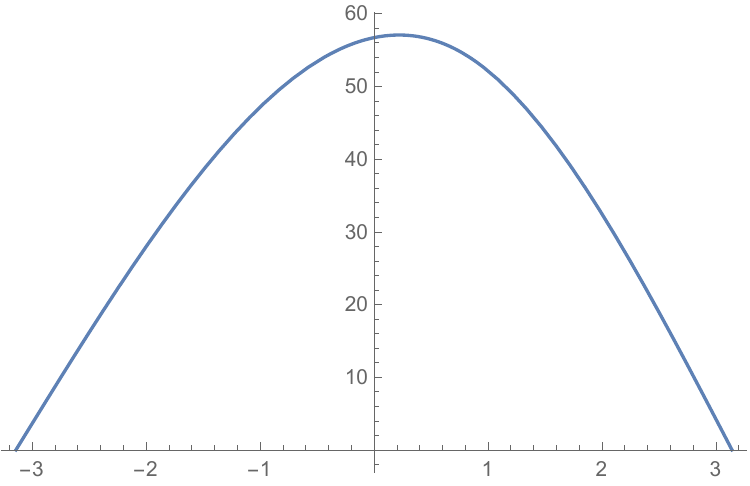}\quad\includegraphics[height=25mm, width=40mm]{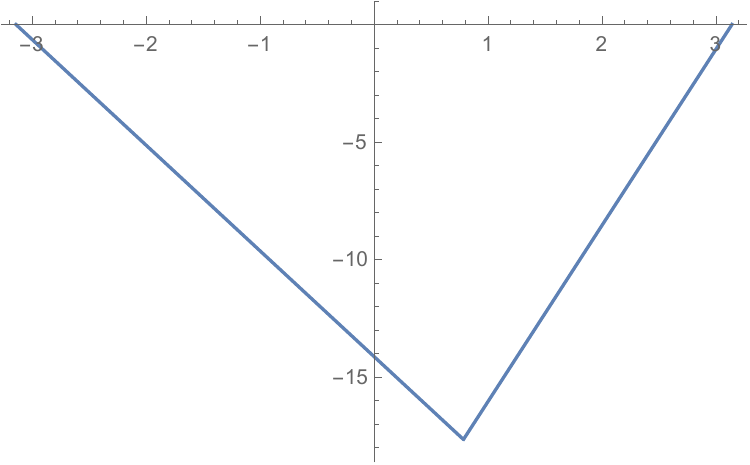}
\caption{From left to right: plots of the functions $v_1$, $v_1''$, $v_2$, $v_2''$ in the case $b=3/4$ and $a=1/4$.}\label{plots}
\end{center}
\end{figure}

We incidentally observe that the orthogonal complement of $H^2_0(I)$ in $H^2 \cap H^1_0(I)$ is spanned by the two functions $\pi^2-x^2$, $x(\pi^2-x^2)$, whose second derivatives are the limits in $L^2(I)$, for $a, b \to 1$, of the second derivatives of two suitable functions in $V(I)^\perp$.
On the other hand, when $a,b\to0$ the space $V(I)$ ``converges'' to the limit space
\neweq{Vstar}
V_*(I):=\{u\in H^2\cap H^1_0(I);\, u(0)=u'(0)=0\},
\endeq
whose orthogonal complement $V_*(I)^\perp$ is spanned by functions $w_1$ and $w_2$ having second derivatives
$$
w_1''(x)=\left\{\begin{array}{ll}
x+\pi & \mbox{if }-\pi\le x\le0\\
\pi-x & \mbox{if }\ 0\le x\le\pi,
\end{array}\right.\qquad\qquad
w_2''(x)=\left\{\begin{array}{ll}
x+\pi & \mbox{if }-\pi\le x<0\\
x-\pi & \mbox{if }\ 0< x\le\pi.
\end{array}\right.
$$
Note that $w_1''$ is continuous, while $w_2''$ is not.
These functions are obtained as limits in $L^2(I)$ of $v_1''+v_2''$ and $v_1''-v_2''$, as defined in \eq{v1} and \eq{v2}, when $a,b\to0$.
It is worth emphasizing that, contrary to $V(I)^\perp$, the space $V_*(I)^\perp$ contains functions which {\em are not} $C^2$.\par
We now pass to study the forced version of \eqref{beamlineare}.
First, we recall that if there are no piers, the equation reads as
\neweq{statforte}
u''''(x)+\gamma u(x)=f(x)\qquad x\in I
\endeq
and the natural functional space where solutions of \eq{statforte} have to be sought
is $H^2\cap H^1_0(I)$, endowed with the scalar product $(u, v) \mapsto \int_I u'' v''$. The notion of weak solution is derived from a variational principle: the total energy $E(u)$ of the beam in position $u$ is the
sum of the bending energy, the restoring energy, and the forcing energy. \par
If the beam does have intermediate piers, the energy is defined on the functional space $V(I)$, that is,
\begin{equation}\label{energiaVI}
E(u)=\frac12 \int_I\Big((u'')^2+\gamma u^2\Big)-\langle f,u\rangle_V\qquad\forall u\in V(I),
\end{equation}
where $\langle\cdot,\cdot\rangle_V$ denotes the duality pairing between $V(I)$ and $V'(I)$, its dual space.
If $f\in L^1(I)$, then the crochet may be replaced by the integral $\int_I fu$.
By computing
the Fr\'echet derivative of $E$ in $V(I)$, we obtain the following definition.
\begin{definition}\label{defweakstat}
Let $f\in V'(I)$. We say that $u\in V(I)$ is a weak solution of \eqref{statforte}-\eqref{bcstat} if
\neweq{weakstat}
\int_I u''v''+\gamma\int_Iuv=\langle f,v\rangle_V\qquad\forall v\in V(I).
\endeq
\end{definition}

In the sequel, we denote by $\delta_{-b\pi},\delta_{a\pi}\in V'(I)$ the Dirac delta distributions at the points $-b\pi$ and $a\pi$.
In the next statement, which will be proved in Section \ref{pfregular}, we discuss the regularity of weak solutions.

\begin{theorem}\label{regular}
Let $\gamma \geq 0$. For all $f\in V'(I)$ there exists a unique weak solution $u\in V(I)$ of \eqref{statforte}-\eqref{bcstat}, according to Definition \ref{defweakstat}.
Moreover, if $f\in C^0(\overline{I})$, then:\par
$(i)$ the solution satisfies $u\in C^4(\overline{I}_-) \cap C^4(\overline{I}_0) \cap C^4(\overline{I}_+) \cap C^2(\overline{I})$ and $u''(\pm\pi)=0$;\par
$(ii)$ there exist $\alpha_f, \beta_f\in\R$ (depending on $f$, $\gamma$, $a$, $b$) such that $u''''+\gamma u=f+\alpha_f\delta_{a\pi}+\beta_f\delta_{-b\pi}$
in distributional sense;\par
$(iii)$ there exists a subspace $X(I)\subset C^0(\overline{I})$ of codimension $2$ such that $u\in C^4(\overline{I})$ if and only if $f\in X(I)$;\par
$(iv)$ we have that $u\in C^4(\overline{I})$ if and only if $u\in C^3(\overline{I})$.
\end{theorem}

From Item $(i)$ we see that a weak solution of \eq{weakstat} makes the beam globally hinged. This means that
its least energy configuration is $C^2$ at the piers and displays no bending at the endpoints. This may lead to nonsmooth
solutions of \eq{weakstat}: indeed, Item $(ii)$ says that

\begin{center}
{\bf if the two-piers beam is subject to a continuous force $f$,}\\
{\bf then each pier yields an additional load equal to some impulse depending on $f$.}
\end{center}
The procedure that we introduce in the proof of Item $(ii)$ explains why, in most cases, the regularity of the solution cannot be improved to
$C^4(\overline{I})$. In particular, it entails the following statement.

\begin{corollary}\label{coroll}
Let $\gamma\ge0$ and $f\in C^0(\overline{I})$. Then the weak solution of \eqref{statforte}-\eqref{bcstat} belongs to $C^4(\overline{I})$ if and only if
it coincides with the unique classical solution $U_f$ of the problem
\neweq{wrong3}
U_f''''(x)+\gamma U_f(x)=f(x)\qquad x \in I,\qquad U_f(-\pi)=U_f(-b\pi)=U_f(a\pi)=U_f(\pi)=0.
\endeq
\end{corollary}

\subsection{Examples and further remarks}

We illustrate here the procedure mentioned in the previous section, applied to two simple and instructive examples helping to better understand the statement of Theorem \ref{regular}.\par
$\bullet$ Take $\gamma=0$ and consider \eq{statforte} with constant load, that is,
\neweq{stationary}
u''''(x)=24\qquad x \in I_- \cup I_0 \cup I_+ ,
\endeq
whose solutions are fourth order polynomials.
If \eq{stationary} is complemented with the boundary and internal conditions \eq{bcstat}, then the function $U_f$ defined by \eq{wrong3}
is given by
$$U_f(x)=(x+b\pi)(x-a\pi)(x^2-\pi^2)$$
which \emph{does not satisfy} \eq{weakstat}, regardless of the values of $a,b < 1$. To see this, it suffices to notice that
$$U_f''(-\pi)=2\pi^2(5-3b+3a-ab)>0,\quad U_f''(\pi)=2\pi^2(5+3b-3a-ab)>0\qquad\forall a,b<1,$$
violating the no-bending boundary behavior in Item $(i)$ of Theorem \ref{regular}. In Figure \ref{falsetrue} we plot the graphs of the functions
$U_f$ and of the weak solution $\overline{u}$ of \eq{stationary} when $b=0.5$ and $a=0.7$: their difference is quite evident.
\begin{figure}[ht]
\begin{center}
\includegraphics[height=40mm, width=60mm]{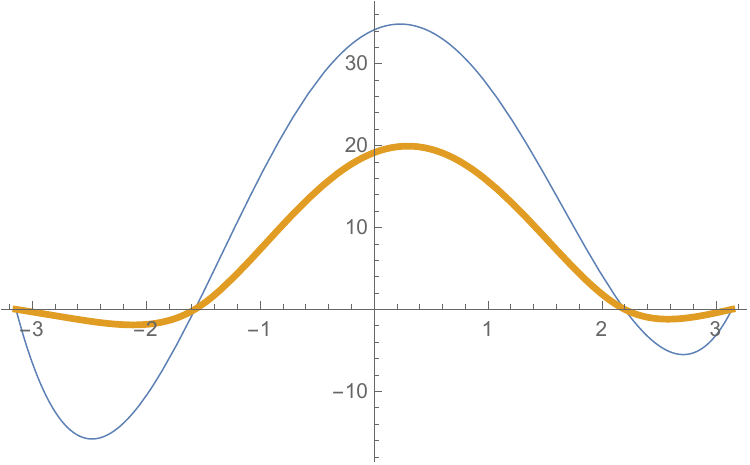}
\caption{Plots of $U_f$ (thin line) and of the weak solution $\overline{u}$ of \eq{stationary} (thick line) for $b=0.5$, $a=0.7$.}\label{falsetrue}
\end{center}
\end{figure}

The function $\overline{u}$, as characterized by Definition \ref{defweakstat}, may be obtained through the procedure explained
in the proof of Theorem \ref{regular}. This solution minimizes the energy defined in \eqref{energiaVI}: let us define
\begin{equation}\label{defacca}
H(a,b):=E(U_f)-E(\overline{u})\qquad\forall(a,b)\in(0,1)^2.
\end{equation}
Then $H(a,b)>0$ in the square $(0,1)^2$ and the graph of the map $(a,b)\mapsto H(a,b)$ is plotted in the left picture in Figure \ref{H}.
\begin{figure}[ht]
\begin{center}
\includegraphics[height=40mm, width=60mm]{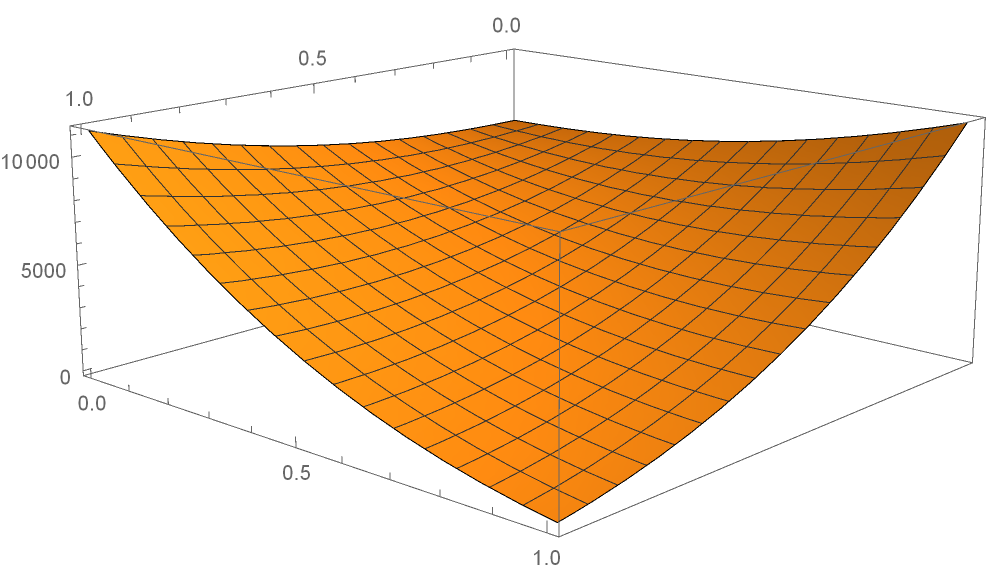}\qquad\includegraphics[height=40mm, width=60mm]{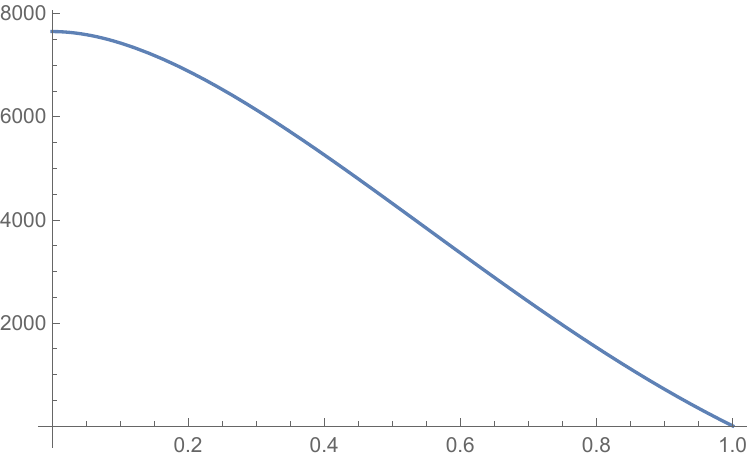}
\caption{Plots of $(a,b)\mapsto H(a,b)$ on $(0,1)^2$ (left) and of $a\mapsto H(a,a)$ on $(0,1)$ (right).}\label{H}
\end{center}
\end{figure}
It turns out that $0=H(1,1)<H(a,b)<H(1,0)=H(0,1)$ for all $(a,b)\in(0,1)^2$. Moreover, the graph is obviously symmetric with respect to the line
$a=b$: in the right picture of Figure \ref{H} we plot the graph of the map $a\mapsto H(a,a)=\frac{(5-a^2)^2(1+2a-3a^2)\pi^5}{2a+1}$ for $a\in(0,1)$,
which corresponds to beams with symmetric side spans. Since $\overline{u}\neq U_f$ for all $a,b$, it never occurs that $(\alpha_f, \beta_f)=(0,0)$
and therefore the constant load $f(x)=24$ does not belong to the subspace $X(I)$ introduced in Item $(iii)$ of Theorem \ref{regular}. Clearly, the coefficients $\alpha_f$ and $\beta_f$
depend on $a$ and $b$;
in Figure \ref{delta24} we plot the graphs of the
map $(a,b)\mapsto\beta_f(a,b)$ for $(a,b)\in(0,1)^2$ and of the map $a\mapsto\beta_f(a,a)=\frac{3(1+a)(a^2-5)\pi}{(1-a)(2a+1)}$ for $a\in(0,1)$.
\begin{figure}[ht]
\begin{center}
\includegraphics[height=40mm, width=60mm]{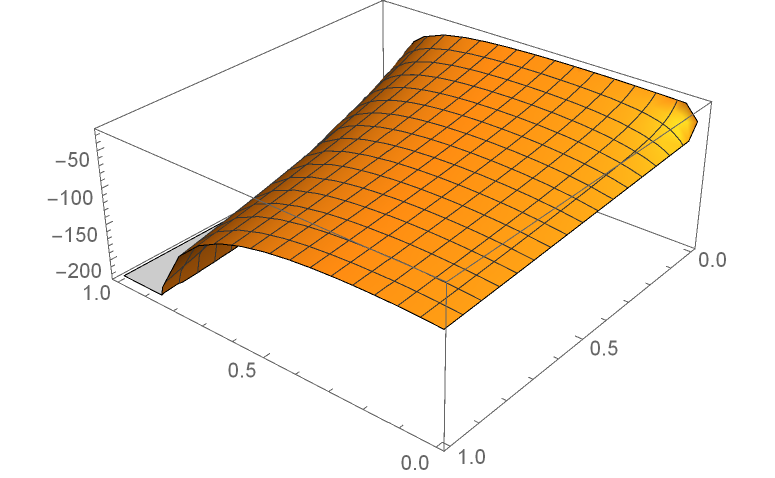}\qquad\includegraphics[height=40mm, width=60mm]{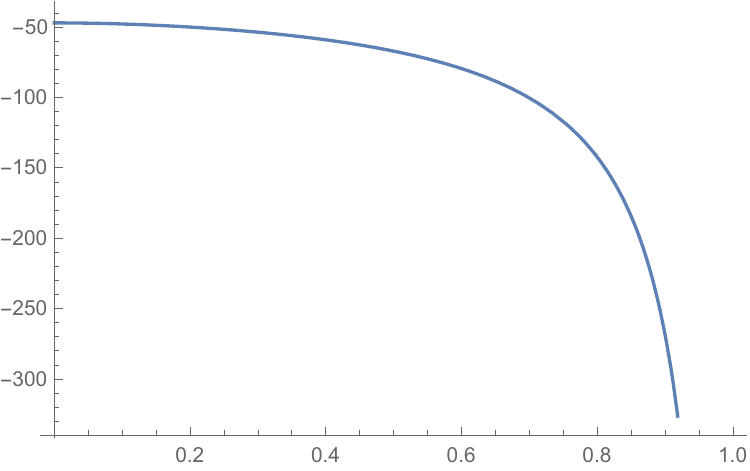}
\caption{Plots of $(a,b)\mapsto\beta_f(a,b)$ on $(0,1)^2$ (left) and of $a\mapsto\beta_f(a,a)$ on $(0,1)$ (right).}\label{delta24}
\end{center}
\end{figure}
Both these functions diverge to $-\infty$ as $a\to1$, meaning that
\begin{center}{\bf the contribution of the piers, in terms of impulses depending on the force $f$,\\
increases and tends to infinity as the piers approach the endpoints of the beam}.
\end{center}

$\bullet$ Take again $\gamma=0$ and consider now \eq{statforte} with a trigonometric load, that is,
\neweq{stationarytrigo}
u''''(x)=m^4\, \sin(mx)\qquad x \in I_- \cup I_0 \cup I_+ \qquad\mbox{for some integer }m=2,3,4...\ .
\endeq
Then,
$$
U_f(x)=\sin(mx)+\left[\frac{\sin(bm\pi)}{1-b^2}+\frac{\sin(am\pi)}{1-a^2}\right]\frac{x(x^2-\pi^2)}{(a+b)\pi^3}
+\left[\frac{b\sin(am\pi)}{1-a^2}-\frac{a\sin(bm\pi)}{1-b^2}\right]\frac{x^2-\pi^2}{(a+b)\pi^2}.
$$
According to Definition \ref{defweakstat} and Corollary \ref{coroll}, this function is a weak solution of \eq{stationarytrigo} if and only
if $U_f''(-\pi)=U_f''(\pi)=0$. After some computations, one sees that this occurs if and only if $\sin(bm\pi)=\sin(am\pi)=0$, that is, $b=h/m$ and $a=k/m$
for any couple of integers $h,k \in \{1,...,m-1\}$. In this case, the weak solution of \eq{stationarytrigo} is $\overline{u}(x)=U_f(x)=\sin(mx)$, which is of class
$C^4(\overline{I})$; in fact, it is analytic in $I$. This example tells us that, for some continuous forces $f$,
\begin{center}
{\bf the regularity of the solution depends on the relative lengths of the three spans.}
\end{center}

Consider again the weak solution $\overline{u}$ of \eq{stationarytrigo}, as characterized by Definition \ref{defweakstat}, and the map \eqref{defacca}.
For $m=3$, in the left picture of Figure \ref{H2} we plot the graph of the map $(a,b)\mapsto H(a,b)$ in the square $(0,1)^2$, while
in the right picture therein we plot the graph of the map $a\mapsto H(a,a)=\frac{3(3+a)\sin^2(3a\pi)}{(1-a)a^2(1+a)^2\pi^3}$ for $a\in(0,1)$.
The zeros of these functions are quite visible.
\begin{figure}[ht]
\begin{center}
\includegraphics[height=40mm, width=60mm]{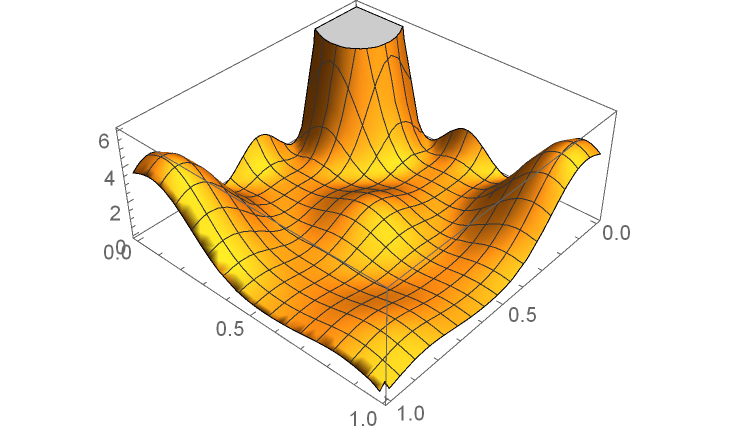}\qquad\includegraphics[height=40mm, width=60mm]{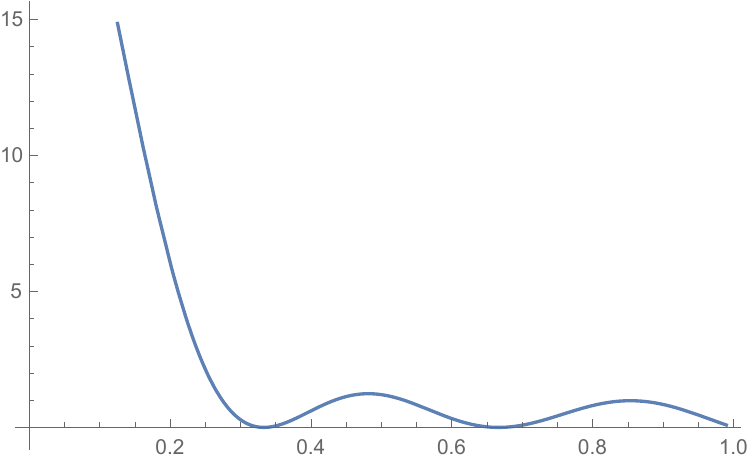}
\caption{Plots of $(a,b)\mapsto H(a,b)$ on $(0,1)^2$ (left) and of $a\mapsto H(a,a)$ on $(0,1)$ (right).}\label{H2}
\end{center}
\end{figure}

As noticed above, $\overline{u}=U_f$ if and only if $a, b\in\{\frac13 ,\frac23 \}$: for these four couples, we have $(\alpha_f, \beta_f)=(0,0)$
and the load $f(x)=81\sin(3x)$ belongs to the subspace $X(I)$, see Item $(iii)$ of Theorem \ref{regular}. By emphasizing the dependence $\alpha_f=\alpha_f(a,b)$
and $\beta_f=\beta_f(a,b)$ as above, in Figure \ref{deltasin} we plot the graphs of the map $(a,b)\mapsto\beta_f(a,b)$ for $(a,b)\in(0,1)^2$ and of the
map $a\mapsto\beta_f(a,a)=\frac{3\sin(3a\pi)}{(1-a)^2a^2\pi^3}$ for $a\in(0,1)$.
We observe that $\beta_f(a, a)$ changes its sign in correspondence of $a=1/3, a=2/3$: in these cases, $f(\pm a\pi)=0$. If instead $f > 0$  (resp., $f < 0$) in one pier, then the impulse due to that pier points downwards (resp., upwards), showing that
\begin{center}
{\bf
the piers ``resist'' to the action of the load $f$.
}
\end{center}
\begin{figure}[ht]
\begin{center}
\includegraphics[height=40mm, width=60mm]{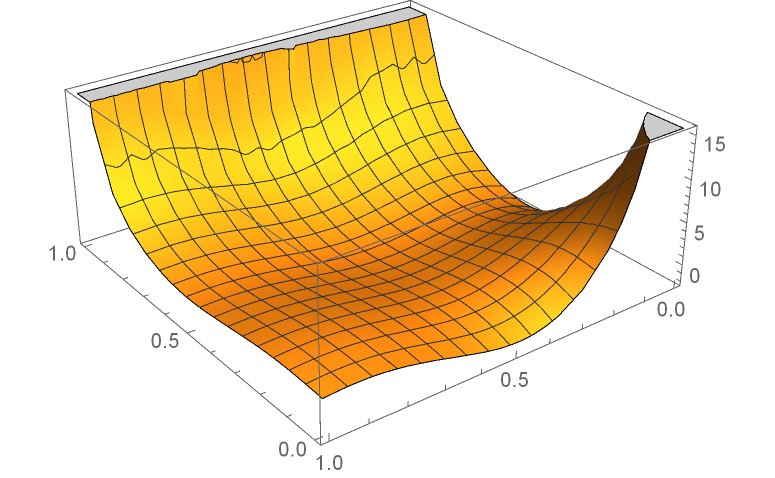}\qquad\includegraphics[height=40mm, width=60mm]{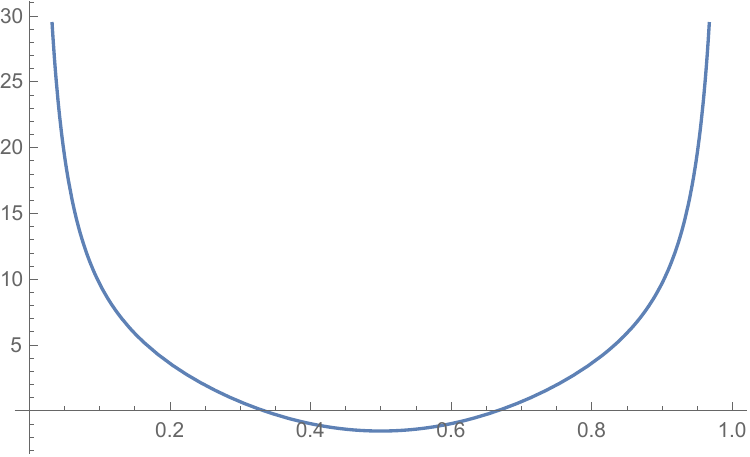}
\caption{Plots of $(a,b)\mapsto\beta_f(a,b)$ on $(0,1)^2$ (left) and of $a\mapsto\beta_f(a,a)$ on $(0,1)$ (right).}\label{deltasin}
\end{center}
\end{figure}

As we have just seen, for some continuous forces $f$ the solution is not globally $C^4$ regardless of $a$ and $b$, for some other $f$ the solution may be globally
$C^4$ only for particular values of $a$ and $b$. Moreover, it may happen that $\beta_f = 0$ and $\alpha_f \neq 0$ (or vice versa), in which case
$u \in C^4(\overline{I_- \cup I_0}) \cap C^4(\overline{I}_+)$  (resp.\, $u \in C^4(\overline{I}_-) \cap C^4(\overline{I_0 \cup I_+})$).\par
$\bullet$
We conclude this section with a simple remark and a related curious example. It is clear that the piers constraint prevents the positivity preserving property
to be true. This means that if $\gamma=0$ and $f\ge0$, one cannot expect the weak solution of \eq{weakstat} to satisfy $\overline{u}\ge0$ as in beams with no piers.
It is however of physical interest to investigate which conditions on $f$ yield a nonnegative solution $\overline{u}$. As an example, take $a=b=1/2$ and consider the
functions
$$f(x)=\tfrac{177}{16}\, \cos(\tfrac{x}{2})\cos^2(x)-17\, \cos(x)\sin(x)\sin(\tfrac{x}{2})
-11\, \cos(\tfrac{x}{2})\sin^2(x)\quad\mbox{and}\quad \overline{u}(x)=\cos^2(x)\cos(\tfrac{x}{2}),$$
whose plots are reported in Figure \ref{fu}.
\begin{figure}[ht]
\begin{center}
\includegraphics[height=40mm, width=60mm]{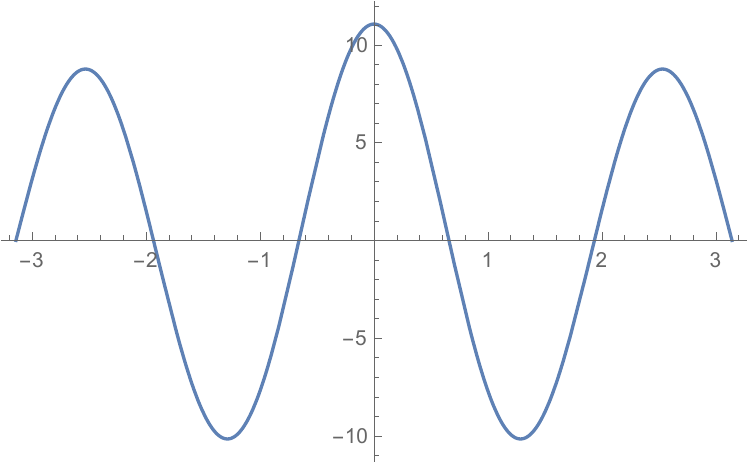}\qquad\qquad\includegraphics[height=40mm, width=60mm]{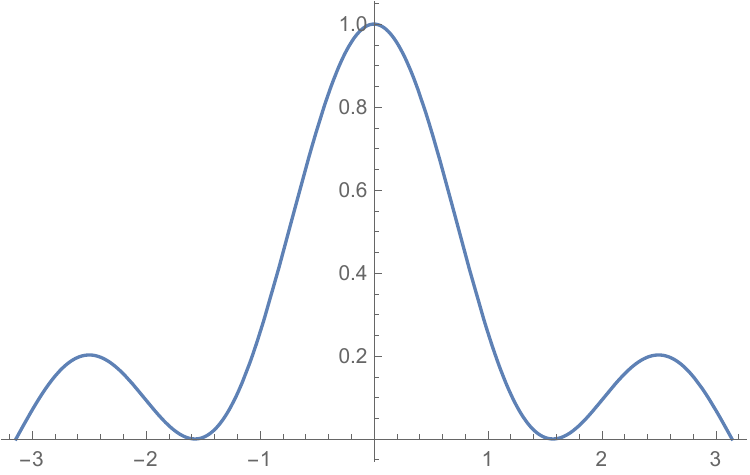}
\caption{Plots of the functions $f$ (left) and $\overline{u}$ (right).}\label{fu}
\end{center}
\end{figure}

One can check that $\overline{u}''''(x)=f(x)$ for all $x\in(-\pi,\pi)$ and that $\overline{u}(\pm\pi)=\overline{u}''(\pm\pi)=\overline{u}(\pm\tfrac\pi2 )=0$: hence, by Corollary \ref{coroll},
$\overline{u}$ is a classical and positive solution of \eq{weakstat}. Clearly, one can change the signs of both $f$ and $\overline{u}$ and stick to the convention \eqref{spostagi첫}. The left picture in Figure \ref{fu} tells us that
\begin{center}
{\bf
the upwards displacement $\overline{u}$ of the beam
with piers is obtained with a sign-changing load $f$ pushing downwards close to the piers and upwards far away from the piers.
}
\end{center}

\subsection{Vibrating modes of a beam with two piers}\label{finitod}

In this section, we determine the modes of vibration for the beam $I=(-\pi, \pi)$ with two intermediate piers. We find the set of the eigenvalues $\mu$
and the corresponding eigenfunctions $e \in V(I)$ solving the problem
\begin{equation}\label{autov0}
\int_I e'' v'' =\mu\int_I e v \qquad \forall v\in V(I).
\end{equation}
Through a regularity argument similar to that in the proof of Theorem \ref{regular}-(i), one can show that any eigenfunction belongs to $C^2(\overline{I})$ and is of class $C^\infty$ on each span $I_-, I_0, I_+$.
Moreover, it necessarily satisfies the no-bending boundary conditions
$$
e''(-\pi) = e''(\pi) = 0
$$
at the endpoints of the beam. If $e \in V(I)$ is an eigenfunction, by taking $v=e$ in \eqref{autov0} one obtains that
\emph{all the eigenvalues are strictly positive}. We thus set $\mu=\lambda^4$ and seek $e\in V(I)$ such that
\begin{equation}\label{autovsym}
\int_I e'' v'' = \lambda^4 \int_I e v \qquad \forall v \in V(I).
\end{equation}

The following result holds; we will provide its proof in Section \ref{pfautovalorias}.

\begin{theorem}\label{autovalorias}
The eigenvalues $\mu$ of problem \eqref{autov0} are given by $\mu=\lambda^4$, for any $\lambda>0$ solving the following equation:
$$
\begin{array}{l}
\sin(2\lambda\pi)\sinh\big[\lambda(1-b)\pi\big]\sinh\big[\lambda(1-a)\pi\big]\sinh^2\big[\lambda(a+b)\pi\big]\\
-\sin\big[\lambda(1-a)\pi\big]\sin\big[\lambda(1+a)\pi\big]\sinh\big[\lambda(1-b)\pi\big]\sinh\big[\lambda(1+b)\pi\big]\sinh\big[\lambda(a+b)\pi\big]\\
+\sin\big[\lambda(1-b)\pi\big]\sin\big[\lambda(1-a)\pi\big]\sin\big[\lambda(a+b)\pi\big]\sinh\big[\lambda(1+b)\pi\big]\sinh\big[\lambda(1+a)\pi\big]\\
-\sin\big[\lambda(1-b)\pi\big]\sin\big[\lambda(1+b)\pi\big]\sinh\big[\lambda(1-a)\pi\big]\sinh\big[\lambda(1+a)\pi\big]\sinh\big[\lambda(a+b)\pi\big]\\
+2\sin\big[\lambda(1-b)\pi\big]\sin\big[\lambda(1-a)\pi\big]\sinh\big[\lambda(1-b)\pi\big]\sinh\big[\lambda(1-a)\pi\big]\sinh\big[\lambda(a+b)\pi\big]\\
-\sin\big[\lambda(1-b)\pi\big]\sin\big[\lambda(1-a)\pi\big]\sin\big[\lambda(a+b)\pi\big]\sinh\big[\lambda(1-b)\pi\big]\sinh\big[\lambda(1-a)\pi\big]=0.
\end{array}
$$
\end{theorem}

We omit writing the explicit form of the corresponding eigenfunctions. We give a complete result only in the
case of symmetric beams, since most suspension bridges have symmetric side spans ($b=a$) satisfying \eq{physicalrange}.
So, assume that $0<b=a< 1$ and let
$$I_-=(-\pi, -a\pi),\qquad I_0=(-a\pi, a\pi),\qquad I_+=(a\pi,\pi).$$
The symmetry of the position of the piers enables us to seek separately even and odd solutions of \eqref{autovsym}.
We obtain the following statement, whose proof will be given in Section \ref{pfsymmetriceigen}.

\begin{theorem}\label{symmetriceigen}
Let $b=a$. The set of all the eigenvalues $\mu=\lambda^4$ of \eqref{autov0} is completely determined by the values of $\lambda > 0$ such that \begin{equation}\label{auto1}
\sin(\lambda \pi)\sinh(\lambda a\pi)\sinh[\lambda(1-a)\pi]=\sinh(\lambda \pi)\sin(\lambda a\pi)\sin[\lambda(1-a)\pi]
\end{equation}
\begin{equation}\label{auto2}
\cos(\lambda \pi)\cosh(\lambda a \pi)\sinh[\lambda(1-a)\pi]=\cosh(\lambda \pi)\cos(\lambda a \pi)\sin[\lambda(1-a)\pi].
\end{equation}
In case \eqref{auto1}, the corresponding eigenfunctions are \textbf{odd} and given by:
\begin{itemize}
\item[-]
$\mathbf{O}_\lambda(x)=\sin(\lambda x)$
if $\lambda \in \mathbb{N}$ (implying both $\lambda a \in \mathbb{N}$ and $\lambda(1-a) \in \mathbb{N}$);
\item[-] the odd extension of
{\small
$$
\mathscr{O}_\lambda(x)=
\left\{
\begin{array}{ll}
\!\!\displaystyle \frac{\sinh[\lambda(1-a)\pi]}{\sinh (\lambda a\pi)} \,  ( \sinh(\lambda a \pi) \sin (\lambda x)-\sin(\lambda a\pi) \sinh (\lambda x)) & \mbox{if } x \in [0, a\pi] \vspace{0.25cm} \\
\!\!\displaystyle \frac{\sin(\lambda a\pi)}{\sin[\lambda(1-a)\pi]} \, (\sin[\lambda(1-a)\pi]\sinh[\lambda(x-\pi)]-\sinh[\lambda(1-a)\pi] \sin[\lambda(x-\pi)]) & \mbox{if } x \in [a\pi, \pi]
\end{array}
\right.
$$}\!
if $\lambda \notin \mathbb{N}$ (implying both $\lambda a \notin \mathbb{N}$ and $\lambda(1-a) \notin \mathbb{N}$).
\end{itemize}
In case \eqref{auto2}, the corresponding eigenfunctions are \textbf{even} and given by
\begin{itemize}
\item[-]
$\mathbf{E}_\lambda(x)=\cos(\lambda x)$
if $\lambda - 1/2 \in \mathbb{N}$ (implying both $\lambda a -1/2 \in \mathbb{N}$ and $\lambda (1-a) \in \mathbb{N}$);
\item[-] the even extension of
{\small
$$
\mathscr{E}_\lambda(x)=
\left\{
\begin{array}{ll}
\!\!\displaystyle \frac{\sinh[\lambda(1-a)\pi]}{\cosh (\lambda a \pi)} \, (\cosh(\lambda a \pi) \cos (\lambda x)  - \cos (\lambda a \pi) \cosh (\lambda x) ) & \mbox{if } x \in [0, a\pi] \vspace{0.25cm}\\
\!\!\displaystyle \frac{\cos(\lambda a \pi)}{\sin[\lambda(1-a)\pi]} \, (\sinh[\lambda(1-a)\pi] \sin[\lambda(\pi-x)]-\sin[\lambda(1-a)\pi]\sinh[\lambda(\pi-x)]) & \mbox{if } x \in [a\pi, \pi]
\end{array}
\right.
$$}\!
if $\lambda - 1/2 \notin \mathbb{N}$ (implying both $\lambda a - 1/2 \notin \mathbb{N}$ and $\lambda (1-a) \notin \mathbb{N}$).
\end{itemize}
\end{theorem}
\par
Some comments about Theorem \ref{symmetriceigen} are in order. Let us notice that the eigenvalues $\lambda^4$ associated with the eigenfunctions $\mathbf{O}_\lambda$ and $\mathbf{E}_\lambda$ correspond
to the values of $\lambda$ for which both sides of the equalities \eqref{auto1} and \eqref{auto2}, respectively, vanish. For this reason, in the case of nonsmooth odd and even eigenfunctions it has necessarily to be $\lambda(1-a)\notin\mathbb{N}$, so that $\mathscr{O}_\lambda$ and $\mathscr{E}_\lambda$ are well-defined. Notice that we fix the eigenfunctions to be all positive in $x=0$ if even, and increasing in $x=0$ if odd. Here the convention \eqref{spostagi첫} plays no role.
\par
We also observe that the eigenfunctions $\mathbf{O}_\lambda$ and $\mathbf{E}_\lambda$ are of class $C^\infty$, while $\mathscr{O}_\lambda$ and $\mathscr{E}_\lambda$ are only $C^2$. In view of the conditions on $\lambda$, no $C^\infty$-eigenfunctions exist if $a \notin \mathbb{Q}$ since, in this case, it cannot be $\lambda \in \mathbb{N}$ and $\lambda a \in \mathbb{N}$ (or $\lambda -1/2 \in \mathbb{N}$ and $\lambda a -1/2 \in \mathbb{N}$) at the same time. On the other hand, if $a=p/q$ with $p, q \in \mathbb{N}$ and $g.c.d.(p, q)=1$, odd $C^\infty$-eigenfunctions appear for $\lambda_r = rq$, for every $r \in \mathbb{N}$. As for even $C^\infty$-eigenfunctions, they exist whenever $a=p/q$, with $g.c.d.(p, q)=1$ and both $p$ and $q$ are odd; in this case, the eigenvalues are given by $\lambda_r=r+1/2$, with $r \in \mathbb{N}$ such that $2r+1$ is a multiple of $q$. Thus, there are infinitely many regular eigenfunctions if $a \in \mathbb{Q}$, but it may happen that they are all odd (as in the case $a=1/2$). The $C^\infty$-eigenfunctions $\mathbf{O}_\lambda$ and $\mathbf{E}_\lambda$ satisfy the strong form of the eigenvalue problem
$$
e''''=\lambda^4 e.
$$
As for merely $C^2$-eigenfunctions, by computing explicitly the third derivative we can formally write equation \eqref{autovsym}, as in  Item $(ii)$ of Theorem \ref{regular}. Precisely,
we have
$$
e'''' = \lambda^4 e  + \alpha_\lambda \delta_{a\pi}  + \beta_\lambda \delta_{-a\pi},
$$
for suitable constants $\alpha_\lambda$ and $\beta_\lambda$ to be determined.

Finally, we notice that conditions \eq{auto1} and \eq{auto2} characterize the {\em couples} $(a, \lambda)$ such that $\mu=\lambda^4$ is an eigenvalue of \eq{autov0}. Since the least eigenvalue $\mu_0=\lambda_0^4$ of \eq{autov0} satisfies the lower bound
\begin{equation}\label{stimabasso}
\lambda_0^4=\min_{v\in V(I)}\ \frac{\int_I(v'')^2}{\int_Iv^2}\ge\min_{v\in H^2\cap H^1_0(I)}\ \frac{\int_I(v'')^2}{\int_Iv^2}=\frac{1}{16},
\end{equation}
the curves implicitly defined by \eq{auto1} and \eq{auto2} in the $(a, \lambda)$-plane all lie above the line $\lambda= 1/2$.
In the next section, we characterize in full detail all such curves, which are depicted in Figure \ref{bello}. Notice that the ones representing the
odd eigenvalues (in darker color in Figure \ref{bello}) are symmetric with respect to $a=1/2$, since \eqref{auto1} does not vary when replacing $a$ by $1-a$.
\begin{figure}[ht]
\begin{center}
\includegraphics[scale=0.7]{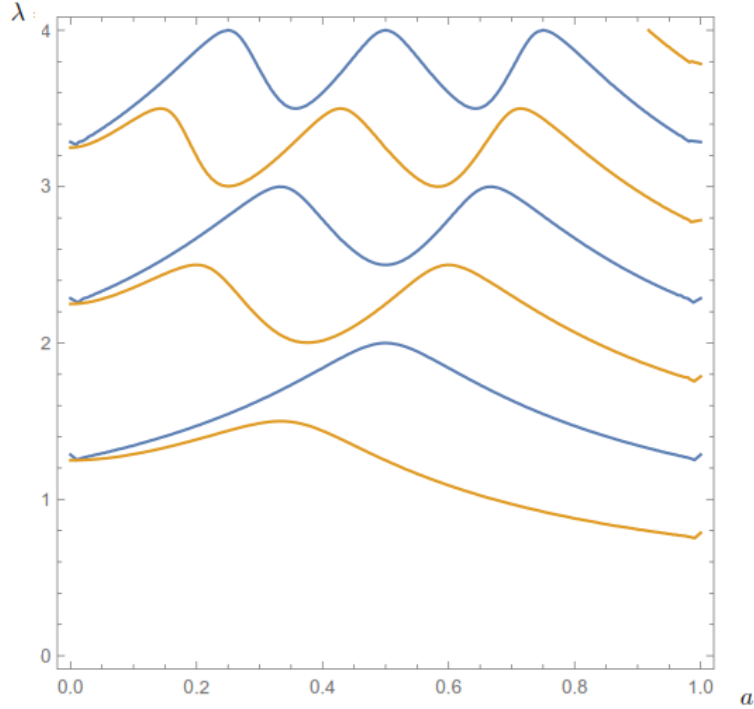}
\caption{The curves implicitly defined by \eq{auto1}-\eq{auto2} in the region $(a, \lambda)\in(0,1)\times(0,6)$.}\label{bello}
\end{center}
\end{figure}

In the case $a=1/2$, Theorem \ref{symmetriceigen} takes the following elegant form.
\begin{corollary}\label{symmetriceigena12}
Let $b=a=1/2$. The eigenvalues $\mu=\lambda^4$ of \eqref{autov0} are completely determined by the values of $\lambda > 0$ such that
$$
\sin(\lambda\pi/2)=0 \quad \textrm{ or } \quad \tan(\lambda\pi/2)=\tanh(\lambda\pi/2) \quad \textrm{ or } \quad \tan(\lambda \pi)=\tanh(\lambda \pi).
$$
In the first case, a corresponding eigenfunction is given by $\mathbf{O}_\lambda(x)=\sin(\lambda x)$, while in the other two cases it is given, respectively, by the odd extension of
$$
\mathscr{O}_\lambda(x)=
\left\{
\begin{array}{ll}
\displaystyle  \frac{\sin (\lambda x)}{\sin(\lambda\pi/2)} -\frac{\sinh (\lambda x)}{\sinh(\lambda\pi/2)} & \mbox{if } x \in [0, \pi/2] \vspace{0.25cm} \\
\displaystyle  \frac{\sinh [\lambda(x-\pi)]}{\sinh(\lambda\pi/2)}-\frac{\sin[\lambda(x-\pi)]}{\sin(\lambda\pi/2)}  & \mbox{if } x \in [\pi/2, \pi],
\end{array}
\right.
$$
and by the even extension of
$$
\mathscr{E}_\lambda(x)=
\left\{
\begin{array}{ll}
\displaystyle \tanh\Big(\frac{\lambda \pi}{2}\Big)\left(\cosh\Big(\frac{\lambda \pi}{2}\Big)\cos(\lambda x) - \cos\Big(\frac{\lambda \pi}{2}\Big)\cosh(\lambda x)\right) & \mbox{if } x \in [0, \pi/2] \vspace{0.25cm}\\
\displaystyle \cotan\Big(\frac{\lambda \pi}{2}\Big) \left(\sinh\Big(\frac{\lambda \pi}{2}\Big)\sin[\lambda(\pi-x)] - \sin\Big(\frac{\lambda \pi}{2}\Big)\sinh[\lambda(\pi-x)]\right) & \mbox{if } x \in [\pi/2, \pi].
\end{array}
\right.
$$.
\end{corollary}

Therefore, for $b=a=1/2$, the eigenvalues are given by
\begin{equation}\label{lambdasimmetrico}
\lambda = 2k \textrm{ or } \lambda \approx 2k + \frac{1}{2} \quad \textrm{(odd)}, \ \qquad \lambda \approx k + \frac{1}{4} \quad \textrm{ (even)}, \ \qquad k=1, 2, \ldots,
\end{equation}
since the function $s \mapsto \tanh(s)$ rapidly converges to $1$.

\subsection{Nodal properties of the vibrating modes}

The main purpose of this section is to classify the eigenfunctions of problem \eqref{autov0}, as given in Theorem \ref{symmetriceigen}, according to their number of zeros (nodal intervals) in $I$.
\par
Preliminarily, it is convenient to consider the eigenvalues of the clamped beam on $I$, that is, the numbers $\mu>0$ for which the problem
\begin{equation}\label{clamped}
\int_{I} e'' v'' = \mu \int_{I} e v  \qquad \forall v \in H^2_0(I)
\end{equation}
admits a nontrivial solution $e$. From Section \ref{3.1} we recall that $H^2_0(I)$ is the limit space of $V(I)$ as $a \to 1$. It is straightforward to verify that $\mu$ is an eigenvalue of \eqref{clamped} if and only if $\mu=\Lambda_n^4$ ($n=0,1,2,\ldots$),
with $\Lambda_n$ defined by
\neweq{clampedeigen}
\left\{\begin{array}{ll}
\tan(\Lambda_{2k} \pi) = -\tanh(\Lambda_{2k} \pi) \qquad & \textrm{with even eigenfunction,}\\
\tan(\Lambda_{2k+1} \pi) = \tanh(\Lambda_{2k+1} \pi) \qquad & \textrm{with odd eigenfunction.}
\end{array}\right.
\end{equation}
The corresponding eigenfunctions $\psi_n$ ($n=0, 1, 2,\ldots$) are all of class $C^\infty(I)$ and are explicitly given by
\neweq{psieigen}
\left\{\begin{array}{ll}
\psi_{2k}(x)=\cosh(\Lambda_{2k} \pi)\cos(\Lambda_{2k} x) - \cos (\Lambda_{2k} \pi)\cosh(\Lambda_{2k} x)\\
\psi_{2k+1}(x)=\sinh (\Lambda_{2k+1} \pi) \sin (\Lambda_{2k+1} x)-\sin(\Lambda_{2k+1} \pi) \sinh (\Lambda_{2k+1} x)
\end{array}\right.\qquad x \in I.
\endeq
Moreover, it will be useful to consider also the limit space $V_*(I)$ (as $a \to 0$) introduced in \eq{Vstar}
and we denote by $\mu=(\Lambda_n^*)^4$ ($n=0, 1, 2, \ldots$) the eigenvalues of
\begin{equation}\label{clamped0}
\int_{I} e'' v'' = \mu \int_{I} e v  \qquad \forall v \in V_*(I).
\end{equation}
Some computations show that each eigenvalue of \eqref{clamped0} is double and that
\begin{equation}\label{coincidenzaaut}
\Lambda_n^*= \Lambda_n \quad \textrm{ if } n \textrm{ is odd, } \qquad \Lambda_n^*= \Lambda_{n+1} \quad \textrm{ if } n \textrm{ is even, }
\end{equation}
namely the eigenvalues of \eqref{clamped0} are exactly the eigenvalues of \eqref{clamped} with odd index $n$, all of them having multiplicity 2. The corresponding eigenfunctions are obtained by
extending by (even and odd) symmetry the restriction to the interval $[0, \pi]$ of $\psi_{2k+1}(\pi-x)$, where $\psi_{2k+1}$ has been defined in \eqref{psieigen}. Notice, in particular, that the so obtained eigenfunctions of \eqref{clamped0} are not all of class $C^2$; in fact, in this case the integration by parts does not imply the matching of the second derivatives in $0$, because the test functions $v$ satisfy $v'(0)=0$, thus canceling the boundary term $u''(0)v'(0)$.
\par
With these preliminaries, we may state one of the main results of this section, which will be proved in Section \ref{pfMichelle}.

\begin{theorem}\label{Michelle}
For any $a \in (0, 1)$, the eigenvalues $\mu=\lambda^4$ of problem \eqref{autovsym} are simple and form a countable set, the corresponding eigenfunctions are of class $C^2$ and form an orthogonal basis of $V(I)$. Moreover, \eqref{auto1} and \eqref{auto2} implicitly define, for $a\in(0, 1)$, a family of
analytic functions $a \mapsto \lambda_n(a)$ which satisfy $\lambda_n(a)\to\Lambda_{n}$ for $a\to1$ and $\lambda_n(a)\to\Lambda_n^*$ for $a\to0$ ($n=0, 1, 2, \ldots$).
\end{theorem}
We observe that, as a straightforward consequence of Theorem \ref{Michelle}, we have
$$
\lim_{a \to 0} \frac{\lambda_{n+1}(a)}{\lambda_n(a)} = \left\{
\begin{array}{ll}
1 & \textnormal{ if } n \textnormal{ is even } \vspace{0.1cm}\\
\displaystyle \frac{\Lambda_{n+2}}{\Lambda_n} & \textnormal{ if } n \textnormal{ is odd. }
\end{array}
\right.
$$
Theorem \ref{Michelle} states that the set of couples $(a,\lambda)$ satisfying either \eq{auto1} or \eq{auto2} is composed by the union of connected branches which are graphs of
regular functions $\lambda=\lambda(a)$; their intersections with any line $a=\bar{a} < 1$ give all the eigenvalues $\mu=\lambda^4$ of problem \eqref{autovsym}
corresponding to the choice $a=\bar{a}$, see again Figure \ref{bello}. It turns out that, even if all the eigenvalues are simple, the spectral gaps can be very small. This means that the corresponding modes of the linear evolution equation \eq{beamlineare} have fairly similar frequencies.
\par
We now turn to the nodal properties of the eigenfunctions. For a given $a \in (0, 1)$, Theorem \ref{Michelle} allows us to sort the eigenvalues in increasing order $\{\lambda_0, \lambda_1, \lambda_2, \ldots\}$ and to label the associated eigenfunctions as $\{e_0, e_1, e_2, \ldots\}$.
 We will always speak about \emph{even} and \emph{odd} eigenfunctions and eigenvalues, referring to such labels.
The placement of the zeros of the eigenfunctions $e_n$, depending on the couple $(a,\lambda)$, is of crucial importance. This was already noticed in the Federal
Report \cite{ammvkwoo}, see the reproduction in Figure \ref{zeroTNB} where an inventory of the modes of oscillation seen at the TNB
is drawn.
\begin{figure}[!h]
\begin{center}
\includegraphics[height=215mm, width=160mm]{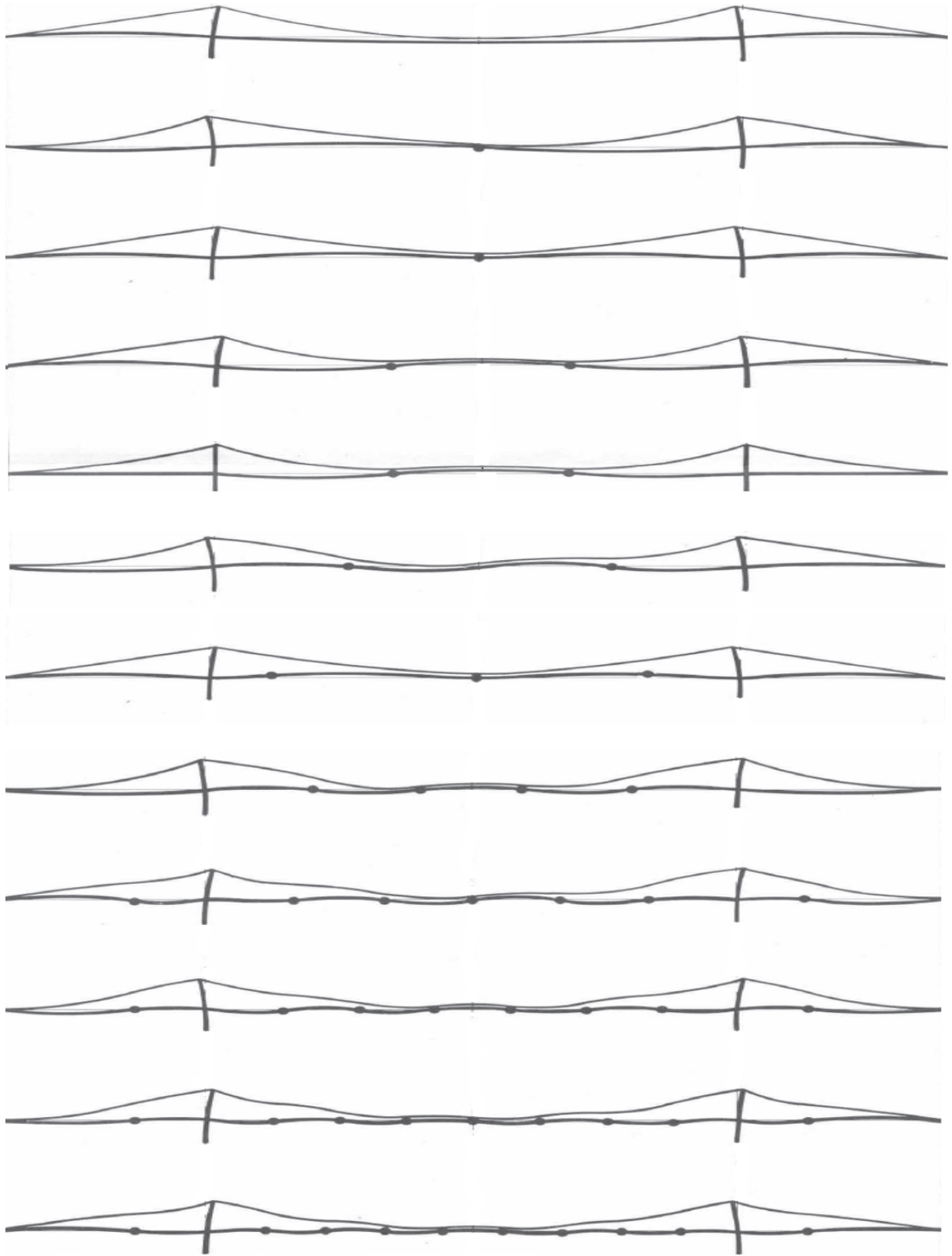}
\caption{Zeros seen at the TNB: hand reproduction of Drawing 4 in \cite{ammvkwoo}.}\label{zeroTNB}
\end{center}
\end{figure}
Since two zeros are always present in the piers, it is necessary to make this precise.
First, it is clear that the $C^\infty$-eigenfunctions $\mathbf{O}_\lambda$ and $\mathbf{E}_\lambda$ cannot have double zeros in any point of $\overline{I}$. On the other hand, using the explicit expression of $\mathscr{O}_\lambda$ (resp., $\mathscr{E}_\lambda$),
it turns out that $x \in I_0$ is a double zero if and only if
\begin{equation}\label{conddoppioa}
\frac{\sin(\lambda x)}{\sinh(\lambda x)} = \frac{\cos(\lambda x)}{\cosh(\lambda x)} = \frac{\sin(\lambda a \pi)}{\sinh(\lambda a \pi)} \quad \left(\textrm{resp.,} -\frac{\sin(\lambda x)}{\sinh(\lambda x)} = \frac{\cos(\lambda x)}{\cosh(\lambda x)} = \frac{\cos(\lambda a \pi)}{\cosh(\lambda a \pi)}\right).
\end{equation}
Since
$$
\left(\frac{\sin(\lambda x)}{\sinh(\lambda x)}\right)' = 0 \quad \textrm{ iff } \quad \frac{\sin(\lambda x)}{\sinh(\lambda x)} = \frac{\cos(\lambda x)}{\cosh(\lambda x)}, \qquad
\left(\frac{\cos(\lambda x)}{\cosh(\lambda x)}\right)' = 0 \quad \textrm{ iff } \quad -\frac{\sin(\lambda x)}{\sinh(\lambda x)} = \frac{\cos(\lambda x)}{\cosh(\lambda x)},
$$
condition \eqref{conddoppioa} may be satisfied only in correspondence of the local minima (or maxima) of the function $\xi(x)=\sin(\lambda x)/\sinh(\lambda x)$ (resp., $\xi(x)=\cos(\lambda x)/\cosh(\lambda x)$). Labeling these stationary points in increasing order as $x_k$, $k=1, 2, \ldots$, one sees that $k \mapsto \vert \xi(x_k) \vert$ is strictly decreasing. Consequently, \eqref{conddoppioa} can be fulfilled only if $x=a\pi$, that is, if the double zero is located at one pier. In a similar way one can reason for the lateral spans, obtaining the following result.
\begin{proposition}\label{nodouble}
Let $e_\lambda$ be an eigenfunction of problem \eqref{autovsym}. Then $e_\lambda$ cannot have double zeros elsewhere than at the piers.
\end{proposition}
If the eigenfunction has a double zero at the piers, its restriction to the central span is clamped, while its restrictions to the side spans are partially hinged and partially clamped.
Corollary \ref{symmetriceigena12} provides an example ($a=1/2$) where the eigenfunctions $\mathscr{O}_\lambda$ of \eqref{autovsym} have this feature, but
no zeros of order greater than $2$ are placed in the piers; in fact, the third derivative is therein defined - with nonzero value - only for the $C^\infty$-eigenfunctions $\mathbf{O}_\lambda$ and $\mathbf{E}_\lambda$. We thus define the number of ``effective'' zeros of an eigenfunction $e_\lambda$ in $I$ by
$$
i(e_\lambda)=
\left\{
\begin{array}{ll}
\#\{x \in I_- \cup I_0 \cup I_ + \mid e_\lambda(x)=0\}  & \textrm{ if } e_\lambda'(a\pi) \neq 0 \\
\#\{x \in I_- \cup I_0 \cup I_ + \mid e_\lambda(x)=0\} + 2 & \textrm{ if } e_\lambda'(a\pi) = 0.
\end{array}
\right.
$$
Hence, if $e_\lambda$ possesses double zeros at the piers, then we count them as two additional simple zeros.\par
The second result of this section gives a complete description of the placement of the zeros of the eigenfunctions; we postpone its lengthy proof to Section \ref{pfconstant}. In order to give the statement, fixed $a \in (0, 1)$ we underline the dependence of the eigenfunctions on $a$, denoting $e_n$ by $e_{\lambda_n(a)}$, $n=0, 1, 2, \ldots$

\begin{theorem}\label{constant}
For $a \in (0, 1)$, it holds that $i(e_n)=n$, for every $n=0, 1, 2, \ldots$. Fixed an integer $n \geq 0$, on decreasing of $a$ the zeros of $e_{\lambda_n(a)}$ move by couples from the central span to the side spans whenever the
curve $\lambda=\lambda_n(a)$ intersects one of the hyperbolas $\{\lambda=\Lambda_k/a\}$, for some integer $k \geq 0$ having the same parity as $n$.
\end{theorem}
\begin{center}
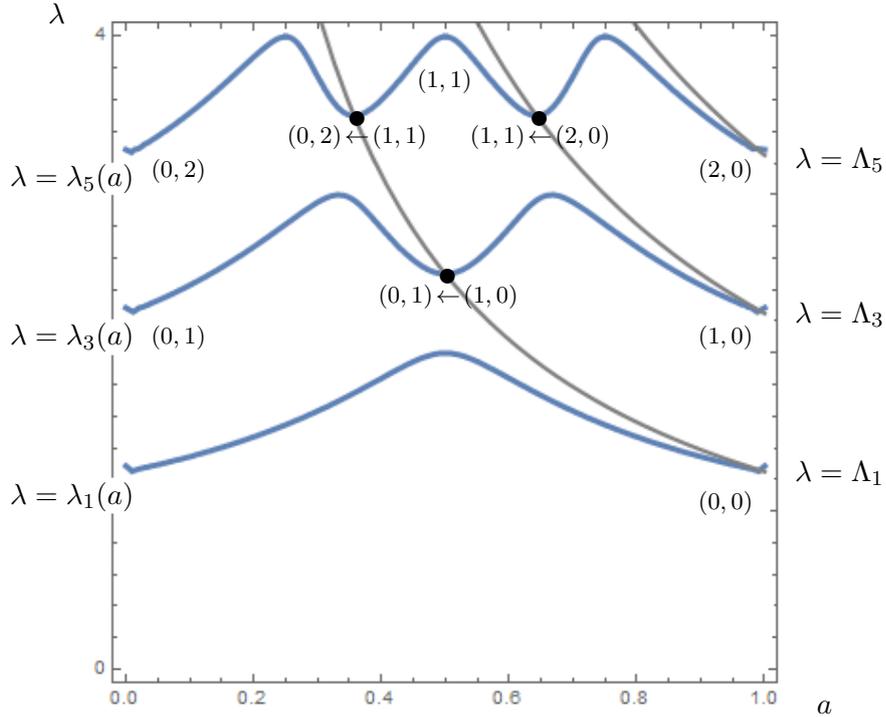
\begin{figure}[!ht]
\begin{center}
\begin{tikzpicture}
\node[inner sep=0pt] (whitehead) at (10,0)
    {\includegraphics[width=9cm]{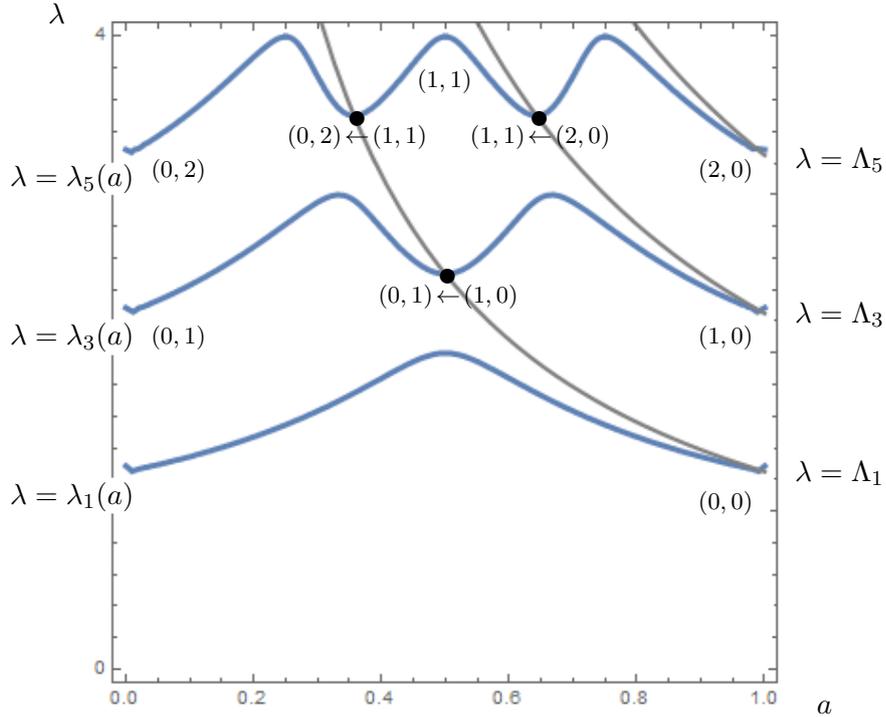}};
\node at (15.1, -4.5){$a$};
\node at (5, 4.7){$\lambda$};

\node at (15.3, 2.75){$\lambda=\Lambda_5$};
\node at (13.8, 2.6){\footnotesize{{$(2,0)$}}};
\node at (11.36, 3.05){\footnotesize{{$(1,1) \!\leftarrow\! (2,0)$}}};
\node at (10.1, 3.8){\footnotesize{{$(1,1)$}}};
\node at (8.95, 3.05){\footnotesize{{$(0,2) \!\leftarrow\! (1,1)$}}};
\node at (6.6, 2.6){\footnotesize{{$(0,2)$}}};
\node at (11.35, 3.28){\footnotesize{{\Large{$\bullet$}}}};
\node at (8.95, 3.28){\footnotesize{{\Large{$\bullet$}}}};
\node at (5.2, 2.5){$\lambda=\lambda_5(a)$};

\node at (15.3, 0.7){$\lambda=\Lambda_3$};
\node at (13.8, 0.4){\footnotesize{{$(1,0)$}}};
\node at (10.15, 0.94){\footnotesize{{$(0,1) \!\leftarrow\! (1,0)$}}};
\node at (6.6, 0.4){\footnotesize{{$(0,1)$}}};
\node at (10.14, 1.2){\footnotesize{{\Large{$\bullet$}}}};
\node at (5.2, 0.3985){$\lambda=\lambda_3(a)$};

\node at (15.3, -1.4){$\lambda=\Lambda_1$};
\node at (13.8, -1.8){\footnotesize{{$(0,0)$}}};
\node at (5.2, -1.71){$\lambda=\lambda_1(a)$};
\end{tikzpicture}
\end{center}
\caption{A visual description of Theorem \ref{constant} for the curves $\lambda=\lambda_{2m+1}(a)$.}\label{figurateo}
\end{figure}
\end{center}

By looking at Figure \ref{figurateo}, we see that the hyperbolas $\{\lambda=\Lambda_k/a\}$ describe a countable set of lines, each of which intersects
the countable set of curves representing the eigenvalues in a countable number of points.
Therefore, double zeros in the piers are possible only for a countable set of values of $a<1$, that is,
\begin{center}
{\bf for almost every $0 < a < 1$ all the eigenfunctions have simple zeros in the piers.}
\end{center}
Theorem \ref{constant} states that a double zero is placed in a pier (and by symmetry also in the other one) each time that the curve $(a, \lambda_n(a))$ crosses the graph of one of the hyperbolas $\{\lambda=\Lambda_{k}/a\}$ (with $k$ having the same parity as $n$). From there on, proceeding in the direction of decreasing $a$, the zeros of $e_n$ move from $I_0$ to $I_+$ (and $I_-$);
 for any odd $n$ and for $a$ sufficiently small such that $\lambda_n(a)$ lies below the hyperbola $\{\lambda=\Lambda_{1}/a\}$, all the zeros of $e_n$ lie in the lateral spans, except for the zero in the origin. For even $n$, the threshold becomes $\lambda_n(a) < \Lambda_0/a$ and no zeros of $e_n$ at all belong to $I_0$ below this threshold. We represent this pattern for some odd eigenfunctions in Figure \ref{figurateo},
where the numbers $(\alpha,\beta)$ in parentheses denote, respectively, the number of zeros of the eigenfunction in $(0,a\pi)$ and in $(a\pi,\pi)$.
The sum $\alpha+\beta$ is constant on each branch.
\par
We conclude this section with a result regarding the asymptotic behavior of the eigenvalues, which will become useful in Section \ref{againa} and will be proved in Section \ref{pfteoasintotica}.
\begin{theorem}\label{teoasintotica}
For every $a \in (0, 1)$, any interval of width $3$ contains at least three values of $\lambda$ for which $\mu=\lambda^4$ is an eigenvalue of
\eqref{autovsym}. As a consequence,
\begin{equation}\label{asintoticaautov}
\lim_{n \to +\infty} \frac{\lambda_{n+1}(a)}{\lambda_n(a)} = 1 \qquad \forall a \in (0, 1).
\end{equation}
\end{theorem}

\subsection{Plots of some particular eigenfunctions}

In this section, we focus on some given values of $a$, for which the eigenvalues are determined numerically, and we plot some pictures of eigenfunctions.
In particular, we consider the cases $a=14/25$ (the ratio of the spans of the TNB, according to \cite{ammvkwoo}) and $a=1/2$ portraited, respectively, in Figures \ref{autofunz2} and \ref{autofunz1}. In all the plots, the dot {\footnotesize $\bullet$} represents the position of the piers.
We recall that for $a=1/2$ there are eigenfunctions with double zeros in the piers, as stated in Corollary \ref{symmetriceigena12}.
\begin{figure}[ht!]
\begin{center}
\includegraphics[height=25mm, width=40mm]{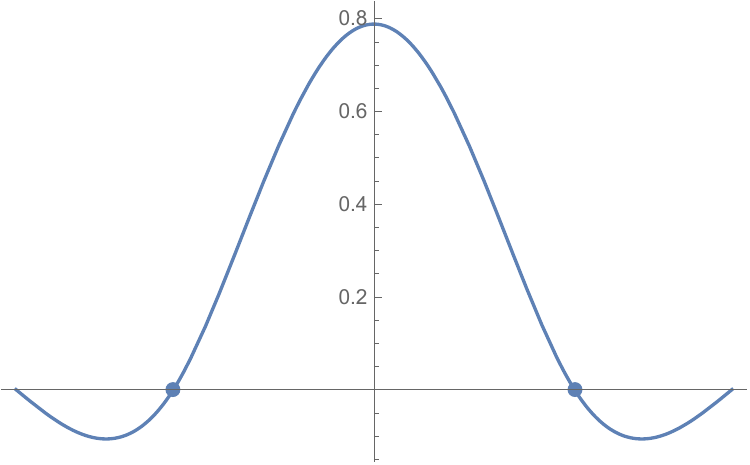}\quad\includegraphics[height=25mm, width=40mm]{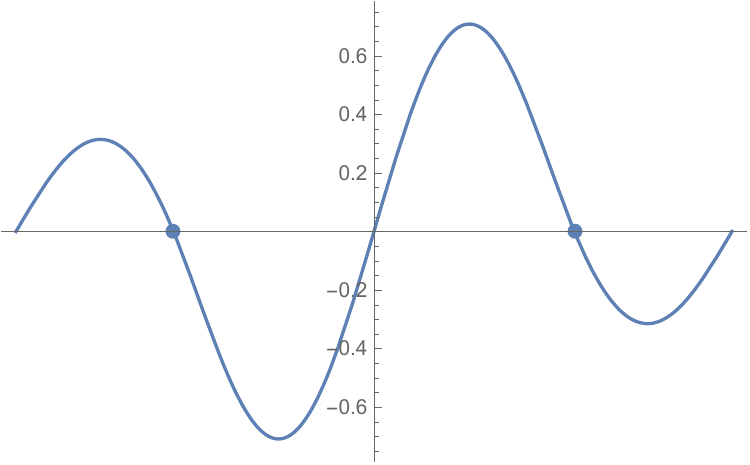}\quad
\includegraphics[height=25mm, width=40mm]{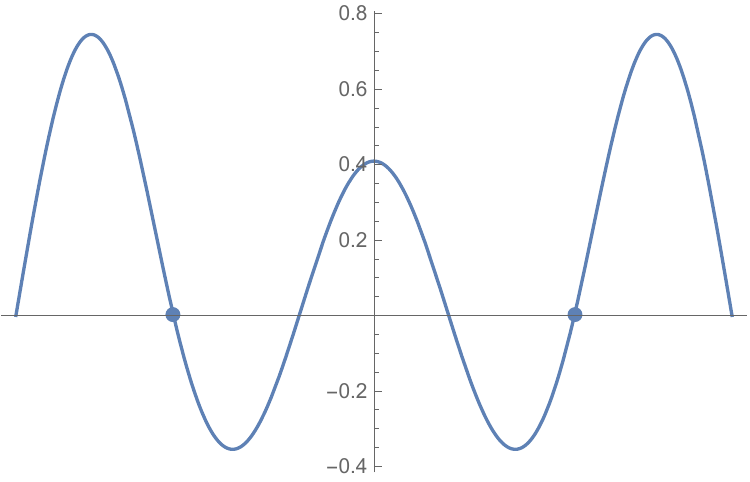}\quad\includegraphics[height=25mm, width=40mm]{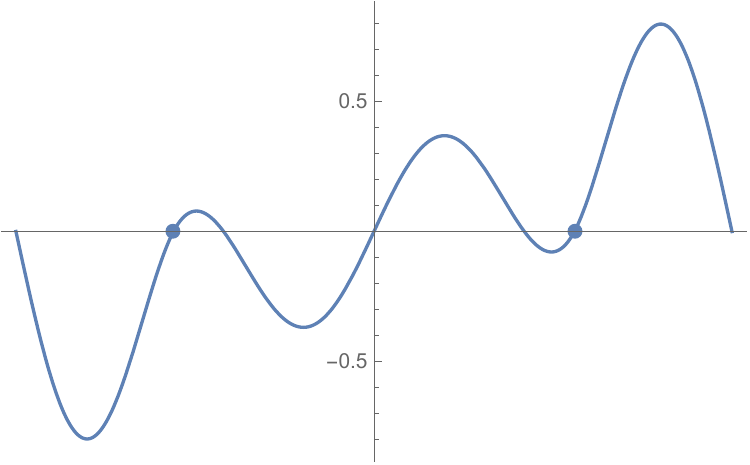}\vspace{2mm}
\includegraphics[height=25mm, width=40mm]{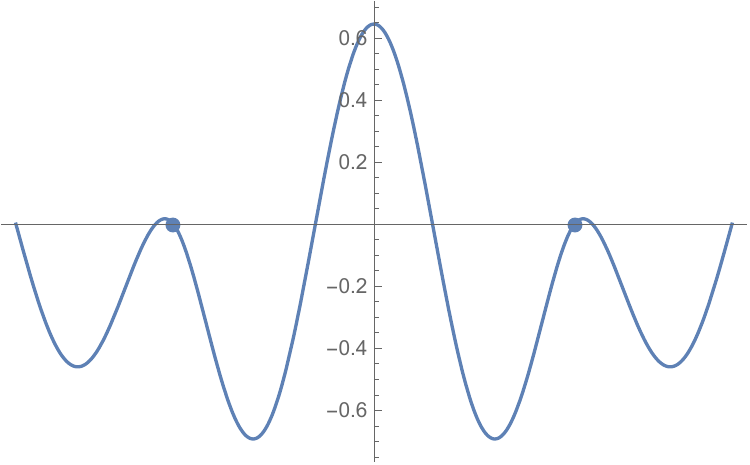}\quad\includegraphics[height=25mm, width=40mm]{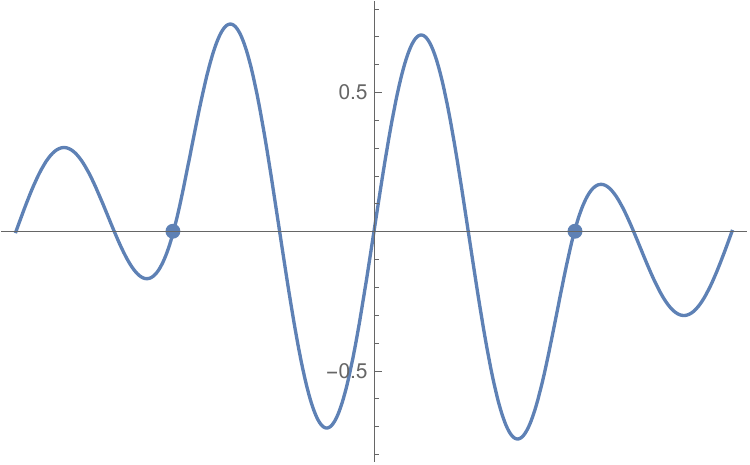}\quad
\includegraphics[height=25mm, width=40mm]{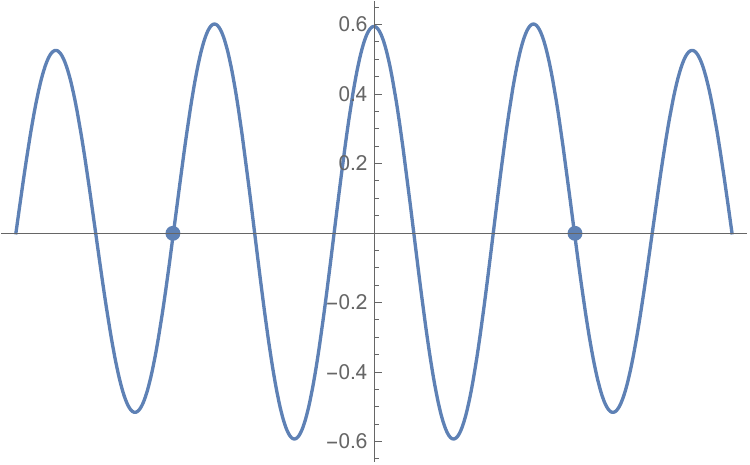}\quad\includegraphics[height=25mm, width=40mm]{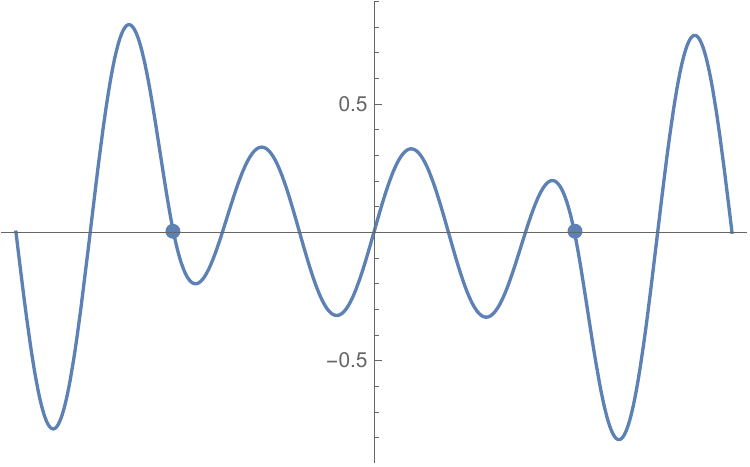}\vspace{2mm}
\includegraphics[height=25mm, width=40mm]{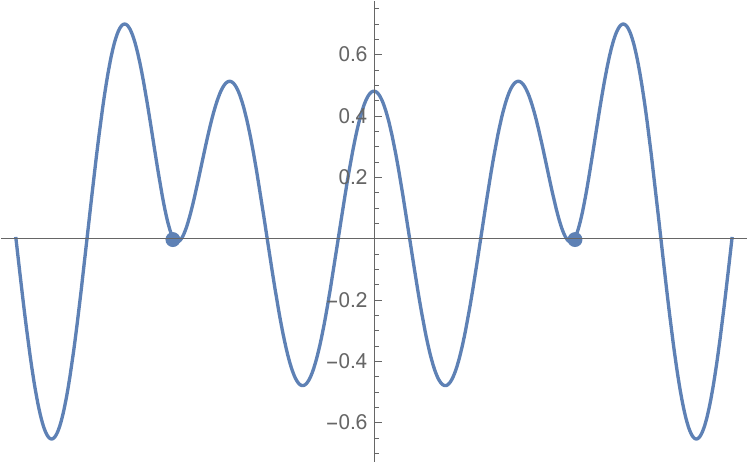}\quad\includegraphics[height=25mm, width=40mm]{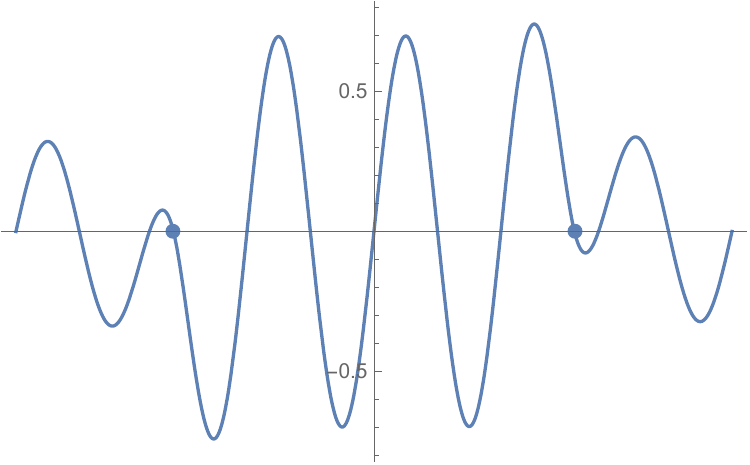}\quad
\includegraphics[height=25mm, width=40mm]{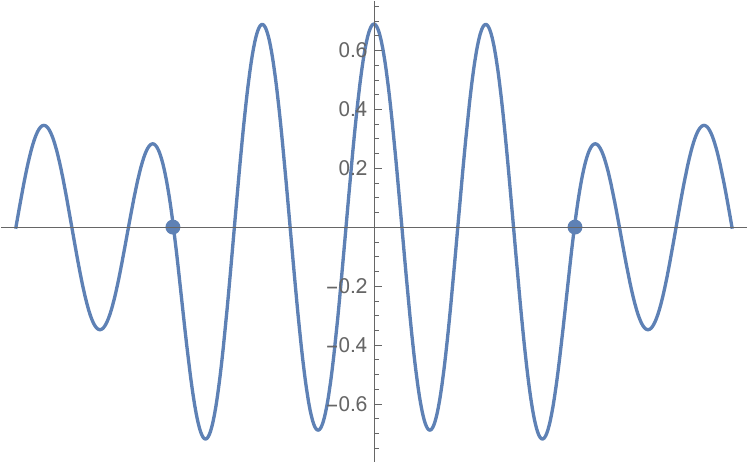}\quad\includegraphics[height=25mm, width=40mm]{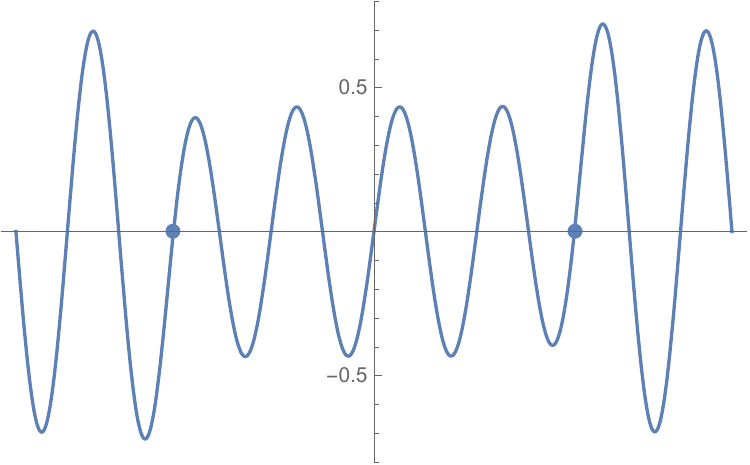}
\caption{The first twelve $L^2$-normalized eigenfunctions of \eqref{autov0} when $a=14/25$.}\label{autofunz2}
\end{center}
\end{figure}

In Table \ref{tabella} we quote the eigenvalues relative to the eigenfunctions plotted in Figures \ref{autofunz2} and \ref{autofunz1}.
\begin{table}[ht!]
{\small
\begin{center}
\begin{tabular}{|c|c|c|c|c|c|c|c|c|c|c|c|c|}
\hline
$a$ & $\mu_0$ & $\mu_1$ & $\mu_2$ & $\mu_3$ & $\mu_4$ & $\mu_5$ & $\mu_6$ & $\mu_7$ & $\mu_8$ & $\mu_9$ & $\mu_{10}$ & $\mu_{11}$ \\
\hline
$14/25$ & 1.74 & 13.8 & 35.5 & 47.3 & 84 & 205 & 409 & 533 & 633 & 1004 & 1684 & 2347 \\
\hline
$1/2$ & 2.44 & 16 & 25.6 & 39 & 112 & 256 & 326 & 410 & 760 & 1296 & 1526 & 1785\\
\hline
\end{tabular}
\caption{The least 12 eigenvalues of \eq{autov0} for $a=14/25$ and $a=1/2$, with an approximation of $1\%$.}\label{tabella}
\end{center}
}
\end{table}

\begin{figure}[ht!]
\begin{center}
\includegraphics[height=25mm, width=40mm]{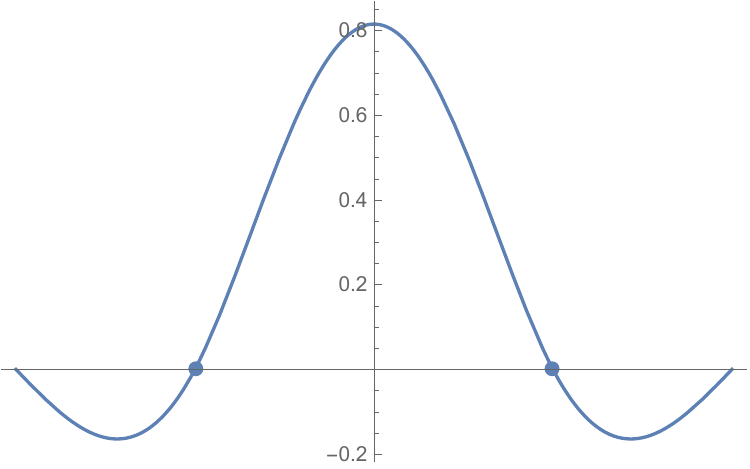}\quad\includegraphics[height=25mm, width=40mm]{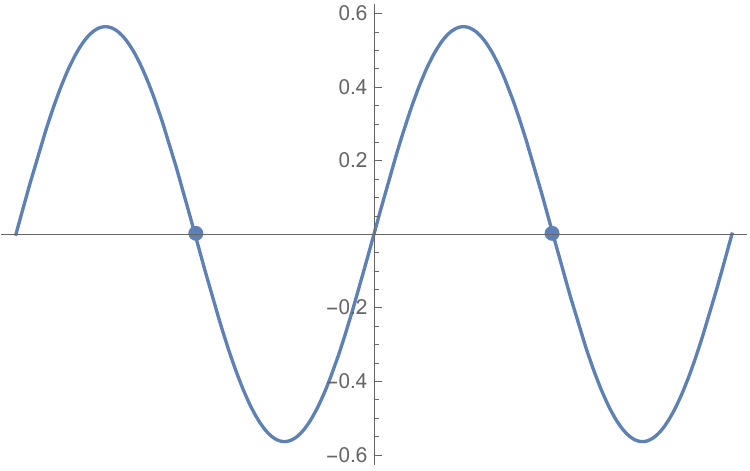}\quad
\includegraphics[height=25mm, width=40mm]{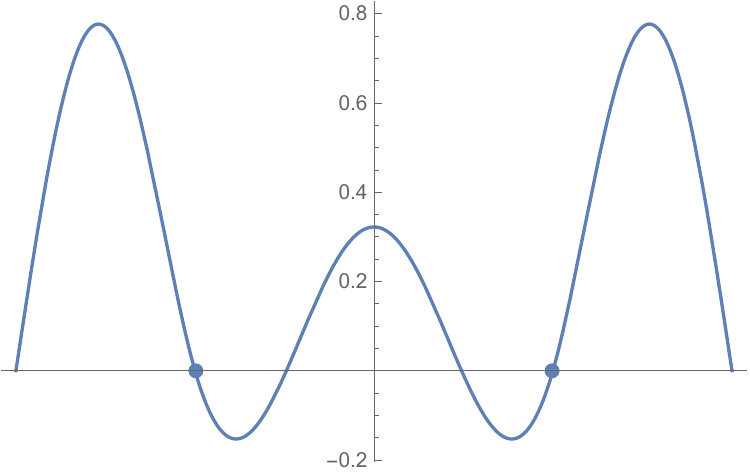}\quad\includegraphics[height=25mm, width=40mm]{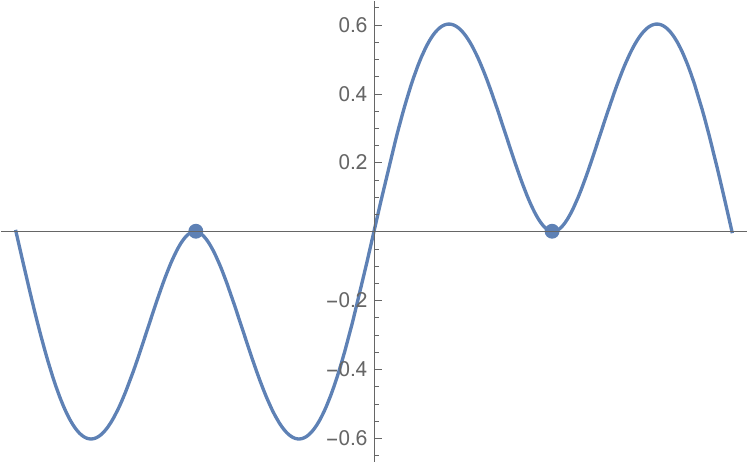}\vspace{2mm}
\includegraphics[height=25mm, width=40mm]{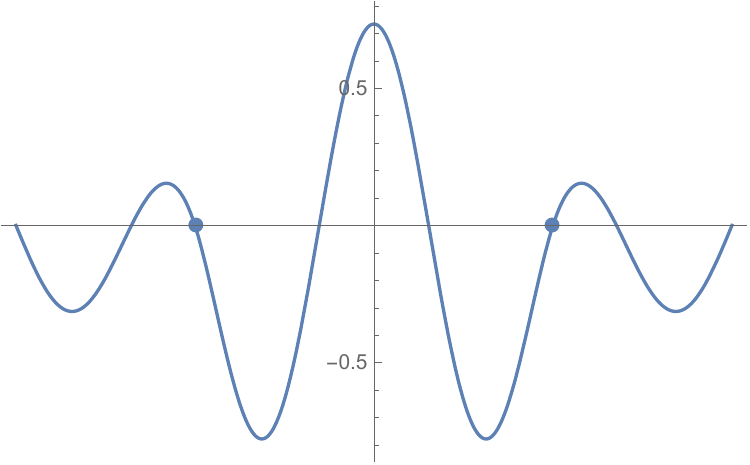}\quad\includegraphics[height=25mm, width=40mm]{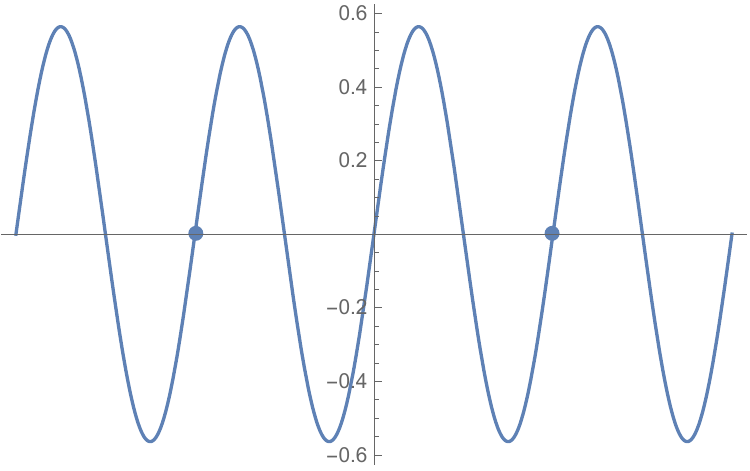}\quad
\includegraphics[height=25mm, width=40mm]{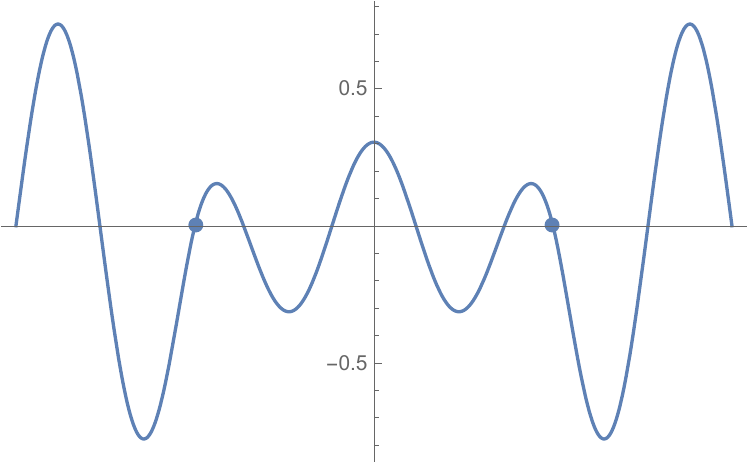}\quad\includegraphics[height=25mm, width=40mm]{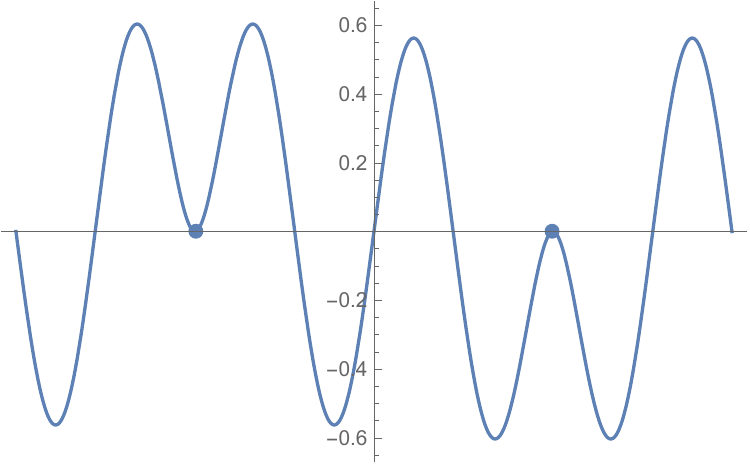}\vspace{2mm}
\includegraphics[height=25mm, width=40mm]{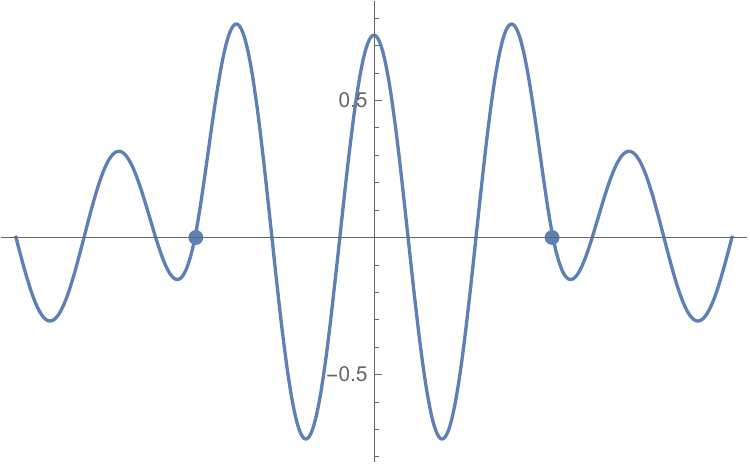}\quad\includegraphics[height=25mm, width=40mm]{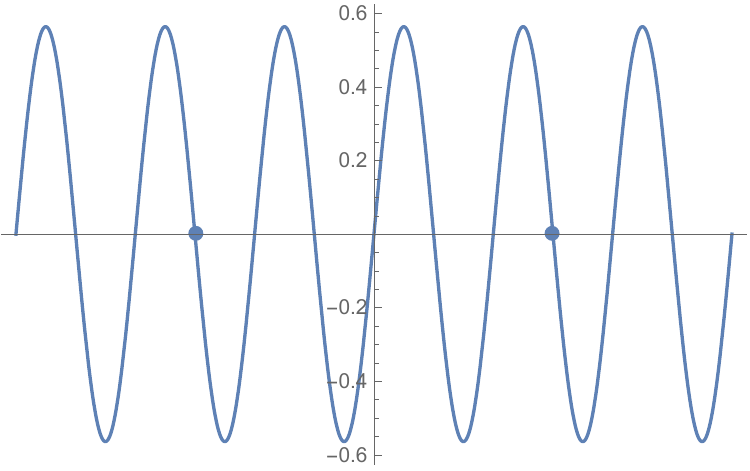}\quad
\includegraphics[height=25mm, width=40mm]{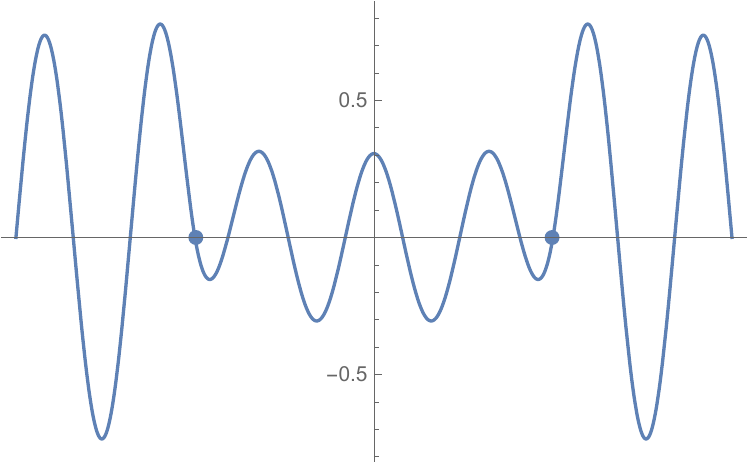}\quad\includegraphics[height=25mm, width=40mm]{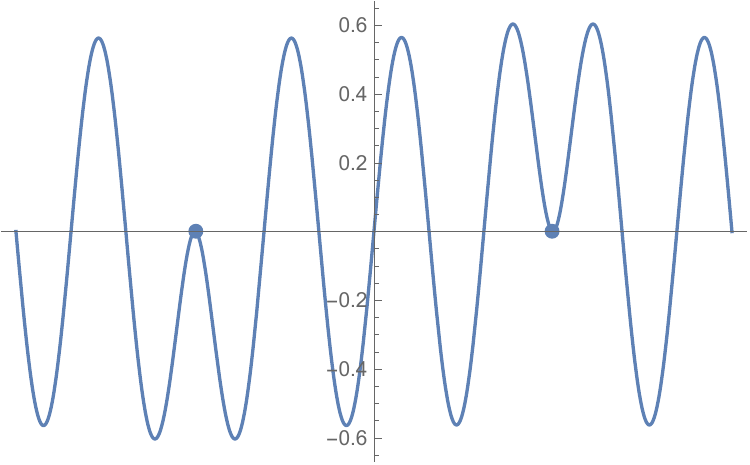}
\caption{The first twelve $L^2$-normalized eigenfunctions of \eqref{autov0} when $a=1/2$.}\label{autofunz1}
\end{center}
\end{figure}

The possibility of having double zeros in the piers naturally leads to wonder whether positive eigenfunctions may exist, namely
$e_\lambda(x)>0$ for every $x \in I_- \cup I_0 \cup I_+$. Of course, this can happen only for an even eigenfunction. Moreover, by Proposition \ref{nodouble} the only possibility is to have double zeros in the piers, since otherwise the eigenfunction would change sign crossing the piers. Due to
the nodal properties of the eigenfunctions stated in Theorem \ref{constant}, this means that the only eigenfunction that can be positive is the third one.
Denoting by $\mu_2(a)=\lambda_2^4(a)$ the third eigenvalue, we are thus led to seek the value of $a$ such that $\lambda_2(a) = \Lambda_0/a$, where $\mu=\Lambda_0^4$ is the least eigenvalue of \eqref{clamped}.
We numerically find that
\begin{center}
{\bf if $a \approx 0.3759$, then $\mu_2 \approx 16.0863$ ($\lambda_2 \approx 2.00269$) and the third eigenfunction is positive; \\
this is the only choice of $(a,\lambda)$ for which an eigenfunction of \eqref{autovsym} is positive.}
\end{center}

In Figure \ref{positiva}, we plot the shape of the third eigenfunction $e_2$ on varying of $0 < a < 1$. It can be seen how the global minima of $e_2$ move when
the side spans are enlarged. In particular, for $a \gtrapprox 0.3759$ (resp.\ $a \lessapprox 0.3759$) the minima are in the central span (resp.\ lateral spans).

\begin{figure}[ht!]
\begin{center}
\includegraphics[height=25mm, width=40mm]{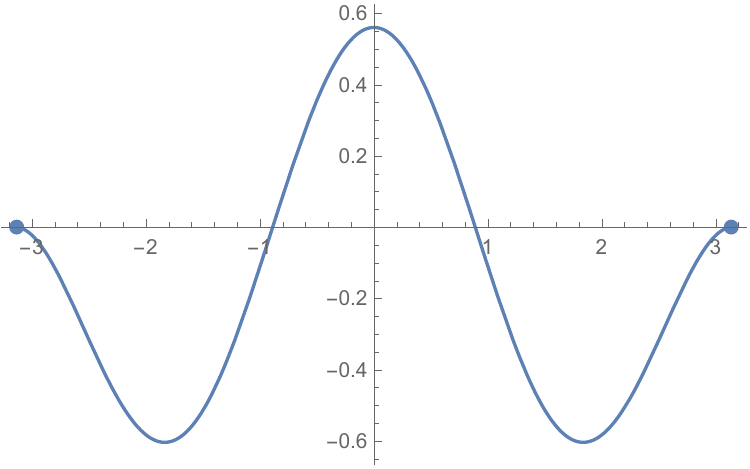}
\quad
\includegraphics[height=25mm, width=40mm]{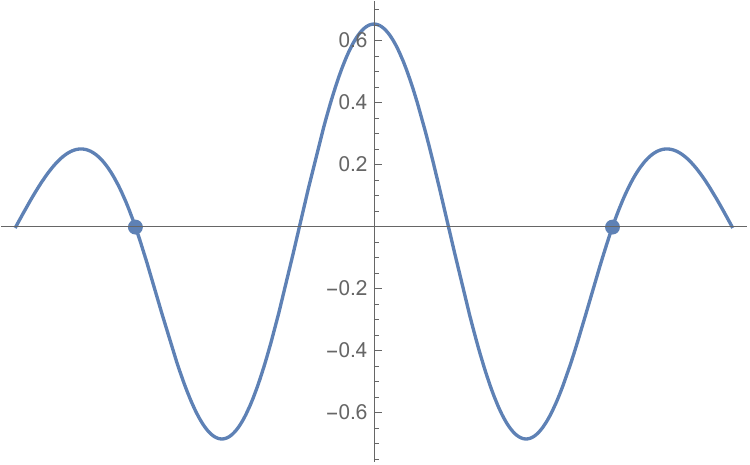}
\quad
\includegraphics[height=25mm, width=40mm]{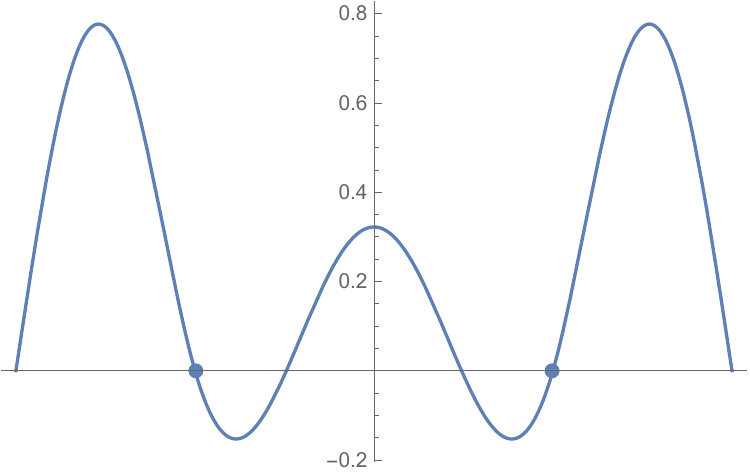}
\center{
\includegraphics[height=30mm, width=50mm]{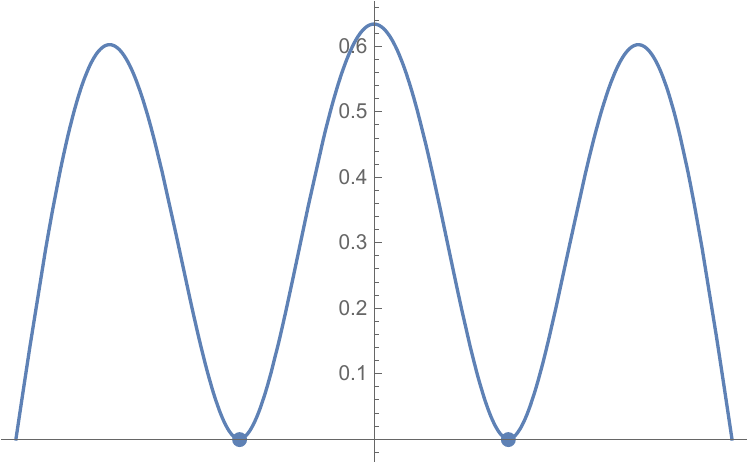}}
\center{
\includegraphics[height=25mm, width=40mm]{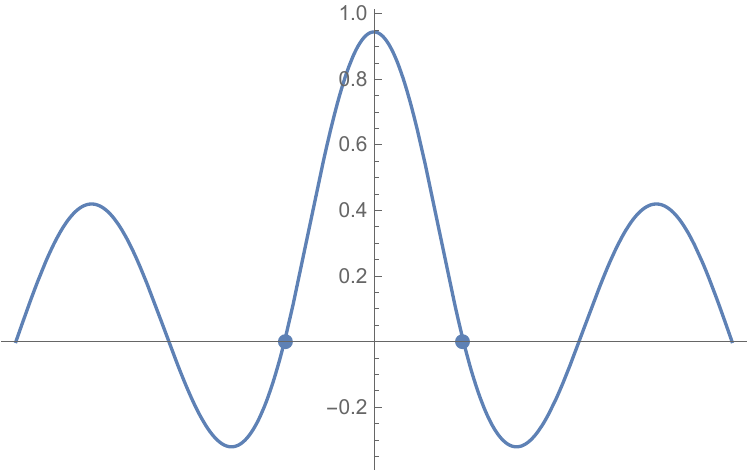}
\quad
\includegraphics[height=25mm, width=40mm]{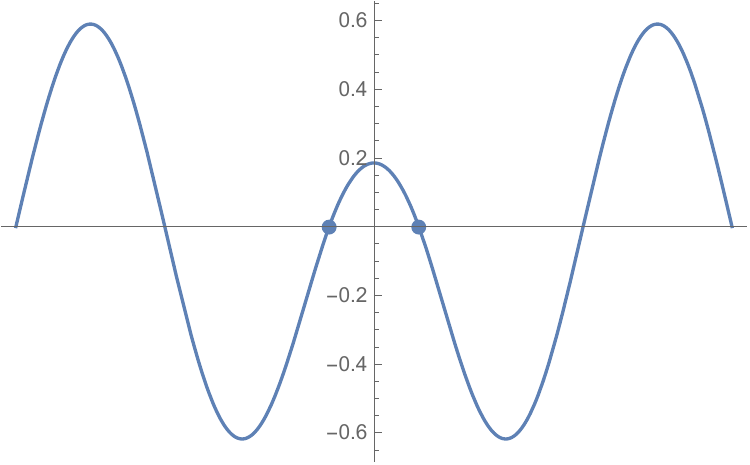}
\quad
\includegraphics[height=25mm, width=40mm]{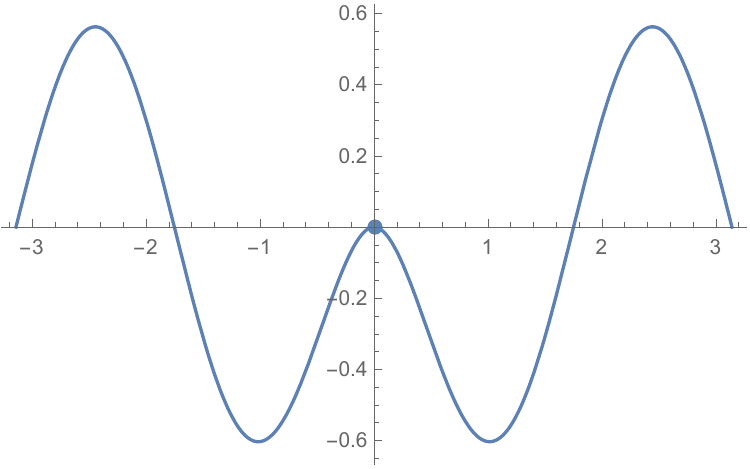}}
\caption{The third eigenfunction $e_2$ for $a=1$ (i.e., for the clamped beam), $a=2/3$, $a=1/2$ (first line), $a=0.3759$ (second line), $a=1/4$, $a=1/8$, $a=0$ (third line).}
\label{positiva}
\end{center}
\end{figure}

\vfill
\eject

\subsection{A related second order eigenvalue problem}\label{secondorder}

In Section \ref{suspbridge} we will study a system of two PDEs for a degenerate plate. The first one is
a beam equation, for which the underlying linear theory has been extensively examined in the previous sections, while the other one is a second order PDE
for which we study here the spectral decomposition.\par
Our analysis starts with the observation that the operator $L$ defined on $V(I)$ by $\langle Lu, v\rangle_V=\int_I u'' v''$
{\em is not} the square of the operator $\mathcal{L}$ defined on $W(I)$ by $\langle \mathcal{L}u, v\rangle_W = \int_I u' v'$, where
\begin{equation}\label{W}
W(I):=\{u\in H^1_0(I);\, u(\pm a\pi)=0\}\,
\end{equation}
and $\langle\cdot,\cdot\rangle_W$ denotes the duality pairing between $W'(I)$, the dual space of $W(I)$, and $W(I)$. Notice that, since $W(I) \subset C(I)$, the pointwise constraints still make sense.
\par
We formalize this observation through the following statement.
\begin{proposition}\label{Upsilon}
Let $a \in (0, 1)$ and let $e_\lambda$ be an eigenfunction of \eqref{autovsym} with associated eigenvalue $\mu=\lambda^4$. Then
\begin{equation}\label{defupsilon}
\Upsilon_\lambda^2:=\frac{\int_I(e_\lambda')^2}{\int_I e_\lambda^2}\, \left\{
\begin{array}{ll}
=\lambda^2 & \textrm{ if } e_\lambda \in C^\infty(I) \\
<\lambda^2 & \textrm{ otherwise.}
\end{array}
\right.
\end{equation}
\end{proposition}

The proof of Proposition \ref{Upsilon} is obtained in two steps.
First, the inequality $\Upsilon_\lambda \leq \lambda$ follows from an integration by parts and from the H\"older inequality:
$$
\int_I (e_\lambda')^2\leq\int_I|e_\lambda e_\lambda''|\leq\left(\int_I e_\lambda^2 \right)^{1/2}\left(\int_I(e_\lambda'')^2\right)^{1/2}=
\lambda^2\int_I e_\lambda^2 .
$$
Then, this inequality is an equality if and only if $e_\lambda$ and $e_\lambda''$ are proportional and, according to Theorem \ref{symmetriceigen}, this happens
if and only if $e_\lambda \in C^\infty(I)$.\par
The number $\Upsilon_\lambda$ in \eqref{defupsilon} may be seen as a ``correction term'' due to the presence of the piers and highlights a striking
difference compared with the beam without piers, for which $\Upsilon_\lambda=\lambda$ for all $\lambda$.
Because of $\Upsilon_\lambda$, there is no coincidence between the eigenvalues of \eq{autov0} and the squares of the ones of
\begin{equation}\label{autov2}
\int_I e' w' = \mu \int_I e w \qquad \forall w \in W(I).
\end{equation}
We provide some spectral results also for \eq{autov2}, giving their proof in Section \ref{pfautofz2}.

\begin{theorem}\label{autofz2}
The eigenvalues $\mu=\kappa^2$ of problem \eqref{autov2} are completely determined by the numbers $\kappa>0$ such that
$$
(i)\quad \sin (\kappa a \pi) \sin[\kappa(1-a)\pi] = 0 \quad \textrm{ or } \quad
(ii)\quad \cos(\kappa a \pi) \sin[\kappa(1-a)\pi] = 0,
$$
that is,
\begin{equation}\label{autoo2}
(i)\quad \kappa \in \frac{\mathbb{N}}{a} \cup \frac{\mathbb{N}}{1-a} \quad \textrm{ or } \quad
(ii)\quad \kappa \in \frac{2\mathbb{N}+1}{2a} \cup \frac{\mathbb{N}}{1-a}.
\end{equation}
$\bullet$ If $\kappa \notin \mathbb{N}/(1-a)$, denoting by $\chi_0$ the characteristic function of $I_0$, then:\par\noindent
- in case $(i)$, $\mu=\kappa^2$ is a simple eigenvalue associated with the odd eigenfunction
$\mathbf{D}_\kappa(x)=\chi_0(x)\sin(\kappa x)$;\par\noindent
- in case $(ii)$, $\mu=\kappa^2$ is a simple eigenvalue associated with the even eigenfunction
$\mathbf{P}_\kappa(x)=\chi_0(x)\cos(\kappa x)$.\par\noindent
$\bullet$ If $\kappa \in \mathbb{N}/(1-a)$, then the following situations may occur:\par\noindent
- if $\kappa \notin \mathbb{N}/a$ and $\kappa \notin (2\mathbb{N}+1)/2a$, then $\mu=\kappa^2$ is a double eigenvalue associated with the eigenfunctions
$$
\mathcal{D}_\kappa(x)=
\left\{
\begin{array}{ll}
\!\! \sin[\kappa(x+\pi)] & \mbox{if } x \in \overline{I}_- \vspace{0.05cm}\\
\!\! 0 & \mbox{if } x \in \overline{I}_0 \vspace{0.05cm} \\
\!\! \sin[\kappa(x-\pi)] & \mbox{if } x \in \overline{I}_+,
\end{array}
\right.
\qquad
\mathcal{P}_\kappa(x)=
\left\{
\begin{array}{ll}
\!\! \sin[\kappa(x+\pi)] & \mbox{if } x \in \overline{I}_- \vspace{0.05cm}\\
\!\! 0 & \mbox{if } x \in \overline{I}_0 \vspace{0.05cm} \\
\!\! \sin[\kappa(\pi-x)] & \mbox{if } x \in \overline{I}_+,
\end{array}
\right.
$$
respectively odd and even;\par\noindent
- if $\kappa\in\mathbb{N}/a$, then $\mu=\kappa^2$ is a triple eigenvalue associated with $\mathcal{D}_\kappa$, $\mathcal{P}_\kappa$ and $\mathbf{D}_\kappa$;\par\noindent
- if $\kappa\in(2\mathbb{N}+1)/2a$, then $\mu=\kappa^2$ is a triple eigenvalue associated with $\mathcal{D}_\kappa$, $\mathcal{P}_\kappa$ and $\mathbf{P}_\kappa$.
\end{theorem}

Notice that \eqref{autoo2}-$(i)$ corresponds to odd eigenfunctions, while \eqref{autoo2}-$(ii)$ to
even ones. The eigenfunctions of \eqref{autov2} are obtained by juxtaposing the eigenfunctions belonging to $H^1_0$ of each span, since there are no smooth junction constraints; for this reason, in general they are not $C^1$. It is also worthwhile noticing that the eigenvalues may be both simple or multiple. Simple eigenvalues are always associated with eigenfunctions being zero on the side spans. Multiple eigenfunctions exist when $\kappa \in \mathbb{N}/(1-a)$ and, in this case, the choice of the associated eigenfunctions is quite arbitrary. In the case of a double eigenvalue, we have chosen to maintain the distinction between odd and even generators, as in Theorem \ref{symmetriceigen}. For triple eigenvalues, we chose to separate the behavior on the central span from the one on the two side spans, at the price of losing regularity. Indeed, in this case $C^\infty$-eigenfunctions exist and coincide with $\mathbf{O}_\lambda$ and $\mathbf{E}_\lambda$ appearing in Theorem \ref{symmetriceigen}, but associated with the eigenvalue $\mu=\lambda^2$: for the eigenvalues $\mu=\lambda^4$ of \eq{autovsym} associated
with $\mathbf{O}_\lambda$ and $\mathbf{E}_\lambda$ one has in fact $\Upsilon_\lambda = \lambda$, see \eqref{defupsilon}. Observe that there are no triple eigenvalues if $a \notin \mathbb{Q}$.
\par This choice of the eigenfunctions is motivated by the possibility of analyzing separately the behavior on the central span, that is the most vulnerable part in bridges.
Other equivalent bases are possible: for instance, one could replace $\mathcal{D}_\kappa$ and $\mathcal{P}_\kappa$ by the functions having only one nontrivial component on $I_-$ and $I_+$, respectively. In Figure \ref{leiperboli} we depict the curves of eigenvalues in the plane $(a, \kappa)$: the bold hyperbolas correspond to $\kappa \in \mathbb{N}/(1-a)$, the dashed lines to $\kappa \in (2\mathbb{N}+1)/2a$ and the dot-dashed ones to $\kappa \in \mathbb{N}/a$.
\par
\begin{figure}[ht]
\begin{center}
\includegraphics[scale=0.75]{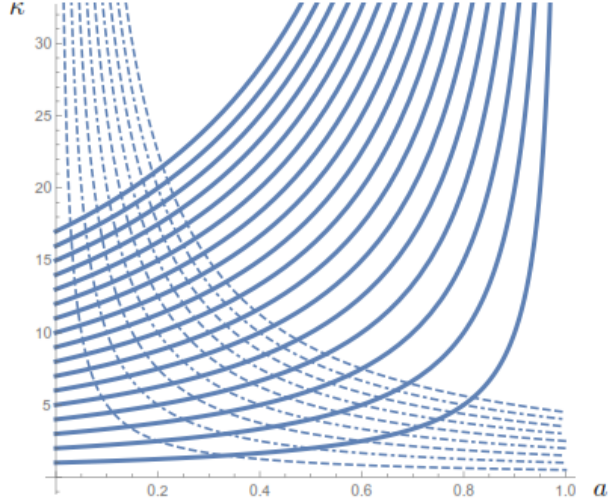}
\end{center}
\caption{A pictorial description of the curves of eigenvalues for \eqref{autov2}, in the $(a, \kappa)$-plane.}\label{leiperboli}
\end{figure}
We close the section providing a formula which identifies simple, double and triple eigenvalues. To this purpose, we introduce the real sequence
$\{\omega_n\}_{n}$ defined by
$$
\omega_n=\frac{n+1}{n+3}, \qquad n \geq -1,
$$
so that $\omega_n \to 1$ for $n \to +\infty$ and the first $\omega_n$'s are equal to $0, \frac{1}{3}, \frac{1}{2}, \frac{3}{5}, \frac{2}{3}$.
We have the following statement, which can be proved by noticing that the simple eigenvalues of \eqref{autov2} are given by the numbers $n/2a$, $n \in \mathbb{N}$.
\begin{proposition}\label{ordine}
The following facts hold:\par\noindent
$\bullet$ if $\omega_{n-1} < a \leq \omega_{n}$, then the first $n$ eigenvalues of \eqref{autov2} are simple $(n \geq 0)$; moreover, if $a=\omega_{n}$, then
the $(n+1)$-th eigenvalue is triple;\par\noindent
$\bullet$ if $\frac{1}{2m+3} \leq a < \frac{1}{2m+1}$, then the first $m$ eigenvalues of \eqref{autov2} are double, with an odd and an even eigenfunction $(m \geq 0)$; moreover, if $a=\frac{1}{2m+3}$, then the $(m+1)$-th eigenvalue is triple.
\end{proposition}

In particular, we infer that the first eigenvalue is simple for $a > 1/3$, triple for $a=1/3$, double for $a<1/3$.
Moreover, the second eigenvalue is simple for $a>1/2$ and triple for $a=1/2$.
Overall, Proposition \ref{ordine} states that the multiplicity increases on low eigenvalues when $a$ is small and, since multiplicity
plays against stability,
\begin{center}
{\bf when dealing with nonlinear problems it appears more convenient to consider large $a$.}
\end{center}

Each bold hyperbola in Figure \ref{leiperboli} carries a double eigenvalue, with one even and one odd eigenfunction, while each of the dashed and dot-dashed hyperbolas therein carries a simple
eigenvalue, alternating even and odd eigenfunctions. In Figures \ref{2tac} and \ref{22} we depict the shape of the first ten/twelve eigenfunctions
for $a=1/2=\omega_1$ and $a=14/25 \in (\omega_1, \omega_2)$, and in Table \ref{tabellax} we quote the corresponding eigenvalues: in case of
multiple eigenvalues we plot first $\mathbf{D}_\kappa$ or $\mathbf{P}_\kappa$ (if they exist), then $\mathcal{D}_\kappa$, finally
$\mathcal{P}_\kappa$.

\begin{table}[ht]
\begin{center}
{\small
\begin{tabular}{|c|c|c|c|c|c|c|c|c|c|c|c|}
\hline
$a$ & $\mu_0$ & $\mu_1$ & $\mu_2$ & $\mu_3$ & $\mu_4$ & $\mu_5$ & $\mu_6$ & $\mu_7$ & $\mu_8$ & $\mu_9$ \\
\hline
$14/25$ & 0.797194 & 3.18876 & 5.1653 & 5.1653 & 7.17474 & 12.7551 & 19.9299 & 20.6611 & 20.6611 & 28.6989  \\
\hline
$1/2$ & 1 & 4 & 4 & 4 & 9 & 16 & 16 & 16 & 25 & 36  \\
\hline
\end{tabular}
}
\caption{The least ten eigenvalues of \eq{autov2} for $a=14/25$ and $a=1/2$.}\label{tabellax}
\end{center}
\end{table}

\begin{figure}[ht!]
\begin{center}
\includegraphics[height=18mm, width=29mm]{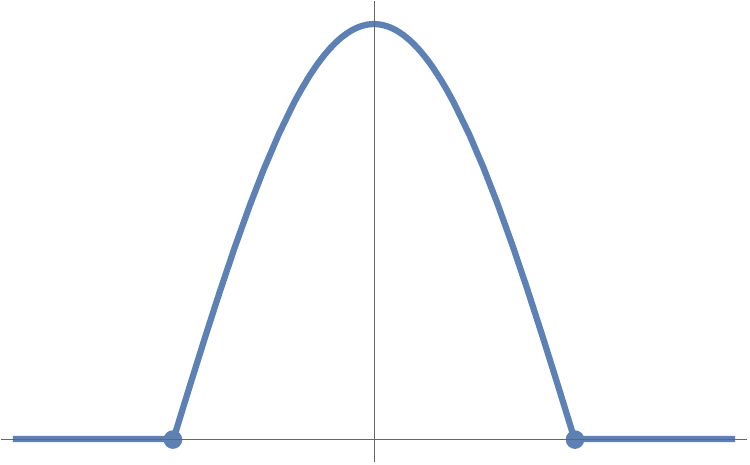}\quad\includegraphics[height=18mm, width=29mm]{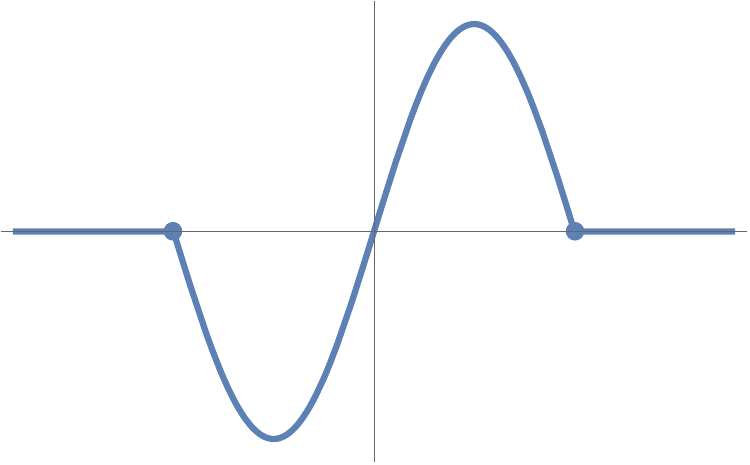}\quad
\includegraphics[height=18mm, width=29mm]{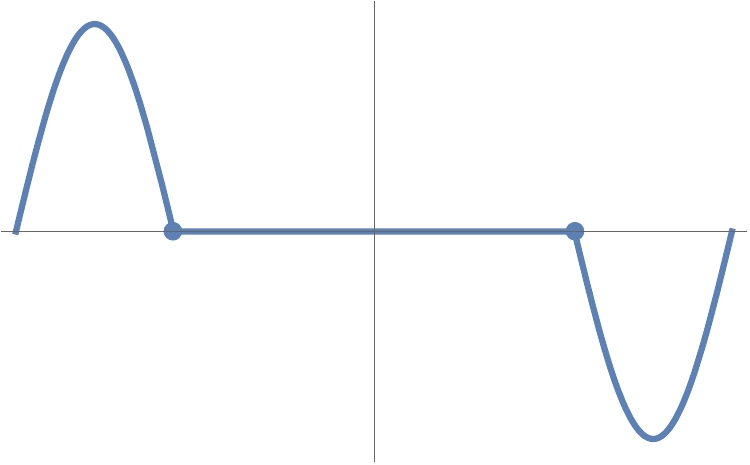}\quad\includegraphics[height=18mm, width=29mm]{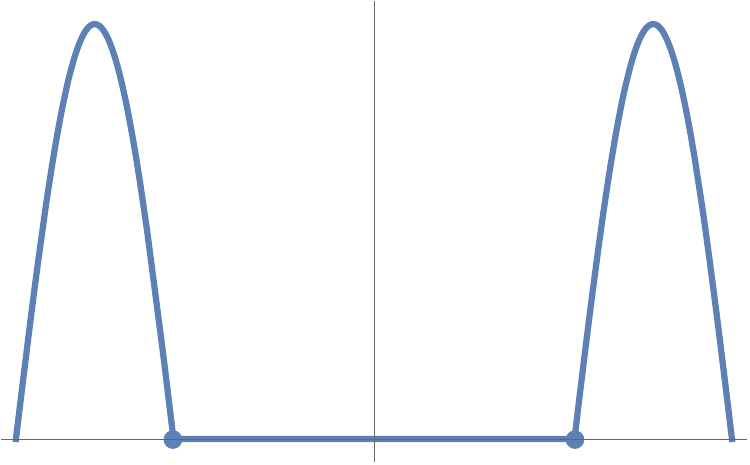}\quad
\includegraphics[height=18mm, width=29mm]{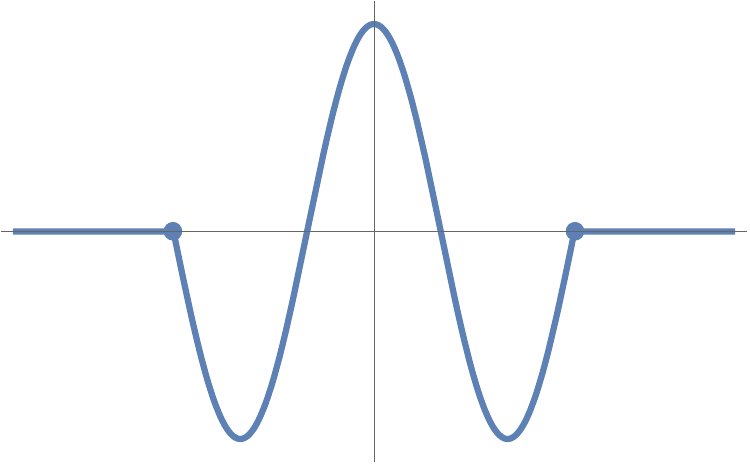}\vspace{0.1cm}\quad\includegraphics[height=18mm, width=29mm]{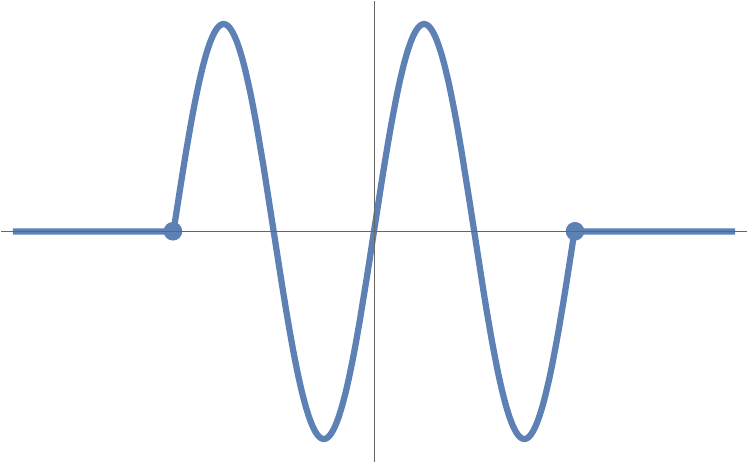}\quad
\includegraphics[height=18mm, width=29mm]{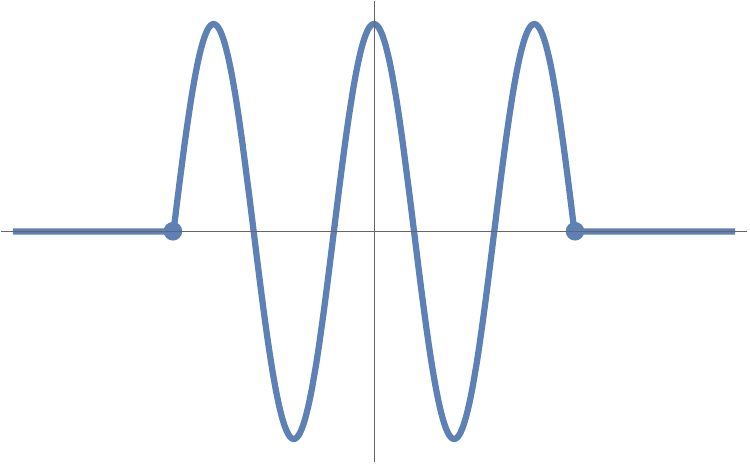}\quad\includegraphics[height=18mm, width=29mm]{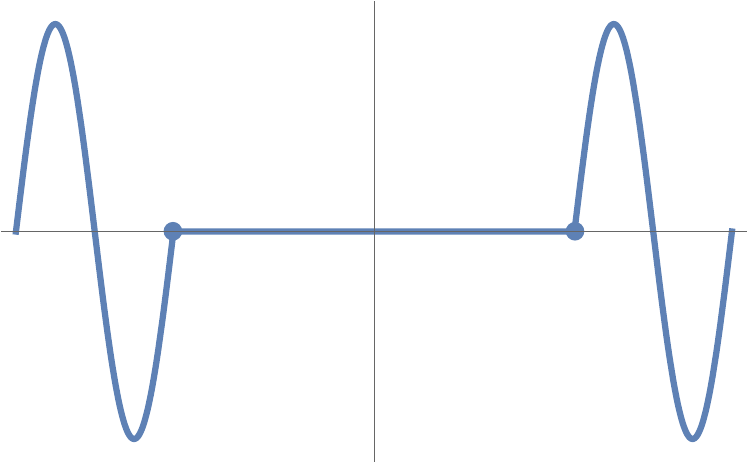}\quad
\includegraphics[height=18mm, width=29mm]{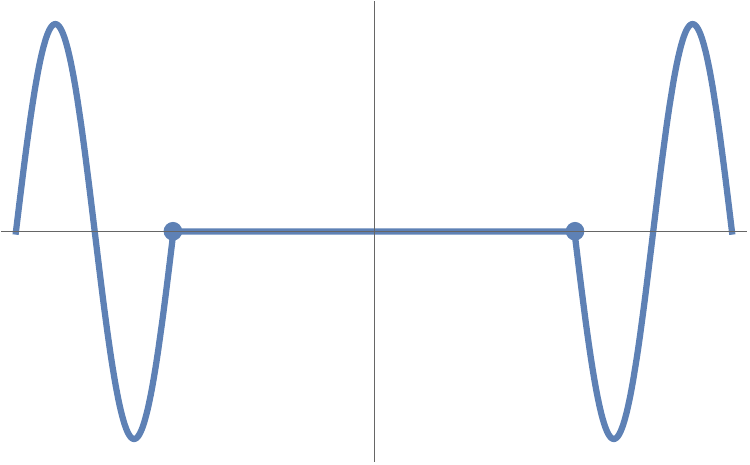}\quad\includegraphics[height=18mm, width=29mm]{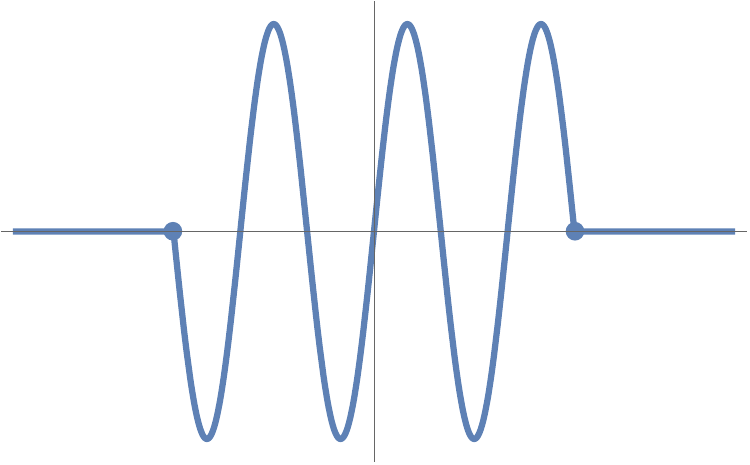}
\caption{The shape of the first ten eigenfunctions of \eqref{autov2} when $a=14/25$.}\label{2tac}
\end{center}
\end{figure}
\begin{figure}[ht!]
\begin{center}
\includegraphics[height=18mm, width=24mm]{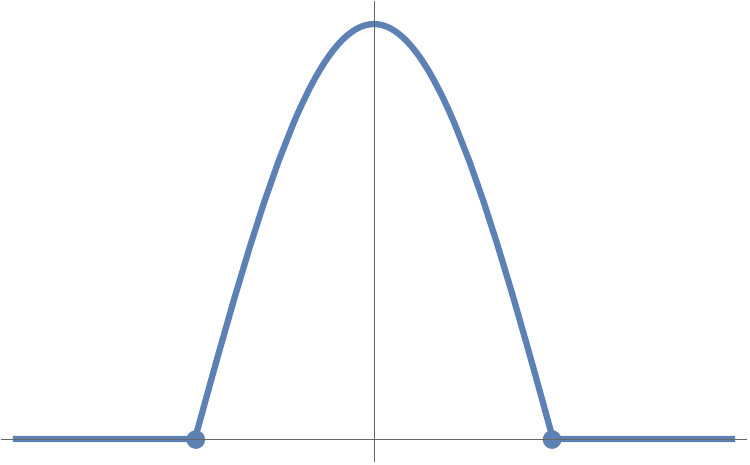}\quad\includegraphics[height=18mm, width=24mm]{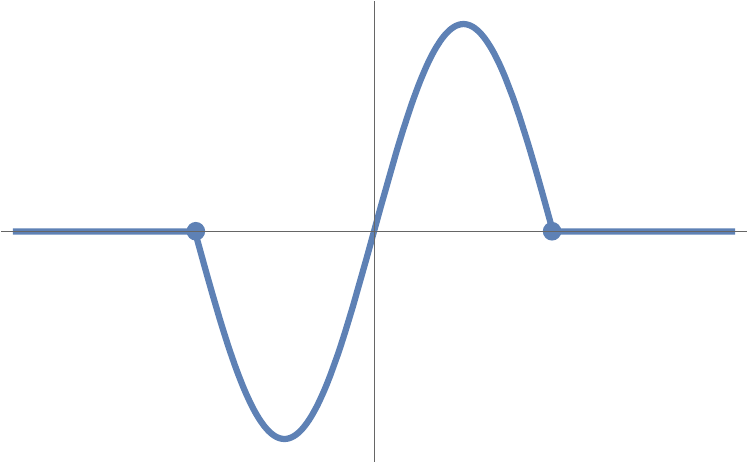}\quad
\includegraphics[height=18mm, width=24mm]{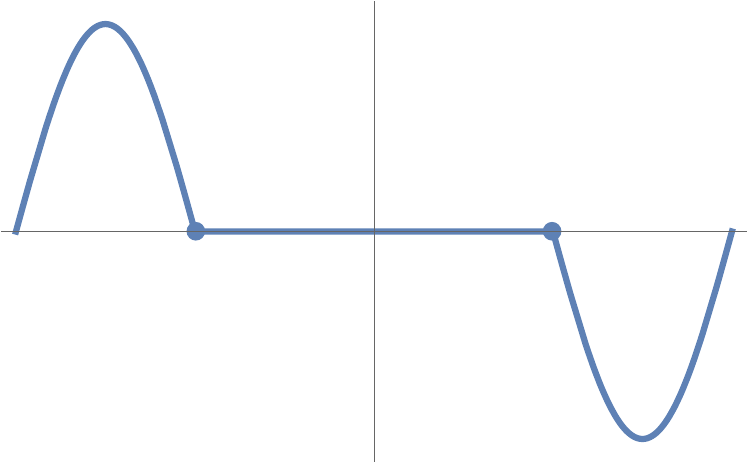}\quad\includegraphics[height=18mm, width=24mm]{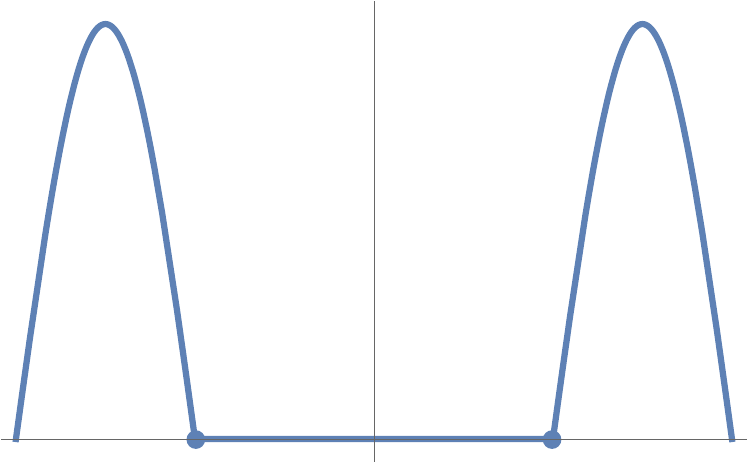}\quad
\includegraphics[height=18mm, width=24mm]{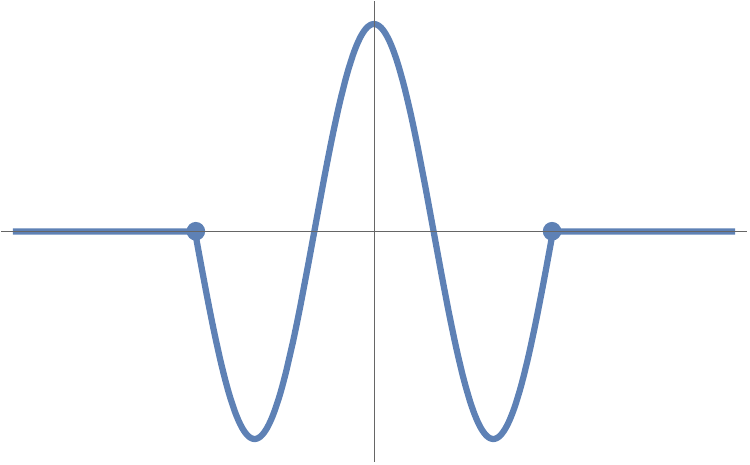}\vspace{0.1cm}\quad\includegraphics[height=18mm, width=24mm]{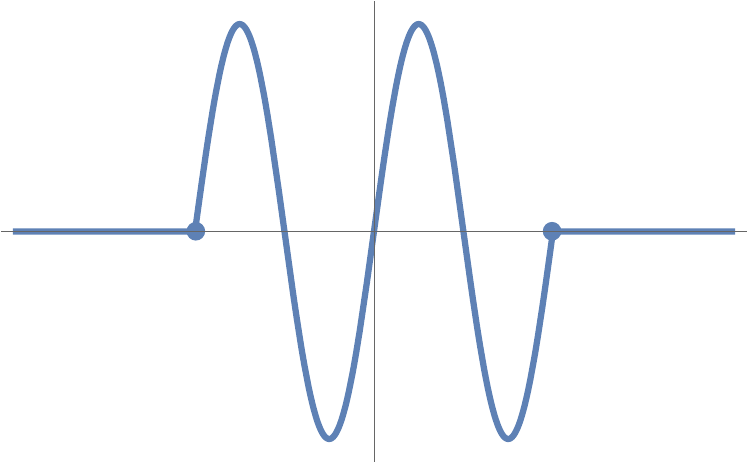}\vspace{0.1cm}\quad
\includegraphics[height=18mm, width=24mm]{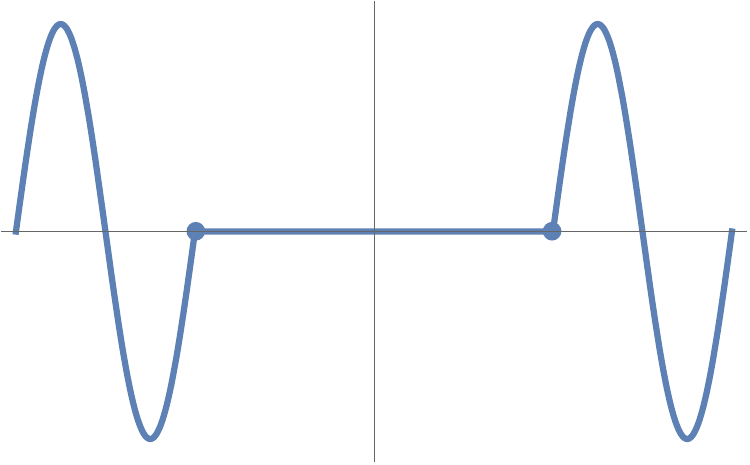}\quad\includegraphics[height=18mm, width=24mm]{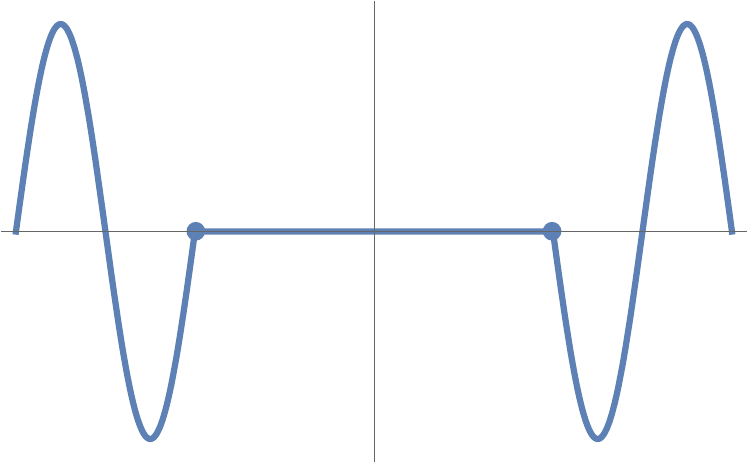}\quad
\includegraphics[height=18mm, width=24mm]{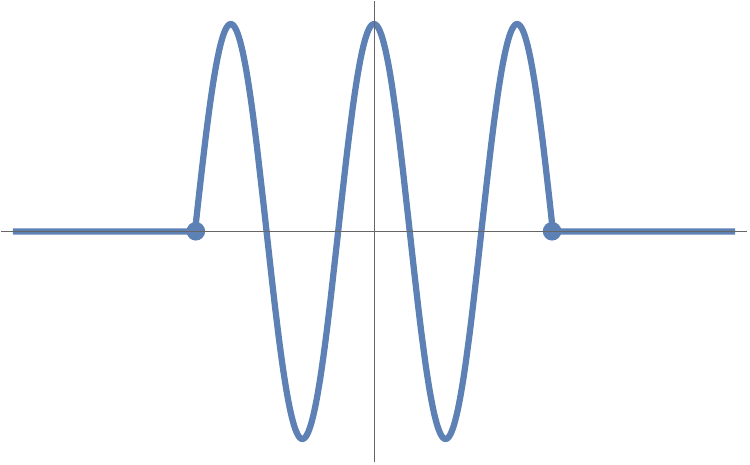}\quad\includegraphics[height=18mm, width=24mm]{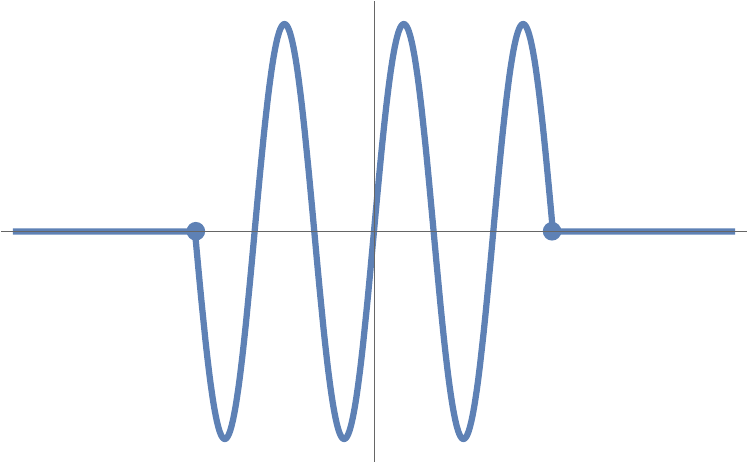}\quad
\includegraphics[height=18mm, width=24mm]{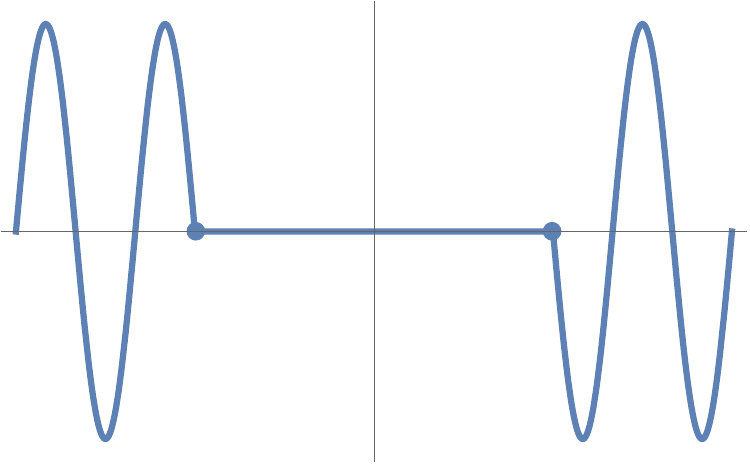}\quad\includegraphics[height=18mm, width=24mm]{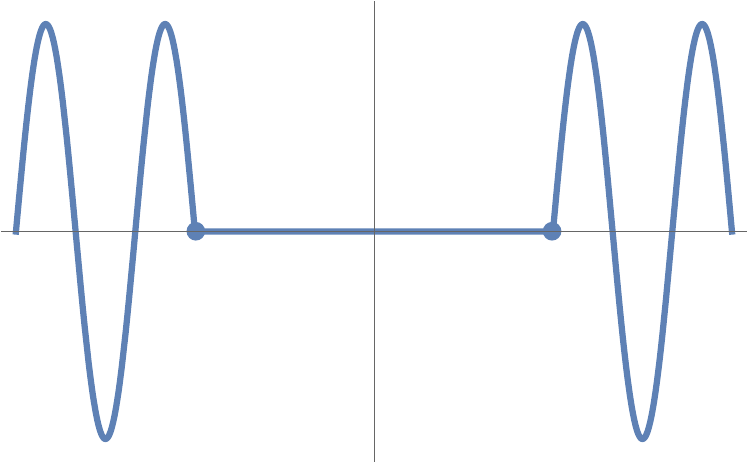}
\caption{The shape of the first twelve eigenfunctions of \eqref{autov2} when $a=1/2$.}\label{22}
\end{center}
\end{figure}

\section{Nonlinear evolution equations for symmetric beams}\label{evolutionbeam}

We consider here a symmetric beam subject to different nonlinear forces. According to suitable definitions which will be given in the course, we study the stability of solutions with a prevailing mode.
The final purposes are to identify the nonlinearity better describing the behavior of real structures and to locate the position of the piers maximizing the stability.

\subsection{From linear to nonlinear beam equations}

The general nonlinear equations describing our models read
\begin{equation}\label{prototipo}
u_{tt}+u_{xxxx}+\gamma_1\Vert u_{xx} \Vert_{L^2}^2 u_{xxxx}-\gamma_2\Vert u_{x} \Vert_{L^2}^2 u_{xx}+\gamma_3\Vert u\Vert_{L^2}^2 u + f(u) = 0
\end{equation}
and take into account nonlocal factors of bending and stretching type and local restoring forces. Equation \eq{prototipo} is considered together with the initial conditions
\neweq{icproto}
u(x,0)=u_0(x),\quad u_t(x,0)=u_1(x)\qquad x\in I.
\endeq
We focus our attention on the ``coercive'' case where
\neweq{assumptions}
\gamma_1,\gamma_2,\gamma_3\ge0,\qquad f\in{\rm Lip}_{\rm loc}(\mathbb{R}),\quad f(s)s\ge0\qquad \forall s \in\R,
\endeq
although the setting discussed in the forthcoming sections may be extended to more general nonlinearities. Note that \eq{assumptions}
includes the case where $\gamma_1=\gamma_2=\gamma_3=0$ and $f(s)=\gamma s$, namely the linear equation \eq{beamlineare}, whose stationary solutions have been studied in Section \ref{3.1}. One of our purposes is to examine the impact of each nonlinearity on the stability analysis: this will be done in Section \ref{sezioneconfronti}.

Let us make precise what is meant by weak solution of \eqref{prototipo}.

\begin{definition}\label{soluzionedebole}
We say that $u\in C^0(\R_+;V(I))\cap  C^1(\R_+;L^2(I))\cap  C^2(\R_+;V'(I))$ is a weak solution of \eqref{prototipo}-\eqref{icproto} if $u_0\in V(I)$, $u_1\in L^2(I)$ and,
for all $v\in V(I)$ and $t>0$, one has
$$
\langle u_{tt}, v \rangle_{V} +(1+\gamma_1\Vert u_{xx} \Vert_{L^2}^2)\int_I u_{xx} v'' + \gamma_2\Vert u_{x} \Vert_{L^2}^2 \int_I u_x v'
+ \gamma_3\Vert u\Vert_{L^2}^2 \int_I uv +\int_I f(u) v = 0.
$$
\end{definition}

We recall that $u(t) \in V(I)$ implies the conditions
\begin{equation}\label{bcbeamnon}
u(-\pi,t)=u(\pi,t)=u(-a\pi,t)=u(a\pi,t)=0 \qquad t \geq 0.
\end{equation}
Existence and uniqueness of a weak solution are somehow straightforward.

\begin{proposition}\label{galerkin1}
Assume \eqref{assumptions}. There exists a unique weak solution $u$ of \eqref{prototipo}-\eqref{icproto} provided that:\par
either $\gamma_1=0$, $u_0\in V(I)$ and $u_1\in L^2(I)$;\par
or $\gamma_1>0$, $u_0\in V(I)$ and $u_1\in L^2(I)$ have a finite number of nontrivial Fourier coefficients.\par\noindent
Moreover, $u\in C^2(\overline{I}\times\R_+)$ and $u_{xx}(-\pi,t)=u_{xx}(\pi,t)=0$ for all $t>0$.
\end{proposition}

By Theorem \ref{Michelle}, the proof can be carried on with a Galerkin method, since the assumptions \eq{assumptions}
yield a coercive stationary problem: it can be obtained by combining arguments from \cite[Theorem 1]{GarGaz} (local case) and \cite{dickey0} (nonlocal case).
The case $\gamma_1>0$ needs stronger assumptions on the initial data because the Galerkin procedure does not allow to control the related nonlocal term. 
Due to the widely discussed lack of regularity at the piers, see Section \ref{functional}, one cannot expect the solutions to be strong
or classical. One reaches the $C^2$-regularity by arguing as in \cite[Lemma 2.2]{HolNec10}, see also the proof of Theorem~\ref{regular}.\par
In order to simplify our analysis, we consider the eigenfunctions of problem \eqref{autov0} normalized in $L^2$, sorting
them according to the order of the corresponding (simple) eigenvalues, as in Theorem \ref{constant}:
$$
\{e_n\}_{n \in \mathbb{N}}, \qquad\|e_n\|_{L^2}=1.
$$
The initial data associated with \eq{prototipo} may be expanded in Fourier series with respect to the basis $\{e_n\}_n$:
\begin{equation}\label{leiniziali}
u(x, 0)=\sum_{n \in \mathbb{N}} \alpha_n e_n, \quad u_t(x, 0)=\sum_{n \in \mathbb{N}} \beta_n e_n.
\end{equation}

The study of \eqref{prototipo} requires to modify some features of linear problems
that are lost when dealing with nonlinear problems; let us briefly comment about this point by restricting our attention to a particular example.
If in \eq{prototipo} we take $\gamma_1=\gamma_2=\gamma_3=0$
and $f(u)=\eps u^3$ for some $\eps\ge0$, then Definition \ref{soluzionedebole} states that a weak solution satisfies
\neweq{linnonlin}
\langle u_{tt}, v \rangle_{V} +\int_I u_{xx} v'' +\eps\int_I u^3v = 0\qquad\forall v\in V(I),\quad t>0.
\endeq
We take initial conditions concentrated on a sole mode, namely
\neweq{iconemode}
u(x,0)=\alpha e_n(x),\quad u_t(x,0)=0\qquad x\in I,
\endeq
for some $\alpha>0$ and some $e_n$ eigenfunction of \eq{autovsym} with associated eigenvalue $\mu_n=\lambda_n^4$. In the {\em linear case} $\eps=0$, the solution of
\eq{linnonlin}-\eq{iconemode} is given by $u(x,t)=\alpha\cos(\lambda_n^2 t)e_n(x)$, which is periodic-in-time and proportional to the initially excited
mode $e_n$.\par
But the nonlinear case $\eps>0$ is extremely more complicated and a full understanding of it is still missing.
The milestone paper by Rabinowitz \cite{rabin} started the systematic study of the existence of periodic solutions for the nonlinear wave equation
\begin{equation}\label{nonlinw}
u_{tt}-u_{xx}+f(x,u)=0\qquad x\in(0,\pi), \quad t > 0 , \quad \qquad u(0,t)=u(\pi,t)=0\qquad t \geq 0.
\end{equation}
Prior to \cite{rabin}, only perturbation techniques were available and the nonlinearity was assumed to be small in some sense. Rabinowitz proved that, under
suitable assumptions on the nonlinearity $f$, this equation admits periodic solutions of any rational period. More recent works, see e.g. \cite{berti} and
\cite{gentile}, imply that there exist periodic solutions of \eq{nonlinw} for all $T$ belonging to sets of almost full measure. Periodic solutions were also
found for the related nonlinear beam equation (with no piers)
$$
u_{tt}+u_{xxxx}+f(x,u)=0\quad x\in(0,\pi), \quad t>0 ,\,\qquad u(0,t)=u(\pi,t)=u_{xx}(0,t)=u_{xx}(\pi,t)=0\quad t\geq 0 ,
$$
see \cite{lee,liu1,liu2}. For this equation, periodic solutions of given period appear difficult to be determined, see
\cite{arioli} for some computer assisted results on this topic. Indeed, it is no longer true that a solution with initial conditions \eq{iconemode} is proportional
to the mode $e_n$, because the energy immediately spreads on infinitely many other modes. A further difference between linear and nonlinear equations is that the latter
may display some {\em instability}, meaning by this that small perturbations may lead to huge differences in the outcomes (in a suitable sense).\par
In order to overcome these difficulties, we argue as follows. First, we remark that the two parameters $\eps$ and $\alpha$  - respectively measuring the strength of the nonlinearity in
\eq{linnonlin} and the initial amplitude in \eq{iconemode} - play a similar role. To see this, set $w=\sqrt{\eps}u$ in \eqref{linnonlin}, so that $w$ satisfies
\begin{equation}\label{riscal}
\left\{\begin{array}{l}
\langle w_{tt}, v \rangle_V +\int_I w_{xx} v'' +\int_I w^3v = 0\qquad\forall v\in V(I), \quad t>0\\
w(x,0)=\alpha\sqrt{\eps}\, e_n(x),\quad w_t(x,0)=0\qquad x\in I,
\end{array}\right.
\end{equation}
where the equation is independent of $\eps$ and the initial amplitude is now $\alpha\sqrt{\eps}$. Therefore, if either $\alpha$ or $\eps$ is small, we are close to a linear regime and \eqref{linnonlin}, as well as \eqref{riscal}, can be tackled with perturbative methods and KAM theory (see, e.g., \cite{Kuk, Way}).
But \eqref{riscal} also shows that for increasing $\alpha$ and $\eps$ the equation will behave more and more nonlinearly.\par
The second step is based on the simple remark that if one chooses initial data yielding a periodic solution of \eq{linnonlin},
then the solution does not spread energy since it coincides with the periodic solution. But, as we already mentioned, it is not easy to find explicit periodic solutions. Therefore, we prefer to still focus on the initial conditions \eq{iconemode}, which give rise to a periodic solution for $\varepsilon=0$, and study in detail
how the energy of \eq{linnonlin} initially concentrated on the mode $e_n$ spreads among the other modes. In Section \ref{4.4}, we show that this may
happen in two different ways, which need to be characterized sharply in order to define the stability of the beam.
It turns out that
two main ingredients influence the instability: the initially excited mode $e_n$ and its amplitude $\alpha$. Our purpose is to study the instability depending
on the mode, on its initial amplitude, and on the position of the piers. Then we aim at finding the placement of the piers maximizing the stability in a suitable sense: this will be done in Sections \ref{nonmischia} and \ref{nonlinevol}.

\subsection{Prevailing modes and invariant subspaces}\label{length}

The notion of instability we are interested in, is restricted to a particular class of solutions of \eq{prototipo}. This class is motivated by the behavior of
actual bridges. From the Report on the TNB collapse \cite[p.20]{ammvkwoo} we learn that, in the months prior to the collapse,
\begin{center}
{\bf one principal mode of oscillation prevailed \\ and the modes of oscillation frequently changed.}
\end{center}
This means that, even if the oscillations were governed by (disordered) forcing and damping, the oscillation itself was quite simple to describe,
see also the sketch in Figure \ref{zeroTNB}. We characterize the solutions of \eq{prototipo} possessing a \emph{prevailing mode} as the solutions having initial
data concentrating most of the initial energy on a sole mode.

\begin{definition}\label{prevalente}
Let $0 < \eta < 1$. We say that a weak solution of \eqref{prototipo}-\eqref{icproto}, with initial data \eqref{leiniziali}, has the $j$-th mode $\eta$-\emph{prevailing} if $j$ is the only integer for which:
\begin{equation}\label{prevalentec2}
\sum_{n\neq j}(\alpha_n^2 + \beta_n^2)< \eta^4 (\alpha_j^2+\beta_j^2).
\end{equation}
For this solution, all the other modes $k \neq j$ are called \emph{residual}. We denote by $\mathbb{P}_j$ the set of all the solutions of \eqref{prototipo} with $\eta$-prevailing mode $j$.
\end{definition}
We will comment on the role of the parameter $\eta$ in Section \ref{4.4}. By now, we only anticipate that a reasonable choice of $\eta$ is given by $\eta=0.1$; from a physical point of view, this highlights an initial difference between prevailing and residual modes of two orders of magnitude. In Table \ref{instabsuperL8} in Section \ref{sezesperim} we also consider other values of $\eta$. Obviously, not all the solutions of \eqref{prototipo} have an $\eta$-prevailing mode.
In the sequel, we restrict our attention to initial conditions \eqref{leiniziali} without kinetic component, that is,
\begin{equation}\label{initnocin}
u(x, 0)=\sum_{n \in \mathbb{N}} \alpha_n e_n, \quad u_t(x, 0)=0,
\end{equation}
for which \eqref{prevalentec2} becomes
\begin{equation}\label{prevalentec}
\sum_{n\neq j}\alpha_n^2 < \eta^4 \alpha_j^2.
\end{equation}
When a fluid hits a bluff body, its flow is modified and goes around the body. Behind the body, due to the lower pressure, the flow creates vortices which
appear somehow periodic-in-time and whose frequency depends on the velocity of the fluid. The asymmetry of the vortices generates a forcing lift which starts the
vertical oscillations of the structure. The frequencies of the vortices are divided in ranges: for each range of frequencies, there exists a vibrating mode of
the structure which captures almost all the input of energy from the vortices \cite{bonedegaz}, giving rise to what we call a solution with a prevailing mode.
This is a simplified explanation, further vortices and more complicated phenomena may appear: however, up to some minor details, this explanation is shared by the
whole engineering community and it has been studied with some precision in wind tunnel tests, see e.g.\ \cite{larsen,scott}. Our analysis starts at this point,
that is, when the oscillation is maintained in amplitude by a somehow perfect equilibrium between the input of energy from the wind and the structural dissipation. Whence, as long as the aerodynamics is involved the frequency of the vortices determines the prevailing mode, while in an isolated
system the prevailing mode is determined by the initial conditions.
This justifies the choice of the initial conditions \eqref{initnocin}.
\par
The behavior of the solutions of \eq{prototipo} having a prevailing mode
strongly depends on the parameters therein. Two fairly different cases should be distinguished:
\neweq{versus}
\gamma_2=0\mbox{ and }f(s)=\gamma s\qquad\mbox{versus}\qquad\gamma_2\neq0\mbox{ or $f$ acts nonlinearly}.
\endeq
By $f$ nonlinear, we mean that there exist no neighborhood of $s=0$ and no $\gamma \geq 0$ for which $f(s) = \gamma s$.
In order to emphasize the difference between the two situations in \eqref{versus}, let us introduce the notion of invariant space.

\begin{definition}\label{invarianti}
A subspace $X \subset V(I)$ is called \emph{invariant} for \eqref{prototipo} under conditions \eqref{initnocin} if
$$
u(x, 0) \in X \quad \Longrightarrow \quad u(x,t) \in X\textrm{ for every }t>0.
$$
\end{definition}

Of course, $\{0\}$ and $V(I)$ are trivial invariant subspaces. However, in the first case in \eq{versus}
there are many more invariant subspaces than in the second one.

\begin{proposition}\label{allinvariant}
The following statements hold:\par\noindent
$\bullet$ if $\gamma_2=0$ and $f(s)=\gamma s$ for some $\gamma\ge0$, then any subspace $X\subset V(I)$ is invariant for \eqref{prototipo};\par\noindent
$\bullet$ if $\gamma_2 \neq 0$ or $f$ is nonlinear, then there are no nontrivial finite-dimensional subspaces for \eqref{prototipo};\par\noindent
$\bullet$ if $f$ is odd, then
$X=\langle e_{2n}\rangle_{n \in \mathbb{N}}$ and $X=\langle e_{2n+1}\rangle_{n \in \mathbb{N}}$
are invariant subspaces for \eqref{prototipo}.
\end{proposition}

The first statement is straightforward and will be clarified in the course, see Section \ref{twoms} below. The second statement may also be obtained with some computations, see Section \ref{howmix} below for the details. The third statement follows by direct computation as well,
see \cite{GarGaz}.
\par
Proposition \ref{allinvariant} states that the second alternative in \eq{versus} forces the invariant subspaces for \eqref{prototipo} to
be infinite-dimensional, whereas in the first case there are finite-dimensional invariant subspaces, including one-dimensional subspaces.
This striking difference is even more evident if we fix some $j\in \mathbb{N}$ and some $\alpha_j\neq0$, and we take initial data as in \eqref{iconemode}:
\neweq{datiunici}
u(x,0)=\alpha_je_j,\quad u_t(x,0)=0.
\endeq
Then, in the first case of \eq{versus}, we may simply find the solution of \eq{prototipo} with $j$-th prevailing mode, having
the form $u(x,t)=\varphi_j(t)e_j(x)$, for some $\varphi_j \in C^2(\mathbb{R}_+)$. On the contrary, in the second case of \eqref{versus}, the same initial conditions give rise to a solution with infinitely
many nontrivial Fourier coefficients $\varphi_n(t)$, as soon as $t>0$.

\subsection{Nonlinear instability and critical energy thresholds}\label{4.4}

We are interested in the following notion of nonlinear instability, introduced in \cite{GarGaz}.

\begin{definition}\label{unstable}
Let $T_W > 0$. We say that a weak solution $u \in \mathbb{P}_j$ of \eqref{prototipo}
is \emph{unstable} before time $T>2T_W$ if there exist
a residual mode $k$ and a time instant $\tau$ with $2T_W < \tau < T$  such that
\neweq{grande}
(i)  \ \ \Vert \varphi_k \Vert_{L^\infty(0, \tau)} > \eta \Vert \varphi_j \Vert_{L^\infty(0, \tau)}
\qquad \mbox{ and } \qquad
(ii) \ \ \frac{\Vert \varphi_k \Vert_{L^\infty(0, \tau)}}{\Vert \varphi_k \Vert_{L^\infty(0, \tau/2)}} > \frac{1}{\eta}
\endeq
(where $\eta$ is the number appearing in Definition \ref{prevalente}). We say that $u$ is \emph{stable} until time $T$ if, for any $k \neq j$, \eqref{grande} is not fulfilled for any
$\tau \in (2T_W, T)$.
\end{definition}

Clearly, this is a ``numerical definition'', which seems difficult to be characterized analytically with great precision. Roughly speaking, this notion of instability extends in a quantitative way the classical linear instability to a nonlinear context. In all our discussion, we will set $\eta=0.1$; with this choice, instability is somehow identified with a sufficiently abrupt change in the order of magnitude of the $k$-th Fourier component of $u \in \mathbb{P}_j$. In Section \ref{algoritmo},
we explain in detail how to fulfill it. As pointed out in \cite{GarGaz}, Definition \ref{unstable} is necessary to
\begin{center}
{\bf distinguish between physiological energy transfers and instability.}
\end{center}
The former occur at any energy level, are unavoidable and start instantaneously, being governed by the natural interactions between modes created by the particular nonlinearity in \eqref{prototipo}. The latter occur and become visible thanks to a (possibly delayed) sudden and violent transfer of energy from some modes to other ones, only occurring at high energies.
Condition $(i)$ in \eqref{grande} refers to the significance of the energy transfer,
while $(ii)$ refers to its abruptness. The interval $[0,T_W]$ in Definition \ref{unstable} represents a transient phase corresponding to the so-called
\emph{Wagner effect} \cite{Wag}, consisting in a time delay in the appearance of the response to a sudden change of the action of an external input in a system.
We will determine a reasonable value for $T_W$ in Section \ref{lininst} (formula \eq{sceltawagner}), after recalling the notion of linear instability.
After having fixed $T_W > 0$ and $T>2T_W$, for all $j \in \mathbb{N}$ the most natural way to define the
\emph{$j$-th energy threshold} of instability before time $T$ for \eqref{prototipo} would be the following:
\neweq{Eja}
E_j(a)=\inf_{u \in \mathbb{P}_j} \{\mathcal{E}(u) \mid u \textrm{ is unstable before time } T\} \qquad \forall a\in(0,1),
\endeq
where
$$
\mathcal{E}(u)= \frac{\Vert u_t \Vert^2_{L^2}}2 +\frac{\Vert u_{xx} \Vert_{L^2}^2}2 +\gamma_1\frac{\Vert u_{xx} \Vert_{L^2}^4}4+
\gamma_2\frac{\Vert u_{x} \Vert_{L^2}^4}4+\gamma_3\frac{\Vert u \Vert_{L^2}^4}4+ \int_I F(u)
$$
is the constant total energy associated with \eqref{prototipo} and $F$ is the primitive of $f$.
But then one may not be able to define the {\em critical energy threshold} of \eq{prototipo} as $\inf_jE_j(a)$,
since this number could be zero: this fact will be further motivated in Section \ref{againa}. A possible way out is then to restrict the attention
to a finite number of modes, a fairly
common procedure in classical engineering literature. Bleich-McCullough-Rosecrans-Vincent \cite[p.23]{bleich} write that {\em out of the
infinite number of possible modes of motion in which a suspension bridge might vibrate, we are interested only in a few, to wit: the ones
having the smaller numbers of loops or half waves}. The justification of this approach has physical roots: Smith-Vincent \cite[p.11]{tac2}
write that {\em the higher modes with their shorter waves involve sharper curvature in the truss and, therefore, greater bending moment at
a given amplitude and accordingly reflect the influence of the truss stiffness to a greater degree than do the lower modes}. This also occurs
in our setting: if $j\to\infty$, this means extremely small oscillations of the prevailing mode. Therefore, the mathematical lack of
compactness corresponds to a physically irrelevant phenomenon involving tiny oscillations of high modes, which can thus be neglected.\par
So, let us fix a finite number of modes. From the Federal Report \cite{ammvkwoo} we learn that, at the TNB, oscillations with more than 10
nodes on the three spans were never seen, see also Figure \ref{zeroTNB}. This means that 12 modes are more than enough to approximate the motion
of real bridges and of beams having similar structural responses. From a mathematical point of view, this finite-dimensional approximation is fully
justified by the Galerkin procedure that can be used to study \eq{prototipo}, see \cite{GarGaz} for more details; moreover, it enables us to
use Proposition \ref{galerkin1} also in the case $\gamma_1>0$. Therefore, we consider the space
$$
V_{12}(I):=\mbox{span}\{e_i;\, i=0,...,11\}
$$
and we point out that, if the first alternative in \eqref{versus} holds, then $V_{12}(I)$ is invariant, see Proposition \ref{allinvariant}.
Then, we seek \emph{approximated solutions} of the projection of equation \eq{prototipo} onto $V_{12}(I)$
in the form
\neweq{tronca}
U^A(x,t)= \sum_{n=0}^{11}\varphi_n(t)e_n(x),
\endeq
obtaining the system of 12 nonlinear ODE's ($n=0, \ldots, 11$)
\neweq{finitodim}
\int_I U^A_{tt} e_n +(1+\gamma_1\Vert U^A_{xx} \Vert_{L^2}^2)\int_I U^A_{xx} e_n'' + \gamma_2\Vert U^A_{x} \Vert_{L^2}^2 \int_I U^A_x e_n'
+ \gamma_3\Vert U^A\Vert_{L^2}^2 \int_I U^Ae_n +\int_I f(U^A)e_n = 0.
\endeq
These equations are derived from Definition \ref{soluzionedebole}, by taking $U^A$ and $v$ both in $V_{12}(I)$ and not in the whole space $V(I)$.
The function $U^A$ is an approximation of a solution $u$ of \eq{prototipo}: if instead of just 12 modes we considered an arbitrary number of modes
$N$ and we let $N\to\infty$, then it would be $U^A\to u$ in a suitable sense, see \cite[Theorem 2]{GarGaz}.\par
For each $j=0,...,11$, we take initial conditions being the projection of \eq{initnocin} on $V_{12}(I)$ satisfying \eq{prevalentec}:
$$
U^A(x,0)=\sum_{n=0}^{11}\alpha_n e_n, \quad U^A_t(x, 0)=0,\qquad\sum_{n\neq j}\alpha_n^2 < \eta^4 \alpha_j^2.
$$
Then, instead of \eq{Eja}, we define
\begin{equation}\label{energiamodoj}
E_j(a)=\inf_{U^A\in \mathbb{P}_j} \{\mathcal{E}_j(U^A) \mid U^A \textrm{ is unstable before time } T\} \qquad \forall a\in(0,1),
\end{equation}
where $\mathcal{E}_j(U^A):=\mathcal{E}(U^A_j)$, $U^A_j$ being the solution of \eqref{finitodim} having initial data \eqref{datiunici}
(we are thus omitting all the energies involving residual modes; since the residual modes are small, such energies are practically negligible and it is $\mathcal{E}_j(U^A) \approx \mathcal{E}(U^A)$).
We are now ready to define what we mean by \emph{energy threshold of instability}.

\begin{definition}\label{threshold}
The energy threshold of instability before time $T$ for \eqref{prototipo} is defined as
\begin{equation}\label{Edodici}
\mathbb{E}_{12}(a) = \inf_{0 \leq j \leq 11} E_j(a).
\end{equation}
\end{definition}

Some remarks are in order.
The dependence of $E_j$ on $a$ is due to the boundary-internal conditions \eqref{bcbeamnon}, which modify the space $V(I)$ where \eqref{finitodim}
is settled. Our purpose is to detect the ``optimal'' value for $a$, for which the threshold $\mathbb{E}_{12}(a)$ is as large as possible,
since this value is a measure of the global stability of the approximated beam, truncated on its lowest 12 modes. When \eq{finitodim} is
governed by energies $\mathcal{E}<\mathbb{E}_{12}(a)$, the oscillations are stable and approximate solutions maintain any prevailing mode before time $T$, whereas
the oscillations of the residual ones remain small and regular. This is no longer true when $\mathcal{E}>\mathbb{E}_{12}(a)$, since there may be a sudden transfer of
energy from the prevailing mode to residual modes. If every approximate solution $U^A\in \mathbb{P}_j$ of \eqref{finitodim} is stable
before time $T$, then $E_j(a)=+\infty$; if this occurs for every $j$, one has $\mathbb{E}_{12}(a)=+\infty$. In this case, the considered beam is totally stable before time $T$ and nothing else
has to be said, otherwise we can prove that the map $a \mapsto \mathbb{E}_{12}(a)$ is continuous (see Section \ref{pfsogliacontinua}).

\begin{theorem}\label{sogliacontinua}
Fix $T > 0$ and assume that $\mathbb{E}_{12}(\bar{a}) < +\infty$ for a certain $\bar{a}\in(0,1)$. Then, $a \mapsto \mathbb{E}_{12}(a)$ is continuous in $\bar{a}$ for \eqref{finitodim}.
\end{theorem}

In our numerical experiments we always found that $\mathbb{E}_{12}(a)<\infty$, which leads to the conjecture that $\mathbb{E}_{12}\in C^0([0,1])$.
\par
The stability analysis is fairly different for the two cases in \eq{versus}, which will be studied separately. For the first case, Proposition \ref{allinvariant} states that every subspace of $V(I)$ is invariant, meaning that the modes do not mix or, equivalently, that there is only a low physiological energy transfer. In the second case of \eqref{versus}, the modes mix and there are significant physiological energy transfers. Section \ref{nonmischia} is devoted to the former situation, which is much simpler and will be used to better understand the essence of Definition \ref{unstable} through the notion of {\em linear stability}. Section \ref{nonlinevol} is instead dedicated to the latter situation.

\subsection{Equations that do not mix the modes}\label{nonmischia}

From now on, depending on the context, we denote the eigenfunctions of \eq{autovsym} both by $e_\lambda$ (with associated eigenvalue $\mu=\lambda^4$) and by $e_n$
(with eigenvalue $\mu_n=\lambda_n^4$).

\subsubsection{Reduction to a two-modes system}\label{twoms}

We assume here to be in the first case of \eq{versus}, namely we assume that \eq{assumptions} holds with
\neweq{nonsocomechiamarla}
\gamma_1+\gamma_3 > 0, \quad \gamma_2=0\quad\mbox{and}\quad f(s)\equiv0.
\endeq
This situation will help us to determine a suitable value for the Wagner time $T_W$, see \eq{sceltawagner} in Section \ref{lininst}.
If we assume \eq{nonsocomechiamarla}, then \eq{prototipo} becomes
\begin{equation}\label{prototipo2}
u_{tt}+u_{xxxx}+\gamma_1\Vert u_{xx} \Vert_{L^2}^2 u_{xxxx}+\gamma_3\Vert u\Vert_{L^2}^2 u = 0,
\end{equation}
to be intended in the weak sense of Definition \ref{soluzionedebole}.
We will treat separately the cases $(\gamma_1,\gamma_3)=(1,0)$ and $(\gamma_1,\gamma_3)=(0,1)$ but, at this stage, we deal with any
nontrivial couple $(\gamma_1,\gamma_3)$ satisfying \eqref{assumptions}. Assuming \eq{nonsocomechiamarla}, we consider the initial conditions
\neweq{datiunici2}
u(x,0)=\delta e_\lambda,\quad u_t(x,0)=0,
\endeq
for some $\delta>0$ and some ($L^2$-normalized) eigenfunction $e_\lambda$ of \eq{autovsym}, relative to the eigenvalue $\mu=\lambda^4$. Then we have the following statement.

\begin{proposition}\label{partform}
The unique weak solution of \eqref{prototipo2} satisfying \eqref{datiunici2} has the form
\neweq{separazione}
u_\lambda(x,t)=W_\lambda(t)e_\lambda(x),
\endeq
where $W_\lambda$ solves
\neweq{duffduff}
\ddot{W}_\lambda(t)+\lambda^4W_\lambda(t)+(\gamma_1\lambda^8+\gamma_3)W_\lambda(t)^3=0,\quad W_\lambda(0)=\delta,\quad\dot{W}_\lambda(0)=0.
\endeq
\end{proposition}

Proposition \ref{partform} may be obtained by combining the uniqueness statement in Proposition \ref{galerkin1} with the orthogonality of the $e_\lambda$'s, both in $L^2$ and in $V$. Together with Proposition \ref{allinvariant}, this result well explains why, under the assumptions \eq{nonsocomechiamarla}, there is no need
to consider approximate solutions such as \eq{tronca}: any subspace of $V(I)$ is invariant, in particular the space $V_{12}(I)$.\par
We call a function $u_\lambda$ in the form \eq{separazione} a \emph{$\lambda$-nonlinear-mode} of \eq{prototipo2}.
In fact, for any eigenvalue $\lambda^4$ there exist infinitely many $\lambda$-nonlinear-modes, one for each value of $\delta$; they are
not proportional to each other and they have different periods. Their shape is described by the solution $W_\lambda$ of \eq{duffduff},
which depends on $\delta$: for this reason, with an abuse of language, we call $W_\lambda$ the $\lambda$-nonlinear-mode of \eq{prototipo2}
of amplitude $\delta$.\par
Motivated again by Proposition \ref{allinvariant}, we now consider solutions of \eq{prototipo2} having two active nonlinear modes. These solutions allow a
simple and precise characterization of the stability of one nonlinear mode with respect to another mode.
We consider solutions of \eq{prototipo2} in the form
\neweq{form3}
u(x,t)=w(t)e_\lambda(x)+z(t)e_\rho(x)
\endeq
for some different eigenfunctions $e_\lambda$ and $e_\rho$ of \eq{autovsym}. Notice that, in view of the mutual orthogonality of the eigenfunctions
(in $L^2$ and in $V$), these solutions appear whenever the initial data (potential and kinetic) are completely concentrated on $e_\lambda$ and
$e_\rho$, since the vector space generated by $e_\lambda$ and $e_\rho$ is invariant for the dynamics of \eqref{prototipo2}, see Proposition
\ref{allinvariant}. After inserting \eq{form3} into \eq{prototipo2}, we reach the following nonlinear system of ODE's:
\neweq{cwgeneral}
\left\{\begin{array}{l}
\ddot{w}(t)+\lambda^4w(t)+\Big((\gamma_1\lambda^8+\gamma_3)w(t)^2+(\gamma_1\lambda^4\rho^4+\gamma_3)z(t)^2\Big)w(t)=0\\
\ddot{z}(t)+\rho^4z(t)+\Big((\gamma_1\lambda^4\rho^4+\gamma_3)w(t)^2+(\gamma_1\rho^8+\gamma_3)z(t)^2\Big)z(t)=0,
\end{array}\right.
\endeq
that we associate with the initial conditions
\neweq{initialsyst3}
w(0)=\delta>0,\quad z(0)=z_0,\quad\dot{w}(0)=\dot{z}(0)=0.
\endeq

If $z_0=0$, then the solution of \eq{cwgeneral}-\eq{initialsyst3} is $(w,z)=(W_\lambda,0)$ where $W_\lambda$ is the
solution of \eq{duffduff}, since in this case there is no physiological energy transfer towards modes that are initially zero. In order to characterize the stability of the nonlinear mode $W_\lambda$ with respect to
the nonlinear mode $W_\rho$ we argue as follows. We take initial data in \eq{initialsyst3} such that
$$
0<|z_0|<\eta^2\delta\quad\mbox{for some }\eta\in(0,1).
$$
This means that the energy of the autonomous system \eq{cwgeneral} is initially almost totally concentrated on the $\lambda$-nonlinear-mode, which thus plays the
role of the prevailing mode, according to Definition \ref{prevalente}.
One should then wonder whether this remains true for all time $t>0$ for the solution of \eqref{cwgeneral}-\eqref{initialsyst3}.\par\smallskip
$\bigstar$ We consider first the nonlinear equation
\neweq{toybeam}
u_{tt}+\big(1+\|u_{xx}\|^2_{L^2}\big) u_{xxxx}=0\qquad x \in I, \quad t>0,
\endeq
which corresponds to \eq{prototipo2} with $(\gamma_1,\gamma_3)=(1,0)$. According to Definition \ref{soluzionedebole}, a weak solution $u$
of \eq{toybeam} satisfies $u\in C^0(\R_+;V(I))\cap C^1(\R_+;L^2(I))\cap C^2(\R_+;V'(I))$ and
$$
\langle u_{tt}, v \rangle_{V}+\big(1+\|u_{xx}\|^2_{L^2}\big)\int_I u_{xx} v'' = 0\qquad\forall v\in V(I), \quad t>0.
$$
Since $(\gamma_1,\gamma_3)=(1,0)$, problem \eq{duffduff} becomes
\neweq{ODED}
\ddot{W}_\lambda(t)+\lambda^4W_\lambda(t)+\lambda^8W_\lambda(t)^3=0, \quad W_\lambda(0)=\delta,\quad\dot{W}_\lambda(0)=0.
\endeq
It is well-known \cite{stoker} that the unique solution of \eq{ODED} is periodic and that $\delta$ is the amplitude of oscillation of $W_\lambda$. Moreover, \eq{cwgeneral} becomes
\neweq{cw3}
\left\{\begin{array}{l}
\ddot{w}(t)+\lambda^4w(t)+\lambda^4\big(\lambda^4w(t)^2+\rho^4z(t)^2\big)w(t)=0\\
\ddot{z}(t)+\rho^4z(t)+\rho^4\big(\lambda^4w(t)^2+\rho^4z(t)^2\big)z(t)=0,
\end{array}\right.
\endeq
which is complemented with the initial conditions \eq{initialsyst3}, with $0<|z_0|<\eta^2\delta$.\par\smallskip
$\blacklozenge$ We then consider the nonlinear equation
\neweq{toybeam2}
u_{tt}+u_{xxxx}+\|u\|^2_{L^2}u=0\qquad x \in I, \quad t>0
\endeq
which corresponds to \eq{prototipo2} with $(\gamma_1,\gamma_3)=(0,1)$. According to Definition \ref{soluzionedebole}, the solution has to be
intended in the following weak sense:
$$
\langle u_{tt}, v \rangle_{V}+\int_I u_{xx} v''+\|u\|^2_{L^2}\int_I uv= 0\qquad\forall v\in V(I),\quad t>0
$$
and satisfies $u\in C^0(\R_+;V(I))\cap C^1(\R_+;L^2(I))\cap C^2(\R_+;V'(I))$.\par
Equation \eq{toybeam2} has a structure similar to \eq{toybeam}, although the stability analysis gives quite different responses.
Since $(\gamma_1, \gamma_3)=(0, 1)$, problem \eqref{duffduff} becomes
\neweq{ODED2}
\ddot{W}_\lambda(t)+\lambda^4 W_\lambda(t)+W_\lambda(t)^3=0, \quad W_\lambda(0)=\delta>0, \quad \dot{W}_\lambda(0)=0;
\endeq
compared with \eq{ODED}, there is no coefficient $\lambda^8$ in front of the cubic term.
Moreover, \eq{cwgeneral} becomes the following system of ODE's:
\neweq{cw32}
\left\{\begin{array}{l}
\ddot{w}(t)+\lambda^4w(t)+\big(w(t)^2+z(t)^2\big)w(t)=0\\
\ddot{z}(t)+\rho^4z(t)+\big(w(t)^2+z(t)^2\big)z(t)=0,
\end{array}\right.
\endeq
to be compared with \eq{cw3}. We add to \eqref{cw32} the initial conditions \eqref{initialsyst3}, with $0<\vert z_0 \vert < \eta^2\delta$.

\subsubsection{Some remarks on the linear instability}\label{lininst}

In order to explain the relationships between Definition \ref{unstable} and the classical Floquet theory, we need to recall and discuss the notion of linear
instability. This is much simpler when the modes of \eq{prototipo} do not mix, for instance under the assumptions \eq{nonsocomechiamarla}. This discussion is the
subject of the next three subsections.\par
The equation in \eq{duffduff} is named after Duffing, due to its first appearance in \cite{duffing}. If we set
$$
b:=\sqrt{\lambda^4+\delta^2(\gamma_1\lambda^8+\gamma_3)}\qquad\mbox{and}\qquad\beta:=\frac{\delta}{b}\sqrt{\frac{\gamma_1\lambda^8+\gamma_3}2},
$$
then from \cite{burg} (see also \cite{stoker}) we learn that
\neweq{explicitsol}
W_\lambda(t)=\delta\, {\rm cn}(bt,\beta),
\endeq
where cn is the Jacobi cosine, so that $W_\lambda$ is periodic with period
\neweq{TW}
T(\delta)=\frac{4}{b}\int_0^{\pi/2}\frac{d\varphi}{\sqrt{1-\beta^2\sin^2\varphi}}.
\endeq
The time-dependent coefficient $W_\lambda$ in \eq{separazione} starts at level $\delta>0$ with null
first derivative, that is, at a maximum point. Then the number $T(\delta)$ in \eq{TW} is the time needed for $W_\lambda$ to complete one cycle and to reach
the subsequent maximum point. Whence, if we consider a solution $u$ of \eq{prototipo} with an $\eta$-prevailing mode (Definition \ref{prevalente}), the number
$T(\delta)$ is a good approximation of the time needed for the Fourier component of the prevailing mode to complete one cycle and hence for the coupling effects between the modes to take place, even
in the case where \eq{nonsocomechiamarla} does not hold. Therefore, we choose this number as the Wagner time
\neweq{sceltawagner}
T_W=T(\delta)
\endeq
and our numerical experiments confirm that this choice for $T_W$ is good enough to highlight the physiological transfer of energy due to the nonlinearities
in \eq{prototipo}. Clearly, $T_W=T(\delta)$ depends on the amount of energy within \eqref{prototipo}.\par
The Duffing equation \eq{duffduff} is extremely useful to study the {\em linear instability} of a solution with a prevailing mode.

\begin{definition}\label{defstabb}
The $\lambda$-nonlinear mode $W_\lambda$ is said to be {\bf linearly stable} ({\bf unstable}) with respect to the $\rho$-nonlinear-mode $W_\rho$ if
$\xi\equiv0$ is a stable (unstable) solution of the linear Hill equation
\neweq{hill33}
\ddot{\xi}(t)+\Big(\rho^4+(\gamma_1\lambda^4\rho^4+\gamma_3)W_\lambda(t)^2\Big)\xi(t)=0.
\endeq
\end{definition}

Equation \eq{hill33} is the linearization of \eq{cwgeneral} around its solution $(w,z)=(W_\lambda,0)$. There exist also alternative
definitions of stability which, in some cases, can be shown to be equivalent; see e.g.\ \cite{ghg}. Moreover, for nonlinear PDE's such as \eq{prototipo2},
Definition \ref{defstabb} is appropriate to characterize the instabilities of the nonlinear modes of the equation, see \cite{babefega,bergaz,ederson}.
\par
An important tool to quantify the instability of a general Hill equation of the form
\begin{equation}\label{generalehill}
\ddot{q}(t) + a(t) q(t) = 0, \qquad a(t+\sigma)=a(t),
\end{equation}
is the \emph{expansion rate} introduced in \cite{arigaz}.
The eigenvalues $\nu_1$ and $\nu_2$ of the monodromy matrix associated with \eqref{generalehill} are the so-called {\em characteristic multipliers}. The instability for \eqref{generalehill} takes place when $\vert \nu_1+\nu_2 \vert > 2$ (see, e.g., \cite{yakubovich}) and the numbers $|\nu_j|^{1/\sigma}$, $j=1,2$, represent the growth rates of the amplitude of oscillation for the two solutions of \eq{generalehill} having initial
values $(q(0),\dot{q}(0))=(1,0)$ and $(q(0),\dot{q}(0))=(0,1)$, respectively, in the time interval $[0,\sigma]$. Then, following \cite{arigaz}, we call \emph{expansion rate} of \eqref{generalehill} the largest growth rate, namely
$$
\mathcal{ER}:=\max\{|\nu_1|,|\nu_2|\}^{1/\sigma}\,
$$
and \emph{expansion rate in time $\tau$} the number
\neweq{ERtau}
\mathcal{ER}_\tau:=\max\{|\nu_1|,|\nu_2|\}^{\tau/\sigma},
\endeq
which represents the amplitude growth in a lapse of time $\tau$.
This will be useful to relate Definitions \ref{unstable} and \ref{defstabb} in the forthcoming study of equation \eqref{prototipo2}.
\par\smallskip
$\bigstar$ Let us first deal with \eq{cw3} for which Definition \ref{defstabb} may be rephrased by saying that the $\lambda$-nonlinear-mode $W_\lambda$
is linearly stable (unstable) with respect to the $\rho$-nonlinear-mode $W_\rho$ if $\xi\equiv0$ is a stable (unstable) solution of the linear Hill equation
\neweq{hill3}
\ddot{\xi}(t)+\rho^4\Big(1+\lambda^4W_\lambda(t)^2\Big)\xi(t)=0.
\endeq

For all $\delta,\lambda,\rho>0$, we introduce the parameters
\neweq{gammaLambda}
\gamma:=\frac{\rho^2}{\lambda^2},\qquad\Lambda_\delta:=2\sqrt2\int_0^{\pi/2}
\sqrt{\frac{1+\delta^2\lambda^4\sin^2\theta}{2+\delta^2\lambda^4+\delta^2\lambda^4\sin^2\theta}}\, d\theta,
\endeq
that play a crucial role in the linear stability of the $\lambda$-nonlinear mode with respect to the $\rho$-nonlinear mode.
Let us also introduce the two sets
\neweq{strisceinfinito}
I_S:=\bigcup_{k=0}^{+\infty}\Big(k(2k+1),(k+1)(2k+1)\Big) \qquad I_U:=\bigcup_{k=0}^{+\infty} \Big((k+1)(2k+1),(k+1)(2k+3)\Big)\, ;
\endeq
note that $\overline{I_S\cup I_U}=[0,\infty)$. The next statement, which will be proved in Section \ref{pftoytheo}, gives sufficient conditions for the stability/instability of \eq{cw3}.

\begin{theorem}\label{toytheo}
Let $\lambda^4\neq\rho^4$ be two eigenvalues of \eqref{autovsym}, let $\Lambda_\delta$ be as in
\eqref{gammaLambda} and $I_S, I_U$ as in \eqref{strisceinfinito}.\par\noindent
$(i)$ The $\lambda$-nonlinear-mode of \eqref{toybeam} of amplitude $\delta$ is linearly stable with respect to the $\rho$-nonlinear-mode
whenever one of the following facts holds:\par
$(i)_1$ $\lambda>\rho$ and $\delta>0$;\par
$(i)_2$ $\lambda<\rho$ and $\delta$ is sufficiently small;\par
$(i)_3$ $\frac{\rho^4}{\lambda^4}\in I_S$ and $\delta$ is sufficiently large;\par
$(i)_4$ there exists $k\in\N$ such that
$$\log(1+\delta^2\lambda^4)<2\cdot\min\left\{\frac{\rho^2}{\lambda^2}\, \Lambda_\delta-k\pi,\, (k+1)\pi-\frac{\rho^2}{\lambda^2}\, \Lambda_\delta\right\}.$$
$(ii)$ The $\lambda$-nonlinear-mode of \eqref{toybeam} of amplitude $\delta$ is linearly unstable with respect to the $\rho$-nonlinear-mode
whenever one of the following facts holds:\par
$(ii)_1$ $1<\frac{\rho^2}{\lambda^2}<\psi_\lambda(\delta)$, where $\psi_\lambda:\R_+\to\R_+$ is a continuous function such that
$$
\psi_\lambda(\delta)>1\quad\forall\delta>0,\qquad\psi_\lambda(\delta)\le1+\left(\frac18 +\frac{1}{2\pi}\right)\delta^2\lambda^4+O(\delta^4)
\mbox{ as }\delta\to0,\qquad\lim_{\delta\to\infty}\psi_\lambda(\delta)=\sqrt{3}\, ;
$$

$(ii)_2$ $\frac{\rho^4}{\lambda^4}\in I_U$ and $\delta$ is sufficiently large.
\end{theorem}

From~\cite[Chapter VIII]{yakubovich} we know that for all $k=1,2,3...$ there exists a resonant tongue $U_k$ emanating from the point $(\delta\lambda^2,\gamma)=(0,k)$.
If $(\delta\lambda^2,\gamma)$ belongs to one of these tongues, then the trivial solution of \eqref{hill3} is unstable.
Since $\Lambda_\delta = \pi(1-\delta^2 \lambda^4/8)+ O(\delta^4)$ as $\delta \to 0$, Item $(i)_4$ implies that if $(\delta\lambda^2,\gamma)\in U_k$, then necessarily
\begin{equation}\label{asintoticalingua}
k+\left(\frac{k}{8}-\frac{1}{2\pi}\right)\delta^2\lambda^4+O(\delta^4)\le\frac{\rho^2}{\lambda^2}
\le k+\left(\frac{k}{8}+\frac{1}{2\pi}\right)\delta^2\lambda^4+O(\delta^4)\quad\mbox{as }\delta\to0.
\end{equation}
This enables us to depict the resonance tongues in a neighborhood of $\delta=0$. In Figure \ref{burdina2} we display the (white) regions of stability
described by Theorem \ref{toytheo}-$(i)$.

\begin{figure}[ht!]
\begin{center}
\includegraphics[scale=0.55]{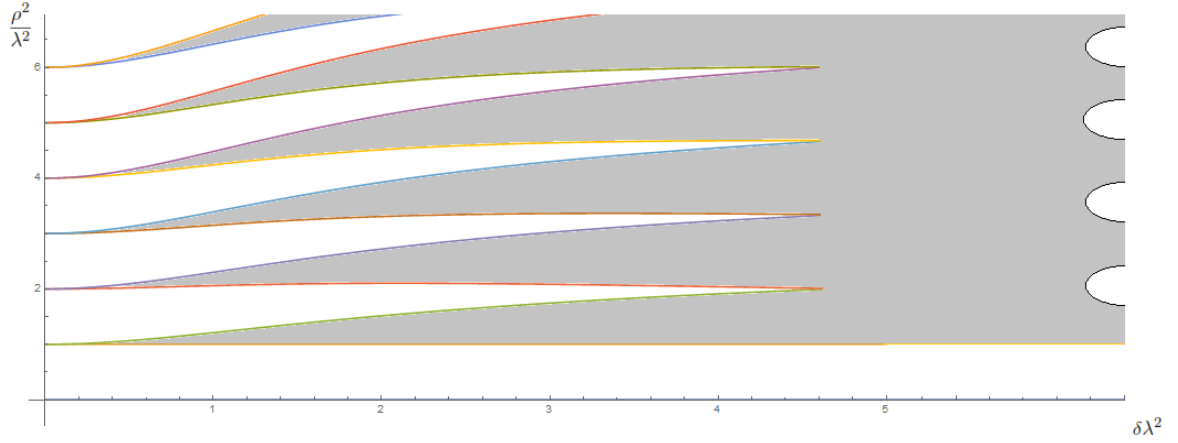}
\caption{Stability regions (white) obtained with the sufficient conditions $(i)$ of Theorem \ref{toytheo}.}\label{burdina2}
\end{center}
\end{figure}
We also numerically obtained a full picture of the resonance tongues of \eq{hill3}. For each value of $\delta=W_\lambda(0)$ we computed
the period of $W_\lambda$ through formula \eqref{TW} and we determined the monodromy matrix of \eq{hill3}. Then we computed the trace of the monodromy matrix after
one period of time, seeking when it was equal to $\pm2$. The results are reported in Figure \ref{monodromymnd1}. It turns out that the resonance tongues are very narrow for small
$\delta$ and they enlarge as $\delta\to\infty$.
\begin{figure}[ht!]
\begin{center}
\includegraphics[scale=0.7]{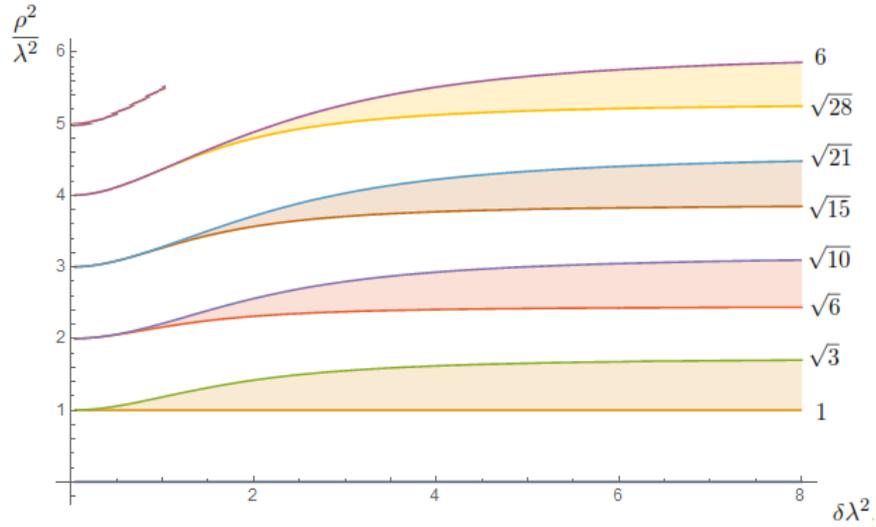}
\caption{Stability regions (white) for~\eq{hill3} obtained numerically through the monodromy matrix.}\label{monodromymnd1}
\end{center}
\end{figure}
As a consequence of Theorem \ref{toytheo} (see also Figure \ref{monodromymnd1}) we see that for all $k\in\N$ the resonance lines emanating
from the points $(0,k)$ of the $(\delta\lambda^2,\rho^2/\lambda^2)$-plane tend towards $\sqrt{k(2k-1)}$ and $\sqrt{k(2k+1)}$ as $\delta\to\infty$.
Therefore, the amplitude at infinity of the $k$-th resonant tongue is given by
$$\sqrt{k(2k+1)}-\sqrt{k(2k-1)}=\frac{2}{\sqrt{2+\frac1k }+\sqrt{2-\frac1k }}\in(\sqrt2 /2,\sqrt3 -1]\approx(0.707,0.732]\qquad\forall k\in\N$$
which shows that they all have, approximately, the same width. A similar bound holds for the width of the stability regions as $\delta\to\infty$.
Once the value of $\gamma=\rho^2/\lambda^2$ is fixed, if $\delta$ increases starting from the point $(0,\gamma)$, some narrow resonance tongues
are crossed before reaching the final stability or instability region characterized by the sets $I_S$ and $I_U$.\par\smallskip
$\blacklozenge$ Let us now turn to the case of \eq{cw32}. Following Definition \ref{defstabb}, we say that the $\lambda$-nonlinear-mode $W_\lambda$ is linearly
stable (unstable) with respect to the $\rho$-nonlinear-mode $W_\rho$ if $\xi\equiv0$ is a stable (unstable) solution of the linear Hill equation
\neweq{hill32}
\ddot{\xi}(t)+\Big(\rho^4+W_\lambda(t)^2\Big)\xi(t)=0.
\endeq

If $W_\lambda$ solves \eq{ODED2} and we set $W_\lambda(t)=\lambda^2\Psi_\lambda(\lambda^2t)$, we see that $\Psi_\lambda$ satisfies
$$
\ddot{\Psi}_\lambda(t)+\Psi_\lambda(t)+\Psi_\lambda(t)^3=0,\qquad \Psi_\lambda(0)=\frac{\delta}{\lambda^2},\qquad \dot{\Psi}_\lambda(0)=0.
$$
The same change of variables shows that the stability of \eq{hill32} is equivalent to the stability of
$$
\ddot{\xi}(t)+\left(\frac{\rho^4}{\lambda^4}+\Psi_\lambda(t)^2\right)\xi(t)=0.
$$
Then we can take advantage of the results in \cite[Section 2]{gaga} and obtain the counterpart of Theorem \ref{toytheo}:

\begin{proposition}\label{toytheo2}
Let $\lambda^4\neq\rho^4$ be two eigenvalues of \eqref{autovsym}. The $\lambda$-nonlinear-mode of \eqref{toybeam2} of amplitude $\delta$ is linearly stable with respect to the $\rho$-nonlinear-mode
if and only if one of the following facts holds:
$$
\lambda>\rho\mbox{ and }\delta>0\qquad\mbox{or}\qquad\lambda<\rho\mbox{ and }\delta \leq \sqrt{2(\rho^4-\lambda^4)}.
$$
\end{proposition}

Obviously, this means that the $\lambda$-nonlinear-mode of \eqref{toybeam2} of amplitude $\delta$ is linearly unstable with respect to the
$\rho$-nonlinear-mode if and only if
$$
\lambda<\rho\quad \mbox{ and }\quad \delta>\sqrt{2(\rho^4-\lambda^4)}.
$$
Hence, Proposition \ref{toytheo2} states that
\begin{center}
{\bf whenever $\lambda<\rho$, the $\lambda$-nonlinear-mode
of \eqref{toybeam2} is linearly unstable with respect to the $\rho$-nonlinear-mode if and only if it oscillates with sufficiently large amplitude.}
\end{center}
\par
We may give a full description of the stability and instability regions for \eq{cw32} (see Figure \ref{parabole}, which should be compared with
Figure \ref{monodromymnd1}).

\begin{figure}[ht!]
\begin{center}
\includegraphics[height=80mm,width=120mm]{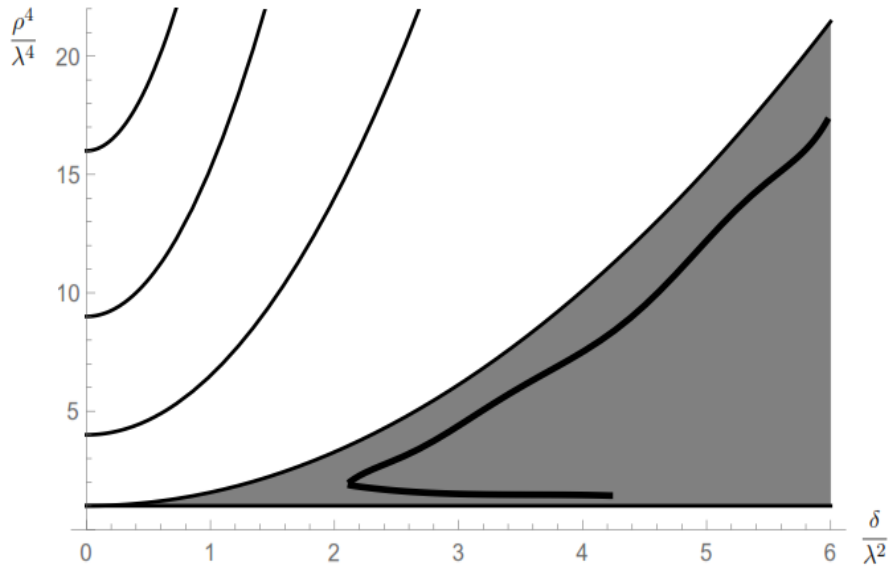}
\caption{Regions of linear stability (white) and instability (gray) for \eqref{cw32}.}\label{parabole}
\end{center}
\end{figure}

The lowest parabola in Figure \ref{parabole} bounds the first instability tongue and has equation
$$
\frac{\rho^4}{\lambda^4}=1+\frac12 \, \left(\frac{\delta}{\lambda^2}\right)^2,
$$
see \cite[Theorem 4]{gaga}. By using the so-called KdV hierarchy, one finds that the instability tongues emanating from $(\delta/\lambda^2, \rho^4/\lambda^4)=(0, n^2)$
(with $n\ge2$) are degenerate, see \cite{gesz} and also the introduction in \cite{goldberg}. This means that, contrary to Figure \ref{monodromymnd1} and except
for the first tongue, there are no open instability
tongues since their upper and lower boundaries coincide. We determined these degenerate tongues numerically, representing them by the black parabolas in
Figure \ref{parabole}. By applying the Burdina criterion \cite{burdina} (see also Lemma \ref{BurdCrit} below) and arguing as for Theorem \ref{toytheo}, we infer that
the degenerate instability tongues emanating from $(\delta/\lambda^2, \rho^4/\lambda^4)=(0, n^2)$ ($n\ge2$) asymptotically satisfy
$$
n^2+\left(\frac{3n^2}{4}-\frac12 -\frac{1}{\pi n}\right)\frac{\delta^2}{\lambda^4}+O(\delta^4)\le\frac{\rho^4}{\lambda^4}\le
n^2+\left(\frac{3n^2}{4}-\frac12 +\frac{1}{\pi n}\right)\frac{\delta^2}{\lambda^4}+O(\delta^4)\quad\mbox{as }\delta\to0.
$$
The characterization of the black line within the gray instability region will be given in Section \ref{instnonlineare}: it represents the thresholds of nonlinear instability.

\subsubsection{The critical energy threshold for the linear instability}\label{againa}

In this section, we compute the critical energy threshold for the linear stability of problems \eqref{toybeam} and \eqref{toybeam2}. To this end, it will be
useful to preliminarily introduce some notations. For all eigenvalues $\lambda^4$ and $\rho^4$ of \eq{autovsym}, we define
\neweq{D}
D(\lambda,\rho):=\inf\{d>0;\, \mbox{\eq{hill33} has unstable solutions whenever }W_\lambda(0)=\delta>d\}.
\endeq
We immediately observe that this quantity may be infinite: for \eq{toybeam} and \eq{toybeam2} respectively, we have
$$
D(\lambda,\rho)\left\{\begin{array}{ll}
=+\infty & \mbox{if }\rho^4/\lambda^4\in I_S\\
<+\infty & \mbox{if }\rho^4/\lambda^4\in I_U
\end{array}\right.\qquad\mbox{and}\qquad
D(\lambda,\rho)=\left\{\begin{array}{ll}
+\infty & \mbox{if }\rho<\lambda\\
\sqrt{2(\rho^4-\lambda^4)} & \mbox{if }\rho>\lambda
\end{array}\right.
$$
in view of Theorem \ref{toytheo} (Figure \ref{monodromymnd1}) for \eq{toybeam} and of Proposition \ref{toytheo2} (Figure \ref{parabole}) for \eq{toybeam2}.\par
Then we define the critical energy as the (constant) energy of the solutions of \eq{ODED} and \eq{ODED2}, respectively, when $\delta=D(\lambda,\rho)$:
\begin{equation}\label{ecritica}
E(\lambda,\rho):=\frac{\lambda^4D(\lambda,\rho)^2}{2}+\frac{\lambda^8D(\lambda,\rho)^4}4, \qquad
E(\lambda,\rho):=\frac{\lambda^4D(\lambda,\rho)^2}{2}+\frac{D(\lambda,\rho)^4}4,
\end{equation}
for \eqref{toybeam} and \eqref{toybeam2}, respectively.
Also this energy can be infinite:
we have, respectively,
\neweq{EEE}
E(\lambda,\rho)\left\{\begin{array}{ll}
=+\infty & \mbox{if }\rho^4/\lambda^4\in I_S\\
<+\infty & \mbox{if }\rho^4/\lambda^4\in I_U
\end{array}\right.\qquad\mbox{and}\qquad
E(\lambda,\rho)=\left\{\begin{array}{ll}
+\infty & \mbox{if }\rho<\lambda\\
(\rho^4-\lambda^4)\rho^4 & \mbox{if }\rho>\lambda.
\end{array}\right.
\endeq

Clearly, the stability of the mode $W_\lambda$ with respect to $W_\rho$ depends on $a$ through $\lambda=\lambda(a)$ and $\rho=\rho(a)$.
For both problems \eq{toybeam} and \eq{toybeam2}, the most natural way to define the energy threshold for linear stability in dependence of $a$ would be
$$\mathbb{E}^{\ell}(a):=\, \inf_{\lambda,\rho} E\big(\lambda(a),\rho(a)\big);$$
unfortunately, this function has several setbacks. First of all, from \eqref{asintoticaautov} and Theorem \ref{toytheo}-$(ii)_1$ we see that,
in case \eq{toybeam}, one has
$$
\lim_{n\to\infty} E\big(\lambda_n(a) ,\lambda_{n+1}(a)\big)=0\qquad\forall a\in(0,1),
$$
which shows that $\mathbb{E}^{\ell}(a)\equiv0$ and thus this threshold is not meaningful. Moreover, as for \eqref{toybeam2}, formula \eqref{EEE} implies that $\mathbb{E}^{\ell}(a)<\infty$
for all $a \in (0, 1)$; on the other hand, \eqref{lambdasimmetrico} guarantees that $\mathbb{E}^{\ell}(1/2) > 0$. A definition of critical energy that has so different behaviors for quite similar
nonlinearities does not appear appropriate. Therefore, we are again led to follow Definition \ref{threshold} and set
\begin{equation}\label{formulalineare}
\mathbb{E}_{12}^\ell(a):=\, \inf_{\lambda_0 \leq \lambda <\rho \leq \lambda_{11}} E\big(\lambda(a),\rho(a)\big)
\end{equation}
where the constraint $\lambda(a)<\rho(a)$ is due to the fact that for $\lambda(a) > \rho(a)$ there is no instability. This means that there
are 66 couples to evaluate and the above infimum is certainly a minimum if at least one of these couples yields finite $E(\lambda(a), \rho(a))$.\par
In order to determine the position $a\in(0,1)$ of the piers which maximizes $\mathbb{E}_{12}^\ell(a)$ (yielding more stability),
one is then interested in showing that the map $a\mapsto\mathbb{E}_{12}^\ell(a)$ is continuous. This can be easily proved for \eq{toybeam2}.

\begin{proposition}\label{energiafacile}
For equation \eqref{toybeam2} the map $a\mapsto\mathbb{E}_{12}^\ell(a)$ is strictly positive and continuous.
\end{proposition}

Proposition \ref{energiafacile} follows by combining the continuity of the maps $a\mapsto\lambda(a)$ (see Theorem \ref{Michelle}) with the explicit
form of the energy in \eq{EEE} (second formula): since $\mathbb{E}_{12}^\ell(a)$ is the minimum of a finite number of positive and continuous functions, it is positive and continuous, as well.\par
However, the map $a\mapsto\mathbb{E}_{12}^\ell(a)$ for \eq{toybeam} may not be continuous, as we now explain through a graphic counterexample.
Let us take the ``inverse'' of the plot in Figure \ref{monodromymnd1}, namely its reflection
about the first bisectrix and let us only plot the four upper boundaries of the resonant tongues, see Figure \ref{inversetongue}.
The black lines represent the value of $\lambda^2D(\lambda,\rho)$ and, since $E(\lambda,\rho)$ increasingly depends on this value, they maintain
the order of $E(\lambda,\rho)$ for different couples $(\lambda,\rho)$. If $\rho^4/\lambda^4\in I_U$, we plot the horizontal segments at the
top of the picture, which means that $\lambda^2D(\lambda,\rho)=+\infty$ and, in turn, also that $E(\lambda,\rho)=+\infty$.
\begin{figure}[ht!]
\begin{center}
\includegraphics[scale=0.8]{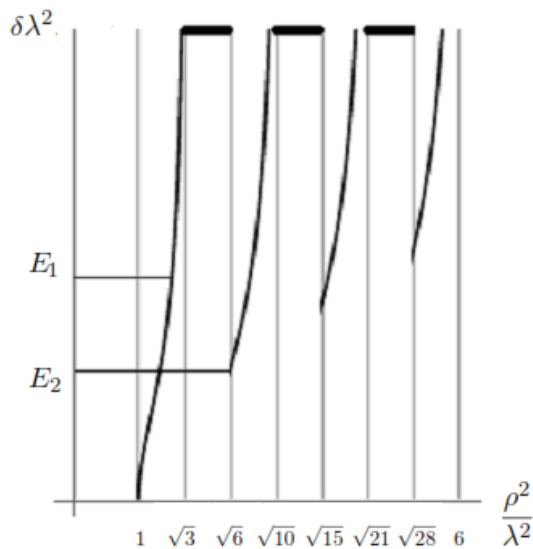}
\caption{Discontinuity of the critical energy threshold $a\mapsto\mathbb{E}_{12}^\ell(a)$.}\label{inversetongue}
\end{center}
\end{figure}
Consider now two couples $(\lambda_i(a),\rho_i(a))$ and $(\lambda_j(a),\rho_j(a))$. From Theorem \ref{Michelle} we know that the two ratios
$\rho_k(a)/\lambda_k(a)$ (for $k=i,j$) depend continuously on $a$. Assume that
$[\rho_i(a)/\lambda_i(a)]^2$ varies in a left neighborhood of $\sqrt3$ in such a way that $\lambda_i^2D(\lambda_i,\rho_i)$ varies in a small
neighborhood of $E_1$ on the vertical axis. Assume also that $[\rho_j(a)/\lambda_j(a)]^2$ varies in an interval centered at $\sqrt6$: then
$\lambda_j^2D(\lambda_j,\rho_j)$ can be either $+\infty$ (in a left neighborhood of $\sqrt6$) or around $E_2$ (in a right neighborhood of
$\sqrt6$). The minimum between the two $\lambda_k^2D(\lambda_k,\rho_k)$ will then be close to $E_1$ in the first case and close to $E_2$ in
the second case, thereby displaying discontinuity.\par
Clearly, this example is based on the fact that the ratios $[\rho(a)/\lambda(a)]^4$ may enter and exit from $I_U$. One could still try to prove
some continuity provided one could show that all the involved ratios remain in the same stability/instability interval, but this seems out
of reach. Anyway, according to Figure \ref{monodromymnd1},
\neweq{ipotesi}
\mbox{\textbf{the resonant lines are increasing functions of $\delta\lambda^2$.}}
\endeq
Proving this statement seems to be a difficult task. But, taking it for granted, it enables to prove the lower semicontinuity of the energy threshold (see Section \ref{pfenergiacontinua}).

\begin{theorem}\label{energiacontinua}
Consider problem \eqref{toybeam}. For every $a\in(0,1)$ we have $0<\mathbb{E}_{12}^\ell(a)<+\infty$. Moreover, if \eqref{ipotesi} holds, then
the map $a\mapsto\mathbb{E}_{12}^\ell(a)$ is lower semicontinuous.
\end{theorem}

\subsubsection{Optimal position of the piers for the linear instability}\label{instabilitalineare}

In this section, we aim at finding the position of the piers which yields more linear stability, for each of the problems \eqref{toybeam} and \eqref{toybeam2}.
In other words, we seek the optimal $a\in(0,1)$ which maximizes $\mathbb{E}_{12}^\ell(a)$. As in the previous subsection, the two equations behave
quite differently and we analyze them separately.\par\smallskip
$\bigstar$ Overall, Theorem \ref{toytheo} states that for \eq{toybeam} the relevant parameter for stability is $\gamma$ defined in \eq{gammaLambda}, that is,
the square root of the ratio of the eigenvalues of \eq{autovsym}. In particular, instability arises if $\gamma^2(a) \in I_U$ (we emphasize here
the dependence on $a$). In order to study the maps $a \mapsto \gamma(a)$, we take advantage of the results in Section \ref{finitod}, in particular
of the numerical computation of the eigenvalues. The ratios turn out to be extremely irregular, with no obvious rule governing their variation. In
Figure \ref{ratios} we plot some of the graphs obtained by interpolation after computing the ratios with step 0.05 for $a$. The shaded horizontal
strips correspond to the square roots of the (endpoints of the) intervals composing $I_U$, see \eqref{strisceinfinito} and Figure \ref{monodromymnd1}. If the value of $\gamma=\lambda_k^2/\lambda_j^2$ is within this range, then for large energies the $\lambda_j$-nonlinear-mode
$W_{\lambda_j}$ is linearly unstable with respect to the $\lambda_k$-nonlinear-mode $W_{\lambda_k}$, otherwise it is linearly stable.

\begin{figure}[ht!]
\begin{center}
\includegraphics[scale=0.4]{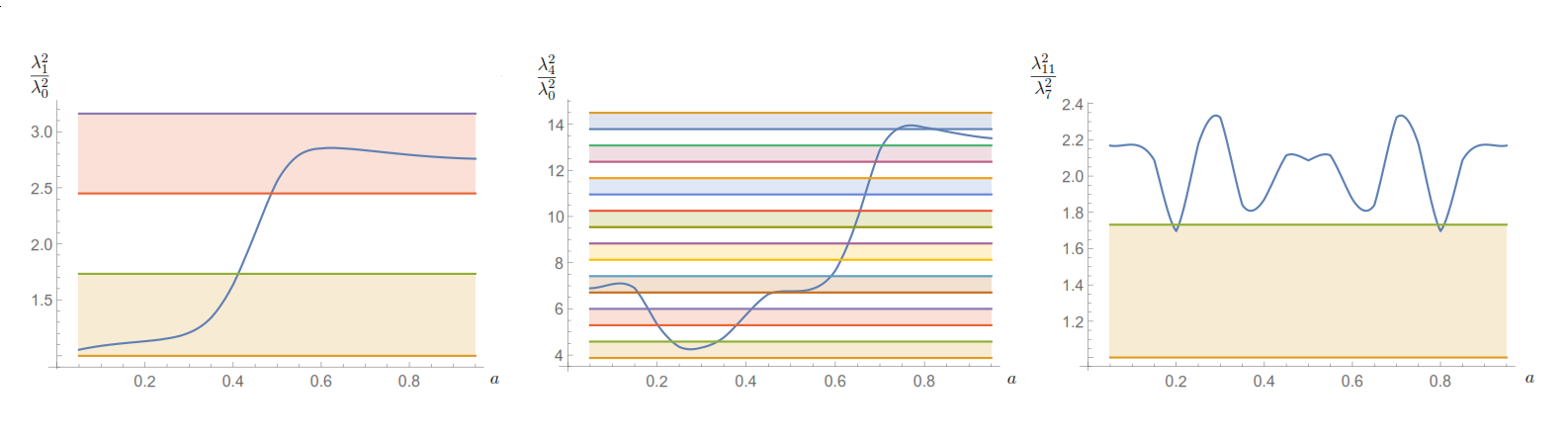}
\caption{Plot of the maps $a\mapsto\lambda_1^2/\lambda_0^2$, $a\mapsto\lambda_4^2/\lambda_0^2$, $a\mapsto\lambda_{11}^2/\lambda_7^2$
(from left to right) and their intersection with the (shaded) instability intervals.}\label{ratios}
\end{center}
\end{figure}

The ratio itself may enter and exit $I_U$ several times, as $a$ varies in $(0,1)$.
We see that $\lambda_1^2/\lambda_0^2$ lies within the stability region for large energies only if (approximately) $0.4<a<0.5$.
The ratio $\lambda_4^2/\lambda_0^2$ is much more complicated, as $a$ varies it enters and exits several times the instability regions: similar behaviors
are visible whenever the ratio between the eigenvalues has large variations for varying $a$. The ratio $\lambda_{11}^2/\lambda_7^2$ is almost always within the stability region.
On the contrary, there exist some ratios which never exit the instability region, see Figure \ref{quozasint}: this means that no choice of $a\in(0,1)$ can prevent the
linear instability for these couples of eigenvalues. Therefore, it is difficult to derive a precise rule telling which values of $a$ yield more stability.
\begin{figure}[ht!]
\begin{center}
\includegraphics[scale=0.4]{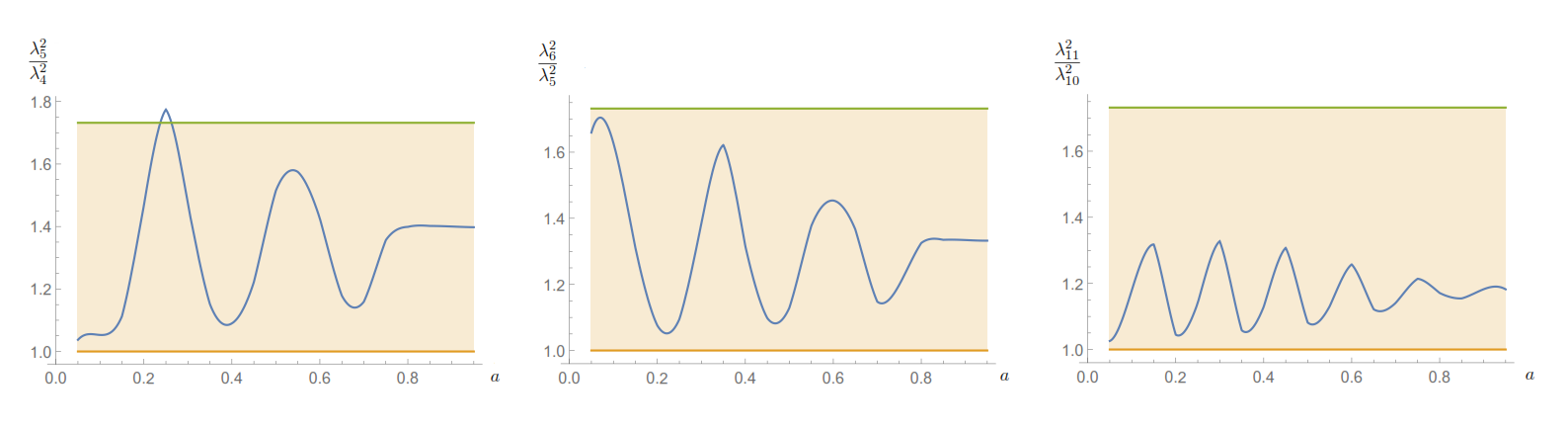}
\caption{
Plot of the maps $a \mapsto \lambda_5^2(a)/\lambda_4^2(a)$, $a \mapsto \lambda_6^2(a)/\lambda_5^2(a)$, $a \mapsto \lambda_{11}^2(a)/\lambda_{10}^2(a)$ (from
left to right) and their intersection with the first (shaded) instability interval.}\label{quozasint}
\end{center}
\end{figure}

Moreover, not all the intervals in $I_U$ yield {\em quantitatively} the same amount of instability.
We order them by means of their instability effects as follows. We compute the absolute value $|\mathcal{T}|$ of the trace of the monodromy matrix
associated with \eq{hill3} when $\gamma=\rho^2/\lambda^2=\frac{\sqrt{k(2k+1)}+\sqrt{k(2k-1)}}{2}$, namely at the midpoint of the limit interval
$[\sqrt{k(2k-1)},\sqrt{k(2k+1)}]$ (as $\delta\to\infty$) of a resonant tongue. We start with $\delta\lambda^2=4$ and we proceed by increasing $\delta\lambda^2$ with step 0.2 until $\delta\lambda^2=12$, then we interpolate. We do this for several different tongues. In the left picture of
Figure \ref{confronto} we display the plots for the lowest six tongues. It appears clearly that the values of $|\mathcal{T}|$ are ordered, the largest one being the
first, and then they decrease as the extremal values of the resonant tongue increase.

\begin{figure}[ht!]
\begin{center}
\includegraphics[scale=0.6]{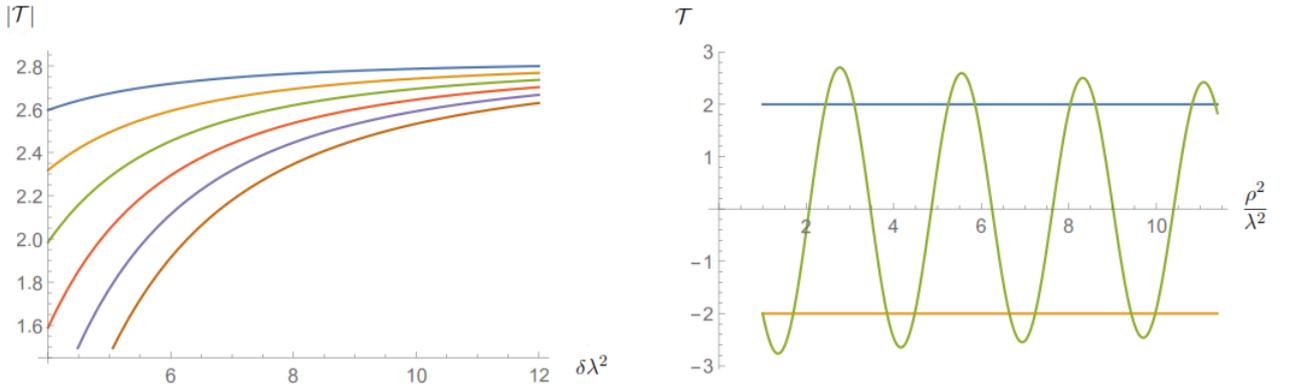}
\caption{Comparison between the traces of monodromy matrices in different resonant tongues.}\label{confronto}
\end{center}
\end{figure}

This provides the following principle:
\begin{center}
{\bf the first resonant tongue yields more instability.}
\end{center}
We made further experiments, studying the behavior of the trace $\mathcal{T}$ (without absolute value) for fixed
$\delta\lambda^2$ on varying of $\rho^2$: starting from $\rho^2=1$ we proceeded with step 0.1 and then we interpolated. In the right picture
of Figure \ref{confronto} we display the plot for $\delta\lambda^2=8$. It turns out that the sequence of the maxima of $\vert \mathcal{T} \vert$ is decreasing. Other values of this parameter showed the same behavior:
$\mathcal{T}$ oscillates above $2$ and below $-2$ with decreasing amplitude. This confirms the above principle.\par
By the first formula in \eqref{ecritica}, the critical energy threshold depends increasingly on $\delta \lambda^2$. Moreover, the critical
amplitude $D(\lambda,\rho)\lambda^2$ depends increasingly on $\rho/\lambda$, thus $\mathbb{E}_{12}^\ell(a)$ takes its maximum and its minimum at
the very same points as $\gamma$: the instability increases as $\gamma$ approaches $1$. In the left picture of Figure \ref{minstability} we display the graph of the map
\neweq{mappaa}
a\mapsto\min\left\{\frac{\lambda_k^4(a)}{\lambda_j^4(a)};\, k=1,...11,\, j=0,...,k-1\right\}.
\endeq
\begin{figure}[ht!]
\begin{center}
\includegraphics[scale=0.58]{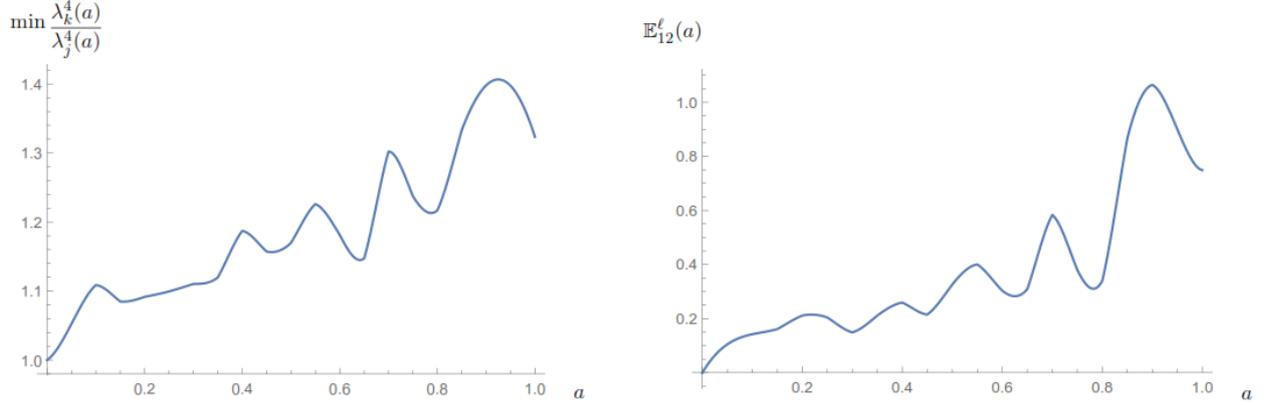}
\caption{Graph of the map \eqref{mappaa} and of $a \mapsto \mathbb{E}_{12}^\ell(a)$ for \eqref{toybeam}.}\label{minstability}
\end{center}
\end{figure}
\newline
It turns out that it is ``almost'' increasing and always below 3, meaning that there exists at least one couple of indexes $j$ and $k$ for which
$\lambda_k^4(a)/\lambda_j^4(a)\in(1,3)$, which is the first instability interval of $I_U$.
We numerically saw that all the ratios $\lambda_{n+1}(a)/\lambda_n(a)$ for $n \geq 5$ satisfy this condition (cf.\ Figure \ref{quozasint}),
for every $a \in (0, 1)$. The right picture in Figure \ref{minstability} depicts the graph of $\mathbb{E}_{12}^\ell(a)$: we see that
\begin{equation}\label{optimalbdg}
\textrm{{\bf for \eqref{toybeam}, the maximum of $\mathbb{E}_{12}^\ell(a)$ is attained for $a\approx 0.9$.
}}
\end{equation}
\par\smallskip
$\blacklozenge$ We now analyze the much simpler case of \eq{toybeam2}. In Figure \ref{zoomenergysuperL2} we display the plot of the map $a\mapsto\mathbb{E}_{12}^\ell(a)$
for \eqref{toybeam2}. It turns out that
\begin{equation}\label{optimall2}
\textrm{{\bf for \eqref{toybeam2}, the maximum of $\mathbb{E}_{12}^\ell(a)$ is attained for
$a\approx0.5$.
}}
\end{equation}

\begin{figure}[ht!]
\begin{center}
\includegraphics[scale=0.58]{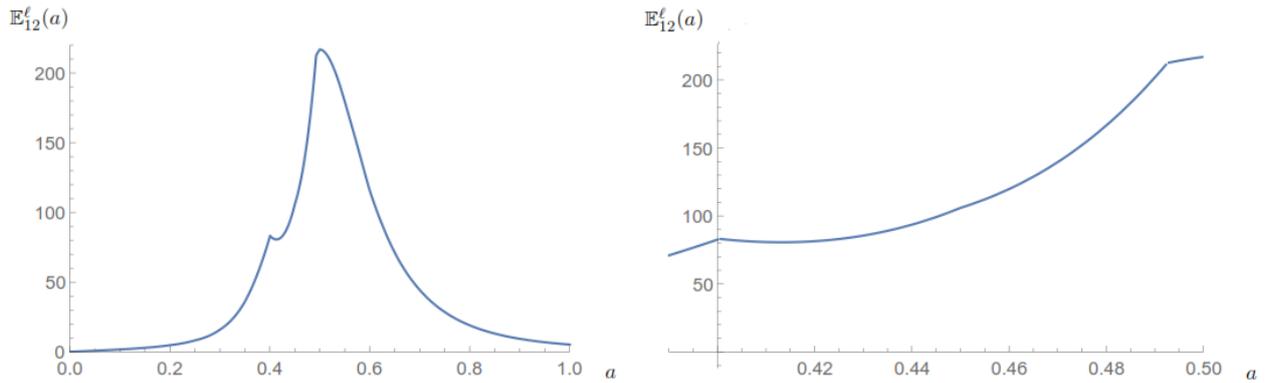}
\caption{Graph of the map $a\mapsto\mathbb{E}_{12}^\ell(a)$ (left), magnified for $a \in [0.4, 0.5]$ (right), for \eqref{toybeam2}.}\label{zoomenergysuperL2}
\end{center}
\end{figure}
As expected, the minimum of the energies $E\big(\lambda(a),\rho(a)\big)$ is attained by couples of low eigenvalues, more precisely by the couple $(\lambda_0(a),\lambda_1(a))$ when
$a\in(0,0.4)\cup(0.5,1)$ (approximately) and by the couple $(\lambda_1(a),\lambda_2(a))$ when $a\in(0.4,0.5)$. This explains the corner close to $a=0.4$. The reason
why these are the minima is easily explained. From \eqref{EEE}, we see that the minimum is necessarily reached by couples of consecutive
eigenvalues $(\lambda^4, \rho^4)$, possibly with a small $\rho$. This suggests that the minimum should always be attained by the
couple $(\lambda_0,\lambda_1)$, with some possible exception. By looking carefully at Figure \ref{bello} when $a\in(0.4,0.5)$ one understands why, in this interval,
the couple $(\lambda_1,\lambda_2)$ yields a smaller energy: the corresponding curves of eigenvalues are very close in this range.

\subsubsection{Back to the nonlinear instability: optimal position of the piers}
\label{instnonlineare}

In this section, we investigate the optimal position of the piers in terms of nonlinear instability and we compare the notions of stability introduced in Definitions \ref{unstable} and \ref{defstabb}, in light of the comments written in the previous sections.
\par
In Section \ref{instabilitalineare}, we have analyzed the linear stability for equations \eqref{toybeam} and \eqref{toybeam2}.
A careful look at Figure \ref{monodromymnd1} highlights that if we maintain fixed the ratio $\rho/\lambda$
and we increase $\delta$ (that is, the energy), then we may cross some tiny instability zones with absolute value of the trace of the
monodromy matrix associated with \eqref{hill3} very close to 2. Here, from a physical point of view, the instability is very weak and difficult
to be detected numerically. For these reasons, we have neglected such regions in our characterization of the amplitude threshold, see \eq{D}; linear instability has a purely theoretical characterization, so there may be cases when its effects are not concretely visible. Even simpler is the picture in Figure \ref{parabole}, where the tiny instability regions are empty and degenerate into double resonant lines.
\par
For both \eqref{toybeam} and \eqref{toybeam2}, we numerically found the critical energy threshold for nonlinear instability, as characterized in Definition \ref{threshold}.
The response with respect to the optimal position of the piers has been examined following the procedure described in Section \ref{algoritmo}, in particular fixing $\eta=0.1$. In Figure \ref{confrontolinn2}, we deal with equation \eqref{toybeam}.
\begin{figure}[ht!]
\begin{center}
\includegraphics[height=63mm,width=158mm]{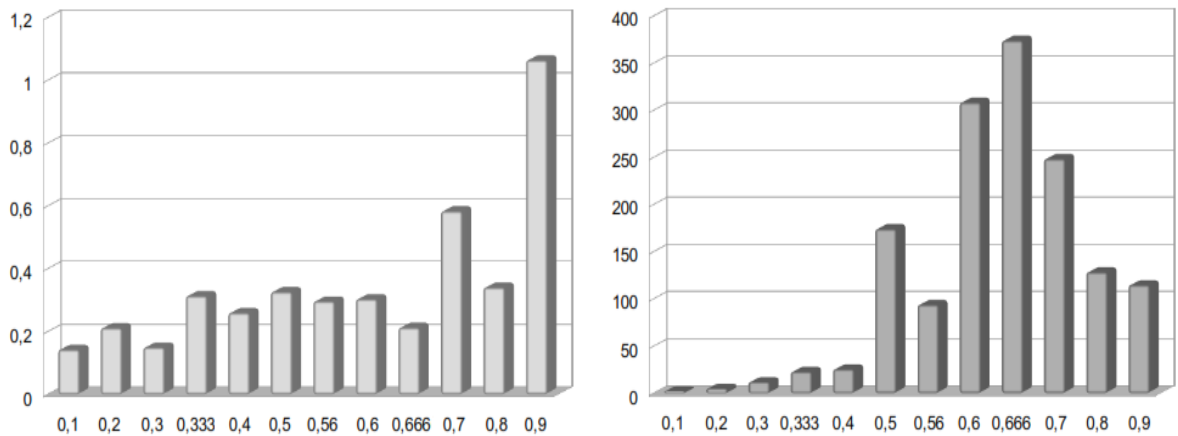}
\caption{Comparison between $a \mapsto \mathbb{E}_{12}^\ell(a)$ (left picture) and $a \mapsto \mathbb{E}_{12}(a)$ (right picture) for \eqref{toybeam}.\label{confrontolinn2}}
\end{center}
\end{figure}

\begin{table}[ht]
\begin{center}
\resizebox{\textwidth}{!}{\footnotesize
\begin{tabular}{|c|c|c|c|c|c|c|c|c|c|c|c|c|c|}
\hline
$a$ & $0.1$ & $0.2$ & $0.3$ & $1/3$ & $0.4$ & $0.5$ & $0.56$ & $0.6$ & $2/3$ & $0.7$ & $0.8$ & $0.9$ & no piers\\
\hline
eq. \eqref{toybeam} & -2.144 & -2.249 & -2.397 & -2.492 & -2.465 & -2.656 & -2.465  & -2.525 & 2.555 & 2.515 & 2.459 & 2.471 & -2.475 \\
\hline
eq. \eqref{toybeam2} & -2.047 & -2.085 &
-2.08 &  -2.069 &  -2.068 & -2.079 &
-2.081  & -2.088 & -2.101 &
-2.103 & -2.128 & -2.16 & -2.289\\
\hline
\end{tabular}
}
\caption{Traces of the monodromy matrix in correspondence of nonlinear instability.}\label{tracce}
\end{center}
\end{table}



\begin{table}[ht]
\begin{center}
\resizebox{\textwidth}{!}{\footnotesize
\begin{tabular}{|c|c|c|c|c|c|c|c|c|c|c|c|c|c|}
\hline
$a$ & $0.1$ & $0.2$ & $0.3$ & $1/3$ & $0.4$ & $0.5$ & $0.56$ & $0.6$ & $2/3$ & $0.7$ & $0.8$ & $0.9$ & no piers \\
\hline
eq. \eqref{toybeam} & 123.548 & 75.831 & 65.793 & 96.85 & 88.532 & 179.408 & 90.627  & 111.555 & 119.632 & 102.219 & 81.17 & 86.2 &  617.994  \\
\hline
eq. \eqref{toybeam2} & 91.828 & 74.486 &
108.586 &  80.212 & 127.101 & 97.985 &
111.482  & 136.453 & 101.399 &
108.093 & 89.127 & 69.748 & 96.017 \\
\hline
\end{tabular}
}
\caption{Expansion rates in correspondence of nonlinear instability.}\label{expansionraten}
\end{center}
\end{table}

In Tables \ref{tracce} and \ref{expansionraten}, we report the trace of the monodromy matrix of \eq{hill33} and the associated expansion rates $\mathcal{ER}_\tau$ defined in \eqref{ERtau} in correspondence of nonlinear instability, namely for the least amplitude of the prevailing mode for which \eqref{grande} holds true (here $\tau$ is the number appearing in Definition \ref{unstable}). We recall that the trace is computed after one period of $W_\lambda^2$, that is, half a period of $W_\lambda$. We notice that the trace is far away from $\pm 2$: indeed, for \eqref{toybeam2} the threshold of instability is obtained in correspondence of the black curve (drawn numerically) inside the region of linear instability in Figure \ref{parabole}. An analogous observation holds for \eqref{toybeam}.
This shows that condition $(ii)$ in \eq{grande} is somehow equivalent to a ``large absolute value of the trace of the monodromy matrix'', namely
\begin{center}
{\bf nonlinear instability corresponds to a sufficiently large Floquet multiplier}
\end{center}
and is detected only if we are ``sufficiently deep inside'' the resonance tongues.
In fact, we can make this more precise, by observing that the expansion rates reported in Table \ref{expansionraten} all lie around 100. This has a simple explanation: Definition \ref{unstable}, with the choice $\eta=0.1$, requires that, in the time interval $[0, \tau]$, the residual mode becomes 10 times larger than in the interval $[0, \tau/2]$. If the considered equation were linear, by the definition of expansion rate in \eqref{ERtau} we would have $\mathcal{ER}_{\tau/2}=\sqrt{\mathcal{ER}_\tau}$ and in order to fulfill the second condition in \eqref{grande} one should have
$$
\sqrt{\mathcal{ER}_\tau}=\frac{\mathcal{ER}_\tau}{\mathcal{ER}_{\tau/2}}= 10,
$$
so that $\mathcal{ER}_\tau=100$. Of course, the presence of a nonlinear term alters the value of $\mathcal{ER}_\tau$, but the general principle remains true also for the nonlinear equation \eqref{prototipo}. Hence,
\begin{center}
{\bf
the expansion rate \eqref{ERtau} provides a fairly precise measure of nonlinear instability.
}
\end{center}
We also tested other choices of $\eta$, such as $\eta=1/15$, $\eta=0.08$, $\eta=0.125$, always approximately finding an expansion rate equal to $1/\eta^2$ (see Tables \ref{instabsupertutti} and \ref{instabsuperL8} in Section \ref{sezesperim}).
\par
Nonlinear instability thus seems more in line with reflecting the occurrence of an abrupt and significant phenomenon inside the considered structure, so that
\begin{center}
{\bf Definition \ref{unstable} is more application-oriented and thus more suitable for practical analysis.}
\end{center}

In Figure \ref{confrontolinn} we deal with \eqref{toybeam2}, comparing the values of $\mathbb{E}_{12}^\ell(a)$ (bright) and $\mathbb{E}_{12}(a)$ (dark) for $a \in (0, 1)$ (left picture) and showing that the behaviors of $\mathbb{E}_{12}(a)$ and $\mathbb{E}_{12}^\ell(a)$ for $a \in (0.4, 0.5)$ are qualitatively the same (right picture). In this case,
\begin{equation}\label{glistessi}
\textrm{{\bf for \eqref{toybeam2}, the optimal values for $a$ are the same with the two definitions of instability.
}}
\end{equation}
The unexpected large value of $\mathbb{E}_{12}(0.1)$ has an explanation, as well; since the corresponding ratio $\lambda_1/\lambda_0 \approx 1.185$ is close to 1,
the trace of the monodromy matrix grows very slowly as $\delta$ increases, therefore it requires a large amount of energy in order to reach a sufficiently large trace
yielding nonlinear instability. Basically, the black curve in Figure \ref{parabole} has an asymptote for $\rho/\lambda=1$.\par
Summarizing, we may conclude that
\begin{equation}\label{indizio}
{\textrm{\bf linear instability is a clue for the possible occurrence of nonlinear instability,}}
\end{equation}
the latter requiring larger energies to occur. In particular, if there is no linear instability (as in the case $\rho<\lambda$, see both
Figures \ref{monodromymnd1} and \ref{parabole}), one should not expect nonlinear instability.

\begin{figure}[ht!]
\begin{center}
\includegraphics[height=63mm,width=158mm]{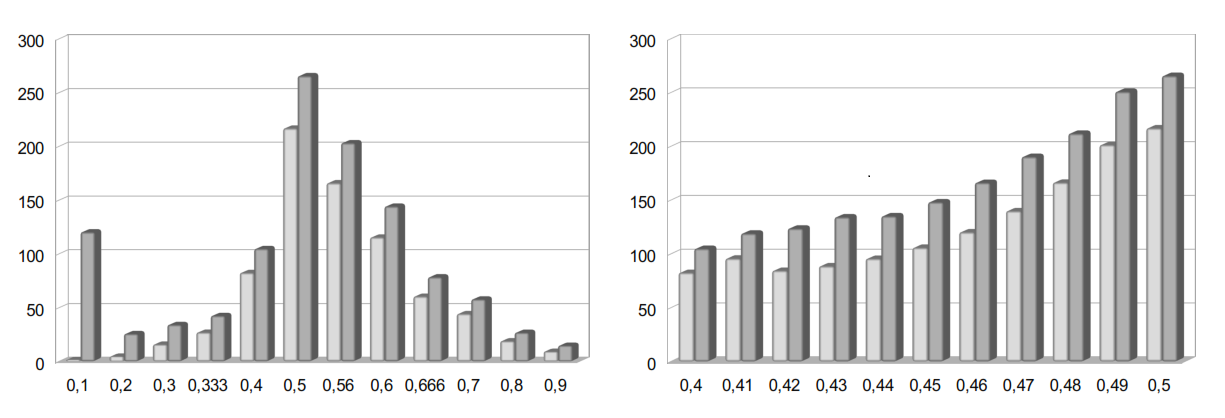}
\caption{Comparison between $a \mapsto \mathbb{E}_{12}^\ell(a)$ (brighter) and $a \mapsto \mathbb{E}_{12}(a)$ (darker) for problem \eqref{toybeam2}.\label{confrontolinn}}
\end{center}
\end{figure}

Beyond the features illustrated so far, there is another crucial concern for choosing the definition of instability to be used. When the modes are mixed and there is no possibility for a bi-modal solution to exist, one cannot define linear instability restricting the attention on a $2\times2$ system.
On the contrary, nonlinear instability may be checked also in cases where no nontrivial invariant subspaces exist for the considered equation. Thus, Definition \ref{unstable} is sufficiently flexible for practical purposes and, in the sequel, we will always stick to it when speaking about instability.\par

\subsection{Equations which tend to mix the modes}\label{nonlinevol}

\subsubsection{How do the modes mix?}\label{howmix}

In this section, we discuss particular situations where the modes are mixed by the nonlinearity, thereby displaying a relevant physiological energy transfer. For simplicity, we only analyze the case of equation \eqref{prototipo} with $\gamma_1=\gamma_3=0$, leading to
\begin{equation}\label{prototipo3}
u_{tt}+u_{xxxx}-\gamma_2\Vert u_{x} \Vert_{L^2}^2 u_{xx}+ f(u) = 0.
\end{equation}
Let us first quickly comment on the nonlinear equation
\neweq{beamstretch}
u_{tt}+u_{xxxx}-\|u_x\|^2_{L^2}\, u_{xx}=0\qquad x \in I, \quad t>0,
\endeq
that corresponds to \eqref{prototipo3} with $\gamma_2=1$ and $f\equiv 0$, thereby fitting in the second case of \eq{versus}.
According to Definition \ref{soluzionedebole}, a weak solution $u$ of \eq{beamstretch} satisfies $u\in C^0(\R_+;V(I))\cap C^1(\R_+;L^2(I))\cap C^2(\R_+;V'(I))$ and
$$
\langle u_{tt}, v \rangle_{V}+\int_I u_{xx} v''+\|u_x\|^2_{L^2}\int_I u_xv' = 0\qquad\forall v\in V(I),\quad t>0.
$$
Although \eq{beamstretch} strongly resembles to \eq{toybeam}, its solutions behave quite differently since they tend to mix the modes, thus reducing the number of invariant subspaces, see Proposition \ref{allinvariant}.
For this reason, the study of two-modes systems is useless, since two-modes solutions do not satisfy \eqref{beamstretch}.
This is also related to the discussion in Section \ref{secondorder}, where we showed that the operator $L$ defined on $V(I)$ by
$\langle Lu, v\rangle_V=\int_I u'' v''$ is not the square of the operator $\mathcal{L}$ defined on $W(I)$ by $\langle \mathcal{L}u, v\rangle_W = \int_I u' v'$. Let
$$
\Delta_{\lambda,\nu}:=\int_I e_\lambda' e_\nu'\ \quad\Longrightarrow\ \quad \left\{\begin{array}{l}
|\Delta_{\lambda,\nu}|\le\Upsilon_\lambda\Upsilon_\nu\mbox{ by the H\"older inequality}\\
\Delta_{\lambda,\nu}=0\mbox{ if $e_\lambda$ and $e_\nu$ have different parities}\\
\Delta_{\lambda,\nu}=0\mbox{ if $e_\lambda$ and $e_\nu$ both belong to $C^\infty(I)$}
\end{array}\right.
$$
and notice that $\Delta_{\lambda,\lambda}=\Upsilon_\lambda^2$, as defined in \eq{defupsilon}; here $e_\lambda$ and $e_\nu$ are $L^2$-normalized eigenfunctions of \eq{autovsym}
relative to the eigenvalues $\mu=\lambda^4$ and $\mu=\nu^4$.\par
If we take initial data $u_0\in V(I)$ and $u_1\in L^2(I)$, then the solution of \eq{beamstretch} can be written in Fourier series as
$$
u(x, t)=\sum_\lambda \varphi_\lambda(t) e_\lambda(x),
$$
where the Fourier coefficients $\varphi_\lambda$ satisfy the infinite-dimensional system of ODEs:
\begin{equation}\label{infinitostretch}
\ddot{\varphi}_\lambda(t)+\lambda^4 \varphi_\lambda(t)+
\bigg[\sum_\rho \Upsilon_\rho^2\varphi_\rho(t)^2+2\sum_{\nu>\rho}\Delta_{\nu,\rho}\varphi_\rho(t)\varphi_\nu(t)\bigg]
\cdot\bigg[\sum_\rho\Delta_{\lambda,\rho}\varphi_\rho(t)\bigg]=0
\end{equation}
which shows that the modes are ``mixed up'' having strong interactions with each other. In particular,
\begin{center}
{\bf even if some mode is initially at rest, it may become nontrivial as $t$ varies and}\\
{\bf there exist no solutions of \eq{beamstretch} with just one active mode, that is, in the form \eq{separazione}.}
\end{center}

We then consider the case where the nonlinearity acts as a local restoring force (namely, $\gamma_1=\gamma_2=\gamma_3=0$ in \eqref{prototipo}).
Motivated by the nonlinearities appearing in the dynamics of suspension bridges \cite{sepe1, plautdavis}, a possible simplified choice is given by $f(u)=u^3$; this leads to the equation
\begin{equation}\label{beamnonlineare}
u_{tt} + u_{xxxx} + u^3 = 0, \qquad x \in I, \quad t > 0,
\end{equation}
complemented with the boundary and internal conditions \eqref{bcbeamnon}.
According to Definition \ref{soluzionedebole}, a weak solution $u$ of \eq{beamnonlineare} satisfies $u\in C^0(\R_+;V(I))\cap  C^1(\R_+;L^2(I))\cap  C^2(\R_+;V'(I))$
and the equality
$$
\langle u_{tt}, v \rangle_{V}+\int_I \Big(u_{xx} v'' + u^3 v\Big)  = 0\qquad\forall v\in V(I),\quad t>0.
$$
Moreover, $u\in C^2(\overline{I}\times\R_+)$ and $u_{xx}(-\pi,t)=u_{xx}(\pi,t)=0$ for all $t>0$, see Proposition
\ref{galerkin1}. By taking initial data
$$
u(x, 0)=u_0(x)=\sum_{n \in \mathbb{N}} \alpha_n e_n, \quad u_t(x, 0)=u_1(x)=\sum_{n \in \mathbb{N}} \beta_n e_n,
$$
every weak solution of \eqref{beamnonlineare} can be expanded in Fourier series as
$$
u(x, t)=\sum_{n \in \mathbb{N}} \varphi_n(t) e_n(x),
$$
and we consider the infinite-dimensional system of ODEs obtained by projecting \eqref{beamnonlineare} onto the eigenspace spanned by each eigenfunction $e_n$:
\begin{equation}\label{sistemainfinito}
\ddot{\varphi}_n(t) + \lambda_n^4 \varphi_n(t) + A_n \varphi_n(t)^3 + \Big(\sum_{m \in \mathbb{N} \atop m \neq n} B_{n, m} \varphi_m(t) \Big) \varphi_n(t)^2 + \Big(\sum_{m, p \in \mathbb{N} \atop m, p \neq n} C_{n, m, p} \varphi_m(t) \varphi_p(t) \Big) \varphi_n(t) + D_n(t)  = 0,
\end{equation}
for $n \in \mathbb{N}$. Here
\begin{equation}\label{coefficienti}
A_n=\int_I e_n^4, \quad B_{n, m}= 3\int_I e_n^3 e_m, \quad C_{n, m, p}= 3 \int_I e_n^2 e_m e_p,
\end{equation}
and $D_n$ does not contain $\varphi_n$. It is worthwhile noticing that all the constants in \eqref{coefficienti} vary with $n$. This is not
the case in hinged beams without piers, for which the equivalent infinite-dimensional ODE system simply reads
$$
\ddot{\varphi}_n(t)+\lambda_n^4 \varphi_n(t) + \frac34 \varphi_n(t)^3 -\frac34 \varphi_{3n}(t)\varphi_n(t)^2 +\mathcal{C}_n(t)\varphi_n(t) +\mathcal{D}_n(t)=0,
$$
see \cite[Lemma 16]{GarGaz}: these equations should be compared with \eq{sistemainfinito}. However, in view of the symmetry of the interval
$I$, we immediately see that the only possible nonzero terms in $D_n$ are the products
of the kind $\varphi_m \varphi_p \varphi_q$, with $m+p+q+n$ even. As a consequence, the following statement holds.
\begin{proposition}\label{fisiologici}
Let $u_0(x)=\alpha_j e_j$, $u_1(x)=\beta_j e_j$ for some $j \in \mathbb{N}$. Then, there exist $C^2$-functions $\varphi_n$, infinitely many of which not identically zero, such that the solution of \eqref{beamnonlineare} is given by
$$
u(x, t)=\sum_{(n-j)\textnormal{ mod } 2=0} \varphi_n(t) e_n(x).
$$
\end{proposition}
This result can be deduced from the uniqueness for the ODE problem, by noticing that, under the assumptions of Proposition \ref{fisiologici}, $\varphi_n \equiv 0$ is a solution of \eqref{sistemainfinito} for every $n$ such that $(n-j) \textnormal{ mod } 2 \neq 0$.
In particular, recalling the discussion after Definition \ref{unstable}, Proposition \ref{fisiologici} states that
\begin{center}
{\bf for \eqref{beamstretch} and \eqref{beamnonlineare}, physiological energy transfers may occur}\\
{\bf only between modes with the same parity.}
\end{center}

\subsubsection{Optimal position of the piers in beams}

In this section, we numerically evaluate the energy thresholds of nonlinear instability of \eqref{beamstretch} and \eqref{beamnonlineare} for different values of $a$, comparing them with the ones for a free beam with no internal piers (see \cite{GarGaz}).
In view of the observations in the previous section, the study of linear stability for equation \eqref{prototipo3} does not make sense, because two-modes systems are not approximations of the original problem. We are thus forced to study only the nonlinear instability. A first problem in this direction is represented by the definition of the Wagner time $T_W$, which has been introduced in \eqref{sceltawagner} for equations \eqref{toybeam} and \eqref{toybeam2} and deeply relies on the existence of uni-modal solutions. However, we proceed similarly by linearizing the solutions of \eqref{prototipo3} with $j$-th prevailing mode around the approximate solution which has only the $j$-th component, obtaining the Duffing equations
\begin{equation}\label{duffst}
\ddot{W}_{\lambda_n}+\lambda_n^4 W_{\lambda_n}+\Upsilon_{\lambda_n}^4 W_{\lambda_n}^3 = 0, \quad  W_{\lambda_n}(0)=\delta, \quad \dot{W}_{\lambda_n}(0)=0
\end{equation}
for \eqref{beamstretch}, with $\Upsilon_\lambda$ defined as in \eqref{defupsilon},
and
\begin{equation}\label{duffcu}
\ddot{W}_{\lambda_n}+\lambda_n^4 W_{\lambda_n}+A_n W_{\lambda_n}^3 = 0, \quad  W_{\lambda_n}(0)=\delta, \quad \dot{W}_{\lambda_n}(0)=0
\end{equation}
for \eqref{beamnonlineare}, with $A_n$ as in \eqref{coefficienti}.
The Wagner time $T_W=T_W(\delta)$ for \eqref{beamstretch} (resp., for \eqref{beamnonlineare}) is then defined as the period of the solution of \eqref{duffst} (resp., \eqref{duffcu}). It can be computed via a formula similar to \eqref{TW}.
\par
To analyze the nonlinear instability, we truncate systems \eqref{infinitostretch} and \eqref{sistemainfinito} at the order $N=12$, as in \eqref{tronca}-\eqref{finitodim}, and numerically integrate the resulting finite-dimensional systems
\begin{equation}\label{sistemaS}
\ddot{\varphi}_n^A(t)+\lambda_n^4 \varphi_n^A(t)+
\bigg[\sum_{m=0 \atop m \neq n}^{11} \Upsilon_{\lambda_m}^2\varphi_m^A(t)^2+\sum_{m, p=0 \atop m \neq p}^{11}\Delta_{\lambda_p, \lambda_m}\varphi_{\lambda_m}^A(t)\varphi_{\lambda_p}^A(t)\bigg]
\cdot\bigg[\sum_{m=0 \atop m \neq n}^{11} \Delta_{\lambda_n,\lambda_m}\varphi_{\lambda_m}^A(t)\bigg]=0
\end{equation}
\begin{equation}\label{sistemaN}
\ddot{\varphi_n^A}(t) + \lambda_n^4 \varphi_n^A(t) + A_n \varphi_n^A(t)^3 + \Big[\sum_{m=0 \atop m \neq n}^{11} B_{n, m} \varphi_m^A(t) \Big]
\varphi_n^A(t)^2 + \Big[\sum_{m, p=0 \atop m, p \neq n}^{11} C_{n, m, p} \varphi_m^A(t) \varphi_p^A(t) \Big]\varphi_n^A(t) + D_n(t)  = 0,
\end{equation}
for $n=0, \ldots, 11$, where $u^A(x, t)=\sum_{n=0}^{11} \varphi_n^A(t) e_n(x)$.
 For every prevailing mode $j$, we determine the $j$-th energy threshold $E_j(a)$; the performed numerical analysis suggests that $E_j(a) < +\infty$ for every $j$ and therefore that $a \mapsto \mathbb{E}_{12}(a)$ is continuous, see Theorem \ref{sogliacontinua}.
\par
In Table \ref{stretchbrief}, we display the energy thresholds of nonlinear instability obtained for \eqref{sistemaS}, see Figure \ref{stretchingfigure} (where ``no'' denotes the case with no piers) and Table \ref{instabstretch} in Section \ref{sezesperim}; we will further comment on them in Section \ref{sezioneconfronti}.
\begin{table}[!ht]
{\small
\begin{center}
\begin{tabular}{|c|c|c|c|c|c|c|c|c|c|c|c|c|c|}
\hline
$a$ & 0.1 & 0.2 & 0.3 & 1/3 & 0.4 & 0.5 & 0.56 & 0.6 & 2/3 & 0.7 & 0.8 & 0.9 & no piers \\
\hline
$\mathbb{E}_{12}(a)$ & $19.8$ & $10.9$  & $7.5$ & $11.3$ & $17.3$ & 35.2 & 49.6 & 77.4 & 143.4 & 168.1 & 832.7 & 659 & 198.3\\
\hline
\end{tabular}
\caption{Energy thresholds of instability for \eqref{sistemaS}.}
\label{stretchbrief}
\end{center}
}
\end{table}
\begin{figure}[ht!]
\begin{center}
\includegraphics[scale=0.55]{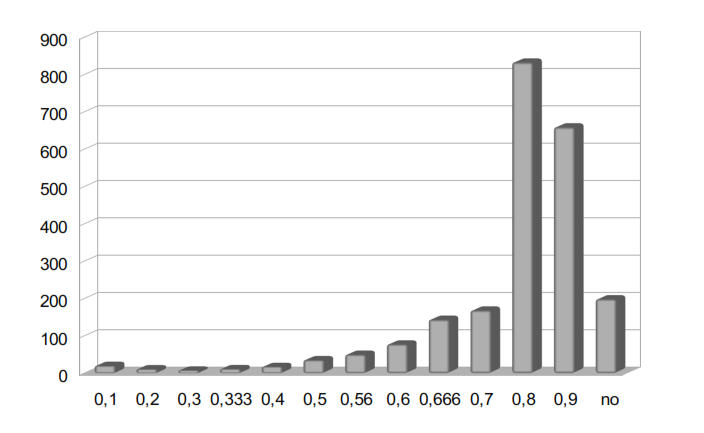}
\end{center}
\caption{Energy thresholds of nonlinear instability for \eqref{sistemaS} on varying of $a$.}
\label{stretchingfigure}
\end{figure}

In the right picture of Figure \ref{figurasoglie}, we deal with system \eqref{sistemaN} and we draw an approximation of the graph of the function $a \mapsto \mathbb{E}_{12}(a)$, where we exploit the data reported in Table \ref{tabellasoglie} (in Section \ref{sezesperim}). We notice that the resulting graph reaches two peaks, one around $0.35$ and the other around $0.5$, the latter being the optimal choice of $a$ for the beam to remain stable. Making a comparison with the instability threshold for the beam without piers, represented by the horizontal line, we deduce that
\begin{eqnarray}\label{claimbeam}
\textrm{{\bf for \eq{sistemaN}, the piers make the beam more stable}} \qquad \qquad & \vspace{-0.2cm} \\
\textrm{{\bf unless the central span is much shorter than the lateral ones.}} &\nonumber
\end{eqnarray}
Looking at Table \ref{tabellasoglie}, we also notice that $\mathbb{E}_{12}(a)$ is given either by $E_1(a)$ or by $E_2(a)$, namely the most fragile mode is  either the second or the third. We can be slightly more precise in this characterization by looking at the left picture of Figure \ref{figurasoglie}, where we plot the two graphs (obtained via interpolation from Table \ref{tabellasoglie}) of $a \mapsto E_1(a)$ and $a \mapsto E_2(a)$. Analyzing the shape of the third eigenfunction for $a \approx 0.354$ and $a \approx 0.517$ (see Figure \ref{fterze}), corresponding to $E_1(a) \approx E_2(a)$,  we notice that
\begin{center}
{\bf for \eq{sistemaN}, the second mode is the most fragile, unless the third \\ has a zero which is ``very close'' to the piers.}
\end{center}
However, this closeness has to be measured in an asymmetric way since a zero on the central span may be considered close to a pier even if its distance from it is larger. This suggests that zeros on the central span are more subject to instability than zeros on the lateral spans.
The interested reader may have a look at the numerical results and at the discussion in Section \ref{sezesperim}.
\begin{figure}[ht!]
\begin{center}
\includegraphics[scale=0.45]{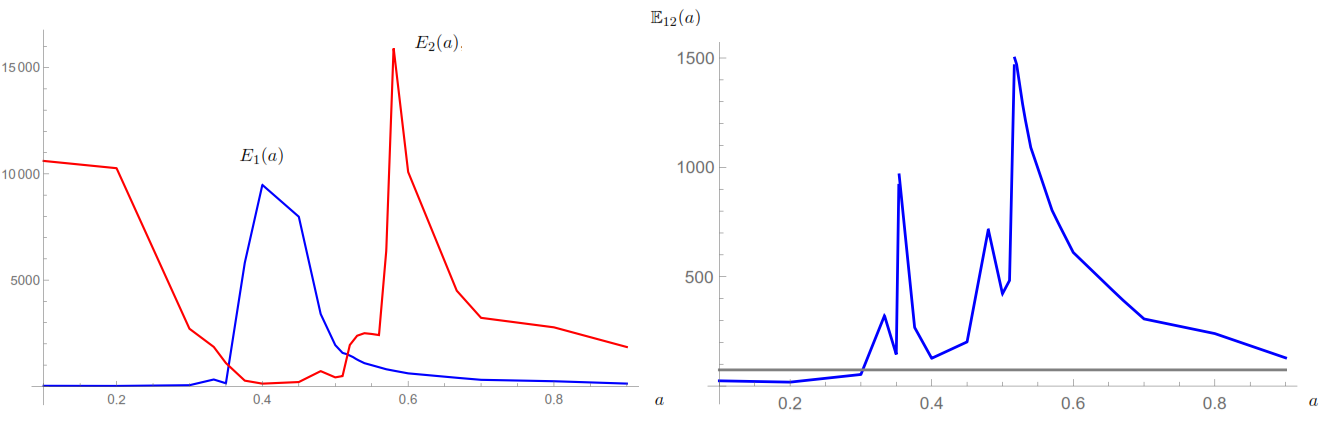}
\end{center}
\caption{On the left, the approximate graphs of $a \mapsto E_1(a)$ and $a \mapsto E_2(a)$. On the right, the approximate energy threshold
of instability for \eqref{sistemaN} on varying of $a$, given by $\min\{E_1(a), E_2(a)\}$. }
\label{figurasoglie}
\end{figure}
\par
\begin{figure}
\begin{center}
\includegraphics[height=35mm, width=15cm]{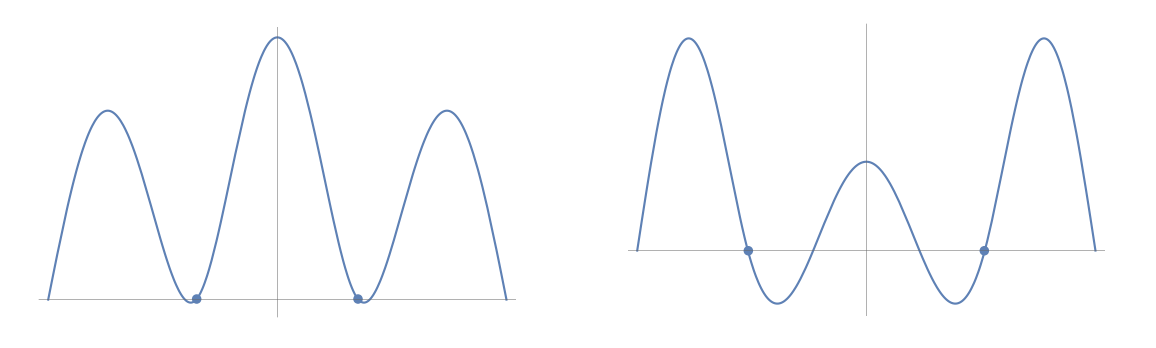}
\end{center}
\caption{The shape of the third eigenfunction for $a=0.354$ (left) and $a=0.517$ (right).}
\label{fterze}
\end{figure}

\subsection{Comparing the nonlinearities}\label{sezioneconfronti}

We discuss here the information collected in the previous sections, seeking which one among the four energies
$$
\mathcal{N}_1(u)=\frac{1}{4} \Big(\int_I u_{xx}^2 \Big)^2, \qquad \mathcal{N}_2(u)=\frac{1}{4} \Big(\int_I u_{x}^2 \Big)^2, \qquad  \mathcal{N}_3(u)=\frac{1}{4} \Big(\int_I u^2 \Big)^2, \qquad  \mathcal{N}_4(u)=\frac{1}{4} \int_I u^4,
$$
introduced in Section \ref{physmod}, appears more suitable for settling a model which describes the behavior of suspension bridges.
In this respect, we use three criteria to evaluate the suitability of each energy:
\begin{itemize}
\item[a)] we expect the results for the optimal position of the piers to be sufficiently robust so as not to depend on the notion of instability considered. In particular, we aim at obtaining the same optimal $a$ using both Definitions \ref{unstable} and \ref{defstabb} (when available), expecting claim \eqref{indizio} to be fully respected;
\item[b)] a nonlinearity with \emph{small physiological energy transfers} between modes (recall the discussion in Section \ref{4.4}) is more in line with what is observed in actual bridges: as already remarked, from \cite{ammvkwoo} we know that ``\emph{one mode of oscillation prevailed}''  and hence the oscillations of a suspension bridge appear in general close to a pure mode;
\item[c)] we expect the optimal $a$ to lie in the physical range \eqref{physicalrange}, representing the proportion between spans which is generally respected in the construction of real bridges. Moreover, the presence of the piers should generally improve the stability of the structure.
\end{itemize}
Let us first focus on the two energies $\mathcal{N}_1$ and $\mathcal{N}_3$, studied in Section \ref{nonmischia}; we have seen in Proposition \ref{allinvariant} that for both of them no physiological energy exchanges occur, since all the subspaces are invariant. However, their behavior with respect to linear stability is fairly different - see Figures \ref{monodromymnd1} and \ref{parabole}, where the instability tongues are depicted. Proposition \ref{toytheo2} states that for $\mathcal{N}_3$ there is always linear instability if the initial amplitude $\delta$ of the prevailing mode is sufficiently large, contrary to what happens for $\mathcal{N}_1$, see Theorem \ref{toytheo}.
\par
Concerning nonlinear instability, Figures \ref{confrontolinn2} and \ref{confrontolinn} bring strong arguments in favor of \eq{toybeam2} and against \eq{toybeam}. In particular,
Figure \ref{confrontolinn2} shows that for \eq{toybeam} it is $\mathbb{E}_{12}^\ell\ll\mathbb{E}_{12}$, namely there is a weak correlation between linear and
nonlinear instability, contrary to \eq{toybeam2} (apart from the case $a=0.1$, which has been explained in Section \ref{instnonlineare}).
Taking also into account \eqref{glistessi}, we thus see that the above criterion a) is fulfilled only for \eqref{toybeam2}.
Furthermore, only for \eqref{toybeam2} the optimal values of $a$ for linear and nonlinear instability both lie in the physical range (compare \eqref{optimalbdg} with \eqref{optimall2}). Finally, from Table \ref{tabellasuperbending1} in Section \ref{sezesperim} we infer that criterion c) is not satisfied for \eqref{toybeam}. Summarizing, we can conclude that
\begin{center}
{\bf the nonlinearity $\Vert u\Vert_{L^2}^2 u$ better describes the nonlocal behavior of a bridge with piers\\
than the nonlinearity $\Vert u_{xx} \Vert_{L^2}^2 u_{xxxx}$.}
\end{center}

In Section \ref{nonlinevol} we analyzed the energies $\mathcal{N}_2$ and $\mathcal{N}_4$, which behave in a much more complex way and, as we have seen, do not allow two-modes solutions of the considered problem. As for $\mathcal{N}_2$, the numerical responses for nonlinear instability reported in Table \ref{stretchbrief} are not in line with criterion c), both for the optimal value of $a$ equal to 0.8 and for the fact that a beam without piers mostly appears more stable than with piers (except for $a=0.8$ and $a=0.9$). It is thus more difficult to control the stability if the stretching effects can move across the piers due to the sliding of the cable (namely, multiplying $u_{xx}$ by the integral of $u_x^2$ on the whole $I$). Criterion c) is instead satisfied by $\mathcal{N}_4$, as we can infer by looking at Figure \ref{figurasoglie} (see also \eqref{claimbeam}). Hence,
\begin{center}
{\bf the nonlinearity $u^3$ better describes the behavior of a bridge with piers\\
than the nonlinearity $\Vert u_{x} \Vert_{L^2}^2 u_{xx}$.}
\end{center}
\begin{table}[ht!]
\begin{center}
{\footnotesize
\begin{tabular}{|c|c|c|c|c|c|c|c|c|c|c|c|c|c|}
\hline
$a$ & $0.1$ & 0.2 & 0.3 & 1/3 & 0.4 & 0.5 & 0.56 & 0.6 & 2/3 & 0.7 & 0.8 & 0.9 & no piers\\
\hline
$\mathbb{E}_{12}(a)$ & 23.1 & 17.7 & 52.1 & 319.2 & 126.8 & 414.1 &  878.6 & 611.6 & 403.2  & 299.1 & 230.3 & 127.5 & 73.4 \\
\hline
$\mathbb{E}_{*}(a)$ & 22.4 & 17.7 & 37.9 & 84.7 & 132.2 & 510  &  448.7 & 332.3 & 196.2  & 151.3 & 76.9 & 43.2 & 8.1 \\
\hline
\end{tabular}
}
\caption{Nonlinear instability for the local problem with $f(u)=u^3$: $12$ vs $2$ modes ($\mathbb{E}_{12}(a)$ vs $\mathbb{E}_*(a)$).}
\label{tabellamodoresiduo}
\end{center}
\end{table}

\begin{table}[ht!]
\begin{center}
{\footnotesize
\begin{tabular}{|c|c|c|c|c|c|c|c|c|c|c|c|c|c|}
\hline
$a$ & $0.1$ & 0.2 & 0.3 & 1/3 & 0.4 & 0.5 & 0.56 & 0.6 & 2/3 & 0.7 & 0.8 & 0.9 & no piers\\
\hline
$\mathbb{E}(a)$ & $120.4$ & $25.7$ & $34.1$ & $42.6$ & $104.9$ & $263.8$/$265.6$  &  $203.3$ & $144.3$ & $78.5$  & $57.9$ & $26.9$ &$14.8$ & $3.8$ \\
\hline
\end{tabular}
}
\caption{Energy thresholds of nonlinear instability for problem \eqref{toybeam2}.}
\label{tabellasuperL2brief}
\end{center}
\end{table}
However, Table \ref{tabellamodoresiduo}, together with Proposition \ref{fisiologici}, highlights the relevance of physiological energy transfers for $\mathcal{N}_4$. Indeed, the striking differences between the 2-modes and the 12-modes systems suggest to expect ``significant physiological spread of energy'' in the infinite-dimensional system. Even more, Proposition \ref{fisiologici} states that the modes having the same parity as the prevailing one are physiologically affected by the nonlinearity and can significantly grow, even if not abruptly (as also observed numerically). Criterion b) above is instead satisfied by $\mathcal{N}_3$, as shown by Table \ref{tabellasuperL2brief}, where we always find the same values of the critical energy except for $a=0.5$ (where the difference, probably due to a small numerical tolerance, is anyway not significant).
\par
Summarizing, criteria a), b) and c) are only satisfied in case of equation \eqref{toybeam2} and we can thus conclude that
\begin{center}
{\bf the nonlocal term $\Vert u\Vert_{L^2}^2 u$ is the nonlinearity better describing \\ the behavior of real structures with piers.
}
\end{center}
For this reason, when analyzing a degenerate plate model in the next section, the nonlinear restoring force due to the cables will be modeled by a suitable version of $\mathcal{N}_3$.

\subsection{Tables and numerical results}\label{sezesperim}

In this section, we report with more details some of the numerical results presented in the previous sections and briefly comment on each of them.
To ease their readability, we here list some notation: $\mathbb{E}_{12}^\ell$ and $\mathbb{E}_{12}$ will respectively denote the energy thresholds of linear (when defined) and nonlinear instability, as in \eqref{formulalineare} and \eqref{Edodici}. The energy threshold of nonlinear instability for a $2$-modes approximation will instead be denoted by $\mathbb{E}_*$. The critical amplitude for nonlinear instability will be denoted by $\delta$, while $\tau$ and $T_W$ will respectively indicate the time needed to observe instability and the Wagner time, cf. Definition \ref{unstable} and \eqref{sceltawagner}. Finally, $e_j$ and $e_k$ will respectively denote the eigenfunctions corresponding to the prevailing and to the residual mode fulfilling Definition \ref{unstable}, keeping the same notation therein.
\par
$\bigstar$ First, we report the numerical results obtained for equation \eqref{toybeam}, where the nonlinearity is represented by a superquadratic bending term. In
Table \ref{instabsuperblin}, we display the energy thresholds of linear stability, while in Tables \ref{tabellasuperbending1} and \ref{tabellasuperbending2} we compare the energy and the amplitude thresholds of nonlinear instability found numerically by integrating the system with $2$ and with $12$ modes, respectively. As described also in Section \ref{algoritmo}, the initial kinetic datum is taken identically equal to zero, while the potential one is taken equal to $0.01$ on each residual mode. We complement the information about the mentioned thresholds with some more details, in particular with the value of $\delta \lambda^2$ in correspondence of nonlinear instability - in order to make a comparison with Table \ref{instabsuperblin}, cf. also Figure \ref{monodromymnd1}. We notice that, though the couple $(j, k)$ is not the same passing from linear to nonlinear instability, yet $\delta \lambda^2$ is always larger for nonlinear instability, confirming claim \eqref{indizio}.
We also observe that higher modes seem to need less time to fulfill Definition \ref{unstable}; in any case, in all our experiments it is $T_W < \tau/2$, meaning that in the determination of the thresholds we are not misled by physiological energy exchanges.
\begin{table}[ht!]
\begin{center}
{\footnotesize
\begin{tabular}{|c|c|c|c|c|c|c|c|c|c|c|c|c|}
\hline
$a$ & $0.1$ & $0.2$ & $0.3$ & $1/3$ & $0.4$ & $0.5$ & $0.56$ & $0.6$ & $2/3$ & $0.7$ & $0.8$ & $0.9$ \\
\hline
$\mathbb{E}_{12}^\ell$ & 0.142 & 0.211 & 0.149 & 0.313 & 0.259 & 0.326 & 0.296 & 0.303 & 0.212 & 0.583 & 0.340 & 1.064 \\
\hline
ratio & $\lambda_7/\lambda_6$ & $\lambda_6/\lambda_5$ &
$\lambda_{10}/\lambda_{9}$ &  $\lambda_7/\lambda_6$ &  $\lambda_5/\lambda_4$&
$\lambda_{10}/\lambda_{9}$  & $\lambda_8/\lambda_7$ & $\lambda_{9}/\lambda_8$ &
$\lambda_{11}/\lambda_{10}$ & $\lambda_5/\lambda_4$ & $\lambda_{9}/\lambda_8$ &
$\lambda_{10}/\lambda_{9}$ \\
\hline
$\delta \lambda^2$	& 0.503 & 0.599 & 0.514 & 0.708 & 0.654 & 0.72 & 0.692 & 0.699 & 0.6 & 0.909 & 0.733 & 1.137 \\
\hline
\end{tabular}
}
\caption{Parameters of linear instability for \eqref{toybeam} with two-modes.}
\label{instabsuperblin}
\end{center}
\end{table}

\begin{table}[ht!]
\begin{center}
\resizebox{\textwidth}{!}{\footnotesize
\begin{tabular}{|c|c|c|c|c|c|c|c|c|c|c|c|c|c|}
\hline
$a$ & 0.1 & 0.2 & 0.3 & 1/3 & 0.4 & 0.5 & 0.56 & 0.6 & 2/3 & 0.7 & 0.8 & 0.9 &
no piers \\
\hline
$\mathbb{E}_*$ &  2.139 & 4.837 & 11.708 & 22.577 & 25.213 & 174.011 & 93.632 & 308.225 & 373.834 & 248.404 & 128.284 & 114.644 & 1483.48\\
\hline
$\delta$ & 0.87 & 0.98 & 1.1 & 1.3 & 0.89 & 1.26 & 0.72 & 0.75 & 6.06 & 5.87 & 6.02 & 6.91 & 34.88\\
\hline
$e_j$ & $e_0$ & $e_0$ & $e_0$ & $e_0$ & $e_1$ & $e_1$ & $e_2$ & $e_3$ & $e_0$ &$e_0$ &$e_0$ & $e_0$ & $e_0$\\
\hline
$e_k$ & $e_1$ & $e_1$ & $e_1$ &$e_1$ &$e_2$ & $e_2$ & $e_3$ & $e_4$ & $e_1$ &$e_1$ &$e_1$  & $e_1$ & $e_1$ \\
\hline
$\delta \lambda^2$ & 1.446 & 1.873 & 2.432& 2.925 & 3.015 & 5.04 & 4.287 & 5.841 & 6.138 & 5.526 & 4.655 & 4.52 & 8.72
\\
\hline
$\tau$  & 15.138 & 7.635 & 4.170 & 3.462 & 2.257 & 1.176 & 0.941 & 0.534 & 3.840 &
4.581 & 6.586 & 8.011 & 15.997 \\
\hline
$T_W$ & $2.377$ & 1.743 & 1.236 & $1.042$ & $0.674$ & $0.357$& $0.279$ & $0.159$ & $1.17$ & $1.392$ & $1.994$& $2.422$ & 3.369 \\
\hline
\end{tabular}
}
\caption{Parameters of nonlinear instability for \eqref{toybeam} with $2$ modes.}
\label{tabellasuperbending1}
\end{center}
\end{table}

\begin{table}[ht!]
\begin{center}
\resizebox{\textwidth}{!}{\footnotesize
\begin{tabular}{|c|c|c|c|c|c|c|c|c|c|c|c|c|}
\hline
$a$ & 0.1 & 0.2 & 0.3 & 1/3 & 0.4 & 0.5 & 0.56 & 0.6 & 2/3 & 0.7 & 0.8 & 0.9  \\
\hline
$\mathbb{E}_{12}$ & 3.087 & 5.699 & 17.378 & 38.131 & 55.172 & 274.462 & 655.22 & 605.881 & 1362.99 & 7175.86 & 274.217 & 216.296 \\
\hline
$\delta$ & 0.98 & 1.03 & 1.23 & 1.5 & 1.1 & 1.12 & 1.19 & 2.05 & 2.96 & 1.07 & 7.33 & 8.15  \\
\hline
$e_j$ & $e_0$ &$e_0$ &$e_0$ &$e_0$ & $e_1$ & $e_2$ & $e_2$ & $e_1$ & $e_1$ & $e_4$ & $e_0$ & $e_0$   \\
\hline
$e_k$ & $e_1$ & $e_1$ &$e_1$ &$e_1$ & $e_2$ & $e_3$ & $e_3$ & $e_5$ & $e_4$ & $e_{10}$ & $e_1$ &$e_1$  \\
\hline
$\delta \lambda^2$ & 1.628 & 1.969 & 2.719 & 3.375 & 3.726 & 5.67 & 7.085 & 6.945 & 8.534 &12.977  & 5.668 & 5.332 \\
\hline
$\tau$ & 15.236 & 8.585 & 4.719 & 3.005 & 1.848 & 0.836 & 0.749 & 2.093 & 2.448 & 6.214 & 7.013 & 8.802 \\
\hline
$T_W$ & 2.206 & 1.682 & 1.127 & 0.919  & 0.558 & 0.252 & 0.173 & 0.31 &0.298 & 0.046 &1.654 & 2.073 \\
\hline
\end{tabular}
}
\caption{Parameters of nonlinear instability for \eqref{toybeam} with $12$ modes.}
\label{tabellasuperbending2}
\end{center}
\end{table}

\smallbreak
\noindent
$\blacklozenge$ We then report the numerical results obtained for equation \eqref{toybeam2}, where the nonlinearity is given by the superquadratic $L^2$-term $\Vert u \Vert_{L^2}^2 u$. In
Table \ref{instabsuperLlin}, we display the energy thresholds of linear instability, while in Table \ref{tabellasuperL2} we report the energy thresholds of nonlinear instability found numerically. Here it is practically not necessary to distinguish between the approximations with $2$ and with $12$ modes, as we have already seen in Section \ref{sezioneconfronti}. The couple ($j$, $k$) for which Definition \ref{unstable} is fulfilled is always the same, $e_0$ being the prevailing mode and $e_1$ being the residual mode displaying instability. This is not any more true for $a$ strictly between $0.4$ and $0.5$, as we have somehow already noticed with Figure \ref{zoomenergysuperL2}; here the couple (prevailing mode, residual mode) is given by $(e_1, e_2)$. In Table \ref{curioso}, we thus find worth zooming on the behavior of the energy thresholds of instability for $a$ belonging to such a range.

\begin{table}[ht!]
\begin{center}
{\footnotesize
\begin{tabular}{|c|c|c|c|c|c|c|c|c|c|c|c|c|}
\hline
$a$ & $0.1$ & $0.2$ & $0.3$ & $1/3$ & $0.4$ & $0.5$ & $0.56$ & $0.6$ & $2/3$ & $0.7$ & $0.8$ & $0.9$ \\
\hline
$\mathbb{E}_{12}^\ell$ & 1.678 & 4.737 & 15.88 & 27.033 & 82.586  & 216.953 & 166.066 & 115.552 & 60.592 & 44.383 & 19.017  &  9.326   \\
\hline 
ratio & $1.185$ & $1.277$ &
$1.456$ & $1.642$ & 2.679 &
6.556  & 7.918 & 8.128 & 8.102
& 8.033 & 7.808 &
7.652 \\
\hline
$\delta$	& 1.012 & 1.424 & 2.112 & 2.55 & 3.793 & 5.207 & 4.908 & 4.487 & 3.817  & 3.53 & 2.853 & 2.386 \\
\hline
\end{tabular}
}
\caption{Parameters of linear instability for \eqref{toybeam2} with $2$ modes.}
\label{instabsuperLlin}
\end{center}
\end{table}

\begin{table}[ht!]
\begin{center}
\resizebox{\textwidth}{!}{\footnotesize
\begin{tabular}{|c|c|c|c|c|c|c|c|c|c|c|c|c|c|}
\hline
$a$ & 0.1 & 0.2 & 0.3 & 1/3 & 0.4 & 0.5 & 0.56 & 0.6 & 2/3 & 0.7 & 0.8 & 0.9 & no piers \\
\hline
$\mathbb{E}_{12}$ & 120.44 & 25.73 & 34.17 & 42.61 & 104.93 & 263.8/265.6
& 203.35 & 144.34 & 78.53 & 57.94 & 26.92 & 14.84 & 3.88 \\
\hline
$\delta$ & 4.4 & 2.67 & 2.79 & 2.99 & 4.08 & 5.49/5.5 & 5.18 & 4.76 & 4.09 & 3.79 & 3.13 & 2.7 & 1.97 \\
\hline
$\tau$ & 15.997 & 15.652 & 15.989 & 15.364 & 14.421 & 10.6/10.55 & 11.345 & 12.36 & 12.665 & 13.701 & 14.304 & 14.082 & 15.964 \\
\hline
$T_W$ & $1.534$ & $2.109$ & $1.93$ & $1.844$ & $1.552$ & $1.277$ & $1.368$
 & $1.492$ & $1.737$ & $1.874$ & $2.27$ & $2.636$ & $3.721$ \\
\hline
\end{tabular}
}
\caption{Parameters of nonlinear instability for problem \eqref{toybeam2}.}
\label{tabellasuperL2}
\end{center}
\end{table}

\begin{table}[ht!]
\begin{center}
{\footnotesize
\begin{tabular}{|c|c|c|c|c|c|c|c|c|c|c|c|}
\hline
$a$ & 0.4 & 0.41 & 0.42 & 0.43 & 0.44 & 0.45 & 0.46 & 0.47 & 0.48 & 0.49 & 0.5 \\
\hline
$\mathbb{E}_{12}$ & 104.932 & 119.12 & 123.761 & 134.277 & 135.221 & 148.216 & 166.302 & 190.427 & 212 & 250.928 & 265.677 \\
\hline
$\delta$ & 4.08 & 3.59 & 3.6 & 3.67 & 3.64 & 3.73 & 3.86 & 4.03 & 4.17 & 4.42 & 5.5 \\
\hline
$\tau$ & 14.421 & 15.335 & 15.134 & 14.764 & 15.939 & 15.604 & 15.211 & 14.787  & 15.665 & 15.189 & 10.554 \\
\hline
$T_W$ & $1.552$ & $1.354$ & $1.334$ & $1.306$ & $1.294$ & $1.266$ & $1.235$ & $1.202$ & $1.177$ & $1.141$ & $1.275$\\
\hline
\end{tabular}
}
\caption{Parameters of nonlinear instability for \eqref{toybeam2}, with $a$ ranging from $0.4$ to $0.5$.}
\label{curioso}
\end{center}
\end{table}

\begin{table}[ht!]
\begin{center}
{\footnotesize
\begin{tabular}{|c|c|c|c|c|c|c|c|c|c|c|c|c|c|}
\hline
$\!\eta\downarrow$ $a\to$\! & 0.2 & 0.3 & 1/3 & 0.4 & 0.5 & 0.56 & 0.6 & 2/3 & 0.7 & 0.8 & 0.9 \\
\hline
1/15 &  $71.23$ & $43.94$ & $56.68$ & $139.38$ & $321.87$ & $251.59$ & $188.89$ & $102.74$ & $77.37$ &  $38.54$ & $21.89$  \\
\hline
0.08 &  $42.32$ & $35.98$ & $48.79$ & $117.43$ & $289.85$ & $226.49$ & $159.83$ & $90.02$ & $67.77$ &  $30.68$ & $17.51$  \\
\hline
0.125 &  $21.68$ & $29.81$ & $38.19$ & $98.27$ & $246.45$ & $190.36$ & $134.3$ & $71.53$ & $53.46$ &  $24.41$ & $13.24$  \\
\hline
\end{tabular}
}
\caption{The energy threshold $\mathbb{E}_*$ for the two-modes problem \eqref{toybeam2}, for different choices of $\eta$.}
\label{instabsupertutti}
\end{center}
\end{table}

\begin{table}[ht!]
\begin{center}
{\footnotesize
\begin{tabular}{|c|c|c|c|c|c|c|c|c|c|c|c|c|}
\hline
$\!\eta\downarrow$ $a\to$\! & 0.2 & 0.3 & 1/3 & 0.4 & 0.5 & 0.56 & 0.6 & 2/3 & 0.7 & 0.8 & 0.9 \\
\hline
1/15 & 3.69 & 3.04 & 3.29 & 4.44   & 5.79 & 5.48 & 5.11 & 4.39 & 4.09 & 3.44 & 2.99  \\
\hline
0.08 & 3.14 & 2.84 & 3.13 & 4.22 & 5.63& 5.33 & 4.89 & 4.24 & 3.95 & 3.24 & 2.82 \\
\hline
0.125 & 2.52 & 2.66 & 2.88 & 4 & 5.39 & 5.09  & 4.67 & 3.99  & 3.71 & 3.05 & 2.62  \\
\hline
\end{tabular}
}
\caption{The amplitude threshold of instability $\delta$  for \eqref{toybeam2}, for different choices of $\eta$.}
\label{instabsuperL8}
\end{center}
\end{table}

Once the superquadratic $L^2$-nonlinearity $\Vert u \Vert_{L^2} u$ has been identified to be the most suitable for our analysis, we have made sure that changing the value of $\eta$ appearing in Definition \ref{prevalente} does not affect the qualitative picture of our results. We chose the three values of $\eta$ equal to $1/15, 0.08, 0.125$, seeking the corresponding nonlinear instability thresholds; we report the obtained values of $\mathbb{E}_{12}$ in Table \ref{instabsupertutti}. We chose not to display the values obtained for $a=0.1$, since they are unrealistically large; this fact has already been given an explanation which relies on the particular position in the picture of the resonance tongues, see Figure \ref{parabole}. In all the other cases, the qualitative behavior of the energy thresholds is the same, as well as the couple ($j$, $k$) of Definition \ref{unstable}; the time in correspondence of which instability is found ranges from 9.74 (for $a=0.56$) to $15.995$ (for $a=0.9$), always largely more than twice the Wagner time (ranging from $1.249$ for $a=0.5$ to $2.532$ for $a=0.9$).

$\bullet$ In Table \ref{instabstretch}, we show the numerical results obtained for system \eqref{sistemaS}, where the data do not seem to follow any clear rule and are quite far from being realistic, as we have commented in Section \ref{sezioneconfronti}. Notice the irregular trend of the critical amplitude and the fact that, the higher the prevailing mode delivering energy is, the less is the time needed for instability to manifest.

\begin{table}[ht!]
\begin{center}
\resizebox{\textwidth}{!}{\footnotesize
\begin{tabular}{|c|c|c|c|c|c|c|c|c|c|c|c|c|c|}
\hline
$a$ & 0.1 & 0.2 & 0.3 & 1/3 & 0.4 & 0.5 & 0.56 & 0.6 & 2/3 & 0.7 & 0.8 & 0.9 & no piers \\
\hline
$\mathbb{E}_{12}$ & 20.216 & 10.966 & 7.501 & 11.379 & 17.394 & 35.227 & 49.654 & 77.429 & 143.492 & 168.135 & 832.751 & 659.026 & 198.391 \\
\hline
$e_j$ &  $e_2$ & $e_0$ & $e_0$  &  $e_0$  & $e_1$  & $e_1$  & $e_2$ & $e_2$ &  $e_4$ & $e_4$ & $e_8$  & $e_4$ & $e_2$ \\
\hline
$e_k$ & $e_3$ & $e_1$ & $e_1$ & $e_1$ & $e_2$ & $e_2$ & $e_3$ & $e_3$ & $e_5$ & $e_5$ & $e_9$ & $e_5$ & $e_3$ \\
\hline
$\delta$ & 0.97
& 1.69 & 1.3 & 1.47 & 1.33 & 1.46 & 1.27 & 1.41 & 1.21 & 1.17 & 1.11 & 2.31 & 3.4 \\
\hline
$\tau$ & 12.035 & 15.191 & 15.983 & 12.739 & 13.526 & 7.214 & 6.909 & 4.703 & 3.903 & 3.448 & 1.655 & 1.862 & 3.922 \\
\hline
$T_W$ & $0.899$ & 2.035 & 1.917 & 1.739 & 1.296 & 0.982 & 0.731 & 0.641 & 0.413 & 0.367 & 0.157 &0.349 & 0.913 \\
\hline
\end{tabular}
}
\caption{Parameters of nonlinear instability for the stretching problem \eqref{sistemaS}.}
\label{instabstretch}
\end{center}
\end{table}

$\bullet$ Finally, we dedicate some more words to \eqref{sistemaN}. First, in Tables \ref{tabellaprecisacubo2} and \ref{tabellaprecisacubo1} we display the energy thresholds of instability found by analyzing the $2$-modes and the $12$-modes systems, respectively. We find curious that the transfer of energy is one-way (from higher to lower modes) except for $a=0.5$ in the $2$-modes system, where the first mode delivers energy to the second one.

\begin{table}[ht!]
\begin{center}
\resizebox{\textwidth}{!}{\footnotesize
\begin{tabular}{|c|c|c|c|c|c|c|c|c|c|c|c|c|c|}
\hline
$a$ & 0.1 & 0.2 & 0.3 & 1/3 & 0.4 & 0.5 & 0.56 & 0.6 & 2/3 & 0.7 & 0.8 & 0.9 & no piers \\
\hline
$\mathbb{E}_{*}$ & 22.435 & 17.704 & 37.951 & 84.721 & 132.215 & 510.073 & 448.787 & 332.383 & 196.2 & 151.349 & 76.977 & 43.204 & 8.167 \\
\hline
$\delta$ & 3.1 & 2.5 & 2.96 & 3.91 & 3.77 & 7.8 & 6.55 & 6.05 & 5.38  & 5.09 & 4.44 & 3.95 & 2.87 \\
\hline
$e_j$ & $e_1$ & $e_1$ & $e_1$ & $e_1$ & $e_2$ & $e_0$ & $e_1$ & $e_1$ & $e_1$ & $e_1$ & $e_1$ & $e_1$ & $e_1$  \\
\hline
$e_k$ & $e_0$ & $e_0$ &$e_0$ &$e_0$ &  $e_1$ &$e_1$ & $e_0$ &$e_0$ &$e_0$ &$e_0$ &$e_0$ &$e_0$ & $e_0$ \\
\hline
$\tau$ & 14.52 & 15.97 & 15.17 & 15.05 & 14.75 & 14.76 & 15.9 & 15.53 & 15.59 & 15.91  & 15.3 & 15.74 & 13.83 \\
\hline
$T_W$ & 2.721 & 2.534 & 2.048 & 1.783 & 1.418 & 1.307 & 1.274 & 1.361 & 1.569 & 1.689  & 2.062 & 2.446 & 4.024 \\
\hline
\end{tabular}
}
\caption{Parameters of nonlinear instability for the local problem \eqref{beamnonlineare} with $2$ modes.}
\label{tabellaprecisacubo2}
\end{center}
\end{table}
\begin{table}[ht!]
\begin{center}
\resizebox{\textwidth}{!}{\footnotesize
\begin{tabular}{|c|c|c|c|c|c|c|c|c|c|c|c|c|c|}
\hline
$a$ & 0.1 & 0.2 & 0.3 & 1/3 & 0.4 & 0.5 & 0.56 & 0.6 & 2/3 & 0.7 & 0.8 & 0.9 & no piers \\
\hline
$\mathbb{E}_{12}$ & $23.196$ & 17.704 & 52.165 & 319.27 & 126.826 & 414.125 & 878.671 & 611.694 & 403.213 & 299.135 & 230.371 & 127.594 & 73.429 \\
\hline
$\delta$ & 3.14 & 2.5 & 3.38 & 6.38 & 3.7 & 5.16 & 8.33 & 7.5 & 6.9 & 6.44 & 6.41 & 5.67 & 5.58 \\
\hline
$e_j$ & $e_1$ & $e_1$ & $e_1$ & $e_1$ & $e_2$ & $e_2$ & $e_1$ & $e_1$ & $e_1$ & $e_1$ & $e_1$ & $e_1$ & $e_1$  \\
\hline
$e_k$ & $e_0$ & $e_0$ &$e_0$ &$e_0$ &  $e_1$ &$e_1$ & $e_0$ &$e_0$ &$e_0$ &$e_0$ &$e_0$ &$e_0$ & $e_0$ \\
\hline
$\tau$ & 14.478 & 15.998 & 15.66 & 15.992 & 14.8 & 15.56 & 15.754 & 15.499 & 15.826 & 15.657 & 15.83 & 14.82 & 15.996 \\
\hline
$T_W$ & 2.707 & 2.534 & 1.976 & 1.437 & 1.422 & 1.081 & 1.134 & 1.22 & 1.373 & 1.489 & 1.667 & 1.98 & 2.488 \\
\hline
\end{tabular}
}
\caption{Parameters of nonlinear instability for the local problem \eqref{beamnonlineare} with $12$ modes.}
\label{tabellaprecisacubo1}
\end{center}
\end{table}

In Table \ref{tabellasoglie}, we determine the instability thresholds for each prevailing mode up to the tenth (reporting a less precise approximation of the values), increasing $\delta$ with step $0.1$. We then relate such thresholds with the position of the zeros of the corresponding eigenfunctions, observing that their distance from the piers seems indeed to play a crucial role. A careful study of the trend for odd and even modes suggests that, for $a$ fixed, the functions $j \mapsto E_{2j}(a)$ and $j \mapsto E_{2j+1}(a)$ are increasing, up to few exceptions $E_{n_i}$. While looking at the shape of the corresponding eigenfunctions $e_{n_i}$, one sees quite clearly that their zeros not lying in the piers are considerably closer to the piers than the ones of the ``preceding'' eigenfunction $e_{n_i-2}$.
We quote these two factors in Tables \ref{comparisonI} and \ref{cpcp2}, reporting on the left the energy thresholds of Table \ref{tabellasoglie} for even and odd modes separately, and on the right the minimal distances $D_n$ of the zeros of $e_n$ (not lying in a pier) from the piers. In correspondence of a double zero, we set $D_n=0$ and we use the letter $D$. Losses of monotonicity of the functions $j \mapsto E_{2j}(a)$, $j \mapsto E_{2j+1}(a)$ are highlighted by bold numbers. In Tables \ref{comparisonI} and \ref{cpcp2} we have highlighted with the symbol * the situations where the zero realizing the minimal distance from the piers lies in the central span. We exclude $D_0$ from this analysis, since the first eigenfunction does not present zeros out of the piers.
We conclude that
\begin{center}
{\bf the functions $j \mapsto E_{2j}$ and $j \mapsto E_{2j+1}$ have a monotone increasing trend; when this fails, the zeros of the corresponding eigenfunction have got considerably closer to the piers.}
\end{center}
Somehow the prevailing mode undergoes more the impulsive effect of the piers if one of its zeros is closer to a pier.
This appears reasonable, since, in this case, a smaller perturbation may be sufficient to ``move'' such zero to another span. Whether the distance from the pier has a different weight if the zero belongs to the lateral or to the central span should be further investigated.


\vfill
\eject

\begin{table}[ht!]
\begin{center}
{\footnotesize
\begin{tabular}{|c|c|c|c|c|c|c|c|c|c|c|}
\hline
$a$ & $E_0$ & $E_1$ & $E_2$ & $E_3$ & $E_4$ & $E_5$ & $E_6$ & $E_7$ & $E_8$ & $E_9$ \\
\hline
$0.1$ & $619$
 & $24$
 & $10609$
& $328$
& $96845$
 & $3913$
&  $641741$
&  $33881$
& $107734$
& $676852$ \\
\hline
$0.2$ & $1352$
 & $18$
 & $10267$
& 1938
& 28998
 & 208477
& 13933
& $183748$
& $186779$
& $250931$\\
\hline
$0.3$ & $3376$
 & $53$
 & $2711$
& $33993$
& $2206$
 & $32431$
& $519474$
& $29515$
& $585070$
& $1861480$\\
\hline
$1/3$ & $5233$
 & $322$
 & $1856$
& $38126$
& $9481$
& $28016$
& $178051$
& $118350$
& $321733$
& $1493750$\\
\hline
$0.4$ & $2922$
 & $9481$
& $127$
& $14815$
& $54148$ & $7009$
& $116255$ & $478827$ & $137040$ & $150139$ \\
\hline
$0.5$ & $4385$  & $1941$& $422$ & $1012$ & $36028$ & $88890$ & $16203$ & $61350$ & $182189$ & $1447500$ \\
\hline
$0.56$ & $4770$
& $900$
& $2413$
& $3560$
 & $12139$
& $65170$
& $154411$ & $40426$
& $73710$
& $842544$
\\
\hline
$0.6$ & $5065$
& $611$
& $10089$
& $1146$
& $2070$
 & $73482$
 & $341889$ & $270462$ & $51749$
& $148354$ \\
\hline
$2/3$ & $2523$
 & $403$
& $4502$
 & $40568$
& $3587$
& $5866$
& $27843$
& $453118$ & $418454$ & $687344$ \\
\hline
$0.7$ & $2061$
& $307$
& $3224$
& $25384$
& $13631$
 & $6123$
 & $16503$ & $445445$ & $1089400$
& $1937780$ \\
\hline
$0.8$ & $3391$
& $240$
& $2777$
& $9736$
& $45075$
 & $169085$
 & $488353$ & $183748$ & $89404$
& $176802$ \\
\hline
$0.9$ & $2375$
& $129$
& $1850$
& $6716$
& $33116$
 & $82774$
 & $234085$ & $591588$ & $1421620$
& $2622600$ \\
\hline
\textnormal{no piers} & $1160$
& $74$
& $194$
& $1697$
& $9741$
& $35347$
& $107364$
& $300251$ & $688532$ & $1261680$\\
\hline
\end{tabular}
}
\caption{Energy thresholds $E_j$ on varying of $a$ for \eqref{sistemaN}.}
\label{tabellasoglie}
\end{center}
\end{table}


\begin{table}[ht!]
\begin{center}
{\footnotesize
\begin{tabular}{|c|}
\hline
$a$ \\
\hline
0.1 \\
\hline
0.2 \\
\hline
0.3 \\
\hline
1/3 \\
\hline
0.4 \\
\hline
0.5 \\
\hline
0.56 \\
\hline
0.6 \\
\hline
2/3 \\
\hline
0.7 \\
\hline
0.8 \\
\hline
0.9 \\
\hline
\end{tabular}
}
{\footnotesize
\begin{tabular}{|c|c|c|c|}
\hline $E_2$ & $E_4$ & $E_6$ & $E_8$ \\
\hline
 $10609$
& $96845$
&  $641741$
& \textbf{107734}
\\
\hline
 10267
& 28998
& \textbf{13933}
& $186779$
\\
\hline
2711
& \textbf{2206}
& $519474$
& $585070$
\\
\hline
 1856
& $9481$
& $178051$
& $321733$
\\
\hline
127
& $54148$
& $116255$ & $137040$ \\
\hline
422 & $36028$ & \textbf{16203} & $182189$  \\
\hline
2413
 & $12139$
& $154411$
& \textbf{73710}
\\
\hline
10089
& \textbf{2070}
 & $341889$ & \textbf{51749}
\\
\hline
$4502$
& \textbf{3587}
& 27843 & $418454$ \\
\hline
$3224$
& $13631$
 & $16503$ & $1089400$\\
\hline
2777
& $45075$
 & $488353$ & \textbf{89404}
\\
\hline
1850
& $33116$
 & $234085$
& $1421620$ \\
\hline
\end{tabular}
}
{\footnotesize
\begin{tabular}{|c|c|c|c|}
\hline
$D_2$ & $D_4$ & $D_6$ & $D_8$ \\
\hline
1.498 & 0.997 & 0.729 & \textbf{0.294}  \\
\hline
1.257 & 0.48 & \textbf{0.075}* & 0.278*  \\
\hline
0.621 & \textbf{0.307}*  & 0.576* & 0.422  \\
\hline
0.325 &
0.491* &  $0.698$*
& 0.129
\\
\hline
0.159* & 0.794* & 0.289  & 0.319* \\
\hline
0.803* & 0.558 & \textbf{0.426}* & $0.345$ \\
\hline
 1.104*
& 0.158
& 0.683 & \textbf{0.07}* \\
\hline
 $1.257$* & \textbf{0.117}* & 0.515 & \textbf{0.379}* \\
\hline
 1.441*
& \textbf{0.651}*  & 0.136 & 0.455\\
\hline
1.522* & 0.843* & 0.091* & 0.297\\
\hline
1.76*  & 1.077* & 0.752* & \textbf{0.476}* \\
\hline
 2.002* & 1.247*  &0.897*  & 0.697* \\
\hline
\end{tabular}
}
\caption{Energy thresholds $E_n$ for \eqref{sistemaN} and minimal distances $D_n$ from the piers for even modes.}
\label{comparisonI}
\end{center}
\end{table}

\vfill
\eject

\begin{table}[ht!]
\begin{center}
{\footnotesize
\begin{tabular}{|c|}
\hline
$a$ \\
\hline
0.1 \\
\hline
0.2 \\
\hline
0.3 \\
\hline
1/3 \\
\hline
0.4 \\
\hline
0.5 \\
\hline
0.56 \\
\hline
0.6 \\
\hline
2/3 \\
\hline
0.7 \\
\hline
0.8 \\
\hline
0.9 \\
\hline
\end{tabular}
}
{\footnotesize
\begin{tabular}{|c|c|c|c|c|}
\hline
$E_1$ & $E_3$ & $E_5$ & $E_7$ & $E_9$ \\
\hline
$24$
& $328$
 & $3913$
&  $33881$
& $676852$ \\
\hline
$18$
& 1938
& 208477
& \textbf{183748}
& $250931$\\
\hline
$53$
& $33993$
 & \textbf{32431}
& \textbf{29515}
& $1861480$\\
\hline
$322$
& 38126
& \textbf{28016}
& $118350$
& 1493750\\
\hline
$9481$
& $14815$
& \textbf{7009}
& 478827 & \textbf{150139} \\
\hline
1941 & \textbf{1012} & 88890 & \textbf{61350} & $1447500$ \\
\hline
$900$
& $3560$
& $65170$
& \textbf{40426}
& $842544$
\\
\hline
$611$
& $1146$
& $73482$
& $270462$
& \textbf{148354} \\
\hline
 $403$
 & 40568
& \textbf{5866}
& $453118$ & $687344$ \\
\hline
$307$
& $25384$
 & \textbf{6123}
 & $445445$
& $1937780$ \\
\hline
$240$
& $9736$
 & $169085$
 & 183748
& \textbf{176802} \\
\hline
$129$
& $6716$
 & $82774$
 & $591588$
& $2622600$ \\
\hline
\end{tabular}
}
{\footnotesize
\begin{tabular}{|c|c|c|c|c|}
\hline
$D_1$ & $D_3$ & $D_5$ & $D_7$ & $D_9$ \\
\hline
$0.314$* & 0.314* & 0.314* & 0.314* &  0.314* \\
\hline
$0.628$* & 0.628* & 0.628*  & \textbf{0.628}* & $0.266$ \\
\hline
$0.942$* & 0.942* & \textbf{0.483} & \textbf{0.142}* & 0.395*  \\
\hline
$1.047$* & $1.047$*
& \textbf{0.184}
& 0.336*
& $0.524$* \\
\hline
$1.257$* & 0.718 & \textbf{0.284}*  & 0.628*  &\textbf{0.064}\\
\hline
$1.571$* & \textbf{D} & $0.785$* &
\textbf{D} & $0.524$* \\
\hline
$1.759$*
& 0.44* &  0.519 & \textbf{0.43}* & 0.211 \\
\hline
$1.885$* & 0.718*  & 0.284 & 0.628* & \textbf{0.064}* \\
\hline
$2.094$*
& 1.047*
& \textbf{0.184}*
& 0.336 & 0.524* \\
\hline
$2.199$* & 1.138* & \textbf{0.483}* & 0.142 & 0.395 \\
\hline
$2.513$* & 1.341*  & 0.893* & 0.628* & \textbf{0.266}* \\
\hline
$2.827$* & 1.54* & 1.045* & 0.785* & 0.625* \\
\hline
\end{tabular}
}
\caption{Energy thresholds $E_n$ for \eqref{sistemaN} and minimal distances $D_n$ from the piers for odd modes.}
\label{cpcp2}
\end{center}
\end{table}

\section{Nonlinear evolution equations for degenerate plates}\label{suspbridge}

\subsection{Weak formulation: well-posedness and torsional stability}\label{weakformTI}

We turn here our attention to possible suspension bridge models. A beam only possesses one degree of freedom (the vertical displacement) but,
for the analysis of a bridge, it is crucial to consider a model allowing to view a torsion of the deck. In this section we propose a degenerate plate-type one, focusing on its
well-posedness and on some of its qualitative properties.\par
In Section \ref{physmod}, we introduced system \eq{system0}, which in adimensional form may be written as
\begin{equation}\label{nonloc}
\left\{
\begin{array}
[c]{l}%
\displaystyle u_{tt}+u_{xxxx}+\Big( \int_I(u^2+\theta^2)\Big)u+2\Big(\int_I u\theta\Big)\theta+f(u+\theta)+f(u-\theta)=0\\
\displaystyle \theta_{tt}-\theta_{xx}+2\Big(\int_I u\theta \Big) u+\Big(\int_I(u^2+\theta^2)\Big)\theta+f(u+\theta)-f(u-\theta)=0.
\end{array}
\right.
\end{equation}
We complement \eq{nonloc} with the boundary-internal-initial conditions \eq{boundary}-\eq{jc1}-\eq{initial2}, that we rewrite here for the
reader's convenience:
\begin{equation}\label{jc2}
u(\pm\pi,t)=u(\pm a\pi,t)=\theta(\pm\pi,t)=\theta(\pm a\pi,t)=0\qquad t\geq 0,
\end{equation}
\begin{equation}\label{ic2}
u(x,0)=u_0(x),\quad u_t(x,0)=u_1(x),\quad\theta(x,0)=\theta_0(x),\quad\theta_t(x,0)=\theta_1(x) \qquad x\in I.
\end{equation}
Similarly as in Definition \ref{soluzionedebole}, we define weak solutions as follows.
\begin{definition}\label{definizionedebole}
We say that the functions
$$
\begin{array}{c}
u\in C^0(\R_+;V(I))\cap  C^1(\R_+;L^2(I))\cap  C^2(\R_+;V'(I))\\
\theta\in C^0(\R_+;W(I))\cap  C^1(\R_+;L^2(I))\cap  C^2(\R_+;W'(I))
\end{array}
$$
are weak solutions of \eqref{nonloc}-\eqref{jc2}-\eqref{ic2} if they satisfy the initial conditions \eqref{ic2} with $u_0 \in V(I)$, $\theta_0 \in W(I)$, $u_1 \in L^2(I)$, $\theta_1 \in L^2(I)$ and if
$$
\langle u_{tt}, \varphi\rangle_V+\int_Iu_{xx}\varphi''+\int_I(u^2+\theta^2)\cdot\int_Iu\varphi+2\int_I u\theta\cdot\int_I\theta\varphi
+\int_I\big(f(u-\theta)+f(u+\theta)\big)\varphi=0,
$$
$$
\langle \theta_{tt},\psi\rangle_W+\int_I\theta_x\psi'+2\int_I u\theta\cdot\int_Iu\psi+\int_I(u^2+\theta^2)\cdot\int_I\theta\psi
+\int_I\big(f(u+\theta)-f(u-\theta)\big)\psi=0,
$$
for all $(\varphi,\psi)\in V(I)\times W(I)$ and all $t>0$ (where the spaces $V(I)$ and $W(I)$ are defined in \eqref{V} and \eqref{W}, respectively).
\end{definition}
The corresponding well-posedness result reads as follows.
\begin{proposition}\label{Galerkin}
Let $u_0\in V(I)$, $\theta_0\in W(I)$, $u_1,\theta_1\in L^2(I)$. Assume that $f$ satisfies \eqref{assumptions} and $|f(s)|\leq C(1+|s|^p)$ for every
$s\in \R\setminus\{0\}$ and for some $p\geq1$. Then there exists a unique weak solution $(u,\theta)$ of \eqref{nonloc}-\eqref{jc2}-\eqref{ic2}.
Moreover, $u\in C^2(\overline{I}\times\R_+)$ and $u_{xx}(-\pi,t)=u_{xx}(\pi,t)=0$ for all $t>0$.
\end{proposition}

We omit the proof of Proposition \ref{Galerkin} because it can be obtained by combining the classical proof for simple beams in \cite{dickey0},
the proof for a full plate in \cite[Theorem 3]{ederson}, and the proof of \cite[Theorems 8 and 11]{holubova}, see also \cite{bergaz}. For the
$C^2$-regularity of $u$ we invoke again \cite[Lemma 2.2]{HolNec10}.\par
Clearly, there exist solutions which have both nontrivial longitudinal and torsional components. However, if $\theta_0(x)=\theta_1(x)=0$ (resp., $u_0(x)=u_1(x)=0$) in \eq{ic2}, then
the solution of \eq{nonloc} satisfies $\theta(x,t)\equiv0$ (resp., $u(x, t) \equiv 0$). Hence, the above zero-initial conditions for $\theta$ or $u$ give rise
to purely longitudinal (resp., purely torsional) solutions. Since the phase space is now $V(I)\times W(I)$, according to Definition \ref{invarianti} this means that
\neweq{invariantigenerici}
\mbox{{\bf the spaces $V(I)\times\{0\}$ and $\{0\}\times W(I)$ are invariant for \eqref{nonloc}}}.
\endeq

Our purpose is to consider solutions which are both longitudinal and torsional but that are ``initially close to a purely longitudinal configuration''.
Therefore, we consider solutions $(u,\theta)$ of \eq{nonloc} with a {\em longitudinal prevailing mode}, see Definition \ref{prevalente}, suitably adapted to take into account also the presence of the torsional components.
In turn, this enables us to study the {\em stability of purely longitudinal modes}, according to the following definition.

\begin{definition}\label{torsunstable}
Let $T_W > 0$. We say that a weak solution $(u,\theta)$ of \eqref{nonloc} having $\eta$-prevailing longitudinal mode $j$,
is \emph{torsionally unstable} before time $T>2T_W$ if there exist a torsional mode $k$ and a time instant $\tau$ with $2T_W < \tau < T$  such that
\neweq{grandetors}
\frac{\Vert \psi_k \Vert_{L^\infty(0, \tau)}}{\Vert \psi_k \Vert_{L^\infty(0, \tau/2)}} > \frac{1}{\eta},
\endeq
where $\eta$ is the number appearing in Definition \ref{prevalente}. We say that $(u,\theta)$ is \emph{torsionally stable} until time $T$ if, for any torsional mode $k$, \eqref{grandetors}
is not fulfilled for any $\tau \in (2T_W, T)$.
\end{definition}

Compared with \eq{grande}, notice that here the indexes $j$ and $k$ do not refer to the same kind of modes, since the former is always longitudinal, while the latter is torsional. Furthermore, there is one condition missing in \eq{grandetors}; it is not required here that the torsional mode becomes significantly large
when compared to the prevailing longitudinal mode, since these two amplitudes measure two different kinds of oscillations. What really measures the torsional instability is the sudden growth of the torsional component, when
compared to itself. Indeed, small torsional oscillations were never observed in suspension bridges, which means that they suddenly switch from tiny invisible
oscillations to large and dangerous oscillations. This is the phenomenon under study, which we call torsional instability.
In order to study the torsional stability we take advantage of some of the material developed in Sections \ref{functional} and \ref{evolutionbeam}.

\subsection{The case of rigid hangers}\label{rigidhangers}

\subsubsection{Linear torsional instability for two-modes systems}

As already mentioned in Section \ref{physmod}, the main contribution to the instability of suspension bridges comes from the sustaining cables.
A possible simplification consists then in considering inextensible hangers, see \cite{luco}. This means that the hangers rigidly connect the deck with the cables
and that the only restoring force acting on the degenerate plate is due to the cables. In this case, \eq{nonloc} becomes
\begin{equation}\label{nohang}
\left\{
\begin{array}
[c]{l}%
\displaystyle u_{tt}+u_{xxxx}+\Big( \int_I(u^2+\theta^2)\Big)u+2\Big(\int_I u\theta\Big)\theta=0\\
\displaystyle \theta_{tt}-\theta_{xx}+2\Big(\int_I u\theta \Big) u+\Big(\int_I(u^2+\theta^2)\Big)\theta=0,
\end{array}
\right.
\end{equation}
which has several common points with \eq{toybeam2} and, therefore, enables us to exploit part of the results obtained in Section \ref{nonmischia}.
From Proposition \ref{allinvariant} we learn that \eq{nohang} has many invariant subspaces and a full analysis of all the cases would be quite
lengthy. Both \eq{invariantigenerici} and Table \ref{tabellasuperL2} in Section \ref{sezesperim} (together with the discussion therein) suggest that the case of a couple of modes, one longitudinal and one torsional, is the most meaningful and is quite representative of
the whole instability phenomenon leading to an energy transfer from longitudinal to torsional modes. Therefore, we reduce again to the case of
a two modes system but, contrary to \eq{cwgeneral}, we obtain a system that {\em does not} necessarily solve \eq{nohang}. Let us explain this
fact with full precision.\par
Let $e_\lambda$ be an $L^2$-normalized eigenfunction of \eq{autovsym} related to the eigenvalue $\mu=\lambda^4$ and let $\eta_\kappa$ be
an $L^2$-normalized eigenfunction of \eq{autov2} related to the eigenvalue $\mu=\kappa^2$, see Theorem \ref{autofz2}. The coupling between
these modes is measured by the coefficient
$$
\Alk=\Alk(a):=\int_Ie_\lambda \eta_\kappa.
$$
Note that
\begin{equation}\label{Alkmin1}
\Alk^2 < 1
\end{equation}
in view of the H\"older inequality (recall that $e_\lambda \not\equiv \eta_\kappa$, since the latter is identically zero on the lateral spans).
Moreover,
\begin{equation}\label{propAlk}
\textrm{if $e_\lambda$ and $\eta_\kappa$ have opposite parities, then $\Alk=0$}.
\end{equation}
But there are also cases where $\Alk\neq0$, see Table \ref{Alka}.  We here choose to consider only the first two torsional modes. Thus, we stick to values of $a$ for which such modes are simple, as provided by Proposition \ref{ordine}; the only exception is represented by $a=0.5$, which is included since it belongs to the range \eqref{physicalrange}. For $a=0.5$, Theorem \ref{autofz2} states that the second eigenvalue has multiplicity 3; the corresponding value of $A_{\lambda_n, \kappa_1}$ in Table \ref{Alka} corresponds to the eigenfunction $\mathbf{D}_{\kappa_1}(x)=\chi_0(x) \sin(2x)$, see the second picture in Figure \ref{22}.

\begin{table}[ht!]
\begin{center}
\resizebox{\textwidth}{!}{\footnotesize
\begin{tabular}{|c|c|c|c|c|c|c|c|c|c|c|c|c|}
\hline
$\!a\downarrow$ $(n,\!m)\to\!$ & $(0,0)$ & $(2,0)$ & $(4,0)$ & $(6,0)$ & $(8,0)$ & $(10,0)$ & $(1,1)$ & $(3,1)$ & $(5,1)$ & $(7,1)$ & $(9,1)$ & $(11,1)$ \\
\hline
$0.5$ & 0.953 & 0.034 & 0.007 & 0.002 & 0.001 & 5 $\cdot 10^{-4}$ & 0.5 & 0.476 & 0 & 0.017 & 0 & 0.003 \\
\hline
$14/25$ & 0.979 & 0.005 & 0.012 & 2 $\cdot 10^{-5}$  & 0.001 & $10^{-4}$ & 0.857 & 0.121 & 0.01 & 0.004 & 0.003 & 2 $\cdot 10^{-4}$\\
\hline
$0.6$ & 0.986 & 0 & 0.011 & 8 $\cdot 10^{-4}$  & 6 $\cdot 10^{-4}$ & 4 $\cdot 10^{-4}$ & 0.936 & 0.041 & 0.016 & 0 & 0.004 & 5 $\cdot 10^{-4}$ \\
\hline
$2/3$ & 0.989 & 0.004 & 0.002 & 0.002 & $10^{-4}$ & 2 $\cdot 10^{-4}$ & 0.975 & 0 & 0.019 & 0.002 & 0 & 0.001 \\
\hline
$0.7$ & 0.989 & 0.006 & 2 $\cdot 10^{-4}$ & 0.002 & 3 $\cdot 10^{-4}$ & $10^{-5}$ & 0.981 & 0.003 & 0.008 & 0.004 & 5 $\cdot 10^{-4}$ & 2 $\cdot 10^{-4}$\\
\hline
$0.8$ & 0.986 & 0.011 & $10^{-4}$ & 9 $\cdot 10^{-5}$  & $10^{-4}$ & 3 $\cdot 10^{-4}$ & 0.981  & 0.014 & 0.001 & 0 & 0.001 & 7 $\cdot 10^{-4}$ \\
\hline
$0.9$ & 0.98  & 0.016 & 0.002  & $4 \cdot 10^{-4}$  & $10^{-4}$ & 5 $\cdot 10^{-5}$ & 0.970 & 0.023 & 0.003 & 0.001 & 3 $\cdot 10^{-4}$ & $10^{-4}$ \\
\hline
\end{tabular}
}
\caption{Some values of $A_{\lambda_n,\kappa_1}=A_{\lambda_n,\kappa_1}(a)$.}\label{Alka}
\end{center}
\end{table}

It appears that only for the couples (0,0) and (1,1) the value of $\Alk$ is relevant, especially if we take into account that the
``coupling between modes'' is measured by $2\Alk^2$, as we now explain.\par
If $\Alk=0$, then the space $\langle e_\lambda\rangle\times\langle\eta_\kappa\rangle$ is invariant and one can seek solutions of \eq{nohang} in
the form
\neweq{formd}
\Big(u(x,t),\theta(x,t)\Big)=\Big(w(t)e_\lambda(x),z(t)\eta_\kappa(x)\Big),
\endeq
to be compared with \eq{form3}. By plugging \eq{formd} into \eq{nohang} we see that the couple $(w,z)$ satisfies the system
\neweq{systemtorsional}
\left\{\begin{array}{l}
\ddot{w}(t)+\lambda^4w(t)+(1+2\Alk^2)z(t)^2w(t)+w(t)^3=0\\
\ddot{z}(t)+\kappa^2z(t)+(1+2\Alk^2)w(t)^2z(t)+z(t)^3=0,
\end{array}\right.
\endeq
where, for later use, we left $\Alk$ explicitly written even if it is zero.\par
If $\Alk\neq0$, then the space $\langle e_\lambda\rangle\times\langle\eta_\kappa\rangle$ is not invariant for \eq{nohang}.
Therefore, there are no solutions of \eq{nohang} in the form \eq{formd}. Nevertheless, since a couple of (longitudinal,\, torsional) modes
is quite representative of the dynamics in view of \eq{invariantigenerici}, since $\Alk$ is quite small in most cases (see Table \ref{Alka}), and since
Table \ref{tabellasuperL2} in Section \ref{sezesperim} (together with the discussion therein) shows that the energy thresholds
do not vary significantly while comparing two-modes systems with twelve-modes systems, we still focus our attention on the two-modes
system \eq{systemtorsional}. Clearly, the resulting solutions should be seen as nonexact (approximated) solutions of \eq{nohang}.\par
We associate with \eq{systemtorsional} the initial conditions
\neweq{initialsystors}
w(0)=\delta>0,\quad z(0)=z_0,\quad\dot{w}(0)=\dot{z}(0)=0
\endeq
and we notice that if $z_0=0$, then the solution of \eq{systemtorsional}-\eq{initialsystors} is $(w,z)=(W_\lambda,0)$ where $W_\lambda$ solves
\eq{ODED2}, see \eq{explicitsol}. Hence, $W_\lambda$ has minimal period as in \eqref{TW}, so that $W_\lambda^2$ has minimal period equal to
\neweq{Tld}
T_\lambda(\delta):=\frac{T(\delta)}{2}= 2\sqrt2\int_0^{\pi/2}\frac{d\phi}{\sqrt{2\lambda^4+\delta^2(1+\sin^2\phi)}};
\endeq
the function $\delta\mapsto T_\lambda(\delta)$ is continuous, decreasing and such that $T_\lambda(0)=\pi/\lambda^2$.\par
We then consider initial conditions \eqref{initialsystors} with $0<|z_0|<\eta^2\delta$, with $\eta$ small, which justifies \eqref{trigo} and the stability analysis of \eqref{nohang}. For this study, we slightly modify Definition \ref{defstabb}.

\begin{definition}\label{defstabtors}
The $\lambda$-longitudinal mode $W_\lambda$ is said to be {\bf linearly stable} ({\bf unstable}) with respect to the $\kappa$-torsional-mode
if $\xi\equiv0$ is a stable (unstable) solution of the linear Hill equation
\neweq{hilltors}
\ddot{\xi}(t)+\Big(\kappa^2+(1+2\Alk^2)W_\lambda(t)^2\Big)\xi(t)=0.
\endeq
\end{definition}

Several remarks are in order. With this definition, we see that the linear stability {\em does not} depend on the positive coefficient $\gamma$ in \eqref{system0} multiplying the nonlocal terms in \eq{nohang}: up to a suitable scaling \eq{hilltors} remains the same.
In fact, \eq{hilltors} {\em is not} a standard Hill equation with a periodic coefficient depending on two parameters measuring
its mean value and its amplitude of oscillation: the presence of the coefficient $\Alk$ prevents the use of classical methods for the study of the stability.
This coefficient depends in a very hidden way on the two parameters $\lambda$ and $\kappa$ and is well-defined only if $\lambda^4$ is an eigenvalue of
\eqref{autovsym} and $\kappa^2$ is an eigenvalue of \eqref{autov2}. Summarizing, there is no way, not even numerically, to determine the resonant tongues because...
they are not defined for all $\lambda$ and $\kappa$! Let us recall that, from a physical point of view, $\Alk$ depends on the coupling between the
torsional and the longitudinal modes involved. Finally, since it may happen both that $\Alk=0$ and $\Alk\neq0$, for the description of the linear stability we
need to distinguish two cases. The case $\Alk=0$ can be directly inferred from Proposition \ref{toytheo2} simply by replacing $\rho^2$ with $\kappa$.

\begin{proposition}\label{toytheo3}
Let $\lambda^4$ and $\kappa^2$ be, respectively, eigenvalues of \eqref{autovsym} and \eqref{autov2} such that $\Alk=0$.
The $\lambda$-longitudinal-mode of \eqref{nohang} of amplitude $\delta$ is linearly stable with respect to the $\kappa$-torsional-mode
if and only if one of the following facts holds:
$$
\lambda^2>\kappa\mbox{ and }\delta>0\qquad\mbox{or}\qquad\lambda^2<\kappa\mbox{ and }\delta \leq \sqrt{2(\kappa^2-\lambda^4)}.
$$
\end{proposition}

Since the linear instability is a clue for the nonlinear instability, Proposition \ref{toytheo3} states that if $\lambda^2 > \kappa$ then \eq{nohang} does not display torsional instability according to Definition \ref{unstable} (in fact, for any $T_W$ and $T$). The stability picture describing the situation
of Proposition \ref{toytheo3} may be derived from Figure \ref{parabole}, again by replacing $\rho^2$ with $\kappa$.\par
When $\Alk\neq0$, it appears out of reach to obtain a statement as precise as Proposition \ref{toytheo3}.
We are able to prove the following, see Section \ref{pftheotors}.

\begin{theorem}\label{theotors}
Let $\lambda^4$ and $\kappa^2$ be, respectively, eigenvalues of \eqref{autovsym} and \eqref{autov2} such that $\Alk\neq0$.\par\noindent
$\bullet$ The $\lambda$-longitudinal-mode of \eqref{systemtorsional} of amplitude $\delta$ is linearly unstable with respect to the $\kappa$-torsional-mode
whenever $\delta$ is sufficiently large.\par\noindent
$\bullet$ If $\kappa^2\neq\lambda^4$, then the $\lambda$-longitudinal-mode of \eqref{systemtorsional} of amplitude $\delta$ is linearly stable with respect to
the $\kappa$-torsional-mode whenever $\delta$ is sufficiently small. In particular, linear stability is guaranteed whenever there exists an integer $m$ such that
both the two following conditions hold:
\neweq{burdinam}
\frac12\log\left(\!1+\tfrac{3\delta^2}{\kappa^2}\!\right)+
2\sqrt{2}\int_0^{\pi/2}\!\!\sqrt{\tfrac{1+3(\delta^2/\kappa^2)\sin^2\phi}{2(\lambda^4/\kappa^2)+(\delta^2/\kappa^2)(1+\sin^2\phi)}}\, d\phi\le(m+1)\pi,
\endeq
\neweq{burdinam2}
2\sqrt{2}\int_0^{\pi/2}\!\!\sqrt{\tfrac{1+(\delta^2/\kappa^2)\sin^2\phi}{2(\lambda^4/\kappa^2)+(\delta^2/\kappa^2)(1+\sin^2\phi)}}\, d\phi
-\frac12\log\left(\!1+\tfrac{3\delta^2}{\kappa^2}\!\right)\ge m\pi.
\endeq
\end{theorem}

The bounds \eq{burdinam} and \eq{burdinam2} give sufficient conditions for the stability and the related regions are depicted in white in Figure \ref{Burdfish}.
\begin{figure}[ht]
\begin{center}
\includegraphics[height=72mm,width=86mm]{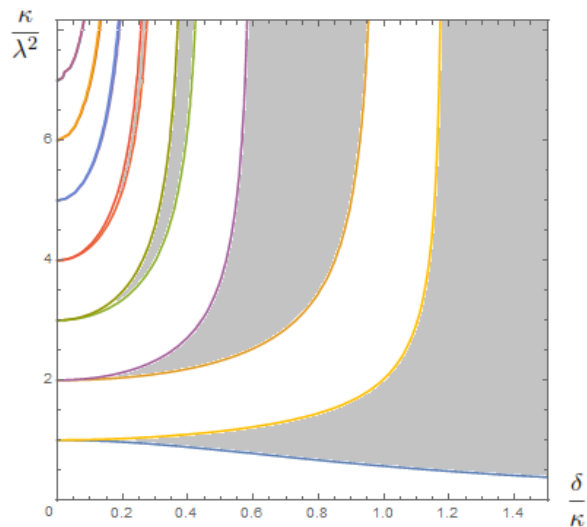}
\caption{In white, subsets of the stability regions for \eq{systemtorsional}.}\label{Burdfish}
\end{center}
\end{figure}
It appears that they are ``deformations'' of the regions in Figure \ref{parabole}. We underline once more that the stability regions may be wider than
in Figure \ref{Burdfish}, although Theorem \ref{theotors} guarantees that they cannot contain horizontal strips, contrary to Figure \ref{parabole}.
In fact, at least for $\Alk\approx0$ (the majority of cases, see Table \ref{Alka}), we expect the instability diagram to be very similar to Figure \ref{parabole}, with possibly tiny instability tongues emanating from the points $(0, n)$, for $n$ integer.
Formula \eq{burdinam2} may be used for all integer $m$ but in the case $m=0$ the estimate can be improved, see \cite[Figure 2]{gaga}.
Assuming that $\kappa<\lambda^2$, we may proceed by using a stability criterion due to Zhukovskii \cite{zhk} (see also \cite[Chapter VIII]{yakubovich}), that
guarantees linear stability provided that
$$
T_\lambda(\delta)^2\Big[\kappa^2+(1+2\Alk^2)\delta^2\Big]\le\pi^2,
$$
where $T_\lambda(\delta)$ is the period of $W_\lambda^2$, see \eq{Tld}. By using \eq{Alkmin1} and by emphasizing the relevant variables, we obtain
the following uniform sufficient condition
$$
2\sqrt{2}\int_0^{\pi/2}\frac{d\phi}{\sqrt{2\lambda^4/\kappa^2+\delta^2/\kappa^2(1+\sin^2\phi)}}\, \sqrt{1+3(\delta^2/\kappa^2)}\le\pi.
$$
Numerically one can verify that this bound is less restrictive than \eq{burdinam2}. Moreover, by deleting $\sin\phi$ in the integral, it gives the following
elegant sufficient condition for the linear stability of \eq{systemtorsional}.
\begin{corollary}
Under the hypotheses of Theorem \ref{theotors}, assume moreover that $\kappa < \lambda^2$. Then, the $\lambda$-longitudinal mode of \eqref{systemtorsional} of amplitude $\delta$ is linearly stable with respect to the $\kappa$-torsional mode if
$$\delta^2\le\frac25 (\lambda^4-\kappa^2).$$
\end{corollary}
We close this section by briefly commenting about the assumption $\kappa^2 \neq \lambda^4$ appearing in the second statement of Theorem \ref{theotors}.
Such a condition is certainly fulfilled if $\lambda$ is associated with a $C^\infty$-eigenfunction of \eq{autovsym} and $\kappa^2$ is a simple eigenvalue of \eq{autov2}.
Indeed, in view of Theorem \ref{autofz2}, we have either $\kappa a \in \mathbb{N}$ or $\kappa a \in \mathbb{N}+1/2$, and in both cases $\kappa (1-a)\notin \mathbb{N}$.
If it were $\kappa=\lambda^2$, writing $\kappa (1-a) = \lambda (\lambda (1-a))$ and $\kappa a = \lambda(\lambda a)$ this would be contradicted. Namely,
\begin{center}
{\bf a simple torsional eigenvalue cannot be a longitudinal eigenvalue  \\ associated with a $C^\infty$-eigenfunction.}
\end{center}
If $\lambda$ is associated with a longitudinal eigenfunction which is not $C^\infty$, there are instead cases where $\kappa^2=\lambda^4$, as Figure \ref{longtors} shows: precisely, this occurs at each intersection between the bold curves (representing the eigenvalues of \eqref{autovsym}) and the dashed/dot-dashed ones (representing the eigenvalues of \eqref{autov2}); notice that, on the vertical axis, $\sqrt[4]{\mu}$ represents either $\lambda$ (bold curves) or $\sqrt{\kappa}$ (dashed/dot-dashed curves). However, we checked numerically that the corresponding values of $a$ are not among the ones of our interest, so this possibility, as well as the case when the considered torsional eigenvalue is multiple, will not be further deepened in our study.
\begin{figure}[ht]
\begin{center}
\includegraphics[scale=0.7]{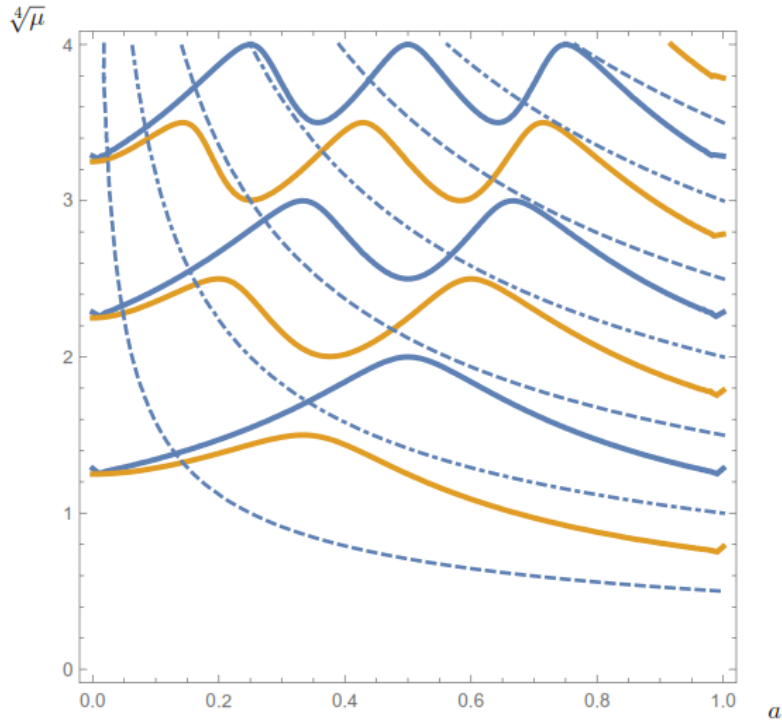}
\caption{Correspondence between longitudinal (bold) and torsional (dashed/dot-dashed) eigenvalues.}\label{longtors}
\end{center}
\end{figure}

\subsubsection{Optimal position of the piers for linear and nonlinear instability}\label{optfish}

A full stability analysis of \eq{nohang}, for any value of $a$ and any couple of (longitudinal, \, torsional) modes covers too many different situations.
Therefore, we restrict our attention on the most relevant cases, both from a mathematical and a physical point of view, briefly commenting the general situation.
The first restriction involves the position of the piers: we take
$$
a\ge\frac12
$$
because this case contains the physical range \eq{physicalrange} and, for $a > 1/2$, the second torsional eigenvalue is simple (see Proposition \ref{ordine}). Actually, the instability analysis in case of multiple torsional eigenvalues is much more complicated, due to the strong interactions which occur between the corresponding Fourier components.
\par
The second torsional mode is of crucial importance. Irvine \cite[Example 4.6, p.180]{irvine} describes the oscillations of the Matukituki Suspension Footbridge
as follows: {\em ...the deck persisted in lurching and twisting wildly until failure occurred, and for part of the time a node was noticeable at midspan}. Moreover, according to the
detailed analysis on the TNB by Smith-Vincent \cite[p.21]{tac2}, this form of torsional oscillation is the only one ever seen: {\em The only torsional mode
which developed under wind action on the bridge or on the model is that with a single node at the center of the main span}. Therefore,
\begin{equation}\label{solosecondo}
\textbf{we restrict our attention to the second torsional mode,}
\end{equation}
that is, we take $\kappa=\kappa_1$.
Concerning the longitudinal mode, we take one among the least 12 modes as in Section \ref{evolutionbeam}, see the motivation in Section \ref{4.4}. This enables us to exploit
several results obtained for the stability of the nonlinear beam equation \eq{toybeam2}.\par
With these restrictions, we tackle the problem of the optimal position of the piers and we proceed numerically, analyzing only the energy transfers towards the second torsional mode. For the linear instability, we maintain \eqref{D} as the characterization of the critical amplitude. We do not expect the picture for linear instability to be significantly different from the one in Figure \ref{parabole} since $A_{\lambda, \kappa}$ is small, see Table \ref{Alka}. On the other hand, the energy threshold of instability is characterized as in Definition \ref{threshold}, where nonlinear instability is meant in the sense of Definition \ref{torsunstable}.
In Table \ref{tfishbone} we quote the energy thresholds of linear and nonlinear instability thus found, according to Definitions \ref{torsunstable} and \ref{defstabtors}.
\begin{table}[ht!]
\begin{center}
{\footnotesize
\begin{tabular}{|c|c|c|c|}
\hline
$a$ & \textrm{linear} & \textrm{nonlinear} & \textrm{prevailing mode} \\
\hline
$0.5$ & 6.276 & 14.933 & $e_0$ \\
\hline
$0.56$ & 4.683 & 14.232 & $e_0$ \\
\hline
$0.6$ & 3.844 & 10.276  & $e_0$ \\
\hline
$2/3$ & 2.783 & 7.713 & $e_0$ \\
\hline
$0.7$ & 2.358 & 7.875 & $e_0$ \\
\hline
$0.8$ & 1.511
&
7.616 & $e_0$ (lin) /$e_1$ (nonlin) \\
\hline
$0.9$ & 1.021
&
3.884 & $e_0$ (lin) /$e_1$ (nonlin) \\
\hline
\end{tabular}
}
\caption{Energy thresholds of torsional instability for \eqref{systemtorsional} with $\kappa=\kappa_1$.}
\label{tfishbone}
\end{center}
\end{table}
\\
Some comments are in order. The most important is that, as in the discussion for the beam carried out in Section \ref{instnonlineare},
\begin{center}
\bf{linear instability is a clue for the occurrence of nonlinear instability}.
\end{center}
We also notice that instability appears first for low prevailing longitudinal modes, since the ratio $\kappa^2/\lambda^4$ becomes very small on growing of $\lambda$; hence, looking at both Figure \ref{parabole} (for $A_{\lambda, \kappa}=0$) and Figure \ref{Burdfish} (for $A_{\lambda, \kappa} \neq 0$), we understand that the more the ratio $\kappa^2/\lambda^4$  approaches $0$, the larger is the energy needed to reach instability. In some cases instability is never reached, this occurring for instance when $A_{\lambda, \kappa}=0$ and $\lambda^2 > \kappa$, since in this case we fall into the first linear stability region in Figure \ref{parabole} (the one below $\rho^4/\lambda^4=1$).
\par
In the next section, we will analyze the role of a further nonlinearity.

\subsection{The case of extensible hangers}\label{slack}

\subsubsection{Choice of the nonlinearity}

Not always the hangers may be considered rigid, which means that they can lose tension if the deck and the cables are too close to each other.
This phenomenon is called {\em slackening} and was observed during the TNB collapse, see \cite[V-12]{ammvkwoo}. The restoring action of
the hangers is described by a {\em local nonlinearity} that should be able to represent a slightly superlinear spring when extended and only
gravity when slacken. Recalling the convention \eqref{spostagi첫}, several nonlinearities of this kind have been considered in literature: among others,
we mention here
\neweq{possibility}
s\mapsto(1+s)^+-1,\qquad s\mapsto e^s-1,\qquad s\mapsto s-1+\sqrt{1+s^2},\qquad s\mapsto\left\{\begin{array}{ll}
\frac{s}{\sqrt{1+s^2}} & \mbox{if }s\le0\\
s &\mbox{if }s\ge0
\end{array}\right.,
\endeq
whose graphs are represented in the pictures of Figure \ref{fourgraphs}: one can notice that their qualitative behavior is the same.
\begin{figure}[ht]
\begin{center}
\includegraphics[height=26mm,width=39mm]{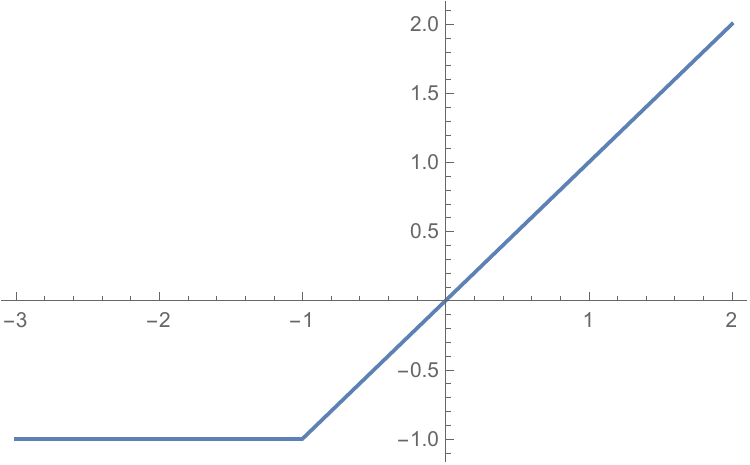}\quad\includegraphics[height=26mm,width=39mm]{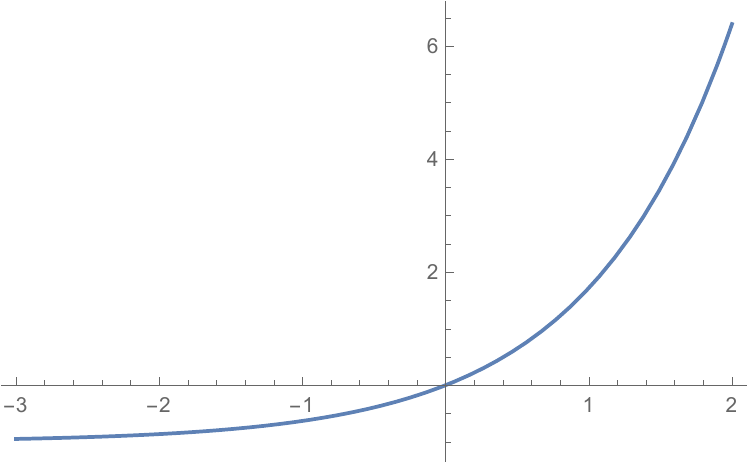}
\quad\includegraphics[height=26mm,width=39mm]{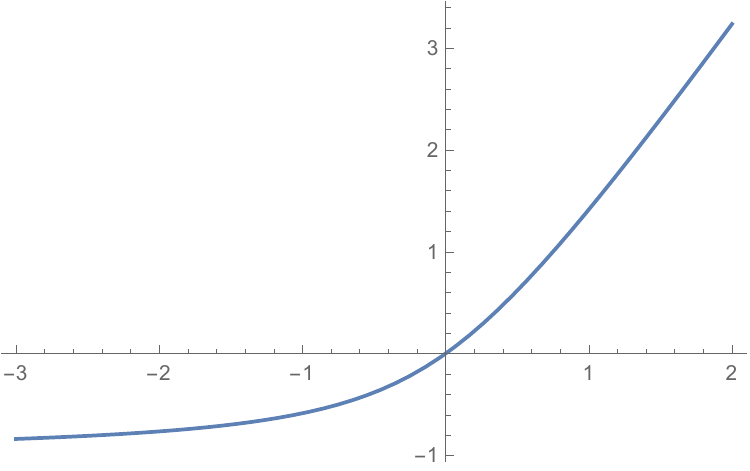}\quad\includegraphics[height=26mm,width=39mm]{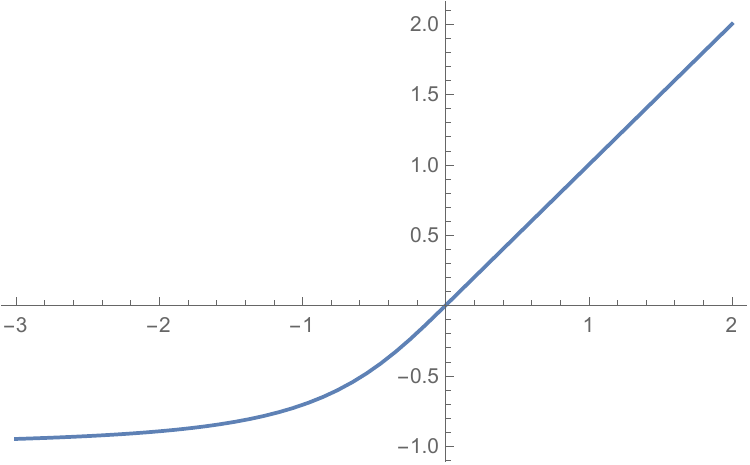}
\end{center}
\caption{Graphs of the functions in \eq{possibility}.}\label{fourgraphs}
\end{figure}

The first nonlinearity in \eq{possibility} was introduced by Lazer-McKenna \cite{lam} and describes a linear restoring force when the hangers
are in tension and a constant gravity force (normalized to be -1) when they slacken. This nonlinearity, that well describes the slackening
mechanism, has two drawbacks. First, it is nonsmooth and this makes numerical experiments very difficult. Second, Brownjohn \cite[p.1364]{brown}
remarks that {\em the hangers are critical elements in a suspension bridge and for large-amplitude motion their behaviour is not well modelled
by either simple on/off stiffness or invariant connections}. Moreover, supported by analytical and experimental studies on the dynamic response
of suspension bridges, Brownjohn \cite{brown} is able to show a strong nonlinear contribution of the cable/hanger effects while
McKenna-Tuama \cite{mckO} write {\em ...we expect the bridge to behave like a stiff spring, with a restoring force that becomes somewhat
superlinear}. For all these reasons, one is led to slightly modify the first nonlinearity in \eq{possibility}. The second, third, fourth
nonlinearities have been considered, respectively, in \cite{lambis}, \cite{March}, \cite{benforgaz}, as possible smooth alternative choices for
the restoring force. They all tend to $-1$ as $s\to-\infty$ and they all have first derivative equal to 1 when $s=0$: this number represents
the {\em Hooke constant of elasticity} of the hangers. Since we aim at emphasizing the role of the rigidity of the hangers, we adopt here
the modified nonlinearity
$$
f_\varsigma(s)=\varsigma s-1+\sqrt{1+\varsigma^2s^2}\qquad(\varsigma>0)
$$
which is a variant of the third nonlinearity in \eq{possibility}. We still have $f_\varsigma(s)\to-1$ as $s\to-\infty$ so that the gravity
constant $-1$ is conserved. Moreover, $f_\varsigma'(0)=\varsigma$ so that $\varsigma$ measures the elasticity of the hangers. Finally,
$f_\varsigma(s)\sim2\varsigma s$ as $s\to+\infty$, showing that $f_\varsigma(s)$ has a slightly superlinear behavior for $s>0$, with a slope going from
$\varsigma$ towards $2\varsigma$. Note also
that $f_\varsigma\to0$ in $L^\infty_{\rm loc}(\R)$ for $\varsigma \to 0$, so that for vanishing $\varsigma$ we are back in the situation of rigid hangers, see Section \ref{rigidhangers}.\par
By taking $f=f_\varsigma$ for some $\varsigma>0$, system \eq{nonloc} becomes
\begin{equation}\label{nonloc2}
\left\{
\begin{array}
[c]{l}%
\!\!\!u_{tt}+u_{xxxx}+\Big(\!\!\int_I(u^2+\theta^2)\!\Big)u+2\Big(\!\!\int_I u\theta\!\Big)\theta+2\varsigma u+
\Big[\sqrt{1+\varsigma^2(u+\theta)^2}+\sqrt{1+\varsigma^2(u-\theta)^2}-2\Big]=0\\
\!\!\!\theta_{tt}-\theta_{xx}+2\Big(\!\!\int_I u\theta\!\Big)u+\Big(\!\!\int_I(u^2+\theta^2)\!\Big)\theta+2\varsigma\theta+
\Big[\sqrt{1+\varsigma^2(u+\theta)^2}-\sqrt{1+\varsigma^2(u-\theta)^2}\Big]=0.
\end{array}
\right.
\end{equation}
We complement \eq{nonloc2} with the boundary-internal-initial conditions \eq{jc2}-\eq{ic2} and we define
weak solutions as in Definition \ref{definizionedebole}. Since $f_\varsigma$ satisfies the two latter conditions in \eq{assumptions}, Proposition \ref{Galerkin} can be applied and problem \eqref{nonloc2} is well-posed.

\subsubsection{Linear and nonlinear instability for the two modes system}

As already mentioned in Section \ref{weakformTI}, the spaces $V(I)\times\{0\}$ and $\{0\}\times W(I)$ are invariant for \eq{nonloc2}:
if we take $(\theta_0,\theta_1)=(0,0)$, then the weak solution $(u,\theta)$ of \eq{nonloc2}-\eq{ic2} has no torsional component
($\theta\equiv0$) while the longitudinal component $u$ satisfies the equation
$$
u_{tt}+u_{xxxx}+\Big(\int_Iu^2\Big)u+2\Big(\varsigma u-1+\sqrt{1+\varsigma^2u^2}\Big)=0\qquad x\in I, \quad t>0,
$$
together with the corresponding conditions in \eq{jc2}. Also this problem should be intended in its weak form, that is,
$$
\langle u_{tt},v\rangle_V+\int_I u_{xx}v''+\int_Iu^2\int_Iuv+2\int_I\Big(\varsigma u-1+\sqrt{1+\varsigma^2u^2}\Big)v=0\qquad\forall v\in V(I),
\quad t>0.
$$

As in Section \ref{rigidhangers}, we reduce our analysis to the case of approximate two-modes solutions of \eq{nonloc2}
in the form \eq{formd}. To this end, for any couple $(\lambda^4,\kappa^2)$ of eigenvalues of \eq{autovsym} and \eq{autov2} we need to introduce
the two functions
$$
\Gamma_\varsigma^{\lambda,\kappa}(w,z):=\int_I\Big[\sqrt{1+\varsigma^2\big(we_\lambda(x)+z\eta_\kappa(x)\big)^2}+
\sqrt{1+\varsigma^2\big(we_\lambda(x)-z\eta_\kappa(x)\big)^2}-2\Big]e_\lambda(x)\, dx,
$$
$$
\Xi_\varsigma^{\lambda,\kappa}(w,z):=\int_I\Big[\sqrt{1+\varsigma^2\big(we_\lambda(x)+z\eta_\kappa(x)\big)^2}-
\sqrt{1+\varsigma^2\big(we_\lambda(x)-z\eta_\kappa(x)\big)^2}\Big]\eta_\kappa(x)\, dx.
$$
These functions have the following symmetries:
\neweq{symmGU}
\begin{array}{ll}
\Gamma_\varsigma^{\lambda,\kappa}(w,z)=\Gamma_\varsigma^{\lambda,\kappa}(-w,z)=\Gamma_\varsigma^{\lambda,\kappa}(w,-z)=\Gamma_\varsigma^{\lambda,\kappa}(-w,-z)\\
\Xi_\varsigma^{\lambda,\kappa}(w,z)=-\Xi_\varsigma^{\lambda,\kappa}(-w,z)=-\Xi_\varsigma^{\lambda,\kappa}(w,-z)=\Xi_\varsigma^{\lambda,\kappa}(-w,-z)
\end{array}\qquad\quad\forall(w,z)\in\R^2.
\endeq

Then the time-dependent coefficients $w, z$ of approximate two-modes solutions of \eq{nonloc2} in the form \eq{formd} satisfy the system
\neweq{mostgeneral}
\left\{\begin{array}{l}
\ddot{w}(t)+(\lambda^4+2\varsigma)w(t)+(1+2\Alk^2)z(t)^2w(t)+w(t)^3+\Gamma_\varsigma^{\lambda,\kappa}(w(t),z(t))=0\\
\ddot{z}(t)+(\kappa^2+2\varsigma)z(t)+(1+2\Alk^2)w(t)^2z(t)+z(t)^3+\Xi_\varsigma^{\lambda,\kappa}(w(t),z(t))=0,
\end{array}\right.
\endeq
to which we associate the initial conditions
\neweq{initialsyst}
w(0)=\delta,\quad z(0)=z_0,\quad\dot{w}(0)=\dot{z}(0)=0.
\endeq
If $z_0=0$, then the solution of \eq{mostgeneral}-\eq{initialsyst} is $(w,z)=(W_\lambda,0)$, where $W_\lambda$ solves
\neweq{superduffing}
\ddot{W}_\lambda(t)+(\lambda^4+2\varsigma)W_\lambda(t)+W_\lambda(t)^3+\Psi_\varsigma^\lambda\big(W_\lambda(t)\big)=0,\qquad
W_\lambda(0)=\delta,\qquad\dot{W}_\lambda(0)=0,
\endeq
where
$$
\Psi_\varsigma^\lambda(w)=2\int_I\Big[\sqrt{1+\varsigma^2w^2e_\lambda(x)^2}-1\Big]e_\lambda(x)\, dx.
$$

The following statement holds.

\begin{proposition}\label{periodic}
Let $g(s):=(\lambda^4+2\varsigma)s+s^3+\Psi_\varsigma^\lambda(s)$ and $G(s):=\int_0^sg(\tau)d\tau$. Then, for any $\delta>0$, the solution $W_\lambda$ of \eqref{superduffing} is periodic with period
\begin{equation}\label{tautau}
\tau_\lambda=\sqrt{2} \int_{\gamma_\delta}^\delta \frac{ds}{\sqrt{G(\delta)-G(s)}},
\end{equation}
where $\gamma_\delta$ is the unique negative solution of $G(\gamma_\delta)=G(\delta)$.
\end{proposition}

To see this, it suffices to notice that
$$
\left|\frac{d}{ds}\Psi_\varsigma^\lambda(s)\right|=2\varsigma\left|\int_I\frac{\varsigma s e_\lambda(x)^3}{\sqrt{1+\varsigma^2s^2e_\lambda(x)^2}}\, dx\right|
\le2\varsigma\int_Ie_\lambda(x)^2dx=2\varsigma,
$$
which shows that $g$ is strictly increasing over $\R$ and
that $G$ is strictly convex. Therefore, the sublevels of the energy
$$\frac{\dot{w}^2}2 + G(w)$$
are convex sets of the phase plane $(w, \dot{w})$ and its level lines are closed curves. This shows that any solution of \eqref{superduffing} is
periodic; a standard computation then provides expression \eqref{tautau}.
\par
In fact, the ``unpleasant parts'' in \eq{mostgeneral} vanish whenever the longitudinal eigenfunction is odd.

\begin{proposition}\label{simplify}
Assume that $e_\lambda$ is an odd eigenfunction of \eqref{autovsym}, as given by Theorem \ref{symmetriceigen}. Then $\Gamma_\varsigma^{\lambda,\kappa}(w,z)=\Xi_\varsigma^{\lambda,\kappa}(w,z)=0$
for any $(w,z)\in\R^2$ and any eigenfunction $\eta_\kappa$ of \eqref{autov2}, as given by Theorem \ref{autofz2}.
\end{proposition}

To prove Proposition \ref{simplify} one needs to distinguish two cases. If $\eta_\kappa$ is also odd, then
$\big(we_\lambda(x)\pm z\eta_\kappa(x)\big)^2$ are even functions so that the integrands in $\Gamma_\varsigma^{\lambda,\kappa}$ and $\Xi_\varsigma^{\lambda,\kappa}$ are odd and the integrals vanish.
If $\eta_\kappa$ is even, then one combines \eq{symmGU} with the change of variables $x\mapsto-x$ within the integrals
that define $\Gamma_\varsigma^{\lambda,\kappa}$ and $\Xi_\varsigma^{\lambda,\kappa}$.\par
Similarly as in Definition \ref{defstabtors}, we characterize the linear stability of longitudinal modes as follows.

\begin{definition}\label{defstabb3}
The mode $W_\lambda$ is said to be {\bf linearly stable} ({\bf unstable}) with respect to the $\kappa$-torsional-mode
if $\xi\equiv0$ is a stable (unstable) solution of the linear Hill equation
\neweq{hill2}
\ddot{\xi}(t)+\Big(\kappa^2+2\varsigma+(1+2\Alk^2)W_\lambda(t)^2+2\varsigma^2B_\varsigma(W_\lambda(t))\Big)\xi(t)=0,
\endeq
where
$$
B_\varsigma(w)=\int_I\frac{e_\lambda(x)\eta_\kappa(x)^2\, w}{\sqrt{1+\varsigma^2w^2e_\lambda(x)^2}}\, dx.
$$
\end{definition}

Equation \eq{hill2} is obtained by linearizing the second equation in system \eq{mostgeneral} around the solution $(w,z)=(W_\lambda,0)$, obtained with
initial datum $z_0=0$. It is a Hill equation, since the coefficient multiplying $\xi$ is periodic, see Proposition \ref{periodic}.
\par
We now state a result which highlights striking differences between odd and even longitudinal modes when slackening occurs; in particular, it will allow us to simplify \eqref{hill2} in some cases. Its proof will be given in
Section \ref{pfBalphanot0}.

\begin{theorem}\label{Balphanot0}
Let $\varsigma \neq 0$ and let $\lambda^4$ and $\kappa^2$ be, respectively, eigenvalues of \eqref{autovsym} and \eqref{autov2}, with associated eigenfunctions $e_\lambda$ and $\eta_\kappa$ provided by Theorems \ref{symmetriceigen} and \ref{autofz2}, respectively.
Then:
\begin{itemize}
\item if $e_\lambda$ is odd, then $B_\varsigma(w)=0$ for all $w\in\R$;
\item if $e_\lambda$ is even and either $\eta_\kappa=\mathbf{D}_\kappa$ or $\eta_\kappa=\mathbf{P}_\kappa$ (being zero on the
side spans), then $B_\varsigma(w)\not\equiv0$ in any neighborhood of $w=0$.
\end{itemize}
\end{theorem}

Theorem \ref{Balphanot0} does not clarify what happens in case of a multiple torsional eigenvalue $\kappa$, when choosing the associated eigenfunctions differently from Theorem \ref{autofz2}. Also the case when $e_\lambda$ is even and $\eta_\kappa$ is an eigenfunction
of \eqref{autov2} other than $\mathbf{D}_\kappa$ or $\mathbf{P}_\kappa$ is not considered. We refer the reader to the discussion in Section \ref{probopen}.
\par
Combining Theorem \ref{Balphanot0} with \eqref{propAlk}, we can simplify \eqref{hill2}.
\begin{corollary}\label{corollarion}
Let $\lambda^4$ and $\kappa^2$ be, respectively, eigenvalues of \eqref{autovsym} and \eqref{autov2}, with associated eigenfunctions $e_\lambda$ and $\eta_\kappa$ provided by Theorems \ref{symmetriceigen} and \ref{autofz2}, respectively. Then:
\begin{itemize}
\item if $e_\lambda$ is odd and $\eta_\kappa$ is even, then equation \eqref{hill2} simplifies to
$$
\ddot{\xi}(t)+\Big(\kappa^2+2\varsigma+W_\lambda(t)^2\Big)\xi(t)=0\,;
$$
\item if $e_\lambda$ is odd and $\eta_\kappa$ is odd, then \eqref{hill2} simplifies to
$$
\ddot{\xi}(t)+\Big(\kappa^2+2\varsigma+(1+2\Alk^2)W_\lambda(t)^2\Big)\xi(t)=0\,;
$$
\item if $e_\lambda$ is even and $\eta_\kappa$ is odd, then \eqref{hill2} simplifies to
$$
\ddot{\xi}(t)+\Big(\kappa^2+2\varsigma+W_\lambda(t)^2+2\varsigma^2B_\varsigma(W_\lambda(t))\Big)\xi(t)=0.
$$
\end{itemize}
\end{corollary}
For the linear stability of system \eq{mostgeneral}, we thus have to take into account all these possible combinations, resulting in the following theorem, which will be proved in Section \ref{pftheotors2}.
\begin{theorem}\label{theotors2}
Let $\varsigma \neq 0$ and let $\lambda^4$ and $\kappa^2$ be, respectively, eigenvalues of \eqref{autovsym} and \eqref{autov2}, with $\kappa^2$ associated with an eigenfunction of the kind $\mathbf{D}_\kappa$ or $\mathbf{P}_\kappa$ (being zero on the side spans). The following hold:
\begin{itemize}
\item the $\lambda$-longitudinal-mode of \eqref{mostgeneral} of amplitude $\delta$ is linearly unstable with respect to the $\kappa$-torsional-mode
whenever $\delta$ is sufficiently large;
\item if $e_\lambda$ is odd, then the $\lambda$-longitudinal-mode of \eqref{mostgeneral} of amplitude $\delta$ is linearly stable with respect to
the $\kappa$-torsional-mode if $\delta$ is sufficiently small, provided that $(\kappa^2+2\varsigma)/(\lambda^4+2\varsigma) \notin \mathbb{N}^2$;
\item if $e_\lambda$ is even, then the $\lambda$-longitudinal-mode of \eqref{mostgeneral} of amplitude $\delta$ is linearly stable with respect to
the $\kappa$-torsional-mode if $\delta$ is sufficiently small, provided that $4(\kappa^2+2\varsigma)/(\lambda^4+2\varsigma) \notin \mathbb{N}^2$.
\end{itemize}
\end{theorem}

This statement should be compared with Theorem \ref{theotors}. Again, it merely gives qualitative sufficient conditions both for linear
stability and instability. Some differences appear here due to the presence of the term $B_\varsigma$, which makes the coefficient of $\xi$ in \eqref{hill2} $\tau_\lambda$ (and not $\tau_\lambda/2$)-periodic, thus causing a discontinuity with respect to the case $\varsigma=0$. This shows that in general any (even small) nonlinearity due to the hangers may destroy the stability diagram for a given system studied without hangers.
\par More results are obtained numerically, as reported in the next section. Therein we also study
numerically the nonlinear stability (according to Definition \ref{torsunstable}) and we seek the optimal position of the piers that maximizes the critical energy.

\subsubsection{Optimal position of the piers in degenerate plates}\label{posottimale}

In this section, we analyze the stability of the second torsional mode (cf. \eqref{solosecondo}) for system \eqref{mostgeneral}. The procedure is similar to the one in Section \ref{optfish}: however, we have already seen that the presence of $f_\varsigma$ makes the situation richer, alternatively involving each of the equations displayed in Corollary \ref{corollarion}.
\par
The situation is already more complicated at the level of linear stability, since we do not have a precise picture of the resonance tongues for equation \eqref{hill2}, especially if $e_\lambda$ is even. In principle, it may happen to cross thick instability regions which alternate to stability ones; hence the definition of critical amplitude given in \eqref{D} would be nonsense. Nevertheless, we could not numerically detect any of them. We thus simply took as critical amplitude threshold of linear instability the least value of $\delta$ for which the absolute value of the monodromy matrix associated with \eqref{hill2} is larger than 2. Concerning nonlinear instability, we proceeded as usual, following the steps in Section \ref{algoritmo}.
\par
As for the choice of $\varsigma$, we are guided by the claims reported in the engineering literature \cite{bartoli, luco} (see Section \ref{physmod}): the cables represent the main source of nonlinearity, so we stick to \emph{small values of $\varsigma$}. A full quantitative analysis would require a huge effort and would be useless without taking into account the exact values of the real parameters inside the considered equations, see Section \ref{probopen} for further comments.
We chose to perform our analysis for $\varsigma \in \{0.1, 0.2, 0.5, 1\}$. The corresponding results are displayed in Table \ref{tabellafinale}. For the reader's convenience, we also report the results from Table \ref{tfishbone} for $\varsigma=0$. The global energy of the system, given by
\begin{eqnarray*}
E_\varsigma(u, \theta) = \frac{1}{2} \int_I (u_t^2\!+\! \theta_t^2\!+\! \theta_x^2 \!+\!u_{xx}^2)
+\frac{1}{4}\left[\int_I(u+\theta)^2\right]^2+\frac{1}{4}\left[\int_I(u-\theta)^2\right]^2 +  \int_{I} F_\varsigma(u+\theta)+\int_I F_\varsigma(u-\theta)
\end{eqnarray*}
with
$$
F_\varsigma(s)=\frac{1}{2}\Big(\varsigma s^2 -2s+s\sqrt{1+\varsigma^2 s^2} + \frac{\textnormal{arcsinh}(\varsigma s)}{\varsigma}\Big),
$$
increases with the hangers rigidity $\varsigma$, thus $E_\varsigma$ is not the most suitable parameter for the comparison of the stability performances on varying of $\varsigma$. Much more meaningful appears the critical amplitude. For this reason, in Table \ref{tabellafinale} we report the two critical amplitudes $\delta_{lin}$ and $\delta$ for linear and nonlinear stability, respectively. Table \ref{tabellafinale} also contains the time $\tau$ in correspondence of which nonlinear instability is observed and the associated expansion rate $\mathcal{ER}_\tau$, which is around $100$, in line with what we observed for the beam.
\par
It appears evident that, for a given $a$, $\varsigma \mapsto \delta_{lin}(\varsigma)$ is decreasing, showing that increasing the elasticity of the hangers lowers the linear stability. The pattern for nonlinear instability is by far less clear.
First, for $a=0.8$ and $a=0.9$ the longitudinal mode having the smallest instability threshold is $e_1$ (odd), while for the other values of $a$ it is $e_0$ (even). Hence, one should not compare the critical amplitudes for $a=0.8$, $a=0.9$ (marked with the symbol $^*$ in Table \ref{tabellafinale}) with the others; actually, they are much smaller. Moreover, they are \emph{increasing} with respect to $\varsigma$, a feature observed
in all our experiments with odd longitudinal modes. As for the other values of $a$, there is no clear emerging pattern; hence, in this case linear stability gives important complementary information. Overall, we observe that
\begin{center}
{\bf
in the range \eqref{physicalrange}, the best placement of the piers for system \eqref{mostgeneral} lies around $a=0.5$.
}
\end{center}

\begin{table}[ht!]
\begin{center}
{\footnotesize
\begin{tabular}{|c|c|c|c|c|c|}
\hline
$a$ & $\varsigma$ & $\delta_{lin}$  & $\delta$ & $\tau$ & $\mathcal{ER}_\tau$ \\
\hline
 & 0 & 1.77  & 2.38 & 15.61 & 70.07 \\
\cline{2-6}
& 0.1 & 1.77 & 2.44 & 15.16 & 69.85 \\
\cline{2-6}
0.5 & 0.2 & 1.77 & 2.51 & 14.68 & 69.45 \\
\cline{2-6}
& 0.5 & 1.75 & 2.54 & 15.96 & 97.78  \\
\cline{2-6}
& 1 & 1.7 & 2.64 & 14.76 & 74.14 \\ 
\hline
 & 0 & 1.71 & 2.45 & 15.95 &  125.61 \\
\cline{2-6}
 & 0.1 & 1.71 & 2.4 & 15.95 & 101.08 \\
\cline{2-6}
0.56 & 0.2 & 1.7 & 2.35 & 15.96 & 83.37 \\
\cline{2-6}
 & 0.5 & 1.68 & 2.41 & 15.08 & 67.43 \\
\cline{2-6}
& 1 & 1.63 & 2.43 & 15.96 & 77.84 \\
\hline
 & 0 & 1.66 & 2.27 & 14.89 & 65.64 \\
\cline{2-6}
& 0.1 & 1.66 & 2.38 & 14.17 & 65.59 \\
\cline{2-6}
0.6 & 0.2 & 1.65 & 2.42 & 15.99 & 116 \\
\cline{2-6}
& 0.5 & 1.63 & 2.28 & 15.94 & 68.05 \\ 
\cline{2-6}
& 1 & 1.58 & 2.49 & 15.97 & 95.98 \\ 
\hline
 & 0 & 1.57   & 2.15 & 15.97 & 78.6 \\
\cline{2-6}
 & 0.1 & 1.57 & 2.13 & 15.89 & 66.21 \\
\cline{2-6}
2/3 & 0.2 & 1.56  & 2.24 & 15.03  & 65.92 \\
\cline{2-6}
& 0.5 & 1.54 & 2.35 & 15.96 & 93.56  \\
\cline{2-6}
& 1 & 1.49 & 2.49 & 14.31 & 61.77 \\
\hline
& 0 & 1.52 & 2.19 & 15.94 & 95.91 \\
\cline{2-6}
 & 0.1 & 1.52 & 2.14 & 15.94 & 75.82\\
\cline{2-6}
0.7 & 0.2 & 1.52 & 2.14 & 15.73 & 66.1 \\
\cline{2-6}
& 0.5 & 1.5 &  2.38 & 15.99 & 103.57 \\
\cline{2-6}
& 1 & 1.45 & 2.44 & 14.62  & 61.53\\
\hline
& 0 & 1.39 &  1.6* & 15.71  & 100.07 \\
\cline{2-6}
 & 0.1 & 1.39 & 1.61* & 15.44 & 99.19 \\
\cline{2-6}
0.8 & 0.2 & 1.39 & 1.61*  & 15.28  & 88.08 \\
\cline{2-6}
& 0.5 & 1.37  & 1.64* & 14.56 & 85.26 \\
\cline{2-6}
& 1 & 1.32 & 1.64* & 15.82 & 92.27 \\
\hline
& 0 & 1.28  & 1.36* & 15.8 & 81.91 \\
\cline{2-6}
 & 0.1 & 1.27  & 1.38* & 15.36  & 82.02 \\
\cline{2-6}
0.9 & 0.2 & 1.27 & 1.4* & 14.97 & 81.62  \\
\cline{2-6}
& 0.5 & 1.25 & 1.45* & 14.05 & 74.36  \\
\cline{2-6}
& 1 & 1.21 & 1.46* & 15.33 &  82.52 \\
\hline
\end{tabular}
}
\caption{Linear and nonlinear instability for \eqref{systemtorsional}.}
\label{tabellafinale}
\end{center}
\end{table}


\vspace{3cm}
\
\eject

\subsection{The algorithm for the computation of the energy thresholds}\label{algoritmo}

In this section, we briefly describe the program we implemented for the numerical simulations in order to check whether nonlinear instability
appears. The program has been written and run using Mathematica$^\copyright$ software. Recalling Definition \ref{unstable},
the aim is to determine whether a residual mode has increased its oscillations so as to subvert of an order of magnitude the inequality in \eqref{prevalentec2} and it has grown sufficiently abruptly on itself. This is done in an iterative way for different values of $a \in (0, 1)$,
as follows.
\begin{itemize}
\item[-] \textit{Step 1}. Fix
$
0 < a < 1, \ N \in \{2, 3, 4, \ldots\}, \ 0 < \eta < 1, \ 0 < tol \ll 1, \ 0 < step \ll 1, \ 0 < r < 1, EN_{\max} > 0
$
and set $j=0$.
\item[-] \textit{Step 2}. Find the first $\delta_0 > r/\eta^2$ such that, setting $u(x, 0)$ and $u_t(x, 0)$ as in \eqref{initnocin} with $\alpha_j=\delta_0$, $\alpha_n=r$ for $n \neq j$, condition \eqref{prevalentec} holds. Set $\delta=\delta_0$.
\item[-] \textit{Step 3}. Define $EN=E_j(a)$ as $\mathcal{E}(U_j^A)$, as discussed after \eqref{energiamodoj}, with $\alpha_j=\delta$; if $EN > EN_{\max}$, go to Step 8, otherwise set $BD=\delta/10$,
compute $T_W(\delta)$ through formula \eqref{sceltawagner} and go to Step 4.
\item[-] \textit{Step 4}. Fix $T > 2\lim_{\delta \to 0^+} T_W(\delta)$ (so that $T > 2T_W(\delta)$ for every $\delta > 0$).
\item[-] \textit{Step 5}.
Numerically integrate (e.g., via the ``NDSolve'' procedure) the finite-dimensional system \eqref{finitodim} on the interval $[0, T]$.
Check whether there exist $k \in \{0, 1, \ldots, N-1\} \setminus \{j\}$ and $\tau \in (2T_W(\delta), T)$ for which $\vert \phi_k(\tau) \vert =BD$:
if yes, proceed to Step 6, if no, proceed to Step 7.
\item[-] \textit{Step 6}. For every $k$ as in Step 5, evaluate
$$
M=\max_{t \in [0, \tau/2]} \vert \phi_k(t) \vert
$$
and check if the ratio $BD/M > 1/\eta - tol$: if yes, compute $E_j(a)$ via the expression in \eqref{energiamodoj} and proceed to Steps 8-9. Otherwise,
set $BD=M/\eta$ and repeat Step 5.
\item[-] \textit{Step 7}. Increase $\delta$ by $step$ and repeat Steps 3-5.
\item[-] \textit{Step 8}. If $j \leq N-2$, increase $j$ by one unit and perform Steps 2-5, otherwise go to Step 9.
\item[-] \textit{Step 9}. Find $\mathbb{E}_{12}(a)$ as in \eqref{Edodici}; then change the value of $a$ and repeat Steps 2-8.
\end{itemize}
For our experiments, we chose $a$ ranging from $0.1$ to $0.9$ with a step of $0.1$, together with the remarkable values $a=1/3$, $a=14/25=0.56$, $a=2/3$, and
$$
N=12, \ j \in \{0, \ldots, 11\}, \ \eta=0.1, \ tol=10^{-8}, \ step=0.01, \ r=0.01.
$$
We integrated the corresponding system on the interval $[0, 16]$, after having checked that the condition $16 > 2T_W(\delta)$ was always fulfilled.
Notice that the presence of $EN_{\max}$ in the above algorithm is only needed to ensure that the algorithm ends; however, in our experiments we continued increasing $\delta$ until we found instability and this process always had success. Possible loop cycles in controlling instability, due to an insufficient machine precision, have been bypassed by further slightly increasing the initial value $\delta$ by $step$.
\par
As for Definition \ref{torsunstable}, if the torsional component is initially set equal to $r$, of course in Step 2 it has to be taken $BD=r/\eta$.

\section{Proofs}

\addtocontents{toc}{\setcounter{tocdepth}{0}}

\secttoc

\subsection{Proof of Theorem \ref{VI}}\label{pfVI}

It is well-known that an orthogonal basis of the space $H^2\cap H^1_0(I)$ is formed by the eigenfunctions of the eigenvalue problem
$$
u''''(x)=\mu u(x) \qquad x\in I,\qquad u(\pm\pi)=u''(\pm\pi)=0.
$$
Some calculus computations show that the eigenvalues and the eigenfunctions are given by
$$
\frac{n^4}{16}\qquad\mbox{and}\qquad \sin\frac{n(x+\pi)}{2}\qquad(n\in\N).
$$
Writing $u\in H^2\cap H^1_0(I)$ in Fourier series with respect to this basis, that is,
$$
u(x)=\sum_{n=1}^\infty\alpha_n\, \sin\frac{n(x+\pi)}{2}\,  \quad \textrm{ with } \quad \alpha_n= \frac{1}{\pi} \int_{-\pi}^{\pi} u(s) \sin\frac{n(s+\pi)}{2} \, ds,
$$
and imposing the pointwise conditions in \eqref{V}, we obtain the following characterization of the space $V(I)$:
$$
u\in V(I)\ \Longleftrightarrow\ \sum_{n=1}^\infty\alpha_n\, \sin\frac{n(1-b)\pi}{2}=\sum_{n=1}^\infty\alpha_n\, \sin\frac{n(1+a)\pi}{2}=0.
$$
This double constraint on the $\alpha_n$'s shows that the space $V(I)$ has codimension 2. On the other hand, let $h \in H^2 \cap H^1_0(I)$ be a continuous and piecewise affine function on $[-\pi,\pi]$, namely
$$
h(x)=\left\{
\begin{array}{ll}
\alpha(x+\pi) & \mbox{ if } x \in [-\pi, -b\pi] \vspace{0.1 cm} \\
\displaystyle \frac{\alpha(b-1)+\beta(1-a)}{a+b}\, x + \frac{\alpha a(1-b)+\beta b(1-a)}{a+b}\, \pi\quad & \mbox{ if } x \in [-b\pi,a\pi] \vspace{0.1cm} \\
\beta (\pi-x) & \mbox{ if } x \in [a\pi,\pi]
\end{array}
\right.
$$
for some $\alpha, \beta \in\R$. Then, by integration by parts on each of the intervals $I_-$, $I_0$, $I_+$, it is immediately checked that
$$
\int_I h u'' = 0 \qquad \forall u \in V(I).
$$
This fully characterizes $V(I)^\perp$. The functions $v_1$ and $v_2$ are found explicitly by integrating twice the function $h$ for a suitable choice of
the integration constants and by imposing vanishing conditions at the endpoints $\pm\pi$.

\subsection{Proof of Theorem \ref{regular}}\label{pfregular}

Existence and uniqueness of the weak solution $u\in V(I)$ follow directly from the Lax-Milgram Theorem, since $\gamma \geq 0$.\par
$(i)$ By arguing as in \cite[Lemma 2.2]{HolNec10} and performing an integration by parts similar to that in \cite[Example 1]{Loc73} in the final part,
we find that a weak solution $u(x)$ of \eq{statforte}-\eq{bcstat} belongs to $C^2(I)$ and is of class $C^4$ on each subinterval $\overline{I}_-$,
$\overline{I}_0$ and $\overline{I}_+$. Moreover, by performing two integration by parts on each subinterval, the terms computed in $-b\pi, a\pi$ compensate (due to the $C^2$-regularity); Definition \ref{defweakstat} is then fulfilled only if $u''(\pm\pi)=0$, due to the arbitrariness of $v \in V(I)$.
\par
$(ii)$ In order to prove this statement, we introduce a general procedure for the analysis of \eq{statforte} that we first describe in detail in the
case $\gamma=0$. Let $f\in C^0(\overline{I})$ and consider the classical solution $U_f$ of the problem
$$
U_f''''(x)=f(x)\qquad x \in I,\qquad U_f(-\pi)=U_f(-b\pi)=U_f(a\pi)=U_f(\pi)=0.
$$
Clearly, $U_f\in C^4(\overline{I})\cap V(I)$ but $U_f$ may not be a weak solution (according to Definition \ref{defweakstat})
because it may fail to fulfill the conditions $U_f''(-\pi)=U_f''(\pi)=0$ required by the just proved Item $(i)$. Wishing to satisfy \eq{weakstat},
we add to $U_f$ three third order polynomials $P_b$, $P_0$, $P_a$ defined, respectively, in $[-\pi,-b\pi]$, $[-b\pi,a\pi]$, $[a\pi,\pi]$, satisfying
\neweq{fourconditions}\begin{array}{cc}
P_b(x)\!=\!(x\!+\!\pi)(x\!+\!b\pi)(Ax\!+\!B\pi),\ P_0(x)\!=\!(x\!+\!b\pi)(x\!-\!a\pi)(Cx\!+\!D\pi),\ P_a(x)\!=\!(x\!-\!a\pi)(x\!-\!\pi)(Ex\!+\!F\pi),\\
P_b'(-b\pi)\!=\!P_0'(-b\pi),\quad P_b''(-b\pi)\!=\!P_0''(-b\pi),\quad P_a'(a\pi)\!=\!P_0'(a\pi),\quad P_a''(a\pi)\!=\!P_0''(a\pi).
\end{array}\endeq

The first line in \eq{fourconditions} guarantees that these polynomials vanish at the endpoints of the interval where they are defined, whereas the second
line in \eq{fourconditions} ensures that they match $C^2$ in $\{-b\pi,a\pi\}$. Then we introduce two further constraints of no bending at the endpoints,
that is,
\neweq{otherconditions}
P_b''(-\pi)=-U_f''(-\pi)=:f_b,\qquad P_a''(\pi)=-U_f''(\pi)=:f_a.
\endeq

The first line in \eq{fourconditions} introduces six unknowns $A$, $B$, $C$, $D$, $E$, $F$, to be determined. The second line in \eq{fourconditions}
gives four equations, whereas \eq{otherconditions} gives two further equations, all being linear and linking the six unknowns. The former four conditions
are ``structural'' since they only depend on $a$ and $b$, while the two latter conditions are ``forced'' since they also depend on the source $f$ through the
function $U_f$. Since these six equations are linearly independent, they uniquely determine the unknowns $A$, $B$, $C$, $D$, $E$, $F$.
Take these constants and replace them into the polynomials in \eq{fourconditions}. Then the solution of \eq{weakstat} has the form
\neweq{weaksolu}
u(x)=U_f(x)+\left\{\begin{array}{ll}
P_b(x) \quad & \mbox{if } x \in [-\pi,-b\pi]\\
P_0(x) \quad & \mbox{if } x \in [-b\pi,a\pi]\\
P_a(x) \quad & \mbox{if } x \in [a\pi,\pi].
\end{array}\right.
\endeq

Next, we analyze the possible discontinuities of the third derivative of this weak solution. In view of \eq{fourconditions} and \eq{weaksolu}, we have
\neweq{gapsthird}
\begin{array}{cc}
u'''(-b\pi^+)-u'''(-b\pi^-)=P_0'''(-b\pi)-P_b'''(-b\pi)=6(C-A)=:\beta_f,\\
u'''(a\pi^+)-u'''(a\pi^-)=P_a'''(a\pi)-P_0'''(a\pi)=6(E-C)=:\alpha_f.
\end{array}
\endeq
Therefore, we may rewrite \eq{statforte} with $\gamma=0$ in the following distributional form:
\neweq{distributional}
u''''=f+\beta_f\delta_{-b\pi}+\alpha_f\delta_{a\pi},
\endeq
which completes the proof of Item $(ii)$ in the case $\gamma=0$.\par
If $\gamma=4\nu^4>0$, consider the classical solution $U_f$ of the problem
\neweq{wrong2}
U_f''''(x)+4\nu^4 U_f(x)=f(x)\qquad x \in I,\qquad U_f(-\pi)=U_f(-b\pi)=U_f(a\pi)=U_f(\pi)=0.
\endeq
Such a solution exists and is unique: one may fulfill the four-point conditions in \eq{wrong2} by adding to any solution of the differential equation in \eq{wrong2} a suitable linear combination of the functions in
$$K:=\Big\{\cos(\nu x)\cosh(\nu x),\, \cos(\nu x)\sinh(\nu x),\, \sin(\nu x)\cosh(\nu x),\, \sin(\nu x)\sinh(\nu x)\Big\},$$
which generate the 4-dimensional kernel of the operator $w\mapsto w''''+4\nu^4w$. Again, $U_f$ may not be a weak solution and, in order to satisfy \eq{weakstat},
we seek three linear combinations of the functions in $K$, that we call $P_b$, $P_0$, $P_a$, defined respectively in $\overline{I}_-$,
$\overline{I}_0$, $\overline{I}_+$, and satisfying the conditions
\neweq{twelveconditions}
\begin{array}{cc}
P_b(-\pi)=P_b(-b\pi)=0,\ P_0(-b\pi)=P_0(a\pi)=0,\ P_a(a\pi)=P_a(\pi)=0,\\
P_b'(-b\pi)=P_0'(-b\pi),\quad P_b''(-b\pi)=P_0''(-b\pi),\quad P_a'(a\pi)=P_0'(a\pi),\quad P_a''(a\pi)=P_0''(a\pi),\\
P_b''(-\pi)=-U_f''(-\pi),\quad P_a''(\pi)=-U_f''(\pi).
\end{array}\endeq

The first line in \eq{twelveconditions} guarantees that these combinations vanish at the endpoints of the interval where they are defined.
The second line in \eq{twelveconditions} ensures that they match $C^2$ in $\{-b\pi,a\pi\}$.
The third line in \eq{twelveconditions} forces the solution to have no bending at the endpoints, in line with Item $(i)$.
Overall, \eq{twelveconditions} contains twelve conditions. There are also twelve unknowns: the four coefficients of the linear combinations of the elements in $K$,
for each of the three functions $P_b$, $P_0$, $P_a$. One finds these unknowns by solving the corresponding linear system and
the weak solution of \eq{weakstat} is again given by \eq{weaksolu}. The same arguments used for \eq{gapsthird} and \eq{distributional} lead to
the distributional equation $u''''+4\nu^4 u=f+\beta_f\delta_{-b\pi}+\alpha_f\delta_{a\pi}$. This completes the proof of Item $(ii)$ also for $\gamma>0$.\par
$(iii)$ The subspace $X(I)\subset C^0(\overline{I})$ is defined by the two (linear) constraints $\beta_f=\alpha_f=0$. \par
$(iv)$ In view of \eq{gapsthird}, we have $u\in C^3(\overline{I})$ if and only if $\beta_f=\alpha_f=0$ which, by the just proved Item $(iii)$, yields
$u\in C^4(\overline{I})$.

\subsection{Proof of Theorem \ref{autovalorias}}\label{pfautovalorias}

Since $e(x)$ is of class $C^4$ on $\overline{I}_-, \overline{I}_0$ and $\overline{I}_+$ in view of Theorem \ref{regular}, we can write
\begin{equation}\label{decomposizione}
e(x)=
\left\{
\begin{array}{ll}
e_b(x)\quad & \mbox{if } x \in [-\pi,-b\pi]\\
e_0(x)\quad & \mbox{if } x \in [-b\pi,a\pi] \\
e_a(x)\quad & \mbox{if } x \in [a\pi,\pi],
\end{array}
\right.
\end{equation}
where $e_b, e_0$ and $e_a$ are classical solutions of $e''''=\lambda^4 e$ on their intervals of definition. Consequently, on each span,
$e(x)$ is a linear combination of the four functions
\begin{equation}\label{labase}
\cos(\lambda x),\quad\sin(\lambda x),\quad\cosh(\lambda x),\quad\sinh(\lambda x).
\end{equation}
We thus seek the coefficients of all the nonzero linear combinations of \eqref{labase} for which $e(x)$ as in \eqref{decomposizione} is a solution
of \eqref{autovsym}.\par
By imposing that $e_b(-\pi)=e_b''(-\pi)=e_b(-b\pi)=0$, one finds that ($K_b \in \mathbb{R}$)
$$
e_b(x)=K_b\Big\{\sinh\big[\lambda(1-b)\pi\big]\sin\big[\lambda(x+\pi)\big]-\sin\big[\lambda(1-b)\pi\big]\sinh\big[\lambda(x+\pi)\big]\Big\}.
$$
By imposing that $e_a(\pi)=e_a''(\pi)=e_a(a\pi)=0$, one finds that ($K_a \in \mathbb{R}$)
$$
e_a(x)=K_a\Big\{\sinh\big[\lambda(1-a)\pi\big]\sin\big[\lambda(x-\pi)\big]-\sin\big[\lambda(1-a)\pi\big]\sinh\big[\lambda(x-\pi)\big]\Big\}.
$$
Finally, by requiring that $e_0(-b\pi)=e_0(a\pi)=0$ one deduces that ($K_1, K_2 \in \mathbb{R}$)
\begin{eqnarray*}
e_0(x) &=& K_1\Big\{\sinh\big[\lambda(a+b)\pi\big]\cos(\lambda x)+\cos(\lambda b\pi)\sinh\big[\lambda(x-a\pi)\big]-
\cos(\lambda a\pi)\sinh\big[\lambda(x+b\pi)\big]\Big\}\\
 &\ & +K_2\Big\{\sinh\big[\lambda(a+b)\pi\big]\sin(\lambda x)-\sin(\lambda b\pi)\sinh\big[\lambda(x-a\pi)\big]-
\sin(\lambda a\pi)\sinh\big[\lambda(x+b\pi)\big]\Big\}.
\end{eqnarray*}
Then \eqref{autovsym} is satisfied if and only if these functions match $C^2$ in $-b\pi$ and in $a\pi$, so that the eigenfunction $e$ in \eq{decomposizione} belongs
to $C^2(\overline{I})$. This leads to a $4\times4$ linear system of $K_b$, $K_a$, $K_1$, $K_2$, and nontrivial solutions are obtained by imposing that the determinant
of such system is equal to $0$. This requirement is precisely the condition in the statement.

\subsection{Proof of Theorem \ref{symmetriceigen}}\label{pfsymmetriceigen}

Thanks to the symmetry of the interval $I$, we can seek odd and even eigenfunctions of \eqref{autovsym} separately. Moreover, we can restrict our attention to the half-interval $[0, \pi]$, aiming at determining $e_0, e_a$ such that
\begin{equation}\label{esym}
e(x)=
\left\{
\begin{array}{ll}
e_0(x)\quad & \mbox{if } x \in [0, a\pi] \\
e_a(x)\quad & \mbox{if } x \in [a\pi, \pi]
\end{array}
\right.
\end{equation}
solves \eqref{autovsym} (recall \eqref{decomposizione}); the expression of $e$ over $[-\pi, 0]$ is then obtained by prolonging \eqref{esym} by symmetry, either even or odd. Following this scheme, we only have to impose the $C^2$-matching of the functions $e_0$ and $e_a$
in $x=a\pi$.\par
Using the same notation as in the proof of Theorem \ref{autovalorias}, when searching for odd solutions this leads to set $K_1 = 0$ and results into the system
\begin{equation}\label{sistemaproof6}
\footnotesize
\!\left\{
\begin{array}{l}
\! \!\! (\cos[\lambda(1\!-\!a)\pi]\!\sinh[\lambda(1\!-\!a)\pi]\!-\!\sin[\lambda(1\!-\!a)\pi]\!\cosh[\lambda(1\!-\!a)\pi])K_a \!=\!2\cosh(\lambda a \pi) (\cos (\lambda a \pi)\!\sinh (\lambda a \pi)\!-\!\sin (\lambda a \pi) \!\cosh (\lambda a \pi))K_2 \\
\! \!\! \sin[\lambda(1\!-\!a)\pi]\sinh[\lambda(1\!-\!a)\pi] K_a + 2\sin(\lambda a \pi)\sinh(\lambda a \pi) \cosh(\lambda a \pi)  K_2 = 0.
\end{array}
\right.
\end{equation}
Nontrivial solutions of this system exist when the associated determinant is equal to zero, namely when \eqref{auto1} holds. We distinguish two cases:
\begin{itemize}
\item[-] if $\lambda \in \mathbb{N}$, then \eqref{auto1} has to hold with both sides equal to $0$, so that necessarily $\lambda a \in \mathbb{N}$ and, consequently, also $\lambda(1-a) \in \mathbb{N}$. The second equation in \eqref{sistemaproof6} is thus identically satisfied, while the first equation reads
$$
\cos(\lambda \pi) \sinh[\lambda(1-a)\pi] K_a = 2\cosh(\lambda a \pi) \sinh (\lambda a \pi) K_2
$$
(recall that $\sin(\lambda a \pi)=0$, so that $\vert\cos(\lambda a \pi)\vert=1$). This yields the eigenfunction $\mathbf{O}_\lambda=\sin(\lambda x)$ for the choices $K_a=\frac{1}{\cos (\lambda \pi)\sinh[\lambda(1-a)\pi]}$, $K_2=\frac{1}{2\cosh(\lambda a \pi)\sinh (\lambda a \pi)}$;
\item[-] if $\lambda \notin \mathbb{N}$, then by \eqref{auto1} it is necessarily $\lambda a \notin \mathbb{N}$ and $\lambda(1-a) \notin \mathbb{N}$.
The explicit form of $\mathscr{O}_\lambda$ can then be derived by using the second equation of system \eqref{sistemaproof6}, choosing $K_a=-\frac{\sin(\lambda a \pi)}{\sin[\lambda(1-a) \pi]}$ and $K_2=\frac{\sinh[\lambda(1-a)\pi]}{2\cosh(\lambda a \pi) \sinh(\lambda a \pi)}$.
\end{itemize}
This completes the proof for odd eigenfunctions.\par
Similarly, the search for even solutions leads to set $K_2 = 0$ and results into the system
$$
\footnotesize
\!\left\{
\begin{array}{l}
\! \! \!(\sin[\lambda(1\!-\!a)\pi]\!\cosh[\lambda(1\!-\!a)\pi]\!-\!\cos[\lambda(1\!-\!a)\pi]\!\sinh[\lambda(1\!-\!a)\pi] )K_a \!=\! 2\sinh(\lambda a \pi) (\cos (\lambda a \pi)\!\sinh (\lambda a \pi)\!+\!\sin (\lambda a \pi)\!\cosh (\lambda a \pi))K_1 \\
\! \! \! \sin[\lambda(1\!-\!a)\pi]\sinh[\lambda(1\!-\!a)\pi]K_a + 2\cos(\lambda a\pi) \sinh(\lambda a \pi) \cosh(\lambda a\pi)  K_1 = 0.
\end{array}
\right.
$$
The proof is here analogous to the one for odd eigenfunctions: the determinant associated with the system is equal to zero if and only if
\eqref{auto2} holds, and a similar analysis to the one performed above yields the result.

\subsection{Proof of Theorem \ref{Michelle}}\label{pfMichelle}

The linear operator $L$ defined on $V(I)$ by $\langle Lu, v \rangle_V=\int_{I} u''v''$ is self-adjoint. Hence, its eigenfunctions form a basis of $V(I)$.
For the rest of the proof, it is convenient to make a change of variables. We set $\Lambda=\lambda a$ and $\alpha=1/a$, so that $\alpha \in (1, +\infty)$ and the eigenvalue problem \eqref{autov0} is rephrased as
\begin{equation}\label{autovautov0}
\int_J e'' v'' =\mu\int_J e v \qquad \forall v\in V(J),
\end{equation}
being $J=(-\alpha\pi, \alpha\pi)$, $\overline{J}=\overline{J}_-\cup \overline{J}_0 \cup \overline{J}_+$, with $\overline{J}_-=[-\alpha\pi, -\pi]$, $\overline{J}_0=[-\pi, \pi]$ and $\overline{J}_+=[\pi, \alpha\pi]$, and $V(J)=\{u\in H^2\cap H^1_0(J);\, u(-\pi)=u(\pi)=0\}\,$. With this procedure, we view the piers as fixed in $\pm \pi$ and move the endpoints $\pm \alpha \pi$ of the beam.
Accordingly, \eqref{auto1} and \eqref{auto2} are changed, respectively, into
\begin{equation}\label{autoauto1}
\sin(\Lambda \alpha \pi)\sinh(\Lambda \pi)\sinh[\Lambda(\alpha-1)\pi]=\sinh(\Lambda \alpha \pi)\sin(\Lambda \pi)\sin[\Lambda(\alpha-1)\pi]
\end{equation}
\begin{equation}\label{autoauto2}
\cos(\Lambda \alpha \pi)\cosh(\Lambda \pi)\sinh[\Lambda(\alpha-1)\pi]=\cosh(\Lambda \alpha \pi)\cos(\Lambda \pi)\sin[\Lambda(\alpha-1)\pi],
\end{equation}
and the corresponding eigenfunctions $\mathscr{O}_\Lambda$ and $\mathscr{E}_\Lambda$ are given, respectively, by the odd and the even extensions of
$$
{\footnotesize
\mathscr{O}_\lambda(x)=
\left\{
\begin{array}{ll}
\!\!\displaystyle \frac{\sinh[\Lambda(\alpha-1)\pi]}{\sinh (\Lambda \pi)} \,  (\sinh (\Lambda \pi) \sin (\Lambda x)- \sin(\Lambda \pi) \sinh (\Lambda x) ) & \mbox{if } x \in [0, \pi] \vspace{0.25cm} \\
\!\!\displaystyle \frac{\sin(\Lambda \pi)}{\sin[\Lambda(\alpha-1)\pi]} \, (\sin[\Lambda(\alpha-1)\pi]\sinh[\Lambda(x-\alpha\pi)]-\sinh[\Lambda(\alpha-1)\pi] \sin[\Lambda(x-\alpha\pi)]) & \mbox{if } x \in [\pi, \alpha\pi],
\end{array}
\right.
}
$$
and
$$
{\footnotesize
\mathscr{E}_\lambda(x)=
\left\{
\begin{array}{ll}
\!\!\displaystyle \frac{\sinh[\Lambda(\alpha-1)\pi]}{\cosh (\Lambda \pi)} \,  (\cosh (\Lambda \pi) \cos (\Lambda x)- \cos(\Lambda \pi) \cosh (\Lambda x) ) & \mbox{if } x \in [0, \pi] \vspace{0.25cm} \\
\!\!\displaystyle \frac{\cos(\Lambda \pi)}{\sin[\Lambda(\alpha-1)\pi]} \, (\sinh[\Lambda(\alpha-1)\pi] \sin[\Lambda(\alpha\pi-x)]- \sin[\Lambda(\alpha-1)\pi]\sinh[\Lambda(\alpha\pi-x)]) & \mbox{if } x \in [\pi, \alpha\pi],
\end{array}
\right.
}
$$
while $\mathbf{O}_\Lambda$ and $\mathbf{E}_\Lambda$ respectively become equal to $\sin(\Lambda x)$ and $\cos(\Lambda x)$ for $x \in J$.
Furthermore, the hyperbolas $\lambda=\Lambda_n/a$ in the $(a, \lambda)$-plane become the horizontal lines $\Lambda=\Lambda_n$ in the $(\alpha, \Lambda)$-plane, and the curves defining the eigenvalues become strictly decreasing, as we state in the following theorem. This is well depicted in Figures \ref{confronto2} and \ref{figurateoteo}.
\begin{figure}[ht]
\begin{center}
\includegraphics[scale=0.54]{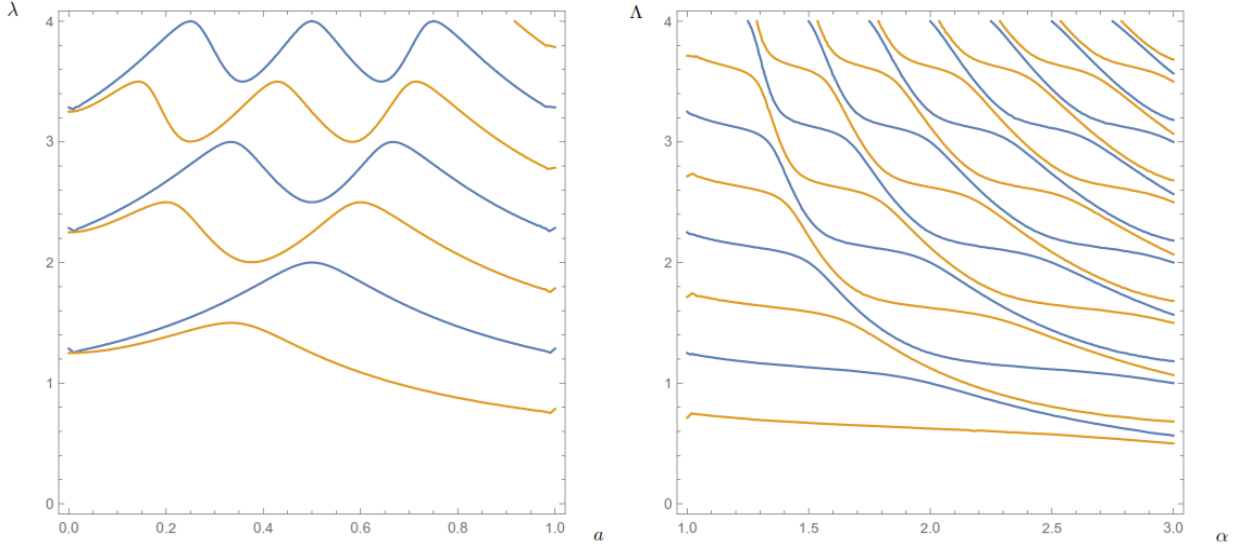}
\caption{The solutions $\lambda=\lambda(a)$ of \eqref{auto1} and \eqref{auto2} (left) and $\Lambda=\Lambda(\alpha)$ of \eqref{autoauto1} and \eqref{autoauto2} (right).}\label{confronto2}
\end{center}
\end{figure}
\par
The statement that we are going to prove in this new framework is the following.
\medbreak
\noindent
\textbf{Theorem \ref{Michelle}'}.
\emph{For any $\alpha > 1$, the eigenvalues $\mu=\Lambda^4$ of problem \eqref{autovautov0} are simple and form a countable set. Moreover, \eqref{autoauto1} and \eqref{autoauto2} implicitly define, for $\alpha\in(1, +\infty)$, a countable family of analytic and strictly decreasing
functions $\alpha \mapsto \Lambda_n(\alpha)$ satisfying $\Lambda_n(\alpha)\to\Lambda_{n}$ (see \eqref{clampedeigen}) for $\alpha\to1$ (n=0, 1, 2, \ldots).}
\medbreak
The proof of Theorem \ref{Michelle}' is based on two lemmas. With the first one, we characterize the eigenvalues of problem \eqref{autovautov0} as functions of $\alpha$. To this end,
we define $I, P: (1, +\infty) \times (0, +\infty)  \to \mathbb{R}$ as
$$\begin{array}{cc}
I(\alpha, \Lambda)= \sin(\Lambda \alpha\pi)\sinh(\Lambda\pi)\sinh[\Lambda(\alpha-1)\pi]-\sinh(\Lambda \alpha\pi)\sin(\Lambda\pi)\sin[\Lambda(\alpha-1)\pi],\\
P(\alpha, \Lambda)= \cos(\Lambda \alpha\pi)\cosh(\Lambda\pi)\sinh[\Lambda(\alpha-1)\pi]-\cosh(\Lambda\alpha\pi)\cos(\Lambda\pi)\sin[\Lambda(\alpha-1)\pi].
\end{array}$$
Moreover, we denote by $C_I$ and $C_P$ the $0$-level sets of $I$ and $P$, respectively, and we observe that they are characterized by \eq{autoauto1} and \eq{autoauto2}.
We prove the following result.

\begin{lemma}\label{IFT}
The equalities $I(\alpha,\Lambda)=0$ and $P(\alpha,\Lambda)=0$ implicitly define a family of analytic strictly decreasing functions $\alpha \mapsto \Lambda(\alpha)$ on each connected
component of $C_I$ and $C_P$, respectively.
\end{lemma}
\begin{proof} We carry on a detailed proof only for the odd eigenfunctions, since the arguments for the even ones are similar.
We set $C_I^0:=C_I \cap \{(\alpha, \Lambda)\mid \Lambda \in \mathbb{N}\}$
and we denote the partial derivatives with respect to $\Lambda$ and $\alpha$ through the subscripts ${}_\Lambda$ and ${}_\alpha$. Since
\begin{eqnarray*}
\footnotesize
I_\alpha(\alpha, \Lambda) & = & \Lambda \pi(\cos(\Lambda \alpha \pi) \sinh(\Lambda \pi) \sinh[\Lambda(\alpha-1)\pi] +\sin(\Lambda \alpha \pi) \sinh(\Lambda \pi) \cosh[\Lambda(\alpha-1)\pi] \\
& & - \cosh(\Lambda \alpha \pi) \sin(\Lambda \pi) \sin[\Lambda(\alpha-1)\pi]-\sinh(\Lambda \alpha \pi) \sin(\Lambda \pi) \cos[\Lambda(\alpha-1)\pi]),
\end{eqnarray*}
noticing that $\Lambda \in \mathbb{N}$ is equivalent to $\Lambda \alpha \in \mathbb{N}$ for the solutions of $I(\alpha, \Lambda)=0$ we immediately have that
\begin{equation}\label{Ia0}
{I_\alpha}_{{\big\vert}_{C_I^0}} = \Lambda \pi \cos(\Lambda \alpha \pi) \sinh(\Lambda \pi) \sinh[\Lambda(\alpha-1)\pi],
\end{equation}
while some computations show that
\begin{equation}\label{Ia}
{I_\alpha}_{{\big\vert}_{C_I \setminus C_I^0}}= \Lambda \pi \sin(\Lambda \pi) \sinh(\Lambda \pi)\frac{\sin[\Lambda(\alpha-1)\pi]^2-\sinh[\Lambda(\alpha-1)\pi]^2}{\sin[\Lambda(\alpha-1)\pi]\sinh[\Lambda(\alpha-1)\pi]}.
\end{equation}

On the other hand,
\begin{equation}\label{Il0}
{I_\Lambda}_{{\big\vert}_{C_I^0}}= \alpha\pi \cos(\Lambda \alpha \pi) \sinh(\Lambda \pi) \sinh[\Lambda(\alpha-1)\pi],
\end{equation}
and
\begin{equation}\label{Il}
{I_\Lambda}_{{\big\vert}_{C_I \setminus C_I^0}}= \frac{\alpha-1}{\Lambda} I_\alpha(\Lambda, \alpha) + \pi \sin[\Lambda(\alpha-1)\pi] \sinh[\Lambda(\alpha-1)\pi] \frac{\sin(\Lambda\pi)^2-\sinh(\Lambda\pi)^2}{\sin(\Lambda\pi)\sinh(\Lambda\pi)}.
\end{equation}
Therefore, $I_\alpha$ and $I_\Lambda$ are everywhere different from zero on $C_I$ and
\begin{equation}\label{segno}
I_\alpha\cdot I_\Lambda > 0 \quad\textrm{for every } (\alpha, \Lambda) \in C_I\, ;
\end{equation}
the sign in \eq{segno} is obtained on $C_I^0$ by using \eqref{Ia0} and \eqref{Il0}, while on $C_I \setminus C_I^0$ we see that \eqref{Ia} and
\eqref{Il} both have the sign of $-\sin(\Lambda\pi) \sin[\Lambda(\alpha-1)\pi]$. As for \eqref{Il}, notice that the second summand therein has the
same sign as the first. Thus \eqref{segno} follows and, by the Implicit Function Theorem, the curves depicted in Figure \ref{bello} are graphs of an
analytic function $\Lambda=\Lambda(\alpha)$ in any neighborhood of each point of $C_I$; moreover, the function $\alpha\mapsto \Lambda(\alpha)$ is strictly decreasing in view of \eqref{segno}.\par
The statement for even eigenfunctions can be obtained by reasoning in an analogous way on the function $P$, changing $C_I^0$ with
$C_P^0:=C_P\cap \Big\{(\alpha, \Lambda)\mid \Lambda -\frac{1}{2} \in \mathbb{N}\Big\}$.
\end{proof}

We now determine the limits of the function $\alpha \mapsto \Lambda(\alpha)$ for $\alpha \to 1$ and $\alpha \to +\infty$.

\begin{lemma}\label{autoclamped}
For any curve $\Lambda=\Lambda(\alpha)$ whose graph is a connected subset of $C_I$ or $C_P$, there exists a unique $n \in \mathbb{N}$ such that
\begin{equation}\label{limiteuno}
\lim_{\alpha \to 1^+} \Lambda(\alpha) = \Lambda_{n}.
\end{equation}
Moreover, this correspondence is one-to-one.
\end{lemma}
\begin{proof}
We restrict again our attention to odd eigenfunctions, the arguments for the even ones being similar. Since $\alpha \mapsto \Lambda(\alpha)$ is
strictly decreasing by Lemma \ref{IFT}, the limit in \eqref{limiteuno} exists and we denote it by $\hat{\Lambda}$. To compute $\hat{\Lambda}$,
we analyze the behavior of $\Lambda(\alpha)$ in a neighborhood of the point $(\alpha, \Lambda)=(1,\hat{\Lambda})$, parametrizing its graph as the curve
$$
\Lambda(s)=\hat{\Lambda} + l(s),\quad \alpha(s)= 1 + \beta(s),\quad(s \geq 0),\quad l(0)=\beta(0)=0.
$$
We first notice that
$$
\frac{\sin[\Lambda\beta(s)\pi]}{\sinh[\Lambda\beta(s)\pi]}=1+o(1)\qquad\mbox{as }s\to0.
$$
Therefore, if we perform an asymptotic expansion of the identity $I(\alpha(s), \Lambda(s)) \equiv 0$ as $s\to0$, we obtain
$$
\sin\big[(\hat{\Lambda}+l(s))(1+\beta(s))\pi\big]\sinh\big[(\hat{\Lambda}+l(s))\pi\big]=\sinh\big[(\hat{\Lambda}+l(s))(1+\beta(s))\pi\big]
\sin\big[(\hat{\Lambda}+l(s))\pi\big](1+o(1)).
$$
Introducing the infinitesimal $\eps(s):=l(s)+\hat{\Lambda}\beta(s)$, the last identity reads
$$
\Big(\sin\big[\hat{\Lambda}\pi\big]+\cos\big[\hat{\Lambda}\pi\big]\eps(s)\pi+o(\eps(s))\Big)
\Big(\sinh\big[\hat{\Lambda}\pi\big]+\cosh\big[\hat{\Lambda}\pi\big]l(s)\pi+o(\eps(s))\Big)
$$
$$
=\Big(\sinh\big[\hat{\Lambda}\pi\big]+\cosh\big[\hat{\Lambda}\pi\big]\eps(s)\pi+o(\eps(s))\Big)
\Big(\sin\big[\hat{\Lambda}\pi\big]+\cos\big[\hat{\Lambda}\pi\big]l(s)\pi+o(\eps(s))\Big)(1+o(1)).
$$
After computing all the products, some terms cancel. Then, by dropping the lower order terms and by recalling that $I(\alpha,\Lambda)\equiv0$, we obtain
$$
\frac{\sin(\hat{\Lambda} \pi)}{\sinh(\hat{\Lambda}\pi)}= \frac{\cos(\hat{\Lambda} \pi)}{\cosh(\hat{\Lambda}\pi)},
$$
namely $\hat{\Lambda}^4$ satisfies \eqref{clampedeigen} and therefore it is an odd eigenvalue of \eqref{clamped}.\par
We now prove that for every eigenvalue $\Lambda_n$ of the clamped problem \eqref{clamped} there exists exactly one connected branch of eigenvalues
of \eq{autovsym} emanating from it. We consider the hyperbolas
$$
\mathcal{H}_n=\{(\alpha, \Lambda) \mid \Lambda \alpha = n \}, \quad n \in \mathbb{N}.
$$
From \eqref{clampedeigen} we deduce that for every couple of consecutive integers $n$ and $n+1$, there exists a unique odd eigenvalue $\hat{\Lambda}$
of the clamped problem \eqref{clamped} such that $(1, \hat{\Lambda})$ belongs to the region delimited by $\mathcal{H}_n$ and $\mathcal{H}_{n+1}$:
precisely, $\hat{\Lambda}= \Lambda_{2n-1}$. Moreover, since
$$
I(\alpha, \Lambda)\!=\!\sin(\Lambda \alpha \pi)\big[\!\sinh(\Lambda \pi) \sinh[\Lambda(\alpha-1)\pi] - \sinh(\Lambda \alpha \pi) \sin(\Lambda \pi) \cos(\Lambda \pi)\big]+
\sinh(\Lambda \alpha \pi) \sin^2(\Lambda \pi) \! \cos(\Lambda \alpha \pi)
$$
changes sign passing from $\mathcal{H}_n$ to $\mathcal{H}_{n+1}$, there exists at least one curve $\Lambda=\Lambda(\alpha)$ of solutions of $I(\alpha, \Lambda)=0$
included therein. In fact, such a curve is unique, since formula \eqref{Il} implies that on each branch emanating from $(1, \Lambda_{2n-1})$ the function $I_\Lambda$ has strictly the same sign: this would be a contradiction in presence of multiple branches, since on each branch it is $I(\alpha, \lambda)=0$.
\end{proof}

Lemmas \ref{IFT} and \ref{autoclamped} prove all the statements of Theorem \ref{Michelle}' and thus all the ones of Theorem \ref{Michelle}, except the limit $\lambda_n(a) \to \Lambda_n^*$ for $a \to 0$. For even eigenfunctions, the limit is obtained simply by letting $a \to 0$ in \eqref{auto2}, noticing that the limit equation is the one defining the eigenvalues $\Lambda_n^*$ (recall \eqref{clampedeigen} and \eqref{coincidenzaaut}). For odd eigenfunctions, this does not work since it produces an identity; the thesis instead follows from the just proved fact that $\lambda_n(a) \to \Lambda_n$ for $a \to 1$ (Lemma \ref{autoclamped}), together with the invariance of \eqref{auto1} upon the substitution $a \mapsto 1-a$.

\subsection{Proof of Theorem \ref{constant}}\label{pfconstant}

As in the proof of Theorem \ref{Michelle}, we work with the variables $\Lambda = \lambda a$ and $\alpha=1/a$. In view of Theorem \ref{Michelle}, we can sort the eigenvalues of \eqref{autovautov0} in increasing order $\{\Lambda_0, \Lambda_1, \Lambda_2,\ldots\}$ and label the corresponding eigenfunctions again as $\{e_0, e_1, e_2, \ldots\}$. Theorem \ref{constant} is now rephrased as follows.
\medbreak
\noindent
\textbf{Theorem \ref{constant}'}. \emph{For $\alpha > 1$, it holds that $i(e_n)=n$, for every $n=0, 1, 2, \ldots$. Moreover, as $\alpha$ increases, the zeros of $e_n=e_{\Lambda_n(\alpha)}$ move by couples from the central span to the side spans whenever the
curve $\Lambda=\Lambda_n(\alpha)$ intersects any of the horizontal lines $\{\Lambda=\Lambda_k\}$, for integers $k \geq 0$ having the same parity as $n$.}
\medbreak
The nice properties of the curves $\Lambda=\Lambda(\alpha)$, which we visualize in Figure \ref{figurateoteo}, bring further evidence of the convenience of the change of variables made above.
\begin{center}
\begin{figure}
\begin{center}
\begin{tikzpicture}
\node[inner sep=0pt] (whitehead) at (10,0)
    {\includegraphics[width=9cm]{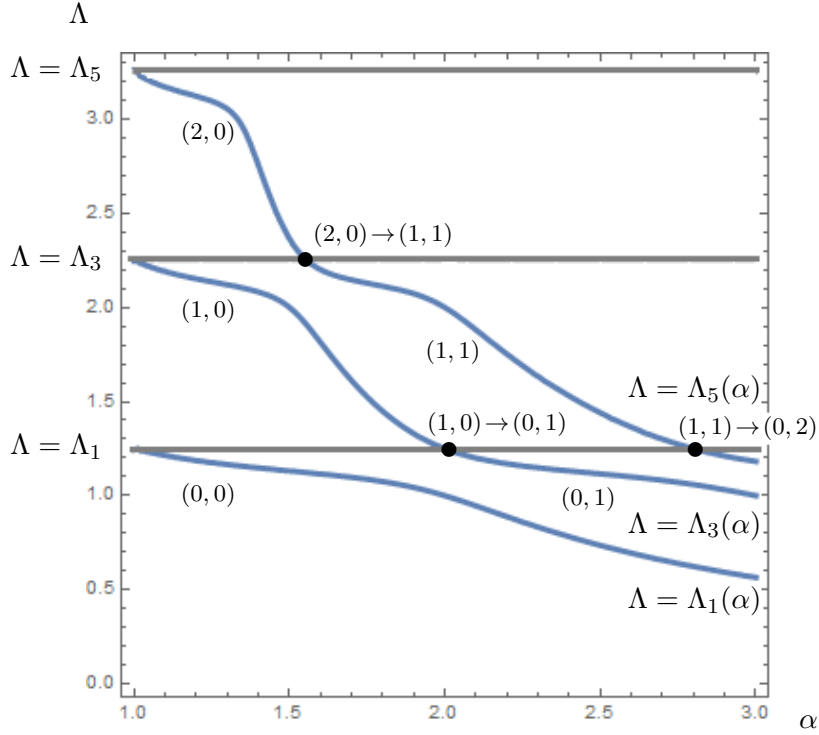}};

\node at (15, -4.4){$\alpha$};
\node at (5.4, 5){$\Lambda$};

\node at (5.1, 4.25){$\Lambda=\Lambda_5$};
\node at (7.1, 3.4){\footnotesize{{$(2,0)$}}};
\node at (9.4, 2.05){\footnotesize{{$(2,0) \!\to\! (1,1)$}}};
\node at (10.32, 0.5){\footnotesize{{$(1,1)$}}};
\node at (14.2, -0.5){\footnotesize{{$(1,1) \!\to\! (0,2)$}}};
\node at (8.38, 1.72){\footnotesize{{\Large{$\bullet$}}}};
\node at (13.51, -0.8){\footnotesize{{\Large{$\bullet$}}}};
\node at (13.5, 0){$\Lambda=\Lambda_5(\alpha)$};

\node at (5.1, 1.75){$\Lambda=\Lambda_3$};
\node at (7.1, 1.05){\footnotesize{{$(1,0)$}}};
\node at (10.9, -0.45){\footnotesize{{$(1,0) \!\to\! (0,1)$}}};
\node at (12.1, -1.45){\footnotesize{{$(0,1)$}}};
\node at (10.27, -0.8){\footnotesize{{\Large{$\bullet$}}}};
\node at (13.5, -1.8){$\Lambda=\Lambda_3(\alpha)$};

\node at (5.1, -0.75){$\Lambda=\Lambda_1$};
\node at (7.1, -1.4){\footnotesize{{$(0,0)$}}};
\node at (13.5, -2.8){$\Lambda=\Lambda_1(\alpha)$};
\end{tikzpicture}
\end{center}
\caption{A visual description of Theorem \ref{Michelle}' for odd eigenfunctions.}\label{figurateoteo}
\end{figure}
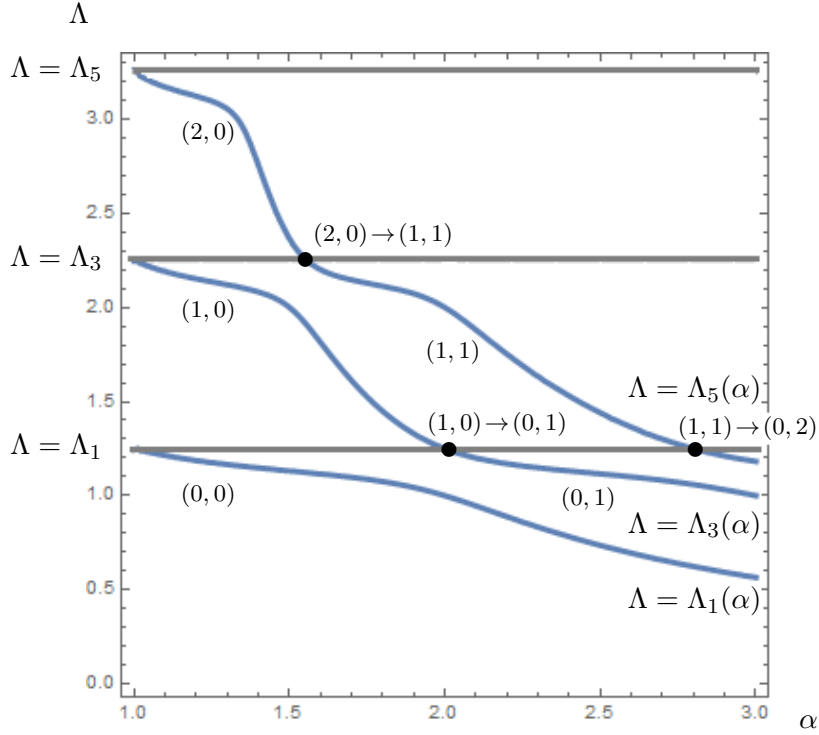
\end{center}
To prove Theorem \ref{constant}', we fix the notations
\begin{equation}\label{oddeven}
k_o(\Lambda)= \left[\Lambda-\frac{1}{2}\right], \qquad k_e(\Lambda)=[\Lambda], \qquad r(\Lambda)=\left[\Lambda(\alpha-1)-\frac{1}{2}\right],
\end{equation}
where $[\cdot]$ denotes the integer part and is set equal to $0$ for negative numbers. Moreover, we introduce the two functions
$$
g(s)=\frac{\sin(s \pi)}{\sinh(s \pi)}-\frac{\cos(s \pi)}{\cosh(s \pi)}, \qquad h(s)=\frac{\sin(s \pi)}{\sinh(s \pi)}+\frac{\cos(s \pi)}{\cosh(s \pi)}.
$$
With these preliminaries, we first state a technical result.

\begin{lemma}\label{zeriaut}
Let $\alpha >1$ be fixed. Let $\mu=\Lambda^4$ be an eigenvalue of \eqref{autovautov0}, with corresponding eigenfunction $e_\Lambda$. If $e_\Lambda$ is \emph{odd},
then its number of zeros in $J_0$ is
\begin{eqnarray*}
\displaystyle 2k_o(\Lambda)+1 \textrm{ if } g(\Lambda) < 0, \qquad 2k_o(\Lambda)+3 \textrm { if }g(\Lambda) > 0, & \quad \mbox{ when $k_o(\Lambda)$ is \emph{odd}}; \\
\displaystyle 2k_o(\Lambda)+1 \textrm{ if } g(\Lambda) > 0, \qquad 2k_o(\Lambda)+3 \textrm { if }g(\Lambda) < 0, & \quad \mbox{ when $k_o(\Lambda)$ is \emph{even}}.
\end{eqnarray*}
If $e_\Lambda$ is \emph{even}, then its number of zeros in $J_0$ is
\begin{eqnarray*}
\displaystyle 2k_e(\Lambda) \textrm{ if } h(\Lambda) < 0, \qquad 2k_e(\Lambda)+2 \textrm { if }h(\Lambda) > 0, & \quad \mbox{ when $k_e(\Lambda)$ is \emph{odd}}; \\
\displaystyle 2k_e(\Lambda) \textrm{ if } h(\Lambda) > 0, \qquad 2k_e(\Lambda)+2 \textrm { if }h(\Lambda) < 0, & \quad \mbox{ when $k_e(\Lambda)$ is \emph{even}}.
\end{eqnarray*}
Finally, the number of zeros of $e_\Lambda$ in $J_-$ (and thus in $J_+$) is
\begin{eqnarray*}
\displaystyle r(\Lambda) \textrm{ if } g[\Lambda(\alpha-1)] < 0, \qquad r(\Lambda)+1 \textrm { if }g[\Lambda(\alpha-1)] > 0, & \quad \mbox{ when $r(\Lambda)$ is \emph{odd}}; \\
\displaystyle r(\Lambda) \textrm{ if } g[\Lambda(\alpha-1)] > 0, \qquad r(\Lambda)+1 \textrm { if }g[\Lambda(\alpha-1)] < 0, & \quad \mbox{ when $r(\Lambda)$ is \emph{even}}.
\end{eqnarray*}
\end{lemma}
\begin{proof}
We carry on a detailed proof only for odd eigenfunctions, the argument for the even ones being similar. In this case, $x=0$ is a zero of $e_\Lambda$
in $J_0$; on this interval, we thus restrict our attention to the sub-interval $(0,\pi)$. Therein, the zeros of $e_\Lambda$ correspond to the zeros of the function
$$
\zeta(x):=\frac{\sin(\Lambda x)}{\sinh(\Lambda x)} - \frac{\sin(\Lambda \pi)}{\sinh(\Lambda \pi)},
$$
which we now count by focusing on the behavior of the sine-function appearing in the numerator. To this end, observe first that $\zeta(x)$ is decreasing and positive for $x \in (0, \pi/2\Lambda]$. Therefore, if $\Lambda \leq 1/2$ then $\zeta(x) > 0$ for every $x \in (0, \pi)$, otherwise $\zeta(x) > 0$ for every $x \in (0, \pi/2\Lambda)$. In the former case, $k_o(\Lambda) = 0$ and the statement is proved since $\tan(\Lambda \pi) > \tanh(\Lambda \pi)$ for $\Lambda \in (0, 1/2]$. In the latter case, in any interval of the kind $[(2s+1)\pi/2\Lambda, (2s+3)\pi/2\Lambda]$, with $s \geq 0$ integer (corresponding to a ``half-cycle'' of the sine function between two consecutive extremal values), $\zeta(x)$ has exactly one zero in view of the Intermediate Value Theorem. Since there are exactly $k_o(\Lambda)$ such intervals strictly contained in $(0, \pi)$, we obtain $k_o(\Lambda)$ zeros in $(0, \pi)$. We then find at least $2k_o(\Lambda)+1$ zeros in $J_0$; however, there may be another zero in the ``residual'' interval $R=((2k_o(\Lambda)+3)\pi/2\Lambda, \pi)$, where the sine function does not perform a complete half-cycle. Indeed, since $\zeta(\pi)=0$, there is a supplementary zero in $(0, \pi)$ if and only if $\zeta$ has in $R$ a local extremum with opposite sign with respect to $\zeta((2k_o(\Lambda)+3)\pi/2\Lambda)$. This happens if and only if $e_\Lambda'(\pi^-) < 0$, so that $g(\Lambda) > 0$ when $k_o$ is odd, and $e_\Lambda'(\pi^-) > 0$, so that $g(\Lambda) < 0$ when $k_o$ is even.
Similar arguments prove the statement for the lateral spans. This completes the proof of the lemma.\end{proof}

In Lemma \ref{zeriaut} we have not considered the cases when $g(\Lambda)=0$ or $h(\Lambda)=0$ (and similar conditions on the lateral spans).
This is due to the fact that
\begin{equation}\label{doppio}
e_\Lambda(\pm \pi) = e_\Lambda'(\pm \pi)=0, \textrm{ that is, a double zero appears in the piers}\quad \Longleftrightarrow \quad g(\Lambda) h(\Lambda) = 0.
\end{equation}
Indeed, it is immediate to check that \eqref{doppio} may be fulfilled only for the eigenfunctions $\mathscr{O}_\Lambda$ and $\mathscr{E}_\Lambda$ for which, respectively, such condition reads
$$
g(\Lambda) = 0 \quad \textrm{ and } \quad  g[\Lambda(\alpha-1)]=0,
$$
$$
h(\Lambda) = 0 \quad \textrm{ and } \quad g[\Lambda(\alpha-1)]=0.
$$
We continue the proof of Theorem \ref{constant}, carrying on the details only for odd eigenfunctions, the case of even eigenfunctions being similar.
Let $\Lambda_{2m+1}$ be an odd eigenvalue
of \eqref{clamped} for some integer $m\ge0$. The strategy is to start from the point $(1,\Lambda_{2m+1})$ of the $(\alpha,\Lambda)$-plane and to show that the number of
zeros along the branch $(\alpha,\Lambda(\alpha))$ is constant, so that it equals $\#\{x \in J_0 \mid \psi_{2m+1}(x)=0\}=2m+1$ as can be checked by recalling
\eq{psieigen}; here, $\Lambda(\alpha)$ is the function found and characterized in Theorem \ref{Michelle}'.\par
To this end, we introduce some notations. We define the open strips determined by the
horizontal lines $\{\Lambda=\Lambda_{2k+1}\}$ in the $(\alpha,\Lambda)$-plane, that is,
\begin{equation}\label{lastriscia}
R_0= \mathbb{R}_+ \times (0, \Lambda_1), \qquad R_k=\mathbb{R}_+ \times (\Lambda_{2k-1}, \Lambda_{2k+1}), \qquad k=1, 2, 3, \ldots;
\end{equation}
notice that, since $\Lambda_n \approx n/2+3/4$ by \eq{clampedeigen}, such strips
contain one horizontal line corresponding to an integer $\Lambda$; more precisely, the line $\{\Lambda=k\}$ is contained in $R_{k-1}$.
\par
We denote by $\mathcal{C}_{2m+1}$ the branch $(\alpha,\Lambda(\alpha))$ that begins at $(1,\Lambda_{2m+1})$ and
continues in the direction of increasing $\alpha$ and, for $k \leq m$,
we set $\mathcal{C}_{2m+1}^k:=\mathcal{C}_{2m+1}\cap R_k$. Notice that, recalling \eqref{clampedeigen}, it holds
\begin{equation}\label{segnoimportante}
\textrm{sgn}\, g(\Lambda) = (-1)^k \quad \textrm{ on } \mathcal{C}_{2m+1}^k.
\end{equation}
By \eqref{doppio}, we know that in correspondence of $\mathcal{C}_{2m+1} \cap \{\Lambda=\Lambda_{2k+1}\}$ the eigenfunction
$e_{\Lambda(\alpha)}$ displays a double zero in the piers.
By \eq{clampedeigen} and \eq{psieigen}, we have that $\Lambda_{2m+1}\approx m+5/4$ and $\psi_{2m+1}$ has $2m+1$ zeros in $J_0$. Hence, if $\alpha-1$ is
sufficiently small, by continuity it turns out that $k_o(\Lambda)=k_o(\Lambda_{2m+1})$ and $r(\Lambda)=0$: therefore, the numbers of zeros on
the central span $J_0$ and on the lateral spans $J_\pm$ are maintained equal to $2m+1$ and $0$, respectively, in a neighborhood of $(1,\Lambda_{2m+1})$
while following $\mathcal{C}_{2m+1}$ (notice that $g$ has constant sign on $\mathcal{C}_{2m+1}^m$, that $g[\Lambda(\alpha-1)] > 0$ for $\alpha$ close to $1$ and that the procedure to count the number of zeros in $J_0$ used to prove Lemma \ref{zeriaut} holds as well for the eigenfunctions of the clamped problem on $[-\pi, \pi]$).
By Lemma \ref{zeriaut}, on growing of $\alpha$ these numbers are ruled
by the maps $k_o(\Lambda(\alpha))$ and $r(\Lambda(\alpha))$ defined in \eqref{oddeven}; indeed, in view of \eq{clampedeigen}, following the (decreasing) branch
$\mathcal{C}_{2m+1}$ the curve $\Lambda=\Lambda(\alpha)$ alternatively crosses the lines $\{\Lambda=n+1/2\}$ for integers $n \leq m$, and $\{\Lambda=\Lambda_{2k+1}\}$ for integers
$k<m$. The following lemmas detect the changes in the number of zeros of the eigenfunctions on each span in correspondence of such crossings.
\begin{lemma}\label{centrale}
The number of zeros of $e_{\Lambda(\alpha)}$ in $J_0$ decreases by $2$ at each crossing of $\Lambda=\Lambda(\alpha)$ with the lines $\{\Lambda=\Lambda_{2k+1}\}$, elsewhere it does not vary.
\end{lemma}
\begin{proof}
Recalling \eqref{segnoimportante}, we see that
whenever the branch $\Lambda=\Lambda(\alpha)$ crosses the lines $\Lambda=\Lambda_{2k+1}$, $k<m$, the integer $k_o(\Lambda)$ remains constant but $g(\Lambda)$ changes sign, so that the number of zeros on $J_0$ changes from $2k_o(
\Lambda)+3$ to $2k_o(\Lambda)+1$. This proves the first part of the statement.

On the other hand, any crossing with the lines $\{\Lambda = n+1/2\}$ does not modify the total number of zeros in the spans, since in this case $k_o(\Lambda)$ changes parity but $g(\Lambda)$ maintains the same sign. To see this, assume first that $m$ is odd; then, the sign of $g(\Lambda)$ is negative in view of \eqref{segnoimportante}, implying that above the line $\Lambda=n+1/2$ it is $k_o(\Lambda) = m$ (odd) and the corresponding $e_\Lambda$ has $2k_o(\Lambda)+1 = 2m+1$ zeros, and below such line $k_o(\Lambda)=m-1$ (even) and $e_\Lambda$ possesses $2k_o(\Lambda)+3=2(m-1)+3=2m+1$ zeros. If $m$ is even, a similar argument holds.
\end{proof}

\begin{lemma}\label{cent->lat}
The numbers of zeros of $e_{\Lambda(\alpha)}$ in $J_-$ and in $J_+$ both increase by one unit at every crossing of $\Lambda=\Lambda(\alpha)$ with the lines $\{\Lambda=\Lambda_{2k+1}\}$, elsewhere they do not vary.
\end{lemma}

\begin{proof}
We begin by stating three facts which follow from some calculus arguments. For the sake of brevity, we only give a hint of their proof.
\par
{\bf Fact (I).} \emph{The hyperbolas
$$
\mathcal{H}_n=\left\{(\alpha, \Lambda) \in \mathbb{R}_+^2 \mid \Lambda \alpha =n\right\}, \quad n=1, 2, \ldots,
$$
are such that $\mathcal{C}_{2m+1}$ is contained in the region between $\mathcal{H}_{m+1}$ and $\mathcal{H}_{m+2}$ and is tangent to $\mathcal{H}_{m+2}$,
see the left picture in Figure \ref{regolari}. The tangency points correspond to couples $(\alpha,\Lambda)$ for which $e_\Lambda\in C^4(\overline{J})$.}
\par
To prove this, one can change variables by introducing the homeomorphism
\begin{equation}\label{deffi}
\Phi: (1, +\infty) \times (0, +\infty) \to \{(x, y) \in \mathbb{R}_+^2 \mid x > y\}, \quad (x, y)=\Phi(\alpha, \Lambda)=(\Lambda \alpha \pi, \Lambda \pi),
\end{equation}
so as to straighten the hyperbolas $\mathcal{H}_n$, which become vertical lines, see the right picture in Figure \ref{regolari}.
One can then check that $\Phi(\mathcal{C}_{2m+1})$ is contained in the region between $\Phi(\mathcal{H}_{m+1})$ and $\Phi(\mathcal{H}_{m+2})$ and
that it is tangent to $\Phi(\mathcal{H}_{m+2})$. The second part of the statement follows from the fact that $\Lambda \alpha=\lambda$.
\begin{figure}[ht]
\begin{center}
\includegraphics[scale=0.54]{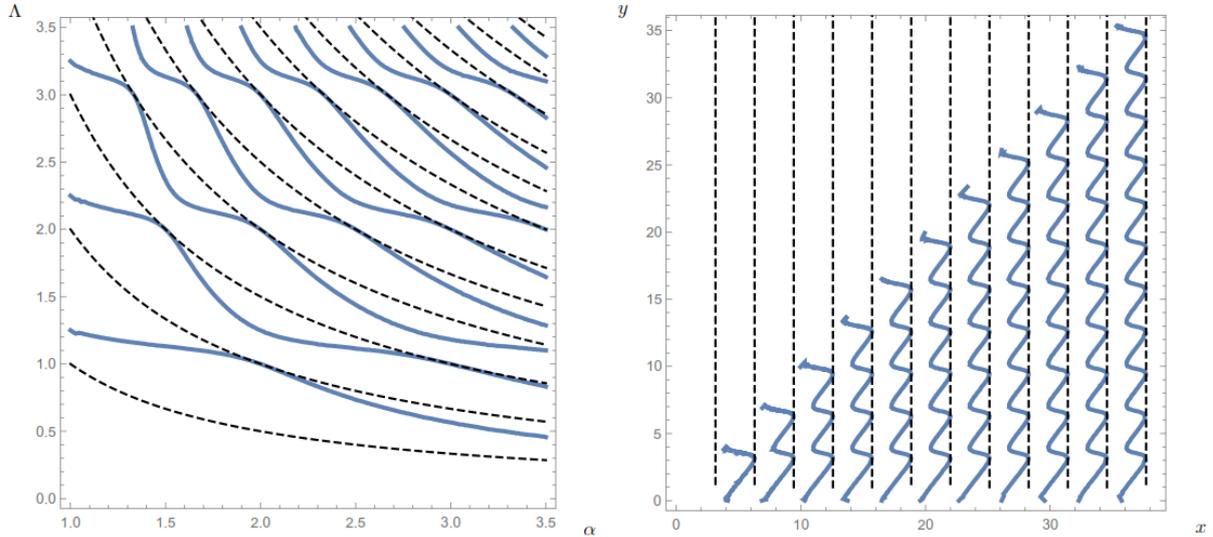}
\caption{The hyperbolas $\mathcal{H}_n$ (left, dashed) and their images (right) through the map $\Phi$ in \eq{deffi}.}\label{regolari}
\end{center}
\end{figure}

{\bf Fact (II).} \emph{On the branch $\alpha \mapsto \Lambda(\alpha)$, it holds $\textnormal{sgn} \, g[\Lambda(\alpha-1)] = -\textnormal{sgn} \, [g(\Lambda)\sin(\Lambda \alpha \pi)]$, so that
$$
\textnormal{sgn}\,g[\Lambda(\alpha-1)] = (-1)^{m+k+2} \quad \emph{ on } \mathcal{C}_{2m+1}^k,
$$
implying in particular}
$$
g[\Lambda(\alpha-1)] > 0 \quad \textrm{ on } \mathcal{C}_{2m+1}^m.
$$
This follows by rewriting \eq{autoauto1} as
\begin{eqnarray*}
& & \sin(\Lambda \pi) \sinh(\Lambda\pi) (\cos[\Lambda(\alpha-1)\pi] \sinh[\Lambda(\alpha-1)\pi]-\cosh[\Lambda(\alpha-1)\pi] \sin[\Lambda(\alpha-1)\pi]) = \\
& & \sin[\Lambda(\alpha-1)\pi] \sinh[\Lambda(\alpha-1)\pi] (\sin(\Lambda\pi) \cosh(\Lambda \pi)-\cos(\Lambda \pi) \sinh(\Lambda \pi)),
\end{eqnarray*}
and using \eqref{segnoimportante} and Fact (I), which says that
$$
\textnormal{sgn}\, \sin(\Lambda \alpha \pi)=(-1)^{m+1} \quad \textrm{ on } \mathcal{C}_{2m+1}.
$$

{\bf Fact (III).} \emph{The hyperbola
$$
\widetilde{\mathcal{H}}_n = \left\{(\alpha, \Lambda) \in \mathbb{R}_+^2 \mid \Lambda(\alpha-1) =n+\frac{1}{2}\right\}, \quad n=0, 1, \ldots
$$
crosses $\mathcal{C}_{2m+1}$ if and only if $n \leq m$. In this case, the intersection between $\widetilde{\mathcal{H}}_n$ and $\mathcal{C}_{2m+1}$ is unique and
takes place in the strip $R_{m-n}$ (recall the notation introduced in \eqref{lastriscia}).}\par
This is quite evident in the left picture in Figure \ref{iperboli}. Through the change of variables $\Phi$ defined in \eqref{deffi} one may again
straighten the considered hyperbolas, as displayed in the right picture in Figure \ref{iperboli}.
Then one rewrites \eq{autoauto1} in this new coordinate system as
\begin{equation}\label{autovgirati}
\sin(x)\sinh(y)\sinh(x-y)=\sinh(x)\sin(y)\sin(x-y).
\end{equation}
Noticing that the regions $\Phi(R_k)$ remain horizontal strips, while the vertical lines $\Phi(\mathcal{H}_n)$ depicted in the right picture in Figure \ref{regolari} determine lower and upper bounds for $x$
on each transformed branch $\Phi(\mathcal{C}_{2m+1})$, one can then show that there exists a unique intersection between $\Phi(\widetilde{\mathcal{H}}_n)$
and $\Phi(\mathcal{C}_{2m+1})$ for $x$ within such bounds. This may be done rigorously by combining equation \eqref{autovgirati} with $x-y=(n+1/2)\pi$:
existence and uniqueness of the solutions of this system follow from monotonicity arguments.
Notice that, in correspondence of the hyperbola $\widetilde{\mathcal{H}}_n$, the number $r(\Lambda)$ changes.
\par\smallskip
\begin{figure}[ht]
\begin{center}
\includegraphics[scale=0.54]{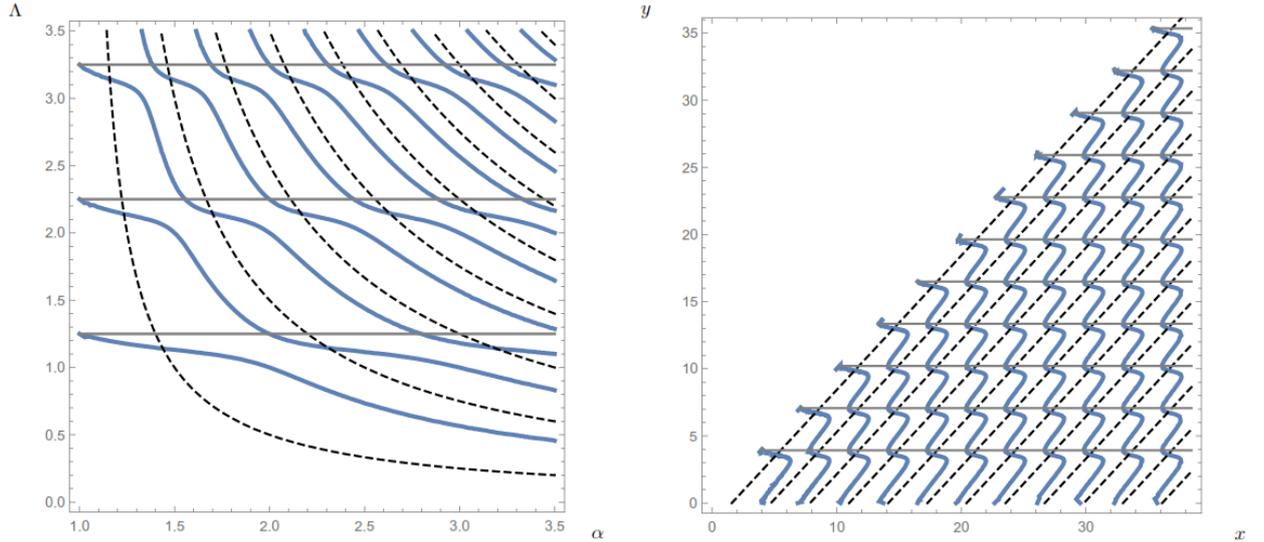}
\caption{The hyperbolas $\widetilde{\mathcal{H}}_n$ and their straightening through the map $\Phi$ in \eq{deffi}.}\label{iperboli}
\end{center}
\end{figure}
With these three facts, we can now complete the proof of the statement. In the strip $R_m$, the branch $\mathcal{C}_{2m+1}$ crosses the hyperbola $\widetilde{\mathcal{H}}_0$ thanks to Fact (III)
but, according to Lemma \ref{zeriaut}, no zeros appear in the lateral span in correspondence of such intersection, since $r(\Lambda)$ remains equal to $0$ and $g[\Lambda(\alpha-1)] > 0$ on $\mathcal{C}_{2m+1}^m$. Thanks to this fact and with a similar argument, using again Lemma \ref{zeriaut} it is possible to show that \emph{none} of the crossings with the hyperbolas $\widetilde{\mathcal{H}}_n$ contributes to modify the number of zeros in the lateral span. Indeed, by Fact (III) the hyperbola $\widetilde{\mathcal{H}}_n$ crosses $\mathcal{C}_{2m+1}$ in $R_{m-n}$, with $r(\Lambda)$ passing from $n-1$ to $n$ in correspondence of the intersection. Moreover, by Fact (II) $\textnormal{sgn}\,g[\Lambda(\alpha-1)]=(-1)^{2m+2-n}$ in a neighborhood of the intersection. Lemma \ref{zeriaut} then applies, showing that the number of zeros of $e_\Lambda$ in $J_-$ and $J_+$ is not affected by this crossing. On the other hand, at each crossing of $\mathcal{C}_{2m+1}$ with $\{\Lambda=\Lambda_{2k-1}\}$, $k \leq m$, the integer $r(\Lambda)$ remains constantly equal to $m-k$, but $\textnormal{sgn}\,g[\Lambda(\alpha-1)]$ passes from $(-1)^{m+k+2}$ to $(-1)^{m+k+1}$, so that by Lemma \ref{zeriaut} $e_{\Lambda(\alpha)}$ acquires a zero in $J_-$ and a zero in $J_+$.
\end{proof}
The proof of Theorem \ref{constant}' follows from Lemmas \ref{centrale} and \ref{cent->lat}.

\subsection{Proof of Theorem \ref{teoasintotica}}\label{pfteoasintotica}

From Fact (I) in the proof of Lemma \ref{cent->lat}, using the notations therein, we know that $\mathcal{C}_{2m+1}$ lies between $\mathcal{H}_{m+1}$ and $\mathcal{H}_{m+2}$, which
means that
\begin{equation}\label{distribeigen}
\lambda_{2m+1}\in [m+1,m+2].
\end{equation}
In turn, this means that any interval of the kind $[m,m+1]$ (for integer $m\ge1$) contains at least an odd eigenvalue
of \eqref{autovsym} (namely, an eigenvalue with an odd label). By the nodal properties of the eigenfunctions stated in Theorem \ref{constant}, odd and even eigenvalues alternate and so, in the interval
$[m,m+2]$, there are at least three eigenvalues, two odd and one even. The thesis follows from the fact that any interval of width 3 contains an interval
of the kind $[m,m+2]$ for integer $m$.

\subsection{Proof of Theorem \ref{autofz2}}\label{pfautofz2}

First, reasoning as in \cite[Lemma 2.2]{holubova}, one can see that any solution of \eqref{autov2} is of class $C^\infty$ on each single span $\overline{I}_-$,
$\overline{I}_0$ and $\overline{I}_+$, but, differently from the fourth order case \eqref{autov0}, \emph{it is not in general $C^1$ on the whole $I$}. Any solution is a linear
combination of the two functions $\cos(\kappa x)$ and $\sin(\kappa x)$ on each span, so that the solutions over $I$ are obtained by extending by symmetry
$$
e(x)=
\left\{
\begin{array}{ll}
e_0(x)\quad & \mbox{if } x \in [0, a\pi] \\
e_a(x)\quad & \mbox{if } x \in [a\pi, \pi],
\end{array}
\right.
$$
where $e_0= H_0 \sin(\kappa x)$, $e_a(x)=K\cos(\kappa x) + H\sin(\kappa x)$ for odd eigenfunctions, while for the even ones it is $e_0= K_0 \cos(\kappa x)$
and $e_a(x)=K\cos(\kappa x) + H\sin(\kappa x)$. Imposing the three vanishing conditions in $a\pi$ and $\pi$, we obtain a $3\times 3$ linear system in the unknowns $H_0, K, H$ (or $K_0, K, H$) and nontrivial solutions are obtained imposing that the determinant
of such system is equal to $0$. All the statements of Theorem \ref{autofz2} are then obtained by mimicking the arguments developed in the proof of
Theorem \ref{symmetriceigen}, with obvious changes.
\endproof
We notice that, after a change of variables $(a, \kappa) \mapsto (\alpha, K)$ similar to the one performed to prove Theorems \ref{Michelle} and \ref{constant}, the eigenvalue problem \eqref{autov2} is transformed into
$$
\int_J e' w' = \mu \int_J e w \qquad \forall w \in W(J),
$$
where $J=(-\alpha \pi, \alpha \pi)$ and
$W(J):=\{u\in H^1_0(J);\, u(\pm \pi)=0\}\,$.
In Figure \ref{leiperboli2}, we depict the eigenvalues curves for such a problem; again, the change of variables seems to simplify the interpretation of the pictures.
\begin{figure}[ht]
\begin{center}
\includegraphics[height=80mm, width=86mm]{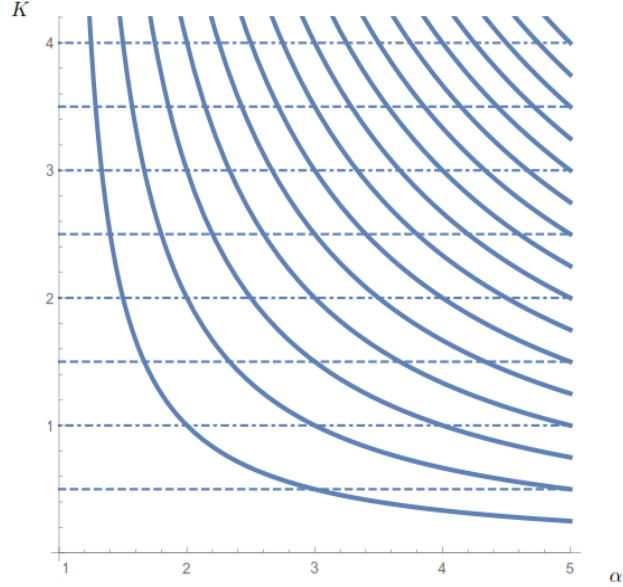}

\end{center}
\caption{A pictorial description of the eigenvalues curves for \eqref{autov2}, in the $(\alpha, K)$-plane.}\label{leiperboli2}
\end{figure}

\subsection{Proof of Theorem \ref{sogliacontinua}}\label{pfsogliacontinua}

Since we consider variable $a\in(0,1)$, throughout this proof we emphasize the dependence of the eigenfunctions and eigenvalues
of \eq{autovsym} on $a$ by denoting them, respectively, by $e_n^a$ and $\lambda_n^a$, for $n=0,...,11$.\par
Assume that $\mathbb{E}_{12}(\bar{a}) < +\infty$ for a certain $\bar{a}\in(0,1)$. By Definition \ref{threshold}, for the choice $a=\bar{a}$
there exist a prevailing mode $j$, a residual mode $k \neq j$, and a time instant $\tau \in (2T_W, T)$ such that, writing
$$U^A_{\bar{a}}(x,t):=\sum_{n=0}^{11}\varphi_n(t)e_n^{\bar{a}}(x),$$
the conditions in Definitions \ref{prevalente} and \ref{unstable} are fulfilled. This means that
\begin{equation}\label{inst1}
\sum_{n=0 \atop n \neq j}^{11}\big[\varphi_n(0)^2 + \dot{\varphi}_n(0)^2\big]<\eta^4(\varphi_j(0)^2 + \dot{\varphi}_j(0)^2),
\quad \Vert \varphi_k \Vert_{L^\infty(0, \tau)} > \eta \Vert \varphi_j \Vert_{L^\infty(0, \tau)},
\quad \frac{\Vert \varphi_k \Vert_{L^\infty(0, \tau)}}{\Vert \varphi_k \Vert_{L^\infty(0, \tau/2)}} > \frac{1}{\eta}.
\end{equation}
The components $\varphi_j$ and $\varphi_k$ solve the two ODEs
\begin{eqnarray*}
&& \!\!\!\!\!\!\!\!\!\! \ddot{\varphi_j}(t) + \mu_j \varphi_j(t)+\gamma_2\Big(\sum_{n\modd 2 \atop = m\modd 2} \varphi_n(t) \varphi_m(t)
\int_I (e_n^{\bar{a}})' (e_m^{\bar{a}})' \Big) \Big(\sum_{l \modd 2 \atop = j \modd 2} \varphi_l(t) \int_I (e_l^{\bar{a}})' (e_j^{\bar{a}})' \Big)\\
+ && \!\!\!\!\!\!\!\!\!\!\gamma_1\Big(\sum_n \mu_n \varphi_n(t)^2\Big) \mu_j \varphi_j(t)+
\gamma_3 \Big(\sum_n \varphi_n(t)^2 \Big) \varphi_j(t) + \int_I f\Big(\sum_n \varphi_n(t) e_n^{\bar{a}} \Big) e_j^{\bar{a}} = 0,
\end{eqnarray*}
and
\begin{eqnarray*}
&& \!\!\!\!\!\!\!\!\!\! \ddot{\varphi_k}(t) + \mu_k \varphi_k(t) + \gamma_2\Big(\sum_{n\modd 2 \atop = m\modd 2} \varphi_n(t) \varphi_m(t)
\int_I (e_n^{\bar{a}})' (e_m^{\bar{a}})' \Big) \Big(\sum_{l \modd 2 \atop = k \modd 2} \varphi_l(t) \int_I (e_l^{\bar{a}})' (e_k^{\bar{a}})' \Big) \\
+ && \!\!\!\!\!\!\!\!\!\! \gamma_1\Big(\sum_n \mu_n \varphi_n(t)^2\Big) \mu_k \varphi_k(t)+
\gamma_3 \Big(\sum_n \varphi_n(t)^2 \Big) \varphi_k(t) + \int_I f\Big(\sum_n \varphi_n(t) e_n^{\bar{a}} \Big) e_k^{\bar{a}} = 0.
\end{eqnarray*}

Consider a sequence $\{a_i\}_i$ such that $a_i\to\bar{a}$ for $i \to +\infty$; then,
\neweq{convergence}
e_n^{a_i}\to e_n^{\bar{a}}\quad\mbox{in }H^2(I)\quad\mbox{for all }n=0,...,11.
\endeq
Indeed, from \eqref{autovsym} and Theorem \ref{Michelle}, it follows that
$$
\Vert e_n^{a_i} \Vert_{V}^2=(\lambda_n^{a_i})^4\to(\lambda_n^{\bar{a}})^4=\Vert e_n^{\bar{a}} \Vert_{V},
$$
after recalling that all the eigenfunctions are $L^2$-normalized. This fact shows that there exists $\bar{e}\in V(I)$ such that
$e_n^{a_i}\rightharpoonup \bar{e}$ in $H^2(I)$, up to a subsequence. Together with the convergence $e_n^{a_i}\to e_n^{\bar{a}}$ in $L^2(I)$
and the convergence of the norms, this proves \eq{convergence}.\par
For every $i \in \mathbb{N}$, consider now system \eq{finitodim} for $a=a_i$ with initial conditions
$$
u^A_i(x,0)=\sum_{n=0}^{11}\varphi_n^i(0) e_n^{a_i}(x), \quad (u^A_i)_t(x, 0)=\sum_{n=0}^{11}\dot{\varphi}_n^i(0) e_n^{a_i}(x),
$$
where $\varphi_n^i(0)$ and $\dot{\varphi}_n^i(0)$ are taken so that the associated total energy is equal to $\mathbb{E}_{12}(\bar{a})$ and so that they fulfill the first condition in \eqref{inst1} (this is possible up to taking $i$ sufficiently large, thanks to the strict inequalities in \eqref{inst1}). Denoting the solution by
$$
u^A_i(x,t)= \sum_{n=0}^{11}\varphi_n^i(t) e_n^{a_i}(x),
$$
the components $\varphi_j^i$ and $\varphi_k^i$ solve the system
\begin{eqnarray*}
&& \!\!\!\!\!\!\!\!\!\!\!\! \ddot{\varphi_j^i}(t) + \mu_j^i \varphi_j^i(t) +\gamma_2\Big(\sum_{n\modd 2 \atop = m\modd 2} \varphi_n^i(t) \varphi_m^i(t) \int_I (e_n^{a_i})' (e_m^{a_i})' \Big) \Big(\sum_{l \modd 2 \atop = j \modd 2} \varphi_l^i(t) \int_I (e_l^{a_i})' (e_j^{a_i})' \Big) \\
+ && \!\!\!\!\!\!\!\!\!\!   \gamma_1\Big(\sum_n \mu_n^i \varphi_n^i(t)^2\Big) \mu_j^i \varphi_j^i(t) +
\gamma_3 \Big(\sum_n \varphi_n^i(t)^2 \Big) \varphi_j^i(t) + \int_I f\Big(\sum_n \varphi_n^i(t) e_n^{a_i} \Big) e_j^{a_i} = 0,
\end{eqnarray*}
and
\begin{eqnarray*}
&&\!\!\!\!\!\!\!\!\!\!\!\! \ddot{\varphi_k^i}(t) + \mu_k^i \varphi_k^i(t) + \gamma_2\Big(\sum_{n\modd 2 \atop = m\modd 2} \varphi_n^i(t) \varphi_m^i(t) \int_I (e_n^{a_i})' (e_m^{a_i})' \Big) \Big(\sum_{l \modd 2 \atop = k \modd 2} \varphi_l^i(t) \int_I (e_l^{a_i})' (e_k^{a_i})' \Big) \\
+ &&\!\!\!\!\!\!\!\!\!\!  \gamma_1\Big(\sum_n \mu_n^i \varphi_n^i(t)^2\Big) \mu_k^i \varphi_k^i(t) +\gamma_3 \Big(\sum_n \varphi_n^i(t)^2 \Big)
\varphi_k^i(t) + \int_I f\Big(\sum_n \varphi_n^i(t) e_n^{a_i} \Big) e_k^{a_i} = 0.
\end{eqnarray*}

From \eq{convergence} and the embedding $H^2(I)\subset C^1(\overline{I})$ we infer the $C^1$-convergence of $e_n^{a_i}$ to $e_n^{\bar{a}}$, so that also all the integral terms converge.
Then, from the Lipschitz continuity of $f$ and the continuous dependence, we infer that $\varphi_n^i \to \varphi_n$ pointwise for every $n$.
Moreover, since all the sequences $\{\varphi_n^i\}_i$ ($n=0,...,11$) are bounded in $H^2$, we infer that $\varphi_n^i \to \varphi_n$ in $C^1$.
Hence, there exists $i_0$ such that, for every $i > i_0$, the components $\varphi_j^i$ and $\varphi_k^i$ satisfy \eqref{inst1}.
Since the energy is preserved, we thus have that
$$
\mathbb{E}_{12}(\bar{a}) \geq \limsup_{i \to +\infty} \mathbb{E}_{12}(a_i),
$$
namely $a\mapsto\mathbb{E}_{12}(a)$ is upper-semicontinuous in $a=\bar{a}$.\par
Assume now by contradiction that
$\liminf_{i \to +\infty}\mathbb{E}_{12}(a_i) < \mathbb{E}_{12}(\bar{a})$ and let $j_i, k_i$ be the prevailing and the residual modes fulfilling
Definition \ref{unstable} for $a=a_i$. Let us restrict to a subsequence $\{a_{i_l}\}_l$ such that $\lim_{l \to +\infty} \mathbb{E}_{12}(a_{i_l})= \liminf_{i \to +\infty}\mathbb{E}_{12}(a_i)<\mathbb{E}_{12}(\bar{a})<+\infty$, so that the sequence $\{\mathbb{E}_{12}(a_{i_l})\}_l$ is bounded.
Since the number of considered modes is finite, passing to a further subsequence $i_{l_m}$ it is possible to extract, from the sequence of couples $(j_{i_l}, k_{i_l})$, a couple of modes $(j, k)$ fulfilling the conditions for instability for every $a=a_{i_{l_m}}$. By continuous dependence, such a couple converges to a couple $(\varphi_j, \varphi_k)$ fulfilling the same requirements for $a=\bar{a}$ and the contradiction is reached. This proves lower semicontinuity and completes the proof.

\subsection{Proof of Theorem \ref{toytheo}}\label{pftoytheo}

Since the proof may be obtained by combining several arguments from \cite{gaga}, we only briefly sketch it. We set
$$\Phi_\lambda(t):=\lambda^2W_\lambda\left(\frac{t}{\lambda^2}\right)$$
and we notice that if $W_\lambda$ satisfies \eq{ODED}, then $\Phi_\lambda$ solves
\neweq{duffingPhi}
\ddot{\Phi}_\lambda(t)+\Phi_\lambda(t)+\Phi_\lambda(t)^3=0,\qquad\Phi_\lambda(0)=\lambda^2\delta,\qquad\dot{\Phi}_\lambda(0)=0.
\endeq
With these transformations, also \eq{hill3} changes and we see that the $\lambda$-nonlinear-mode of \eqref{toybeam} of amplitude $\delta$ is linearly stable with respect to the
$\rho$-nonlinear-mode $W_\rho$ if and only if $\xi\equiv0$ is a stable solution of
\neweq{hillscalata}
\ddot{\xi}(t)+\gamma^2\Big(1+\Phi_\lambda(t)^2\Big)\xi(t)=0.
\endeq

Then we recall a criterion due to Burdina \cite{burdina} (see also \cite[Test 3, \S 3, Chapter VIII]{yakubovich}),
which yields a sufficient condition for the stability of some Hill equations.

\begin{lemma}\label{BurdCrit}
Let $T>0$ and let $p$ be a continuous, $T$-periodic and strictly positive function having a unique maximum point and a unique minimum point
in $[0,T)$. If there exists $k\in\N$ such that
\begin{equation}\label{condizioneBurdina}
k\pi+\frac12\log\frac{\max p}{\min p}<\int_0^T\sqrt{p(t)}\, dt<(k+1)\pi-\frac12\log\frac{\max p}{\min p},
\end{equation}
then the trivial solution of the Hill equation $\ddot{\xi}(t)+p(t)\xi(t)=0$ is stable.
\end{lemma}

We apply Lemma \ref{BurdCrit} to the Hill equation \eq{hillscalata}, that is, with $p(t)=\gamma^2\big(1+\Phi_\lambda(t)^2\big)$ so that
$$
\log\frac{\max p}{\min p}=\log\big(1+\delta^2\lambda^4\big).
$$
Indeed, from Burgreen~\cite{burg} we know that
$$
\Phi_\lambda(t)=\delta\lambda^2\, {\rm cn}\Bigg[t\, \sqrt{1+\delta^2\lambda^4},\frac{\delta\lambda^2}{\sqrt{2(1+\delta^2\lambda^4)}}\Bigg],
$$
where cn is the Jacobi cosine. The period of $\Phi_\lambda$ can be computed via \eqref{TW}, so that the period of $\Phi_\lambda^2$ is given by
\begin{equation}\label{Tomega}
T_\lambda(\delta)=2\sqrt2 \int_0^{\pi/2}\frac{d\phi}{\sqrt{2+\delta^2\lambda^4(1+\sin^2\phi)}}
\end{equation}
and $\lim_{\delta\to 0}T_\lambda(\delta)=\pi$.
\par
Moreover, from the energy conservation in \eq{duffingPhi} we see that
$$2\dot\Phi_\lambda(t)^2=\Big(\delta^2\lambda^4-\Phi_\lambda(t)^2\Big)\Big(2+\delta^2\lambda^4+\Phi_\lambda(t)^2\Big)\, ;$$
therefore, thanks to symmetries and with the change of variables $\Phi_\lambda(t)=\delta\lambda^2\sin\theta$, we obtain
$$
\int_0^{T_\lambda(\delta)}\sqrt{1+\Phi_\lambda(t)^2}\, dt=2\int_0^{T_\lambda(\delta)/2}\sqrt{1+\Phi_\lambda(t)^2}\, dt=\Lambda_\delta,
$$
see \cite{gaga} for the details; here, $\Lambda_\delta$ is as in \eq{gammaLambda}.
Item $(i)_4$ then follows by applying Lemma \ref{BurdCrit} to \eq{hillscalata}.\par
By using the very same method as in \cite[Theorem 3.1]{cazw} (see also \cite[Theorem 8]{babefega}), we may obtain the following statement.

\begin{lemma}\label{known}
Let $I_S$ and $I_U$ as in \eqref{strisceinfinito}. For every $\gamma>0$ there exists $\bar{\delta}_\gamma>0$ such that,
for all $\delta>\bar{\delta}_\gamma$:\par\noindent
$(a)$ if $\gamma^2\in I_U$, then the trivial solution of equation~\eqref{hillscalata} is unstable;\par\noindent
$(b)$ if $\gamma^2\in I_S$, then the trivial solution of equation~\eqref{hillscalata} is stable.
\end{lemma}

Lemma \ref{known} proves Items $(i)_3$ and $(ii)_2$ of Theorem \ref{toytheo}.\par
When $\gamma=1$ we see that $\xi=\Phi_\lambda$ solves \eq{hillscalata}; even if the case $\lambda=\rho$ has no meaning for \eq{cw3},
this tells us that $\gamma=1$ is part of the boundary of the resonance tongue $U_1$ emanating from the point $(\delta\lambda^2,\gamma)=(0,1)$:
this proves Item $(i)_1$. Moreover, combined with Lemma \ref{known} and with the asymptotic behavior
$$
\Lambda_\delta = \pi-\frac\pi8 \, \delta^2\lambda^4+O(\delta^4)\qquad\mbox{as }\delta\to0,
$$
Items $(i)_1$ and $(i)_4$ show that the resonant tongue $U_1$ of \eqref{hill3} emanating from $(\delta\lambda^2,\gamma)=(0,1)$ is given by
$$
U_1 =\big\{(\delta\lambda^2,\gamma)\in\mathbb{R}_+^2:\, 1<\gamma<\psi_\lambda(\delta)\big\}
$$
where $\psi_\lambda$ satisfies the properties stated in Item $(ii)_1$.\par
Finally, note that Item $(i)_4$ implies that if $k\in\N$ $(k\ge2)$ and $U_k$ is the resonant tongue of \eqref{hill3} emanating
from $(\delta\lambda^2,\gamma)=(0,k)$ and if $(\delta\lambda^2,\gamma)\in U_k$, then necessarily
\eqref{asintoticalingua} holds. This proves Item $(i)_2$
and completes the proof of Theorem \ref{toytheo}.

\subsection{Proof of Theorem \ref{energiacontinua}}\label{pfenergiacontinua}

To prove the first part of the statement, it suffices to show that, for every $a\in(0,1)$, at least one of the admissible ratios $\rho^4(a)/\lambda^4(a)$ for
the computation of $\mathbb{E}_{12}^\ell(a)$ belongs to the first instability interval $(1,3)$ of $I_U$.
If $\lambda_{11}^2(a)/\lambda_{10}^2(a) < \sqrt{3}$, we are done. Otherwise, $\lambda_{10}^2(a)\le\lambda_{11}^2(a)/\sqrt{3}$ and from \eqref{distribeigen}
we know that
$$5 \leq \lambda_9(a) \leq 6 \leq \lambda_{11}(a) \leq 7,
$$
implying $\lambda_{10}^2(a)\leq49/\sqrt3 $ and, in turn,
$$\frac{\lambda_{10}^2(a)}{\lambda_9^2(a)}\le\frac{49/\sqrt3 }{25}<\sqrt3 .$$
Therefore, $0<\mathbb{E}_{12}^\ell(a)<+\infty$ for every $a\in(0,1)$.\par
For the proof of the lower semicontinuity, we start by noticing that the couples of eigenvalues of \eq{autovsym} candidates to minimize $E$, defined in \eqref{ecritica}, have
to be sought among consecutive eigenvalues. So, fix some $\overline{a}\in(0,1)$ and a couple of consecutive eigenvalues $\lambda_n(a)$ and $\lambda_{n+1}(a)$
with $a\to\overline{a}$. Three cases may occur.\par
$\bullet$ First case: $\lambda_{n+1}^4(\overline{a})/\lambda_n^4(\overline{a})\in I_S$ (open interval).\par
Then $\lambda_{n+1}^4(a)/\lambda_n^4(a)\in I_S$ for $a$ sufficiently close to $\overline{a}$ in view of Theorem \ref{Michelle}. Next, we need to use tools from
the classical Floquet theory, see e.g.\ \cite{yakubovich}. Following Definition \ref{defstabb}, the values of $D(\lambda_n(a),\lambda_{n+1}(a))$ and
$E(\lambda_n(a),\lambda_{n+1}(a))$ depend on the stability of \eq{hillscalata}. So, consider the two solutions $\xi_1$ and $\xi_2$ of \eq{hillscalata} satisfying,
respectively, the initial conditions
$$
\xi_1(0)=1,\quad\dot{\xi}_1(0)=0,\qquad\xi_2(0)=0,\quad\dot{\xi}_2(0)=1.
$$

From \eq{Tomega} we know that the period of $\Phi_\lambda^2$ is equal to $T_\lambda(\delta)$,
so that $\delta\mapsto T_\lambda(\delta)$ is a continuous function. Then, the trace of the monodromy matrix associated with \eq{hillscalata} is given by
$$\xi_1\big(T_\lambda(\delta)\big)+\dot{\xi}_2\big(T_\lambda(\delta)\big),$$
hence $D(\lambda_n(a),\lambda_{n+1}(a))$ in \eqref{D} has the following equivalent characterization:
$$D(\lambda_n(a),\lambda_{n+1}(a))=\inf\big\{d>0;\, \big|\xi_1\big(T_\lambda(\delta)\big)+\dot{\xi}_2\big(T_\lambda(\delta)\big)\big|>2\mbox{ whenever }W_{\lambda_n}(0)=\delta>d\big\}.$$

Since the resonant lines are continuous \cite{yakubovich}, we have that
$D(\lambda_n(a),\lambda_{n+1}(a))\to D(\lambda_n(\overline{a}),\lambda_{n+1}(\overline{a}))$. By definition of $E$, this shows that
the map $a\mapsto E(\lambda_n(a),\lambda_{n+1}(a))$ is continuous in all the values of $a$ for which $\lambda_{n+1}^4(a)/\lambda_n^4(a)\in I_S$.\par
$\bullet$ Second case: $\lambda_{n+1}^4(\overline{a})/\lambda_n^4(\overline{a})\in I_U$ (open interval).\par
This case is simple, since $E(\lambda_n(a),\lambda_{n+1}(a))\equiv\infty$ in a neighborhood of $\overline{a}$, which means that the couple
$(\lambda_n(a),\lambda_{n+1}(a))$ does not compete to achieve the minimum in $E$.\par
$\bullet$ Third case: $\lambda_{n+1}^4(\overline{a})/\lambda_n^4(\overline{a})\not\in I_S\cup I_U$.\par
From \eq{strisceinfinito} we know that there exists an integer $k\ge2$ such that
$\lambda_{n+1}^4(\overline{a})/\lambda_n^4(\overline{a})=k(k+1)/2$. By Theorem \ref{Michelle} we then infer that
$$\lim_{a\to\overline{a}}\ \frac{\lambda_{n+1}^4(a)}{\lambda_n^4(a)}\ =\ \frac{k(k+1)}{2}.$$
Whence, from the continuity of the resonant lines \cite{yakubovich} and using \eq{ipotesi}, we have that
$$\liminf_{a\to\overline{a}}E(\lambda_n(a),\lambda_{n+1}(a))=E(\lambda_n(\overline{a}),\lambda_{n+1}(\overline{a})).$$
This shows that the map $a\mapsto E(\lambda_n(a),\lambda_{n+1}(a))$ is lower semicontinuous in this third case. Incidentally, we observe that \eq{ipotesi}
is needed precisely for the third case.\par
By combining the analysis for the three cases, we infer that
$$\mbox{all the maps }a\mapsto E(\lambda_n(a),\lambda_{n+1}(a))\mbox{ are lower semicontinuous in every point }\overline{a}\in(0,1).$$
Since $\mathbb{E}_{12}^\ell(a)$ is the minimum of a finite number of lower semicontinuous functions, it is itself lower semicontinuous.

\subsection{Proof of Theorem \ref{theotors}}\label{pftheotors}

The linear instability for large $\delta$ is a consequence of the following result, essentially due to Cazenave-Weissler \cite{cazw}.

\begin{lemma}\label{asymptotics}
Assume that $A_{\lambda,\kappa}\neq 0$. Then the $\lambda$-longitudinal-mode of \eqref{nohang} of amplitude $\delta$ is linearly unstable with respect to
the $\kappa$-torsional-mode provided that $\delta$ is sufficiently large.
\end{lemma}
\begin{proof} For all $\eps>0$ we define $Z_\eps$ and $\xi_\eps$ by
$$
W_\lambda(t)=\frac{\lambda^2}{\sqrt\eps }\, Z_\eps\left(\frac{\lambda^2}{\sqrt\eps }\, t\right),\qquad
\xi(t)=\xi_\eps\left(\frac{\lambda^2}{\sqrt\eps }\, t\right).
$$
If $W_\lambda$ satisfies \eq{ODED2}, then $Z_\eps$ solves the equation
\neweq{duffingeps}
\ddot{Z}_\eps(t)+\eps Z_\eps(t)+Z_\eps(t)^3=0
\endeq
and, according to Definition \ref{defstabtors}, the linear stability of \eq{systemtorsional} depends on the stability of the following
Hill equation:
\neweq{hilleps}
\ddot{\xi}_\eps(t)+(1+2A_{\lambda,\kappa}^2)\Big(\frac{\kappa^2\, \eps}{(1+2A_{\lambda,\kappa}^2)\lambda^4}+Z_\eps(t)^2\Big)\xi_\eps(t)=0.
\endeq
By letting $\eps\to0$, we see that \eq{duffingeps} and \eq{hilleps} ``converge'' respectively to the limit problems
\neweq{limitproblems}
\ddot{Z}(t)+Z(t)^3=0,\qquad \ddot{\xi}(t)+(1+2A_{\lambda,\kappa}^2)Z(t)^2\, \xi(t)=0.
\endeq
These limit equations are precisely (3.1) and (3.2) in \cite{cazw} with $\gamma=1+2A_{\lambda,\kappa}^2 \in (1, 3)$ in view of \eqref{Alkmin1}.
Therefore, \cite[Theorem 3.1]{cazw} applies and the statement follows.\end{proof}

Then we turn to the delicate part of Theorem \ref{theotors}, namely the stability result for small $\delta$.
We apply Lemma \ref{BurdCrit} to the Hill equation \eq{hilltors}: we take $p(t)=\kappa^2+(1+2A_{\lambda,\kappa}^2)W_\lambda(t)^2$ so that
$$
\log\frac{\max p}{\min p}=\log\left(1+\frac{(1+2A_{\lambda,\kappa}^2)\delta^2}{\kappa^2}\right).
$$
From \eq{explicitsol}-\eq{TW} (with $\gamma_1=0$ and $\gamma_3=1$) we know that
$$
W_\lambda(t)=\delta\, {\rm cn}\Bigg[t\, \sqrt{\lambda^4+\delta^2},\frac{\delta}{\sqrt{2(\lambda^4+\delta^2)}}\Bigg]
$$
and that the period of $W_\lambda^2$ is given by \eq{Tld}. Moreover, from the energy conservation in \eq{ODED2} we see that
$$2\dot{W}_\lambda(t)^2=\Big(\delta^2-W_\lambda(t)^2\Big)\Big(2\lambda^4+\delta^2+W_\lambda(t)^2\Big)\, ;$$
therefore, thanks to symmetries and with the change of variables $W_\lambda(t)=\delta\sin\phi$, we obtain
\begin{eqnarray*}
\int_0^{T_\lambda(\delta)}\sqrt{\kappa^2+(1+2A_{\lambda,\kappa}^2)W_\lambda(t)^2}\, dt &=& 2\int_0^{T_\lambda(\delta)/2}\sqrt{\kappa^2+(1+2A_{\lambda,\kappa}^2)W_\lambda(t)^2}\, dt\\
&=& 2\sqrt{2}\int_0^{\pi/2}\sqrt{\frac{\kappa^2+(1+2A_{\lambda,\kappa}^2)\delta^2\sin^2\phi}{2\lambda^4+\delta^2+\delta^2\sin^2\phi}}\, d\phi.
\end{eqnarray*}

Lemma \ref{BurdCrit} then states that the trivial solution of \eq{hilltors} is stable provided that there exists $m\in\N$ such that
\neweq{burdina22}
m\pi+\frac12\log\left(\!1+\tfrac{(1+2\Alk^2)\delta^2}{\kappa^2}\!\right)<
2\sqrt{2}\int_0^{\pi/2}\!\!\sqrt{\tfrac{\kappa^2+(1+2\Alk^2)\delta^2\sin^2\phi}{2\lambda^4+\delta^2+\delta^2\sin^2\phi}}\, d\phi
<(m+1)\pi-\frac12\log\left(\!1+\tfrac{(1+2\Alk^2)\delta^2}{\kappa^2}\!\right).
\endeq

Condition \eq{burdina22} is by far less clear than the corresponding condition in Theorem \ref{toytheo}. As $\delta\to0$ it becomes
$$
m\pi+(1+2A_{\lambda,\kappa}^2)\frac{\delta^2}{2\kappa^2}<\pi\, \frac{\kappa}{\lambda^2}+\frac{\delta^2}{2\kappa^2}
\left(\frac{\pi(1+2A_{\lambda,\kappa}^2)}{2}\, \frac{\kappa}{\lambda^2}-\frac{3\pi}{4}\left(\frac{\kappa}{\lambda^2}\right)^3\right)
<(m+1)\pi-(1+2A_{\lambda,\kappa}^2)\frac{\delta^2}{2\kappa^2}.
$$
This is a sufficient condition for linear stability for small $\delta$ and enables us to obtain the following necessary condition for
linear instability for small $\delta$: there exists $m\in\N$ such that
$$
\frac{\delta^2}{2\kappa^2}\, F_-\left(\frac{\kappa}{\lambda^2}\right)<
\pi\left(\frac{\kappa}{\lambda^2}-m\right)<
\frac{\delta^2}{2\kappa^2}\, F_+\left(\frac{\kappa}{\lambda^2}\right),
$$
where
$$
F_\pm(s):=\frac{3\pi}{4}\, s^3-\frac{\pi(1+2\Alk^2)}{2}\, s\pm(1+2\Alk^2)\qquad(s\ge0).
$$
Simple calculus arguments show that:\par
-- there exists a unique $\bar{s}>0$ such that $F_-(s)<0$ if $s\in[0,\bar{s})$, $F_-(\bar{s})=0$, $F_-(s)>0$ if $s\in(\bar{s},\infty)$; we notice that from \eq{Alkmin1} and since $\Alk\neq0$, we have
$$
\frac{3\pi}{4}\, s^3-\frac{3\pi}{2}\, s-3=:G^1(s)<F_-(s)<G^0(s):=\frac{3\pi}{4}\, s^3-\frac{\pi}{2}\, s-1\qquad\forall s>0.
$$
Both $G^0$ and $G^1$ admit a unique zero. Since $G^1(2)=3\pi-3>0$ and $G^0(1)=-1+\pi/4<0$, we infer that $\bar{s}\in(1,2)$ for any value of $\Alk$. \par
-- since $\Alk^2<1<\frac12 (\frac{81}{2\pi^2}-1)$, we have that $F_+(s)>0$ for all $s\ge0$.
\begin{figure}[h!]
\begin{center}
\includegraphics[scale=0.5]{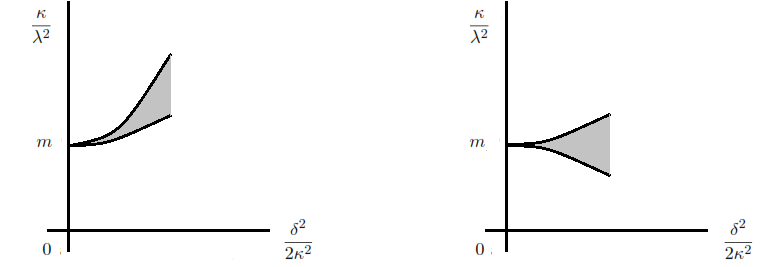}
\caption{Local bounds for the resonance tongue $U_m$ (gray) emanating from $(0,m)$.}\label{resline}
\end{center}
\end{figure}

Therefore, the resonant tongue $U_m$ emanating from the point $(\frac{\delta^2}{2\kappa^2},\frac{\kappa}{\lambda^2})=(0,m)$
necessarily lies (as $\delta\to0$) in one of the gray regions depicted in Figure \ref{resline}, which occur if $\kappa/\lambda^2\in\mathbb{N}$ and, respectively,
\neweq{verybad}
\frac{\kappa}{\lambda^2} >\bar{s} \quad \textrm{ or } \quad \frac{\kappa}{\lambda^2}<\bar{s}.
\endeq
Since $\bar{s} \in (1, 2)$, the only case
where the second situation in \eq{verybad} may occur is when $\kappa=\lambda^2$, which is excluded by the assumptions. Thus, only the case of the left picture in Figure \ref{resline} is possible, proving the second statement of Theorem \ref{theotors}.
We have so proved that the $\lambda$-longitudinal-mode of \eqref{nohang} of amplitude $\delta$ is linearly stable with respect to the $\kappa$-torsional-mode
whenever $\delta$ is sufficiently small. We now make quantitative this smallness requirement: we get rid of the ``mysterious term'' $\Alk$ and we obtain uniform
bounds for $\delta$ not depending on it. We first notice that, since $1+2\Alk^2<3$ in view of \eq{Alkmin1}, the right inequality in \eq{burdina22} is
certainly satisfied provided that \eq{burdinam} holds, where we have emphasized the relevant variables $\kappa/\lambda^2$ and $\delta/\kappa$. For the left
inequality in \eq{burdina22} we use both the bounds $1<1+2\Alk^2<3$ and we see that it is certainly satisfied provided that \eq{burdinam2} holds.

\subsection{Proof of Theorem \ref{Balphanot0}}\label{pfBalphanot0}

If $e_\lambda$ is odd, the statement follows easily from the fact that, if $\eta_\kappa$ is an eigenfunction of \eqref{autov2} as in Theorem \ref{autofz2},
then $\eta_\kappa^2$ is even and the integrand defining $B_\varsigma$ is odd whenever $e_\lambda$ is odd.\par
The difficult part is the second item, which we now prove.
According to Theorem \ref{symmetriceigen}, if $e_\lambda$ is even we need to distinguish two cases.\par\smallskip
$\bullet$ {\bf Case (I):} $e_\lambda(x)=\mathbf{E}_\lambda(x)=\cos(\lambda x)$.\par
In this case, the following calculus formula turns out to be useful:
$$
\lambda\neq2\kappa\ \Longrightarrow\ \int_0^{a\pi}\cos(\lambda x)\cos(2\kappa x)\, dx=
\frac{\lambda\sin(\lambda a\pi)\cos(2\kappa a\pi)-2\kappa\cos(\lambda a\pi)\sin(2\kappa a\pi)}{\lambda^2-4\kappa^2}.
$$
In particular, since $\lambda a - 1/2 \in \mathbb{N}$ (see Theorem \ref{symmetriceigen}), this yields $\cos(\lambda a\pi)=0$ and hence
\neweq{bypartslk}
\lambda\neq2\kappa\ \Longrightarrow\ \int_0^{a\pi}\cos(\lambda x)\cos(2\kappa x)\, dx=
\frac{\lambda\sin(\lambda a\pi)\cos(2\kappa a\pi)}{\lambda^2-4\kappa^2}.
\endeq

Since $B_\varsigma(w)=B_1(\varsigma w)/\varsigma$, it suffices to prove that $B_1(w)\not\equiv0$ in any neighborhood of $w=0$.
And since $B_1(w)=w\int_Ie_\lambda\eta_\kappa^2+o(w^2)$ as $w\to0$, in order to prove this fact it suffices to show that
\neweq{suffclaim}
J:=\int_Ie_\lambda(x)\eta_\kappa(x)^2\, dx=2\int_0^{a\pi}e_\lambda(x)\eta_\kappa(x)^2\, dx\neq0.
\endeq

Theorem \ref{autofz2} suggests to distinguish two further subcases:
$$
\textnormal{(Ia)}\quad \eta_\kappa(x)=\chi_0(x)\sin(\kappa x)\qquad \textnormal{(Ib)}\quad \eta_\kappa(x)=\chi_0(x)\cos(\kappa x).
$$
In Case (Ia) we have $\sin(\kappa a\pi)=0$ and therefore
\neweq{tech1}
\kappa a\in\mathbb{N},\qquad\cos(2\kappa a\pi)=1.
\endeq
From Theorem \ref{symmetriceigen} and \eq{tech1} we know that there exist two integers $m$ and $n$ such that $\lambda=(2m+1)/2a$ and $\kappa=n/a$. Therefore,
$$
4\kappa^2-\lambda^2=\frac{16n^2-4m^2-4m-1}{4a^2}\neq0
$$
since the first three terms in the numerator are multiples of $4$ while -1 is not. This shows that \eq{bypartslk} applies and, by \eq{bypartslk} and
using again \eq{tech1}, we obtain
$$
J\!=\!2\int_0^{a\pi}\!\!\cos(\lambda x)\sin^2(\kappa x)dx\!=\!\int_0^{a\pi}\!\!\cos(\lambda x)[1-\cos(2\kappa x)]dx\!=\!\frac{\sin(\lambda a\pi)}{\lambda}
-\frac{\lambda\sin(\lambda a\pi)}{\lambda^2-4\kappa^2}\!=\!\frac{4\kappa^2\sin(\lambda a\pi)}{\lambda(4\kappa^2-\lambda^2)}.
$$
Since $\sin(\lambda a\pi)=\pm1$, this proves \eq{suffclaim} in Case (Ia).\par
In Case (Ib) we have $\cos(\kappa a\pi)=0$ and therefore
\neweq{tech2}
\kappa a-\frac12 \in\mathbb{N},\qquad\cos(2\kappa a\pi)=-1.
\endeq
From Theorem \ref{symmetriceigen} and \eq{tech2} we know that there exist two integers $m$ and $n$ such that $\lambda=(2m+1)/2a$ and $\kappa=(2n+1)/2a$.
Therefore,
$$
4\kappa^2-\lambda^2=\frac{16n^2+16n-4m^2-4m+3}{4a^2}\neq0
$$
since the first four terms in the numerator are multiples of $4$ while 3 is not. This shows that \eq{bypartslk} applies and, by \eq{bypartslk} and
using again \eq{tech2}, we obtain
$$
J\!=\!2\int_0^{a\pi}\!\!\cos(\lambda x)\cos^2(\kappa x)dx\!=\!\int_0^{a\pi}\!\!\cos(\lambda x)[1+\cos(2\kappa x)]dx\!=\!\frac{\sin(\lambda a\pi)}{\lambda}
-\frac{\lambda\sin(\lambda a\pi)}{\lambda^2-4\kappa^2}\!=\!\frac{4\kappa^2\sin(\lambda a\pi)}{\lambda(4\kappa^2-\lambda^2)}.
$$
Since $\sin(\lambda a\pi)=\pm1$, this proves \eq{suffclaim} also in Case (Ib).\par
This completes the proof of Theorem \ref{Balphanot0} for $e_\lambda$ even in Case (I).\par\smallskip
$\bullet$ {\bf Case (II):} $e_\lambda(x)=\mathscr{E}_\lambda(x)=C[\cos(\lambda x)-\frac{\cos(\lambda a\pi)}{\cosh(\lambda a\pi)}\cosh(\lambda x)]$
if $x\in[0,a\pi]$.\par
In this case, we start with a technical statement that probably holds in a more general form, but which is strong enough for our purposes.

\begin{lemma}\label{calculus1}
Assume that $0<|A|<1$ and consider the function $f(t)=\cos t+A\cosh t$. Then, $f$ admits a finite number of critical points (local extrema)
in $[0,\infty)$, say $\{t_1,...,t_m\}$ for some integer $m\ge1$. Moreover, $f$ cannot have a maximum $t_i$ at positive level and a minimum $t_j$ at negative level such that $|f(t_j)|=f(t_i)$.
%
\end{lemma}
\begin{proof}
The first statement follows by noticing that $f'(t)=-\sin t+A\sinh t$ so that $f'(0)=0$ (proving $m\ge1$) and $f'$ has eventually the same sign as $A$.
As for the second statement, we first claim that
\begin{equation}\label{primocl}
\textnormal{if } |f(t_i)|=|f(t_j)| \textnormal{ for some } i\neq j, \textnormal{ then}\,\cos t_i\cos t_j>0.
\end{equation}
To see this, put together the three facts
$$
f(t_i)^2=f(t_j)^2,\quad f'(t_i)=0\Leftrightarrow \sin t_i=A\sinh t_i,\quad f'(t_j)=0\Leftrightarrow \sin t_j=A\sinh t_j,
$$
to obtain $\cos t_i\cosh t_i=\cos t_j\cosh t_j$, which shows that either $\cos t_i=\cos t_j=0$ or $\cos t_i\cos t_j>0$. The claim follows if we
exclude the first possibility; by contradiction, if it were true then $|\sin t_i|=|\sin t_j|=1$ and, since $f'(t_i)=f'(t_j)=0$,
this would imply that $\sinh t_i=\sinh t_j$, contradicting $t_i\neq t_j$.\par
Assume now by contradiction that $f$ has a maximum $t_i$ at positive level and a minimum $t_j$ at negative level satisfying $|f(t_j)|=f(t_i)$. Then, from $f(t_i)f(t_j)<0$ we infer that
\neweq{primastretta}
0>\cos t_i\cos t_j+A^2\cosh t_i\cosh t_j+A\cos t_i\cosh t_j+A\cos t_j\cosh t_i>A(\cos t_i\cosh t_j+\cos t_j\cosh t_i),
\endeq
where we used \eqref{primocl}. On the other hand, from $f''(t_i)f''(t_j)\le0$ and using again \eqref{primocl}, we infer that
$$
0\ge\cos t_i\cos t_j+A^2\cosh t_i\cosh t_j-A\cos t_i\cosh t_j-A\cos t_j\cosh t_i>-A(\cos t_i\cosh t_j+\cos t_j\cosh t_i),
$$
which contradicts \eq{primastretta}, concluding the proof.
\end{proof}
We now continue the analysis of Case (II). Consider the function
$$
A(w):=\frac{B_1(w)}{w}=\int_0^{a\pi}\frac{e_\lambda(x)\eta_\kappa(x)^2}{\sqrt{1+w^2e_\lambda(x)^2}}\, dx\quad\forall w\neq0,\qquad
A(0)=\int_0^{a\pi} e_\lambda(x)\eta_\kappa(x)^2\, dx,
$$
so that Theorem \ref{Balphanot0} will be proved if we show that $A(w)\not\equiv0$ in any neighborhood of $w=0$.
Assume by contradiction this to be false, namely
\neweq{notclaimAn}
A(w)\equiv0\mbox{ in some neighborhood $U$ of $w=0$}\quad\Longrightarrow\quad A^{(k)}(w)\equiv0\mbox{ in }U\ \forall k\in\mathbb{N},
\endeq
where $A^{(k)}(w)$ denotes the $k$-th derivative of $A$ with respect to $w$. We have
$$
A'(w)=-w\int_0^{a\pi}\frac{e_\lambda(x)^3\eta_\kappa(x)^2}{[1+w^2e_\lambda(x)^2]^{3/2}}\, dx
$$
and by \eq{notclaimAn} we deduce that
$$
A_1(w):=\int_0^{a\pi}\frac{e_\lambda(x)^3\eta_\kappa(x)^2}{[1+w^2e_\lambda(x)^2]^{3/2}}\, dx\equiv0\mbox{ in }U,\qquad
A_1(0)=\int_0^{a\pi}e_\lambda(x)^3\eta_\kappa(x)^2\, dx=0.
$$
In turn, by differentiating $A_1(w)$, we deduce that
$$
A_2(w):=\int_0^{a\pi}\frac{e_\lambda(x)^5\eta_\kappa(x)^2}{[1+w^2e_\lambda(x)^2]^{5/2}}\, dx\equiv0\mbox{ in }U,\qquad
A_2(0)=\int_0^{a\pi}e_\lambda(x)^5\eta_\kappa(x)^2\, dx=0.
$$
By iterating this procedure, we obtain that
\neweq{tuttenulle}
\int_0^{a\pi}e_\lambda(x)^{2k+1}\eta_\kappa(x)^2\, dx=0\qquad\forall k\in\mathbb{N}.
\endeq

By Lemma \ref{calculus1} we may find $K\neq0$ such that $g_\lambda(x):=Ke_\lambda(x)$ satisfies
$$
\exists\, \overline{x}\in[0,a\pi)\quad\mbox{s.t.}\quad g_\lambda(\overline{x})=\|g\|_{L^\infty(0,a\pi)}>1\ ,\qquad g_\lambda(x)\ge-1\ \forall x\in[0,a\pi).
$$
Lemma \ref{calculus1} does not specify if such $\overline{x}$ is unique, nevertheless we may consider the nonempty open set $H$ where $g_\lambda(x)>1$:
this may be an interval (in case of uniqueness of $\overline{x}$) or the union of a finite number of intervals (one for each $\overline{x}$).
From \eq{tuttenulle} we infer that
$$
\int_0^{a\pi}g_\lambda(x)^{2k+1}\eta_\kappa(x)^2\, dx=0\qquad\forall k\in\mathbb{N}.
$$
But, denoting $H_0=(0,a\pi)\setminus H$, we have that
$$
\int_0^{a\pi}g_\lambda(x)^{2k+1}\eta_\kappa(x)^2\, dx=\int_{H_0}g_\lambda(x)^{2k+1}\eta_\kappa(x)^2\, dx
+\int_{H}g_\lambda(x)^{2k+1}\eta_\kappa(x)^2\, dx=0\qquad\forall k\in\mathbb{N}.
$$
By letting $k\to\infty$, the first term converges to 0 by the Lebesgue Theorem (recall that $|g_\lambda(x)|<1$ a.e.\ in $H_0$) while the second term
diverges to $+\infty$ (recall that $g_\lambda(x)>1$ in $H$). This gives a contradiction.\par
Therefore, \eq{notclaimAn} does not hold and this proves Theorem \ref{Balphanot0} for even eigenfunctions $e_\lambda$ also in Case (II).

\subsection{Proof of Theorem \ref{theotors2}}\label{pftheotors2}

In order to prove instability for large $\delta$, we proceed as in Lemma \ref{asymptotics}, with a few changes.
For all $\eps>0$ we define $Z_\eps$ and $\xi_\eps$ by
$$
W_\lambda(t)=\sqrt{\frac{\lambda^4+2\varsigma}{\eps}}\, Z_\eps\left(\sqrt{\frac{\lambda^4+2\varsigma}{\eps}}\, t\right),\qquad
\xi(t)=\xi_\eps\left(\sqrt{\frac{\lambda^4+2\varsigma}{\eps}}\, t\right).
$$
If $W_\lambda$ satisfies \eq{superduffing}, then $Z_\eps$ solves the equation
\neweq{duffingeps2}
\ddot{Z}_\eps(t)+\eps Z_\eps(t)+Z_\eps(t)^3+\frac{2\eps}{\lambda^4+2\varsigma}\int_I\left(\sqrt{\frac{\eps}{\lambda^4+2\varsigma}+\varsigma^2Z_\eps(t)^2e_\lambda(x)^2}-
\sqrt{\frac{\eps}{\lambda^4+2\varsigma}}\right)
e_\lambda(x)dx=0
\endeq
and, according to Definition \ref{defstabb3}, the linear stability of \eq{mostgeneral} depends on the stability of the following Hill equation:
\neweq{hilleps2}
\ddot{\xi}_\eps(t)+\left[\frac{(\kappa^2+2\varsigma)\eps}{\lambda^4+2\varsigma}+(1+2A_{\lambda,\kappa}^2)Z_\eps(t)^2
+\frac{2\eps\varsigma^2Z_\eps(t)}{\sqrt{\lambda^4+2\varsigma}}\int_I\frac{e_\lambda(x)\eta_\kappa(x)^2\ dx}{\sqrt{\eps+\varsigma^2(\lambda^4+2\varsigma)Z_\eps(t)^2e_\lambda(x)^2}}
\right]\xi_\eps(t)=0.
\endeq
By letting $\eps\to0$ we see that \eq{duffingeps2} and \eq{hilleps2} ``converge'' again to the limit problems \eq{limitproblems}
and we conclude as for Lemma \ref{asymptotics}, by invoking \cite[Theorem 3.1]{cazw}.
\par
To finish the proof of Theorem \ref{theotors2}, we apply the Burdina criterion (Lemma \ref{BurdCrit}) to the periodic coefficient of the Hill equation \eqref{hill2}, given by
$p(t)=\kappa^2+ 2\varsigma+(1+2A_{\lambda, \kappa}^2) W_\lambda^2(t) + 2 \varsigma^2 B_\varsigma(W_\lambda(t))$. By Theorem \ref{Balphanot0}, we know that $B_\varsigma (W_\lambda(t))\not\equiv 0$ if and only if $e_\lambda$ is even. Thus, $p(t)$ has period equal to $\tau_\lambda/2$ if $e_\lambda$ is odd, and $\tau_\lambda$ if $e_\lambda$ is even (note that such expressions depend on $\delta$). For $\delta \to 0$ these periods converge, respectively, to $2\pi/\sqrt{\lambda^4+2\varsigma}$ and to $\pi/\sqrt{\lambda^4+2\varsigma}$ while $p(t)$ converges to $\kappa^2 + 2\varsigma$. Hence, recalling \eqref{condizioneBurdina}, for $\delta \to 0$ we are in a stability region if there exists $k \in \mathbb{N}$ such that
$$
k\pi < \frac{2\pi}{\sqrt{\lambda^4+2\varsigma}} \sqrt{\kappa^2+2\varsigma} < (k+1)\pi \qquad \textrm{ if } e_\lambda \textrm{ is even },
$$
and
$$
k\pi < \frac{\pi}{\sqrt{\lambda^4+2\varsigma}} \sqrt{\kappa^2+2\varsigma} < (k+1)\pi \qquad \textrm{ if } e_\lambda \textrm{ is odd },
$$
from which the statement for small $\delta$ follows.

\addtocontents{toc}{\setcounter{tocdepth}{4}}

\section{Final comments and perspectives}\label{finalsect}

\subsection{Conclusions}

In the first part of the present paper, Section \ref{functional}, we have introduced all the basic tools for the analysis of hinged beams
with two piers. The functional and variational setting, as well as the spectral analysis, highlighted some curious phenomena such as the
explicit form of the underlying functional space and the lack of regularity of weak solutions of the related equations. We showed that {\em each
pier reduces by one the dimension of the functional space and inserts a Dirac delta within the equation}. Of particular
relevance for the rest of the paper and for future developments is the behavior of the eigenvalues of \eq{autovsym} as the position of the
piers varies. We gave a full picture of the behavior of the nodes of the eigenfunctions of \eq{autovsym}, a feature that is essential also for
engineers in order to study the oscillations of a bridge, see once again Figure \ref{zeroTNB}. Moreover, we showed that the functional space
(of codimension two) does not allow to view the fourth order eigenvalue problem \eq{autovsym} as the ``squared'' second order
eigenvalue problem \eq{autov2}, a fact that highlights how {\em the stretching energy propagates in a disordered way across the piers}.\par
In Section \ref{evolutionbeam} we reached two main conclusions: the best stability performances are obtained when the piers are placed in the
physical range \eq{physicalrange} and the nonlinear term in \eq{prototipo} that better describes the behavior of actual structures is
$\Vert u\Vert_{L^2}^2 u$. These conclusions were reached after a lengthy and delicate analysis of each nonlinearity and after studying the
stability of the prevailing modes (Definition \ref{prevalente}) from several points of view. At each stage of our analysis, we kept an eye on
what was observed by the witnesses of the TNB collapse \cite{ammvkwoo}. This led us to use a suitable notion of stability (Definition \ref{unstable})
and to focus our attention on the particular class of solutions having a prevailing mode. Also the conclusion on the ``best'' nonlinearity
takes into account the witnesses of the TNB collapse, in particular the fact that {\em the system should have small physiological transfers of
energy between modes}. The nonlinearity $\Vert u\Vert_{L^2}^2 u$ satisfies this feature and has a nonlocal behavior as in real structures.\par
These results for nonlinear beams turned out to be of crucial importance for the analysis of a more realistic and involved model, the
so-called fish-bone model, studied in Section \ref{suspbridge}, which better describes the behavior of a suspension bridge. Indeed, for this model
the theoretical tools have limited strength and, if one wishes to have an idea of the phenomena governing its stability, then one should have
in mind what happens for simpler models where the results are much more precise. We were able to show the impact of the two main nonlinear forces
acting on the deck of a bridge, the restoring forces due to the sustaining cables and to the hangers. We saw that odd and even longitudinal modes respond differently to the slackening of the hangers and that an increment of the elastic constant $\varsigma$ plays against linear torsional stability. \par
Overall, our results enable us to give some suggestions for the designers of future bridges:\par
$\bullet$ design bridges with piers in the physical range \eq{physicalrange};\par
$\bullet$ avoid the use of ``slipping'' decks or cables across the piers since we saw that the stretching energy mixes all the modes and
makes much more difficult the stability analysis;\par
$\bullet$ avoid the use of too elastic hangers.\par
In the next subsections we list some problems left open in the present paper and some possible future developments towards a better understanding
of the instability in suspension bridges.

\subsection{Some open questions}\label{probopen}

\hspace{6pt}
$\bullet$ Definition \ref{prevalente} leaves some arbitrariness in the choice of $\alpha_n$ ($n \neq j$) in \eqref{prevalentec}. In all our experiments, we took $\alpha_n = 0.01$ for $n \neq j$, regardless of the value of $\alpha_j$. It would be interesting to analyze how the energy threshold depends on the choice of the $\alpha_n$'s; our conjecture is that if we increase the initial value of the residual components, then the energy threshold of nonlinear instability will also increase (recall that condition $(ii)$ in \eqref{grande} needs to be satisfied) but the hierarchy between modes will be maintained. Clearly, this \emph{does not} mean that one should have larger residual modes to improve stability! The important feature is instead that the ``weakest'' couple of modes would not depend on how the $\alpha_n$'s are chosen, this being the most important response for the stability analysis; in fact, the detection of this couple may help a designer improve the performances of the structure.\par\smallskip
$\bullet$ It would be interesting to study \eqref{prototipo2} with $\gamma_1 \gamma_3 > 0$. How does the instability diagram look like? Maybe something in between Figures \ref{monodromymnd1} and \ref{parabole}?
\par\smallskip
$\bullet$ The stability analysis performed in Section \ref{optfish} should be completed with the remaining cases. What happens in presence of
multiple torsional
eigenfunctions and what happens for $a<1/2$? Also, torsional modes other than the second should be considered. Finally, let us also point out that a full analysis
of the Hill equation \eq{hilltors} appears mathematically quite challenging. Assume that $W_\lambda$ solves \eq{ODED2}, see \eq{explicitsol}, and consider the equation
$$
\ddot{\xi}(t)+\Big(\kappa^2+\gamma W_\lambda(t)^2\Big)\xi(t)=0
$$
for some $\gamma>0$. The case $\gamma=1$ is fully described by Proposition \ref{toytheo2}, while the case $\gamma\in(1,3)$ is partially described by Theorem
\ref{theotors}. What happens for other values of $\gamma$? In the proof of Theorem \ref{theotors} we have seen that the behavior at infinity of the resonant
tongues is governed by the Cazenave-Weissler intervals \eq{strisceinfinito}, but what about the behavior for $\delta$ small?
\par
\smallskip
$\bullet$ The stability analysis performed in Section \ref{posottimale} leaves many unanswered questions. A complete description of the behavior of
the optimal energy threshold on varying of $\varsigma$ is a challenging task. From what we saw in our numerical experiments, the general trend seems
to be that, for $\varsigma$ increasing, the amplitude threshold of linear stability decreases, while much more investigation is needed to clarify the behavior of nonlinear instability. A full quantitative description of these phenomena is still missing.
\par
\smallskip
$\bullet$ Theorem \ref{Balphanot0} leaves several questions open. First, we remark that if the eigenfunctions $\eta_\kappa$ are different from those in Theorem \ref{autofz2} it may be that $B_\varsigma(w)\not\equiv0$ also if $e_\lambda$ is odd: indeed, consider the second eigenfunction $e_1$ of \eq{autovsym} (which is odd) and an eigenfunction of the kind $\mathcal{D}_\kappa+\mathcal{P}_\kappa$ of \eq{autov2} (which is neither odd nor even). Second, Theorem \ref{Balphanot0} {\em does not} hold if
$\eta_\kappa=\mathcal{D}_\kappa$ or $\eta_\kappa=\mathcal{P}_\kappa$, see Theorem \ref{autofz2}. To see this, take $a=1/5$ and
$e_\lambda(x)=\eta_\kappa(x)=\cos(5x/2)$ so that $\lambda=\kappa=5/2$. With this choice, we have $B_\varsigma(w)\equiv0$. We have here exploited the fact
that $e_\lambda$ has ``an integer number of periods'' (one in this case) on the side span. In fact, if $a\not\in\mathbb{Q}$ this cannot happen and we
believe that Theorem \ref{Balphanot0} holds also for $\eta_\kappa$ not vanishing on the side spans. We leave these as open problems.

\subsection{Future developments}

\hskip10pt $\bullet$ In this paper we have only focused on the structural aspect of the instability problem. The next step is to take into account
the fluid-structure interaction \cite{aero,lasiecka} and to insert into the model suitable damping effects.
We do not expect the aerodynamics to modify the response on the optimal position of the piers but it will probably
modify the {\em quantitative response} in terms of energy thresholds.\par\smallskip

$\bullet$ Motivated by the design of most bridges, we have mainly considered the case of symmetric side spans ($b=a$). The main advantage
of this restriction is that one can deal with even and odd eigenfunctions, see the discussion in Section \ref{finitod}. But some suspension bridges, such as
the three Kurushima Bridges (see \cite[Figure 2.4.6]{jurado}) have asymmetric side spans. In this respect, let us also recall a forgotten suggestion from the 19th
century: while commenting the collapse of the Brighton Chain Pier (1836), Russell \cite{russell} claims that {\em the remedies I have proposed, are those by which
such destructive vibrations would have been rendered impossible}. His remedies were to alter the place of the cross bars and to put stays below the bridge which
should be put {\em at distances not perfectly equal}. His scope was to break symmetry in the longitudinal oscillating modes of the deck. Therefore,
the optimal position of the piers should also be discussed in the asymmetric framework, it is not even clear if symmetry yields better stability performances:
a full analysis for the case $b\neq a$ and the comparison with symmetric side spans appears quite important and challenging.\par\smallskip

$\bullet$ Several suspension bridges, such as the San Francisco-Oakland Bay Bridge (see \cite[Fig. 15.10]{podolny}), have more than two intermediate piers.
Some of the results in the present paper may be extended to the case of multiple intermediate piers. With the same proof as for Theorem \ref{VI}, one can show
that the subspace of $H^2\cap H^1_0(I)$ with $n$ interior vanishing constraints has codimension $n$. Also Theorem \ref{regular} continues to hold with some
obvious changes. However, a different and general procedure seems necessary to prove smoothness of a solution, since problem \eq{wrong3} has exactly the same number of
constraints (four) as the order of the ODE.
Furthermore, the spectral analysis carried out in Section \ref{functional}, including the discussion about the possible existence of positive eigenfunctions (see Figure \ref{positiva}) appears much longer and possibly more difficult than in the case of only two piers. In this respect, anyway, taking into account the definition of $i(e_\lambda)$ in Section \ref{functional}, in presence of $n$ piers it would be reasonable that the only positive eigenfunction is the $(n+1)$-th. But the main difficulty is certainly to determine the optimal length of the secondary spans in order
to minimize the dangerous energy exchanges within the main span.\par\smallskip

$\bullet$ The opportunity to introduce full plate models should be evaluated. From Ventsel-Krauthammer \cite[$\S$ 1.1]{ventsel} we learn that plates may be
classified according to the ratio between width and thickness. By taking these values from the Report \cite{ammvkwoo}, we deduce that {\em the collapsed TNB may
be considered as a thin plate}, see \cite[$\S$ 5.2.1]{bookgaz} for a detailed discussion. The behavior of rectangular thin plates subject to a variety of
boundary conditions is studied in \cite{braess,grunau,gruswe,sweers}. The role and the optimal position of the intermediate piers could be studied
also for thin plate models since they appear more appropriate for decks with a large width.
\par
\smallskip
$\bullet$ The nonlocal behavior of structures such as suspension bridges may intervene also in the differential operators and in other terms, for instance dampings. The stability analysis, both for beams and degenerate plates, appears quite challenging in general frameworks. A possibility would be to compare the results in the present paper with results obtained for different beam and plate models with Woinowsky-Krieger-type nonlinearity, such as the ones considered in \cite{autpucsal,pucsal}. For these equations, a full analysis of the nonlinear stability appears extremely difficult and this is a further reason why one should always keep an eye on linear stability.
\par\smallskip
$\bullet$ In plates without piers, the stretching effects can be controlled, see \cite{bonedegaz}. As already mentioned, in presence of the piers the stability analysis is much more complicated, see Section \ref{nonlinevol}. Therefore, a possible suggestion is to design the entire structure in such a way that the
stretching effects in a point remain confined to the span to which it belongs. This means that, instead of \eq{beamstretch}, one should
study the stability for the nonlinear nonlocal equation
$$
u_{tt}+u_{xxxx}-\Big(\|u_x\|^2_{L^2(I_-)}\chi_-+\|u_x\|^2_{L^2(I_0)}\chi_0+\|u_x\|^2_{L^2(I_+)}\chi_+\Big)u_{xx}=0\qquad x \in I,\quad
t>0,
$$
that is, an equation where the stretching effects act separately on each span. Here, $\chi_-$, $\chi_0$, $\chi_+$ denote the characteristic
functions of $I_-$, $I_0$, $I_+$.

\bigskip
\textbf{Acknowledgement}. The Authors are grateful to Alberto Farina for an important remark about Theorem \ref{regular} and to Clelia Marchionna for
raising their attention on the papers \cite{gesz,goldberg} which led to an improvement of Proposition \ref{toytheo2}.\par
Both the authors are supported by the Gruppo Nazionale per l'Analisi Matematica, la Probabilit\`a e le loro Applicazioni (GNAMPA) of the
Istituto Nazionale di Alta Matematica (INdAM).
The second Author is also partially supported by the PRIN project {\em Equazioni alle derivate parziali di tipo ellittico e parabolico: aspetti geometrici, disuguaglianze collegate, e applicazioni.}

\end{document}